\documentclass[11pt]{amsart}

\usepackage[text={428pt,660pt},centering]{geometry}
\usepackage{esint,amssymb}
\usepackage{mathtools}
\usepackage[colorlinks=true, pdfstartview=FitV, linkcolor=blue, citecolor=blue, urlcolor=blue,pagebackref=false]{hyperref}
\usepackage{microtype}
\usepackage{dsfont}

\makeatletter
\newtheorem*{rep@theorem}{\rep@title}
\newcommand{\newreptheorem}[2]{%
\newenvironment{rep#1}[1]{%
 \def\rep@title{#2 \ref{##1}}%
 \begin{rep@theorem}}%
 {\end{rep@theorem}}}
\makeatother

\newtheorem{proposition}{Proposition}
\newtheorem{theorem}[proposition]{Theorem}
\newreptheorem{theorem}{Theorem}
\newtheorem{lemma}[proposition]{Lemma}

\theoremstyle{remark}
\newtheorem{remark}[proposition]{Remark}

\theoremstyle{definition}
\newtheorem{definition}[proposition]{Definition}

\numberwithin{equation}{section}
\numberwithin{proposition}{section}
\numberwithin{figure}{section}
\numberwithin{table}{section}

\newcommand{\Z}{\mathbb{Z}}
\newcommand{\N}{\mathbb{N}}

\newcommand{\R}{\mathbb{R}}
\newcommand{\inte}[1]{%
  {\kern0pt#1}^{\mathrm{o}}%
}

\newcommand{\E}{\mathbb{E}}
\renewcommand{\P}{\mathbb{P}}

\newcommand{\Zd}{{\mathbb{Z}^d}}
\newcommand{\Rd}{{\mathbb{R}^d}}
\newcommand{\B}{\mathbb{B}}

\newcommand{\ep}{\varepsilon}
\newcommand{\e}{\mathfrak{e}}

\renewcommand{\subset}{\subseteq}

 \newcommand{\cu}{\square}

\DeclareMathOperator{\dist}{dist}
\DeclareMathOperator{\diam}{diam}
\DeclareMathOperator{\tr}{tr}
\DeclareMathOperator{\divg}{div}
\DeclareMathOperator{\supp}{supp}

\newcommand{\var}{\mathbb{V}\mathrm{ar}}

\DeclareMathOperator{\Leb}{Leb}

\DeclareMathOperator{\im}{im}

\renewcommand{\div}{\divg}
\newcommand{\indc}{\mathds{1}}

\begin{document}

\keywords{$\nabla \phi$ model, renormalization}
\subjclass[2010]{}
\date{\today}

\title{Quantitative homogenization of the disordered $\nabla \phi$ model}

\begin{abstract}
We study the $\nabla \phi$ model with uniformly convex Hamiltonian $\mathcal{H} (\phi) := \sum  V(\nabla \phi)$ and prove a quantitative rate of convergence for the finite-volume surface tension as well as a quantitative rate estimate for the $L^2$-norm for the field subject to affine boundary condition. One of our motivations is to develop a new toolbox for studying this problem that does not rely on the Helffer-Sj\"ostrand representation. Instead, we make use of the variational formulation of the partition function, the notion of displacement convexity from the theory of optimal transport, and the recently developed theory of quantitative stochastic homogenization.
\end{abstract}

\author[P. Dario]{Paul Dario}
\address[P. Dario]{CEREMADE, Universit\'e Paris-Dauphine, PSL Research University, Paris, France \& DMA, Ecole normale sup\'erieure, PSL Research University, Paris, France}
\email{paul.dario@ens.fr}

\keywords{}
\subjclass[2010]{}
\date{\today}

\maketitle

\setcounter{tocdepth}{1}
\tableofcontents

\section{Introduction} \label{section1}
Many physical phenomena exhibit a transition between two pure phases, especially at low temperature. The first mathematical model to understand the macroscopic shape of the interface separating the two phases was introduced by Wulff in 1901 in~\cite{wul1901} to describe the shape of a crystal at equilibrium: it characterizes the interfaces as minimizers of the Wulff functional, defined by, for a subset $E\subseteq \Rd$,
\begin{equation} \label{Wulfffunctional}
W \left( E \right) := \int_{\partial E} \sigma \left( \mathbf{n}(x) \right) \, d x,
\end{equation}
where $\mathbf{n}$ is the outward normal to $\partial E$ at $x$, $\sigma$ is a surface tension, under a volume constraint $\mathrm{vol}(E) = v$. The minimizer of the Wulff functional is called the Wulff shape. From a mathematical point of view, the interfaces are macroscopic objects, and one would like to describe them using models from statistical mechanics which are defined on a microscopic level. Many important results in this direction were obtained on various models in the 90s; in~\cite{ACC90}, Alexander, Chayes and Chayes derived a Wulff construction for the two dimensional supercritical Bernoulli bond percolation. In the monograph~\cite{DKS92}, Dobrushin, Koteck\'y and Shlosman, studied the two dimensional ferromagnetic Ising model at low temperature with periodic boundary conditions. These results were later extended to every temperature below the critical one, and we refer to the works of Ioffe~\cite{Io94, Io95}, Schonmann, Shlosman~\cite{SS95} and Pfister, Velenik ~\cite{PfVe97} and by Ioffe and Schonmann in~\cite{IS98}. In dimension $3$, Cerf proved in~\cite{Ce00} a Wulff construction for the supercritical Bernoulli bound percolation. Bodineau in~\cite{Bo99} proved a similar result for the Ising model in any dimension $d \geq 3$ at low temperature. Cerf and Pisztora in~\cite{CePi00} proved a Wulff construction for Ising in dimension larger than $3$ for temperatures below a limit of slab-thresholds.

In this article, we consider the discrete Ginzburg-Landau or $\nabla \phi$-model in dimension $d \geq 2$, which is a model of stochastic interface defined as follows. We assume that the interface is a discrete object which has only one degree of freedom; it is represented by a set of the form $\left\{ (x , \phi(x)) \,:\, x \in \Zd \right\} \subseteq \Zd \times \R$, where $\phi$ is a map from $\Zd$ to $\R$ which encodes the height of the interface. We associate to a configuration $\phi$ an energy computed through the Hamiltonian,
\begin{equation*}
H \left(  \phi \right) := \sum_{|x - y| = 1} V \left( \phi(x) - \phi(y) \right),
\end{equation*}
where $V : \R \rightarrow \R$ is an elastic potential satisfying the properties 
\begin{enumerate}
\item[(i)] $V$ is even: $V(x) = V(-x)$ for each $x \in \R$,
\item[(ii)] $V$ is uniformly convex: there exists an ellipticity parameter $\lambda \in (0,1] $ such that for each $x , y \in \R$,
\begin{equation} \label{chaptire0.elaunifconv}
\lambda |x - y |^2 \leq V(x) + V(y) - 2 V\left(  \frac{x+y}{2} \right) \leq \frac1{\lambda} |x - y |^2.
\end{equation}
\end{enumerate}
To this system, one can associate a Langevin dynamics, given by the stochastic differential equation
\begin{equation*}
d \phi_t (x) = - \sum_{|y - x| = 1} V' \left( \phi_t(x) - \phi_t (y) \right) dt + \sqrt{2} d B_t \left( x \right), ~ x \in \Zd,
\end{equation*}
where $\left( B_t \left( x \right) \right)_{x \in \Zd}$ is a family of independent normalized Brownian motions. This dynamics has an invariant Gibbs measure formally given by  the formula
\begin{equation} \label{intro.formalgibbs}
\frac{1}{Z} \exp \left( - H \left( \phi \right) \right) \prod_x d \phi(x), 
\end{equation}
where the parameter $Z$ is the partition function, chosen such that the measure~\eqref{intro.formalgibbs} is a probability measure.

A typical result one wishes to prove is that a properly rescaled version of the interface approaches, over large scales, a deterministic shape  and to characterize this deterministic object in the spirit of a Wulff construction (see~\cite{BAD96, FS97, DGI00}). A quantity of interest is thus the surface tension of the model which is defined by the following procedure. We first define the finite volume surface tension according to the formula, for every $p \in \Rd$,
\begin{equation*}
\nu \left( Q_r , p \right) := - \frac{1}{r^d} \log \left( \int e^{-H(\phi)} \prod_{x \in Q_r} d \phi(x) \prod_{x \in \partial Q_r} \delta_{ p \cdot x} \left( d \phi (x) \right) \right).
\end{equation*}
In~\cite{FS97}, Funaki and Spohn proved that as the size of the cube tends to infinity, the finite-volume surface tension converges, i.e.,
\begin{equation} \label{qualitatveratevu}
\nu \left( Q_r, p \right)  \underset{ r  \rightarrow \infty}{\longrightarrow} \bar \nu \left( p \right).
\end{equation}
The limit $\bar \nu$ is the surface tension of the model. The proof relies on a subadditivity argument and is thus qualitative. The main result of this article is to obtain an algebraic rate of convergence for the finite-volume surface tension and is stated below.

\begin{theorem}[Quantitative convergence of the finite-volume surface tension] \label{QuantitativeconvergencetothGibbsstate}
There exist a constant $C := C(d , \lambda) < \infty$ and an exponent $\alpha := \alpha(d , \lambda) > 0$ such that for each $p \in \Rd$,
\begin{equation} \label{quantconvnu}
\left| \nu (Q_r, p) - \bar \nu (p) \right| \leq C r^{- \alpha} (1 + |p|^2).
\end{equation}
\end{theorem}

The second main result of this article is to deduce a quantitative $L^2$ sublinearity estimate for the field $\phi$ distributed according to the Gibbs measure~\eqref{intro.formalgibbs}, in a cube with affine boundary conditions.

\begin{theorem}[$L^2$ contraction of the Gibbs measure] \label{conhvergencetoaffine}
For each $p\in \Rd$, we let $l_p$ be the affine function of slope $p$, i.e., for each $x \in \Zd$, $l_p(x) = p \cdot x$. We let $\phi_{r,p} : Q_r \mapsto \R$ be a random variable distributed according to the Gibbs measure
\begin{equation*}
\P\left( d \phi \right) := \frac{1}{Z} e^{-H(l_p + \phi)} \prod_{x \in Q_r} d \phi(x) \prod_{x \in \partial Q_r} \delta_{0} \left( d \phi (x) \right),
\end{equation*}
where $Z$ is the partition function chosen such that $\P$ is a probability measure. Then there exist a constant $C := C(d , \lambda) < \infty$ and an exponent $\alpha := \alpha (d, \lambda) > 0$ such that
\begin{equation*}
\frac{1}{r^2}\E \left[  \frac1{r^d} \sum_{x \in Q_r}   \left| \phi_{r,p}(x) \right|^2  \right] \leq C r^{- \alpha} \left( 1 + |p|^2\right).
\end{equation*}
\end{theorem}

These theorems are the first quantitative results for convergence of the surface tension of the $\nabla \phi$-model and should be seen as a first step to develop a quantitative theory on this model. A first reasonable objective would be to extend these results to more general conditions, instead of the affine boundary condition presented in Theorem~\ref{conhvergencetoaffine}. Such a result would prove that the properly rescaled version of the interface $\phi$ converges quantitatively to a deterministic interface, which is a critical point of the Wulff functional involving the surface tension $\bar \nu$.

The result obtained is suboptimal since we only obtain an algebraic rate of convergence for some small exponent $\alpha > 0$. In the specific case of the Gaussian free field, i.e., when $V_e(x) = x^2$, the finite volume surface tension can be computed explicitly and the convergence takes place at the following rate
\begin{equation*}
\left| \nu_{\mathrm{GFF}} (Q_r, p) - \bar \nu_{\mathrm{GFF}}  (p) \right| \sim C r^{-1}.
\end{equation*}
We expect that this rate should also be the optimal rate in the case of uniformly convex potentials considered in this article.

The $\nabla \phi$-model and its large-scale properties have already been studied in several works. A common tool to study this model is the Helffer-Sj\"ostrand PDE representation which originates in the work of Helffer and Sj\"ostrand~\cite{helsjo}. Naddaf and Spencer in~\cite{NS2} proved a central limit theorem for this model by homogenizing the infinite dimensional elliptic PDE obtained from the Helffer-Sj\"ostrand representation. Funaki and Spohn in~\cite{FS97} studied the dynamics of this model and proved that the suitably rescaled version of the dynamic field $\phi_t$ converges toward a deterministic field which is solution to a parabolic PDE involving the surface tension. Deuschel Giacomin and Ioffe in~\cite{DGI00} established the large scale $L^2$ convergence of the interface to some deterministic function, which can be characterized as a minimizer of a Wulff functional, as well as a large deviation principle. These result were later extended by Funaki and Sakagawa in~\cite{FS04}. In 2001, Giacomin, Olla and Spohn established in~\cite{GOS01} a central limit theorem for the Langevin dynamics associated to this model. More recently Miller in~\cite{Mi} proved a central limit theorem for the fluctuation field around a macroscopic tilt.

The Helffer-Sj\"ostrand representation is a very powerful tool, but may also face some limitations. The PDE operator arising in this representation contains a divergence-form part whose coefficients are given by $V''$. In case when $V''$ is singular, or of a varying sign, then it is rather unclear how to proceed (see however \cite{CDM,BS,CD}). Besides the specific results to be proved in this paper, we are interested in developing new tools to study the $\nabla \phi$ model that completely forego any reference to the Helffer-Sj\"ostrand representation. We rely instead on the variational formulation of the free energy, and of the displacement convexity of the associated functional; to the best of our knowledge, it is the first time that tools from optimal transport are being used to study this model.  

The mechanism by which we obtain a rate of convergence is inspired by recent developments in the homogenization of divergence-form operators with random coefficients. The first results in this context date back to the early 1980s, with the results of Kozlov~\cite{Koz79}, Papanicolaou-Varadhan~\cite{PV1} and Yurinski\u{\i}~\cite{Y22} who were able to prove qualitative homogenization for linear elliptic equations under very general assumptions on the coefficient field. These results were later extended by Dal Maso and Modica in~\cite{DM1,DM2} to the nonlinear setting. Obtaining quantitative rates of convergence has been the subject of much recent study over the past few years. Some notable progress were achieved by Gloria, Neukamm and Otto~\cite{GO1, GO2,GNO} and by Armstrong, Kuusi, Mourrat and Smart~\cite{ AKM1, AKM2, armstrong2017quantitative, AS, AM}.

While most of the theory developed to understand stochastic homogenization focuses on linear elliptic equations, the closest analogy with the $\nabla \phi$ interface model is the stochastic homogenization of nonlinear equations. In this setting, the results are more sparse: one can mention the work of Armstrong, Mourrat and Smart~\cite{AM,AS} who quantified the works of Dal Maso and Modica~\cite{DM1,DM2}. More recently, Arsmtrong, Ferguson and Kuusi~\cite{AFK18} were able to adapt part of the theory developed in the linear setting to the nonlinear setting.  

While the theory is usually presented in the case of uniformly elliptic environment, the extension of the theory to the setting degenerate or perforated environment has been a subject of attention (see for instance~\cite{LNO, bella2018liouville, GiMo18, GHV18, AD2}). We hope that ideas and strategies presented in these works could be useful to obtain information on some $\nabla \phi$-models with non-uniformly elliptic potential.


\subsection{Notations and assumptions}
\subsubsection{Notations for the lattice and cubes} In dimension $d \geq 2$, let $\Zd$ be the standard $d$-dimensional hypercubic lattice, $\B_d := \left\{ (x,y)~:~ x,y \in \Zd, |x-y| = 1 \right\}$ the set of unoriented nearest neighbors, or \emph{bonds}, and $E_d$ be the set of oriented nearest neighbors, or \emph{edges}. We denote the canonical basis of $\Rd$ by $\{ \e_1 , \ldots , \e_d \}$. For $x, y \in \Zd$, we write $x\sim y$ if $x$ and $y$ are nearest neighbors. For $x \in \Zd$ and $r > 0$, we denote by $B(x,r) \subseteq \Zd$ the discrete ball of center $x$ and of radius $r$. We usually denote a generic bond by $e$. For a given subset $U$ of $\Zd$, we denote by $\B_d(U)$ the bonds of $U$, i.e,
\begin{equation*}
\B_d (U) := \left\{ (x,y) \in \B_d ~:~ x \in U, \, y \in U  ~\mbox{and}~ x \sim y \right\}.
\end{equation*}
If we wish to talk about edges, we use an arrow to distinguish them from the bonds. We denote by $\overset{\rightarrow}{e}$ a generic edge. Similarly, we denote by $E_d$ the edges of $U$,
\begin{equation*}
E_d(U) := \left\{ \overrightarrow{xy} \in E_d ~:~ x \in U, \, y \in U  ~\mbox{and}~ x \sim y \right\}.
\end{equation*}
We denote by $\partial U$ the discrete boundary of $U$, defined by
\begin{equation*}
\partial U := \left\{ x \in U : \exists y \in \Zd, ~ y\sim x ~\mbox{and}~ y \notin U \right\}
\end{equation*}
and by $\inte U$ the discrete interior of $U$,
\begin{equation*}
\inte U := U \setminus \partial U.
\end{equation*} 
We also denote by $\left| U \right|$ the cardinality of $U$, we refer to this quantity as the (discrete) volume of $U$. For $N \in \N$, we write $N \Zd$ to refer to the set $\{N x \, :\, x \in \Zd \} \subseteq \Zd$.
A cube of $\Zd$ is a set of the form
\begin{equation*}
\Zd \cap \left( z + [0, N]^d \right), ~ z \in \Zd, ~ N \in \N.
\end{equation*}
We define the size of a cube given in the previous display above to be the integer $N +1$. For $n \in \N$, we denote by $\cu_n$ the discrete triadic cube of size $3^n$, 
\begin{equation*}
\cu_n := \left( - \frac{3^n}{2}, \frac{3^n}{2} \right)^d \cap \Zd.
\end{equation*}
We say that a cube $\cu$ is a triadic cube if it can be written
\begin{equation*}
\cu = (z + \cu_n) , ~ \mbox{for some} ~ n \in \N,~ \mbox{and}~z \in 3^n \Zd.
\end{equation*}
Note that two triadic cubes are either disjoint or included in one another. Moreover for each $n \in \N$, the family of triadic cubes of size $3^n$ forms a partition of $\Zd$. A caveat must be mentioned here, the family of triadic cubes $\left( z + \cu_n \right)_{z \in 3^n \Zd}$ forms a partition of $\Zd$ but the family of edges $\left( \B \left( z + \cu_n \right) \right)_{z \in 3^n \Zd}$ does not form a partition of $\B_d$, indeed the edges connecting two triadic cubes are missing, i.e., the edges of the set
\begin{equation*}
\left\{ (x , y) \in \B_d ~:~ \exists z \in 3^n \Zd, \, x \in (z + \cu_n) ~\mbox{and}~ y \notin (z + \cu_n) \right\}.
\end{equation*}
We mention that the volume of a discrete triadic cube of the form $z + \cu_n$ is $3^{dn}$.

\smallskip

Given two integers $m,n \in \N$ with $m < n$, we denote by
\begin{equation} \label{def.mathcalzmn}
\mathcal{Z}_{m,n} := 3^m \Zd \cap \cu_n,
\end{equation}
we also frequently use the shortcut notation
\begin{equation*}
\mathcal{Z}_{n} := \mathcal{Z}_{n,n+1} = 3^n \Zd \cap \cu_{n+1}.
\end{equation*}
These sets have the property that $\left( z + \cu_m \right)_{z \in \mathcal{Z}_{m,n}}$ is a partition of $\cu_m$. In particular $\left( z + \cu_n \right)_{z \in \mathcal{Z}_{n}}$ is a partition of $\cu_{n+1}$.

\subsubsection{Notations for functions} For a bounded subset $U \subseteq \Zd$, and a function $\phi : U \rightarrow \R$, we denote by $(\phi)_U$ its mean defined by the formula
\begin{equation*}
(\phi)_U := \frac1{|U|}\sum_{x \in U} \phi(x).
\end{equation*}
We let $h^1_0 (U)$ and $\mathring h^1 (U)$ be the set of functions from $U$ to $\R$ with value zero on the boundary of $U$ and mean zero respectively, i.e.,
\begin{equation*}
h^1_0 (U) := \left\{ \phi : U \rightarrow \R ~:~ u = 0~ \mbox{on}~ \partial U \right\}
\end{equation*}
and
\begin{equation*}
\mathring h^1 (U) := \left\{ \psi : U \rightarrow \R ~:~ \left( \psi \right)_{U} = 0 \right\}.
\end{equation*} 
These spaces are finite dimensional and their dimension is given by the formulas
\begin{equation*}
\dim h^1_0 (U) = \left| U \setminus \partial U \right| ~\mbox{and}~ \dim \mathring h^1 (U) = |U| -1. 
\end{equation*}
We sometimes need to restrict functions, to this end we introduce the following notation, for any subsets $U,V \subseteq \Zd$ satisfying $V \subseteq U$, and any function $\phi : U \rightarrow \R$, we denote by $\phi_{|V}$ the restriction of $\phi$ to $V$.
A vector field $G$ on $U$ is a function 
\begin{equation*}
G : E_d (U) \rightarrow \R
\end{equation*}
which is antisymmetric, that is, $G\left( \overrightarrow{xy}\right) = - G\left(\overrightarrow{yx}\right)$ for each $x,y \in E_d(U)$. Given a function $ \phi : U \rightarrow \R$, we define its gradient by, for each $\overrightarrow e = \overrightarrow{xy} \in E_d (U)$,
\begin{equation*}
\nabla \phi \left(\overrightarrow e\right) = \phi(y) - \phi(x).
\end{equation*}
The divergence of a vector field $G$ is the function from $U$ to $\R$ defined by, for each $x \in U$,
\begin{equation*}
\div G (x) = \sum_{y \in U, y \sim x} G\left( \overrightarrow{xy} \right).
\end{equation*}
We also define the discrete Laplacian $\Delta$ of a function $\phi : U \rightarrow \R$ by the formula, for each $x \in U$,
\begin{equation*}
\Delta \phi(x) = \sum_{y \in U, y \sim x} \left( \phi(y) - \phi(x)  \right).
\end{equation*}
For $p \in \Rd$, we also denote by $p$ the constant vector field given by
\begin{equation} \label{constantvectfield}
p(x,y) :=p \cdot (x - y).
\end{equation}
Given two vector fields $F$ and $G$, we define their product to be the function defined on the set of unoriented edges by
\begin{equation*}
F \cdot G (x,y) = F(\overrightarrow{xy})  G(\overrightarrow{xy}).
\end{equation*}
This notation is frequently applied when $F$ is a constant vector $q$ and when $G$ is the gradient of a function $\nabla \psi$, so we write, for each $(x,y) \in \B_d$,
\begin{equation*}
q \cdot \nabla \psi (x,y) = q (x,y)  \nabla \psi (x,y) .
\end{equation*}
We also often use the shortcut notation
\begin{equation*}
\sum_{e \subseteq U} ~\mbox{to mean}~ \sum_{e \in \B_d (U)}.
\end{equation*}
If one assumes additionally that $U$ is bounded, then for any vector field $F : E_d (U) \rightarrow \R$, we denote by $\left\langle F\right\rangle_U$ the unique vector in $\Rd$ such that, for each $p \in \Rd$
\begin{equation*}
p \cdot \left\langle F\right\rangle_U = \frac1{|U|} \sum_{e \subseteq U} p \cdot F (e).
\end{equation*}

\subsubsection{Notations for vector spaces and scalar products} 
Let $V$ be a finite dimensional real vector space equipped with a scalar product $(\cdot , \cdot)_V$, this space can be endowed with a canonical Lebesgue measure denoted by $\Leb_V$. This measure is simply denoted by $dx$ when we integrate on $V$, i.e., we write, for any measurable, integrable or non-negative, function $f: V \rightarrow \R$, 
\begin{equation} \label{formuladxLebv}
\int_V f(x) \, dx ~\mbox{to mean}~ \int_V  f(x) \, \Leb_V (dx).
\end{equation}
For any linear subspace $H \subseteq V$, we denote by $H^{\perp}$ the orthogonal complement of $H$. Given $H, K \subseteq V$, we use the notation
\begin{equation*}
V = H \overset{\perp}{\oplus} K ~\mbox{if}~ V = H \oplus K ~\mbox{and}~ \forall (h,k) \in H \times K, \, (h ,k)_V = 0.
\end{equation*}
Note that from the scalar product on $V$, one can define scalar products on $H$ and $H^\perp$ naturally by restricting the scalar product on $V$ to these spaces. Consequently, the spaces $H$ and $H^\perp$ are equipped with Lebesgue measures denoted by $\Leb_H$ and $\Leb_{H^\perp}$. These measures are related to the Lebesgue measure on $V$ by the relation
\begin{equation} \label{decomplebmeasure}
\Leb_V = \Leb_H \otimes \Leb_{H^\perp},
\end{equation}
where the notation $\otimes$ is used to denote the standard product of measures. Note that we used in the previous notation the equality $V = H\overset{\perp}{\oplus} H^{\perp}$ to obtain a canonical isomorphism between the spaces $V$ (on which $\Leb_V$ is defined), and the space $H \times H^\perp$ (on which  $\Leb_H \otimes \Leb_{H^\perp}$ is defined).

Given a bounded subset $U \subseteq \Zd$, we equip any linear space $V$ of functions from $U$ to $\R$ with the standard $L^2$ scalar product, i.e., for any $\phi, \psi \in V$, we define
\begin{equation*}
(\phi , \psi)_V := \sum_{x \in U} \phi(x) \psi(x).
\end{equation*}
This in particular applies to the spaces $h^1_0 (U)$ and $\mathring h^1 (U)$. From now on, we consider that these spaces are equipped with a scalar product and consequently with Lebesgue measure denoted by $\Leb_{h^1_0 (U)}$ and $\Leb_{\mathring h^1 (U)}$, or simply by $dx$ when we use the notation convention~\eqref{formuladxLebv}.

\subsubsection{Notations for measures and random variables}
For any finite dimensional real vector space $V$, we denote by $\mathcal{P}(V)$ the set of probability measures on $V$ equipped with its Borel $\sigma$-algebra denoted by $\mathcal{B}(V)$. For a pair of finite dimensional real vector spaces $V$ and $W$, and a measure $\pi \in \mathcal{P}(V \times W)$, the first marginal of $\pi$ is the probability measure $\mu \in \mathcal{P}(V)$ defined by, for each $A \in \mathcal{B}(V)$,
\begin{equation*}
\mu(A) := \pi \left( A \times W\right),
\end{equation*}
we similarly define the second marginal as a measure in $\mathcal{P}(W)$. Given two probability measures $\mu \in \mathcal{P}(V)$ and $\nu \in \mathcal{P}(W)$, we denote by $\Pi(\mu,\nu)$ the set of probability measures of $\P \left( V \times W \right)$ whose first marginal is $\mu$ and second marginal is $\nu$, i.e.,
\begin{multline*}
\Pi (\mu , \nu) := \left\{ \pi \in \mathcal{P}(V \times W) \, : \, \forall (A , B) \in \left( \mathcal{B}(V),  \mathcal{B}(W) \right), ~ \pi \left( A \times W \right) = \mu (A) \right. \\ \left.  \mbox{and}~ \pi \left( V \times B \right) = \nu (B)  \right\}.
\end{multline*}
We define a coupling between two probability measures $\mu \in  \mathcal{P}(V)$ and $\nu \in  \mathcal{P}(W)$ to be a measure in $\Pi (\mu , \nu)$. For a generic random variable $X$, we denote by $\P_X$ its law. Given two random variables $X$ and $Y$, a coupling between $X$ and $Y$ is a random variable whose law belongs to $\Pi(\P_X , \P_Y)$.

We mention that in this article we do not assume that there is an underlying probability space $(\Omega, \mathcal{F}, \P)$ on which all the random variables are defined. We use the random variables as proxy for their laws to simplify the notations. A consequence of this is that  given two random variables $X$ and $Y$, we have to be careful to first define a coupling between $X$ and $Y$ if one wants to work with the random variables $X + Y$, $XY$ etc.

Given a measurable space $(X , \mathcal{F})$ and two $\sigma$-finite measures $\mu$ and $\nu$ on $X$, we write $\mu \ll \nu$ to mean that $\mu$ is absolutely continuous with respect to $\nu$, and denote by $\frac{d \mu}{d \nu}$ the Radon-Nikodym derivative of $\mu$ with respect to $\nu.$ If we are given a second measurable space $(Y,\mathcal{F}')$ and a measurable map $T : X \rightarrow Y$, we denote by $T_*\mu$ the pushforward of the measure $\mu$ by the map $T$.

\subsubsection{Notations for the $\nabla \phi$ model}
We consider a family of functions $(V_i)_{i = 1 , \ldots , d} \in C^2 (\R)$ satisfying the following assumptions, for each $i \in \{1 , \ldots, d \}$,
\begin{enumerate}
\item \textit{Symmetry:} for each $x \in \R$, $V_i(x) = V_i(-x)$;
\item \textit{Uniform convexity:} there exists $\lambda \in (0,1)$ such that $\lambda < V_i'' < \frac 1\lambda$;
\item \textit{Normalization:} The value of $V_i$ at $0$ is fixed: $V_i(0) = 0$.
\end{enumerate} 
Assumption (2) implies, for each $p_1, p_2 \in \R$,
\begin{equation} \label{unifconvVi}
\lambda \left| p_1 - p_2 \right|^2 \leq V_i \left(p_1\right) +  V_i\left(p_2\right) - 2 V_i\left(\frac{p_1+ p_2}2\right) \leq \frac1\lambda \left| p_1 - p_2 \right|^2.
\end{equation}
From the first and second assumptions, we see that the functions $V_i$ have a unique minimum achieved in $0$.
The third assumption is not necessary and can be easily removed, but thanks to this assumption we have the convenient inequality, for each $x \in \R$,
\begin{equation*}
\lambda |x|^2 \leq V_i(x) \leq \frac{1}{\lambda} |x|^2.
\end{equation*}
For each bond $e \in \B_d$, we define $V_e := V_i$ where $ i $ is the unique integer in $ \{1 , \ldots, d \}$ such that $ e$ can be written $(x , x+\e_i) $ for some $ x \in \Zd.$ We then define the partition function, for each bounded subset $U \subseteq \Zd$ and for each $p \in \Rd$,
\begin{equation} \label{partfctnu}
Z_p \left(U \right) := \int_{h^1_0(U)} \exp \left(  - \sum_{e \subset U} V_e( p(e) + \nabla \phi(e)) \right) \, d\phi,
\end{equation}
where we recall that the notation $p(e)$ denotes the constant vector field introduced in~\eqref{constantvectfield} and $d \phi$ stands for the Lebesgue measure on $h^1_0(U)$. Note that thanks to the symmetry of the functions $V_i$, even if the vector field $p + \nabla \phi$ is defined for edges, the quantity $V_e( p(e) + \nabla \phi(e)) )$ can be defined for bonds. From this, one can define the finite-volume surface tension
\begin{equation} \label{expformnu}
\nu (U , p) := - \frac1{|U|} \ln Z_p \left(U \right)
\end{equation}
and the probability measure on $h^1_0(U)$
\begin{equation*}
\P_{U,p}(d \phi) :=  \frac{\exp \left(  - \sum_{e \subset U} V_e( p(e) + \nabla \phi(e)) \right) d \phi}{Z_p \left(U \right)}.
\end{equation*}
We also denote by $\phi_{U,p}$ a random variable of law $\P_{U,p}$. These objects are frequently used with triadic cubes, we thus define the shortcut notations, for $n \in \N$,
\begin{equation*}
\P_{n,p} : = \P_{\cu_n,p},
\end{equation*}
and by $\phi_{n,p}$ a random variable of law $\P_{n,p}$. We also define the quantity, for each $q \in \Rd$,
\begin{equation} \label{partfctnustar}
Z_q^* \left(U \right) := \int_{\mathring h^1(U)} \exp \left(  - \sum_{e \subset U} \left( V_e( \nabla \psi(e)) - q \cdot \nabla \psi (e) \right)\right) \, d\psi,
\end{equation}
as well as the quantity
\begin{equation} \label{expformnustar}
\nu^* (U , q) := \frac1{|U|} \ln Z_q^* \left(U \right),
\end{equation}
and the probability measure
\begin{equation*}
\P_{U,q}^*(d \psi) :=  \frac{ \exp \left(  - \sum_{e \subset U} \left( V_e( \nabla \psi(e)) - q \cdot \nabla \psi (e) \right) \right) \, d\psi}{Z_q^* \left(U \right)}.
\end{equation*}
We denote by $\psi_{U,p}$ a random variable of law $\P_{U,p}^*$ and we frequently write $\P_{n,q}^*$ and $\psi_{n,p}$ instead of $\P_{\cu_n,q}^*$ and $\psi_{\cu_n,p}$.

\medskip

\subsubsection{Convention for constants and exponents}
Throughout this article, the symbols $c$ and $C$ denote positive constants which may vary from line to line. These constants may depend solely on the parameters $d$, the dimension of the space, and $\lambda$, the ellipticity bound on the second derivative of the function $V_e$. Similarly we use the symbols $\alpha$ and $\beta$ to denote positive exponents which may vary from line to line and depend only on $d$ and $\lambda$. Usually, we use $C$ for large constants (whose value is expected to belong to $[1, \infty)$) and $c$ for small constants (whose value is expected to be in $(0,1]$). The values of the exponents $\alpha$ and $\beta$ are always expected to be small.

 We also frequently write $C := C(d,\lambda) < \infty$ to mean that the constant $C$ depends only on the parameters $d,\lambda$ and that its value is expected to be large. We may also write $C := C(d) < \infty$ or $C := C(\lambda) < \infty$ if the constant $C$ depends only on $d$ (resp. $\lambda$). For small constants or exponents we use the notations $c := c(d , \lambda) > 0$, $\alpha := \alpha(d , \lambda) > 0$, $\beta := \beta(d , \lambda) > 0$.

\subsection{Main result}
The objective of this article is to obtain a quantitative rate of convergence for the finite-volume surface tension $\nu (\cu_n, p)$. We also obtain an algebraic rate of convergence for the quantity $ \nu^*$ and prove that the two surface tensions $\nu$ and $ \nu^*$ are, in the limit, convex dual. This is stated in the following theorem.
\begin{reptheorem}{QuantitativeconvergencetothGibbsstate}[Quantitative convergence of the surface tensions] 
There exist a constant $C := C(d , \lambda) < \infty$ and an exponent $\alpha := \alpha(d , \lambda) > 0$ such that for each $p,q \in \Rd$, there exist two real numbers $ \bar \nu (p)$ and $\bar \nu^* (q)$ such that
\begin{equation} 
\left| \nu (\cu_n, p) - \bar \nu (p) \right| \leq C 3^{-\alpha n} (1 + |p|^2),
\end{equation}
and 
\begin{equation} \label{quantconvnustar}
\left| \nu^* (\cu_n, q) - \bar \nu^* (q) \right| \leq C 3^{-\alpha n} (1 + |q|^2).
\end{equation}
Moreover the functions $p \rightarrow \bar \nu (p)$ and $q \rightarrow \bar \nu^* (q)$ are uniformly convex, i.e., there exists a constant $C := C(d , \lambda) < \infty$ such that for each $p_1 , p_2 \in \Rd$
\begin{equation} \label{unifconvbarnu}
\frac1{C} \left| p_1 - p_2 \right|^2 \leq \bar \nu \left(p_1\right) + \bar \nu \left( p_2 \right) - 2 \bar \nu \left( \frac{p_1 + p_2}2\right) \leq C \left| p_1 - p_2 \right|^2,
\end{equation}
for each $q_1 , q_2 \in \Rd$,
\begin{equation} \label{unifconvbarnustar}
\frac1{C} \left| q_1 - q_2 \right|^2 \leq \bar \nu^* \left(q_1\right) +\bar \nu^* \left( q_2 \right) - 2 \bar \nu^* \left( \frac{q_1 + q_2}2\right) \leq C \left| q_1 - q_2 \right|^2,
\end{equation}
and are dual convex, i.e., for each $q \in \Rd$
\begin{equation} \label{dualconvnunustar}
\bar \nu^* (q) = \sup_{p \in \Rd} \left( - \bar \nu (p) + p \cdot q \right).
\end{equation}
\end{reptheorem}

\begin{remark}
The uniform convexity property~\eqref{unifconvbarnu} is already known and was proved by Deuschel, Giacomin and Ioffe in~\cite[Lemma 3.6]{DGI00}.
\end{remark}

\begin{remark}
To build some intuition on the quantities $\nu$ and $\nu^*$ and see that they are convex dual, one can consider the specific case of the Gaussian free field, i.e., when $V_e (x) := \beta x^2$, for some strictly positive coefficient $\beta$. In this case, the value of the surface tensions can be explicitly computed and one has, for each $p,q \in \Rd$,
\begin{equation*}
\nu \left( \cu_n , p\right) = \beta |p|^2 + C_{n, \mathrm{Dir}} ~\mbox{and}~ \nu^* \left( \cu_n , q\right) = \frac{1}{4\beta} |q|^2 - C_{n, \mathrm{ Neu}},
\end{equation*}
where the constants $ C_{n, \mathrm{Dir}}$ and $C_{n, \mathrm{ Neu}}$ can be computed explicitly in terms of the parameter $\beta$ and of the eigenvalues of the discrete Laplacian on the cube $\cu_n$ with Dirichlet and Neumann boundary conditions, respectively. As the size of the cube $\cu_n$ tends to infinity, these two constants converge to the same constant which we denote by $C := C(\beta , d) < \infty$. From this argument, one deduces the following formula for the surface tensions $\bar \nu$ and $\bar \nu^*$, for each $p,q \in \Rd$,
\begin{equation*}
\bar \nu \left( p \right) :=   \beta |p|^2 + C ~\mbox{and} ~ \bar \nu^* \left( q \right) :=  \frac{1}{4\beta} |q|^2 - C,
\end{equation*}
which shows in particular that the two functions $\bar \nu$ and $\bar \nu^*$ are convex dual.
\end{remark}

\begin{remark}
These results give a quantitative rate of convergence for the surface tension which is suboptimal: we only obtain an algebraic rate of convergence for some small exponent $\alpha > 0$. The reason which justifies this exponent $\alpha$ is the following. The main idea of the proof is to obtain an inequality of the form
\begin{equation} \label{simpleversionconvdual}
 \nu \left( \cu_{n+1}, p \right) - \bar \nu (p) \leq C \left( \nu \left( \cu_{n},p \right) - \nu\left( \cu_{n+1}, p\right) \right),
\end{equation} 
which implies the algebraic rate of convergence thanks to the two following arguments:
\begin{enumerate}
\item the inequality~\eqref{simpleversionconvdual} can be rearranged to obtain
\begin{equation} \label{algebraicratesimpleversion}
\nu \left( \cu_{n+1}, p \right) - \bar \nu (p)  \leq \frac{C}{C+1} \left( \nu \left( \cu_{n} , p \right) -  \bar \nu (p)  \right);
\end{equation}
\item since the sequence $\left( \nu \left( \cu_{n}, p \right)  \right)_{n \in \N}$ is decreasing (see Proposition~\ref{propofnunustar} or~\cite[Lemma II.1]{FS97}), the value $\left( \nu \left( \cu_{n}, p \right)  - \bar \nu (p) \right)$ is non-negative and we deduce from the estimate~\eqref{algebraicratesimpleversion},
\begin{equation*}
\nu \left( \cu_{n}, p \right) - \bar \nu (p) \leq  \left( \frac{C}{C+1} \right)^n  \left( \nu \left( \cu_{0} , p \right) -  \bar \nu (p)  \right).
\end{equation*}
Since the number $\frac{C}{C+1}$ is strictly smaller than $1$, we obtain an algebraic rate of convergence with the exponent $\alpha := \frac{1}{\ln 3} \ln \frac{C+1}{C}$. Since we do not have a good control on the value of the constant $C$ in the proof, the exponent $\alpha$ obtained is not explicit.
\end{enumerate}
The precise statements are more involved and make use of the dual surface tension $\nu^*$. We refer to Section~\ref{section4.3}, and more specifically to the result of Proposition~\ref{lemma3.3}, and to Section~\ref{section4.4gradphi} for the details.
\end{remark}

From this, we deduce that the random variables $\phi_{n,p}$ and $\psi_{n,q}$ are close to affine functions in the expectation of the $L^2$-norm. Before stating the result, we note that~\eqref{unifconvbarnu} and~\eqref{unifconvbarnustar} imply that the functions $p \rightarrow \bar \nu(p)$ and $q \rightarrow \bar \nu^*(q)$ are $C^{1,1}$ and we denote their gradients by $\nabla_p \bar \nu$ and $\nabla_q \bar \nu^*$. 
\begin{reptheorem}{conhvergencetoaffine}[$L^2$ contraction of the Gibbs measure] 
There exist a constant $C := C(d , \lambda) < \infty$ and an exponent $\alpha := \alpha (d, \lambda) > 0$ such that for each $n \in \N$, $p, q \in \Rd$,
\begin{multline*}
\frac{1}{\left(\diam \cu_n\right)^2}\E \left[  \frac1{\left| \cu_n \right|} \sum_{x \in \cu_n} \left(  \left| \phi_{n,p}(x) \right|^2 + \left| \psi_{n,q}(x) -  \nabla_q \bar \nu^* (q ) \cdot x \right|^2 \right) \right] \\ \leq C 3^{- \alpha n} \left( 1 + |p|^2 + |q|^2 \right).
\end{multline*}
\end{reptheorem}

\begin{remark}
While the results of the two theorems are stated for deterministic potentials $V_e$, we expect that similar results, with similar proofs, should hold in the case of random potentials, with a stationary law and strong enough mixing conditions. Indeed the techniques developed in~\cite{armstrong2017quantitative}, which we try to adapt here and which rely on similar subadditive quantities, are designed to work in a random setting and their methods should be applicable to treat the random potential case. This generalization will affect the proofs by adding an additional layer of notations; moreover some concentration inequalities will be required to conclude.
\end{remark}

\begin{remark}
Qualitative versions of the previous statements, and in particular a qualitative version of Theorem~\ref{QuantitativeconvergencetothGibbsstate} can be obtained by a (much softer) subadditivity argument and was established by Funaki and Spohn in~\cite[Lemma II.1]{FS97}. Their result is stated in this article in Proposition~\ref{propofnunustar}. Obtaining quantitative rate of convergence is a different problem since the subadditivity arguments are purely qualitative. To obtain such results, one has to work with the two subadditive quantities $\nu$ and $\nu^*$ together; the idea is to show that they are approximately convex dual, to quantify the defect of convex duality and to deduce a rate of convergence for these quantities.
\end{remark}

\subsection{Strategy of the proof}
The strategy of the proof is to use the ideas from the theory of quantitative stochastic homogenization, and in particular the ideas developed by Armstrong, Kuusi, Mourrat and Smart in~\cite{armstrong2017quantitative} and~\cite{AS} to the setting of the $\nabla \phi$ model. To this end, we introduce the two quantities $\nu$ and $\nu^*$, which are in some sort equivalent to the subadditive quantities with the same notation used in~\cite[Chapters 1 and 2]{armstrong2017quantitative}. 

The idea is then to find a variational formulation for these quantities to rewrite them as a minimization problem of a convex functional, this is done in Section~\ref{varformulanunu*}. This functional involves a term of entropy and it is not a priori clear that it is convex. To solve this issue, we appeal to the theory of optimal transport and more specifically to the notion of displacement convexity to obtain some sort of convexity for the entropy. 

Once this is done, we are able to collect some properties about $\nu$ and $\nu^*$ which match the basic properties of the equivalent quantities in stochastic homogenization, and one can for instance compare Propostion~\ref{propofnunustar} with Lemma 1.1 of~\cite{armstrong2017quantitative}. One can then exploit the convex duality between $\nu$ and $\nu^*$ to obtain a quantitative rate of convergence for these quantities as it is done in Section~\ref{section4}.

\subsection{Outline of the paper}
The rest of the article is organized as follows. In Section~\ref{section2}, we collect some preliminary results which are useful to prove the main theorems. Specifically, we introduce the differential entropy of a measure and state some of its properties, we also record some definitions from the theory of optimal transport and state the main result we need to borrow from this theory, namely the displacement convexity of the entropy, Proposition~\ref{p.disconv}. We also introduce the variational formulation for $\nu$ and $\bar \nu^*$ in Section~\ref{varformulanunu*}. In Section~\ref{Couplinglemmas}, we state and prove a technical lemma which allows to construct suitable coupling between random variables. We then complete Section~\ref{section2} by stating some functional inequalities on the lattice $\Zd$, in particular the multiscale Poincar\'e inequality which is an important ingredient in the proofs of Theorems~\ref{QuantitativeconvergencetothGibbsstate} and~\ref{conhvergencetoaffine}.

In Section~\ref{section3}, we use the tools from Section~\ref{section2} to prove a series of properties on the quantities $\nu$ and $\nu^*$, summarized in Proposition~\ref{propofnunustar}.

In Section~\ref{section4}, we combine the tools collected in Section~\ref{section2} with the results proved in Section~\ref{section3} to first prove that the variance of the random variable $\left( \nabla \psi_{n,q} \right)_{\cu_n}$ contracts, this is done in Lemma~\ref{contslope}. We then deduce from this result and the multiscale Poincar\'e inequality that the random variable $\psi_{n,q}$ is close to an affine function in the $L^2$-norm, this is the subject of Proposition~\ref{p.ustarminimizerareflat}. We then use these results combined with a patching construction, reminiscent to the one performed in~\cite{AS}, to prove Theorems~\ref{QuantitativeconvergencetothGibbsstate} and~\ref{conhvergencetoaffine}.

Appendix~\ref{appendixA} is devoted to the proof of some technical estimates useful in Sections~\ref{section2} and~\ref{section3}.

Appendix~\ref{appendixB} is devoted to the proof of some inequalities from the theory of elliptic equations adapted to the setting of the $\nabla \phi$ model. More precisely, we prove a version of the Caccioppoli inequality, the reverse H\"older inequality and the Meyers estimate for the $\nabla \phi$-model.

\subsection{Acknowledgments} I would like to thank Jean-Christophe Mourrat for helpful discussions and comments.

\section{Preliminaries} \label{section2}

\subsection{The entropy and some of its properties}
In this section, we define one of the main tools used in this article, the differential entropy, we then collect a few properties of this quantity which are used in the rest of the article.

\begin{definition}[Differential entropy] \label{definitionentropy}
Let $V$ be a finite dimensional vector space equipped with a scalar product. Denote by $\mathcal{B}(V)$ the Borel $\sigma$-algebra associated with $V$. Consider the Lebesgue measure on $V$ and denote it by $\mathrm{Leb}$. For each probability measure $\P$ on $V$, we define its entropy according to the formula
\begin{equation*}
H \left( \P \right) := \left\{ \begin{aligned}
\int_V \frac{d \P}{d \mathrm{Leb}}(x) \ln \left( \frac{d \P}{d \mathrm{Leb}}(x) \right) \, dx &~ \mbox{if} ~ \P \ll \mathrm{Leb} ~\mbox{and}~ \frac{d \P}{d \mathrm{Leb}}  \ln \left( \frac{d \P}{d \mathrm{Leb}} \right) \in L^1 \left( V \right), \\
+ \infty &~ \mbox{otherwise} .
\end{aligned} \right.
\end{equation*} 
\end{definition}

\begin{remark}
\begin{itemize}
\item In this article we implicitly extend the function $x \rightarrow x \ln x$ by $0$ at $0$.
\item We emphasize that the usual definition of the differential entropy is stated with the function $x \rightarrow - x \ln x$ instead of the function $x \rightarrow  x \ln x$. Adopting the other sign convention is more meaningful in this article because we want the entropy to be convex in the sense of displacement convexity as it is explained in the following sections.
\end{itemize}
\end{remark}

We now record a few properties about the entropy. We first study how the entropy behaves under translation and affine change of variables. These properties are standard and fairly simple to prove, the details are thus omitted.

\begin{proposition}[Translation and linear change of variables of the entropy] \label{transandchofvar}
Let $X$ be a random variable taking values in a finite dimensional real vector space $V$ equipped with a scalar product. Denote by $\P_X$ the law of $X$. For each $a \in V$, we denote by $\P_{X+a}$ the law of the random variable $X+a$. Then we have
\begin{equation} \label{transentropy}
H\left( \P_{X + a} \right) = H\left( \P_{X} \right).
\end{equation}
Now consider a linear map $L$ from $V$ to $V$. We denote by $\P_{L(X)}$ the law of the random variable $L(X)$. Then we have
\begin{equation} \label{linearmapentropy}
H \left( \P_{L(X)} \right) = H \left( \P_X \right) - \ln \left| \det L \right|.
\end{equation}
In particular if $L$ is non invertible then $\det L = 0$ and $H \left( \P_{L(X)} \right) = \infty$.
\end{proposition}

Let $V , W$ be two finite dimensional real vector spaces equipped with scalar products denoted by $\left( \cdot , \cdot \right)_{V}$ and $\left( \cdot , \cdot \right)_{W}$. Denote by $\Leb_V$ and $\Leb_W$ the Lebesgue measures on $V$ and $W$. Consider the space $V \times W$. Define a scalar product on this space by, for each $v,v' \in V$ and each $w , w' \in W$,
\begin{equation} \label{defproductscalarprod}
\left( (v, w) , (v', w') \right)_{V \times W} =  \left(v , v' \right)_{V} +  \left(w , w' \right)_{W}.
\end{equation}
Then the Lebesgue measure on $V \oplus W$ satisfies
\begin{equation*}
\Leb_{V \times W} = \Leb_{V} \otimes \Leb_{W}.
\end{equation*}
The following proposition gives a property about the entropy of a pair of random variables.

\begin{proposition} \label{entropycouplesum}
Let $V , W$ be two finite dimensional real vector spaces equipped with scalar products. Consider the space $V \times W$ equipped with the scalar product defined in~\eqref{defproductscalarprod}. Let $X$ and $Y$ be two random variables valued in respectively $V$ and $W$. Assume that $H \left( \P_X \right) < \infty$, $H \left( \P_Y \right) < \infty$ and that we are given a coupling $(X,Y)$ between $X$ and $Y$. Then we have
\begin{equation*}
H \left( \P_{(X , Y)} \right) \geq H \left( \P_X \right) +  H \left( \P_Y \right),
\end{equation*}
with equality if and only if the random variables $X$ and $Y$ are independent.
\end{proposition}

\begin{remark}
This inequality states that the entropy of two random variable is minimal when $X$ and $Y$ are independent. This is due to the sign convention adopted in Definition~\ref{definitionentropy}.
\end{remark}

\begin{proof}
These estimates can be obtained using the convexity of the function $x \rightarrow x \ln x$ and the Jensen inequality. The proof is standard and the details are omitted.
\end{proof}

Frequently in this article, the previous proposition is used with the following formulation.

\begin{proposition} \label{entropycouplesum2}
Let $U$ be a finite dimensional vector space equipped with a scalar product and assume that we are given two linear subspaces of $U$, denoted by $V$ and $W$, such that
\begin{equation*}
U = V \overset{\perp}{\oplus} W.
\end{equation*}
Assume moreover that we are given two random variables $X$ and $Y$ taking values respectively in $V$ and $W$. Assume that $H \left( \P_X \right) < \infty$, $H \left( \P_Y \right) < \infty$ and that we are given a coupling $(X,Y)$ between $X$ and $Y$. Then we have
\begin{equation*}
H \left( \P_{X + Y} \right) \geq H \left( \P_X \right) +  H \left( \P_Y \right),
\end{equation*}
where the entropy of $X+Y$ (resp. $X$ and $Y$) is computed with respect to the Lebesgue measure on $U$ (resp. $V$ and $W$). Moreover there is equality in the previous display if and only if $X$ and $Y$ are independent.
\end{proposition}

\begin{proof}
This proposition is a consequence of Proposition~\ref{entropycouplesum} and the fact that there exists a canonical isometry between $U$ and $V \times W$ given by
\begin{equation} \label{canonicalbij}
\Phi : \left\{ \begin{aligned}
V \times W & \rightarrow U \\
(v, w) & \mapsto v + w.
\end{aligned} \right.
\end{equation}
\end{proof}

\subsection{Variational formula for the surface tensions $\nu$ and $\nu^*$} \label{varformulanunu*}
One of the important ingredients of this article is to introduce a convex functional $\mathcal{F}_{n,p}$ (resp. $\mathcal{F}_{n,q}^*$) defined on set of the probability measures on the space $h^1_0(\cu_n)$ (resp. $\mathring h^1(\cu_n)$) such that $\P_{n,p}$ (resp. $\P_{n,q}^*$) is the minimizer of $\mathcal{F}_{n,p}$ (resp. $\mathcal{F}_{n,q}^*$). The convexity of $\mathcal{F}_{n,p}$ (resp. $\mathcal{F}_{n,q}^*$) allows to perform a perturbative analysis around its minimizer, i.e., the measure $\P_{n,p}$ (resp. $\P_{n,q}^*$), and to obtain quantitative estimates which turn out to be crucial in the proof of Theorem~\ref{QuantitativeconvergencetothGibbsstate}.

\begin{definition} \label{chapitre444.energyentrpyfunctional}
For each $n \in \N$ and each $p , q \in \Rd$, we define
\begin{equation*}
\mathcal{F}_{n,p} : \left\{ \begin{aligned}
\mathcal{P}\left( h^1_0 (\cu_n) \right) & \rightarrow \R \\
\P & \mapsto \E \left[ \sum_{e \subseteq \cu_n} V_e\left( p \cdot e + \nabla \phi_e \right) \right] + H \left( \P \right),
\end{aligned} \right.
\end{equation*}
where $\phi$ is a random variable of law $\P$. Similarly, we define
\begin{equation*}
\mathcal{F}_{n,q}^* : \left\{ \begin{aligned}
\mathcal{P}\left( \mathring h^1 (\cu_n) \right) & \rightarrow \R \\
\P^* & \mapsto \E \left[ \sum_{e \subseteq \cu_n} \left( V_e\left( \nabla \psi(e) \right) - q \cdot  \nabla \psi(e)\right) \right] + H \left( \P^* \right),
\end{aligned} \right.
\end{equation*}
where $\psi$ is a random variable of law $\P^*$.
\end{definition}

The main property about this functional is stated in the following proposition.

\begin{proposition} \label{varfornuprop}
Let $V$ be a finite dimensional real vector space equipped with a scalar product. We denote by $\mathcal{B}(V)$ be the Borel set associated to $V$. For any measurable function $f : V \rightarrow \R$ bounded from below, one has the formula
\begin{equation} \label{Jformula}
- \log \int_{V} \exp \left( - f (x) \right) \, dx = \inf_{\P \in \mathcal{P} \left( V \right)} \left(  \int_{\P} f(x) \P (dx) + H \left( \P \right) \right),
\end{equation}
where the integral in the left-hand side is computed with respect to the Lebesgue measure on $V$.

As a consequence, one has the following formula, for each $n \in \N$ and each $p \in \Rd$,
\begin{equation} \label{varfornu}
\nu \left( \cu_n , p  \right)  = \inf_{\P \in \mathcal{P} \left( h^1_0 (\cu_n) \right)} \left(  \frac1{|\cu_n|}\E \left[ \sum_{e \subseteq \cu_n} V_e\left( p \cdot e + \nabla \phi(e) \right) \right] + \frac1{|\cu_n|} H \left( \P \right) \right),
\end{equation}
where in the previous formula, $\phi$ is a random variable of law $\P$. Moreover the minimum is attained for the measure $\P_{n,p}$. Similarly, one has, for each $q \in \Rd$,
\begin{equation} \label{varfornustar}
\nu^* \left( \cu_n , q  \right)  = \sup_{\P^* \in \mathcal{P}  \left( \mathring h^1 (\cu_n) \right)} \left( \frac1{|\cu_n|} \E \left[- \sum_{e \subseteq \cu_n} \left( V_e\left(  \nabla \psi(e) \right) - q \cdot \nabla \psi (e) \right) \right] - \frac1{|\cu_n|} H \left( \P^* \right) \right),
\end{equation}
where in the previous formula, $\psi$ is a random variable of law $\P^*$. Moreover the minimum is attained for the measure $\P_{n,q}^*$.
\end{proposition}

\begin{proof}
We first prove~\eqref{Jformula} and decompose the proof into three steps.

\medskip

\textit{Step 1.} Let $\P$ be a probability measure on $V$, we want to show that
\begin{equation} \label{orvarformulageq}
\int_{V} f(x) \P (dx) + H \left( \P \right) \geq - \log \int_{V} \exp \left( - f (x) \right) \, dx.
\end{equation}
First note that if $H(\P) = \infty$, then the term on the left-hand side is equal to infinity and the inequality is satisfied. Thus one can assume $H \left( \P\right) < \infty$, this implies that $\P $ is absolutely continuous with respect to the Lebesgue measure on $V$ and we denote by $h$ its density. In particular, we have
\begin{equation*}
H \left( \P \right) = \int_V h(x) \ln h(x) \, dx.
\end{equation*}
Similarly, one can assume that $\int_V f(x) h(x) \, dx < \infty$ otherwise the estimate~\eqref{orvarformulageq} is automatically verified. 
Using that $h$ is a probability density and the Jensen inequality, one obtains
\begin{equation*}
\exp \left( - \int_V f(x) h(x) \, dx - H \left( \P \right)  \right) \leq \int_V \exp \left(- f(x) - \ln h(x) \right) h(x) \, dx.
\end{equation*}
We then denote
\begin{equation*}
A := \left\{ x \in V ~:~ h(x) > 0 \right\} \in \mathcal{B}(V),
\end{equation*}
so that
\begin{align*}
\exp \left( - \int_V f(x) h(x) \, dx - H \left( \P \right)  \right) & \leq \int_{V} \indc_{A}(x)  \exp \left( - f(x)\right) h(x)^{-1} h(x) \, dx \\
												& \leq \int_{V} \indc_{A}(x)  \exp \left( - f(x)\right) \, dx \\
												& \leq \int_{V} \exp \left( - f(x)\right) \, dx .
\end{align*}
This is precisely~\eqref{orvarformulageq}.

\medskip

\textit{Step 2.} We assume that
\begin{equation} \label{assumptionfstep2}
\int_V \exp \left(- f(x) \right) \, dx < \infty ~ \mbox{and}~ \int_V \left| f(x) \right|  \exp \left(- f(x) \right) \, dx < \infty
\end{equation}
and construct a probability measure $\P \in \mathcal{P} \left( V \right)$ satisfying
\begin{equation} \label{orvarformulageqother}
\int_{\P} f(x) \P (dx) + H \left( \P \right) = - \log \int_{V} \exp \left( - f (x) \right) \, dx.
\end{equation}
In this case, we define 
\begin{equation*}
\P := \frac{\exp \left( - f(x) \right)}{\int_{V} \exp \left( - f (x) \right) \, dx} \, dx.
\end{equation*}
It is clear that $\P$ is absolutely continuous with respect to the Lebesgue measure $\Leb_V$ and, from the assumptions~\eqref{assumptionfstep2}, that $H \left( \P \right) <\infty$. An explicit computation gives
\begin{align*}
\int_{V} f(x) \P(dx) + H(\P) & = \int_{V} f(x) - f(x) - \ln \left( \int_{V} \exp \left( - f (x) \right) \, dx \right) \P(dx) \\
						& = \ln \left( \int_{V} \exp \left( - f (x) \right) \, dx \right).
\end{align*}
This completes the proof of~\eqref{orvarformulageqother}.
\medskip

\textit{Step 3.} In this step, we assume that 
\begin{equation*}
\int_V \exp \left(- f(x) \right) \, dx = \infty  ~\mbox{or}~ \int_V \left| f(x) \right|  \exp \left(- f(x) \right) \, dx = \infty
\end{equation*}
and we construct a sequence of probability measures $\P_n $ such that
\begin{equation*}
\int_V f(x) \P_n (dx) + H \left( \P_n \right)  \underset{n \rightarrow \infty}{\longrightarrow} - \log \int_{V} \exp \left( - f (x) \right) \, dx,
\end{equation*}
where we used the convention
\begin{equation*}
- \log \int_{V} \exp \left( - f (x) \right) \, dx = - \infty ~\mbox{if}~ \int_{V} \exp \left( - f (x) \right) \, dx = \infty.
\end{equation*}
To this end, we define, for each $n \in \N$,
\begin{equation*}
\P_n := \frac{\exp \left( - f(x) \right)\indc_{\{|x| \leq n \mbox{ and } f(x) \leq n    \}} }{\int_{V} \exp \left( - f (x) \right) \indc_{\{|x| \leq n \mbox{ and } f(x) \leq n  \}} \, dx} \, dx.
\end{equation*}
With this definition, we compute
\begin{equation*}
\int_V f(x) \P_n (dx) + H \left( \P_n \right) = - \ln  \left( \int_{V} \exp \left( - f (x) \right) \indc_{\{|x| \leq n \mbox{ and } f(x) \leq n   \}} \, dx \right).
\end{equation*}
Sending $n \rightarrow \infty$ gives the result. This completes the proof of~\eqref{Jformula}.

\medskip

We now deduce the equality~\eqref{varfornu} from the identity~\eqref{Jformula}. We first note that, thanks to the bound $V_e(x) \leq \frac1\lambda |x|^2$, there exists a constant $C := C(d, \lambda) < \infty$ such that, for each $n \in \N$ and each $\mathbf{\phi} \in h^1_0 \left( \cu_n\right)$,
\begin{equation} \label{simplepoincare25}
\sum_{e \subseteq \cu_n} V_e\left( p \cdot e + \nabla \mathbf{\phi}(e) \right) \leq C |p|^2 + C\sum_{x \in \cu_n} \left| \mathbf{\phi} (x) \right|^2.
\end{equation}
Using this inequality, we deduce that one can apply the identity~\eqref{Jformula} with $V = h^1_0 \left( \cu_n \right)$ equipped with the standard $L^2$ scalar product and $f(\mathbf{\phi} ) = \sum_{e \subseteq \cu_n} V_e\left( p \cdot e + \nabla \mathbf{\phi}(e) \right)$. Moreover, using the estimate~\eqref{simplepoincare25}, one has
\begin{equation*}
\int_{h^1_0 \left( \cu_n \right)} \exp \left(- f(\mathbf{\phi}) \right) \, dx < \infty ~ \mbox{and}~ \int_{h^1_0 \left( \cu_n \right)} \left| f(\mathbf{\phi}) \right|  \exp \left(- f(\mathbf{\phi}) \right) \, d\mathbf{\phi} < \infty.
\end{equation*}
Thus applying the result proved in Steps 1 and 2, one obtains
\begin{equation*}
\nu \left( \cu_n , p  \right)  = \inf_{\P \in \mathcal{P} \left( h^1_0 (\cu_n) \right)} \frac1{|\cu_n|}\E \left[ \sum_{e \subseteq \cu_n} V_e\left( p \cdot e + \nabla \phi(e) \right) \right] + \frac1{|\cu_n|} H \left( \P \right)
\end{equation*}
and the minimum is attained for the measure
\begin{equation*}
\P_{n,p}(d \mathbf{\phi}) = \frac{ \exp \left(  - \sum_{e \subset U} V_e( p \cdot e + \nabla \mathbf{\phi}(e)) \right) \, d \mathbf{\phi}}{\int_{h^1_0 \left(\cu_n \right)} \exp \left(  - \sum_{e \subset U} V_e( p \cdot e + \nabla \mathbf{\phi}(e)) \right) \, d \mathbf{\phi}}.
\end{equation*}
The proof of~\eqref{varfornustar} is similar and the details are left to the reader.
\end{proof}

\subsection{Optimal transport and displacement convexity}

In this section, we introduce a few definitions about optimal transport and state one of the main tools of this article, namely the displacement convexity. We first give a definition of the optimal coupling. We refer to~\cite[Proposition 2.1 and Theorem 2.12]{Vi} for this definition.

\begin{definition}\label{def.opttrans}
Let $U$ be a finite dimensional real vector space equipped with a scalar product. We denote by $\left| \cdot \right|$ the norm associated to this scalar product. Let $X$ and $Y$ be two random variables taking values in $U$ and denote their laws by $\P_X$ and $\P_Y$ respectively. Assume additionally that $\P_X$ and $\P_Y$ have a finite second moment, i.e,
\begin{equation} \label{finitesecondmoment}
\E \left[ \left| X \right|^2 \right] < \infty ~\mbox{and}~\E \left[ \left| Y \right|^2 \right] < \infty,
\end{equation}
and that they are absolutely continuous with respect to the Lebesgue measure on $V$. Then the minimization problem
\begin{equation*}
\inf_{\mu \in \Pi \left(\P_X, \P_Y \right)} \int_{U} \left|x - y \right|^2 \, \mu(dx , dy)
\end{equation*}
admits a unique minimizer denoted by $\mu_{(X,Y)}$. This probability measure is called the optimal coupling between $X$ and $Y$.
\end{definition}
For $t \in [0,1]$, we denote by $T_t$ the mapping
\begin{equation*}
T_t := \left\{ \begin{aligned}
U \times U & \rightarrow U \\
(x,y) & \mapsto (1-t) x + t y,
\end{aligned} \right.
\end{equation*}
and for two random variables $X$ and $Y$ taking values in $U$ with finite second moment, we denote by
\begin{equation*}
\mu_t := \left( T_t \right)_* \mu_{(X,Y)}.
\end{equation*}
This is the law of $(1-t) X + t Y$ when the coupling between $X$ and $Y$ is the optimal coupling. The main property we need to use is called the displacement convexity and stated in the following proposition. We refer to \cite[Chapter 5]{Vi} for this theorem but it is due to McCann~\cite{MC}.

\begin{proposition}[Displacement convexity, Theorem 5.15 of~\cite{Vi}] \label{p.disconv}
Let $U$ be a finite dimensional real vector space equipped with a scalar product, let $X$ and $Y$ be two random variables taking values in $U$ with finite second moment, i.e., satisfying~\eqref{finitesecondmoment}, then the function $t \rightarrow H \left( \mu_t \right)$ is convex, i.e., for each $t \in [0,1]$,
\begin{equation*}
H \left( \mu_t \right) \leq (1-t) H \left( \P_X \right) + t H \left( \P_Y \right).
\end{equation*}
\end{proposition}

\subsection{Coupling lemmas} \label{Couplinglemmas}
Thanks to optimal transport theory and particularly thanks to Definition~\ref{def.opttrans}, we are able to couple two random variables. The next question which arises is to find a way to couple three random variables. Broadly speaking, the question we need to answer is the following: assume that we are given three random variables $X$, $Y$ and $Z$, a coupling between $X$ and $Y$ and another coupling between $Y$ and $Z$, can we find a coupling between $X$, $Y$ and $Z$? This question can be positively answered thanks to the following proposition.

\begin{proposition} \label{couplinglemma}
Let $(E_1, \mathcal{B}_1)$,  $(E_2, \mathcal{B}_2)$, $(E_3, \mathcal{B}_3)$ be three Polish spaces equipped with their Borel $\sigma$-algebras. Assume that we are given three probability measures $\P_X$ on $E_1$, $\P_Y$ on $E_2$ and $\P_Z$ on $E_3$ as well as a coupling $\P_{(X,Y)}$ between $\P_X$, $\P_Y$ and a coupling $\P_{(Y,Z)}$ between $\P_Y$, $\P_Z$, that is to say two measures on $(E_1 \times E_2, \mathcal{B}_1  \otimes \mathcal{B}_2)$ and $(E_2 \times E_3, \mathcal{B}_2  \otimes \mathcal{B}_3)$  satisfying, for each $( B_1, B_2 , B_3) \in \left( \mathcal{B}_1, \mathcal{B}_2, \mathcal{B}_3 \right) $,
\begin{multline*}
\P_{(X,Y)} \left( B_1 \times E_2 \right) = \P_X \left( B_1\right) , ~ \P_{(X,Y)} \left(E \times B_2 \right) = \P_Y \left(B_2\right)  ~ \mbox{and} \\ \P_{(Y,Z)} \left( B_2 \times E \right) = \P_Y \left(B_2 \right), ~ \P_{(Y,Z)} \left( E \times B_3 \right) = \P_Z \left(B_3\right),
\end{multline*}
then there exists a probability measure $\P_{(X,Y,Z)}$ on $(E_1 \times E_2 \times E_3, \mathcal{B}_1  \otimes \mathcal{B}_2 \otimes \mathcal{B}_3)$ such that for each $B_{12} \in  \mathcal{B}_1  \otimes \mathcal{B}_2$ and each $B_{23} \in  \mathcal{B}_2  \otimes \mathcal{B}_3$,
\begin{equation} \label{propPXYZ}
\P_{(X,Y,Z)} \left( B_{12} \times E_3 \right) = \P_{(X,Y)} \left( B_{12} \right) ~\mbox{and}~ \P_{(X,Y,Z)} \left( E_1 \times B_{23} \right) = \P_{(Y,Z)} \left( B_{23} \right).
\end{equation}
\end{proposition}

\begin{remark}
As was mentioned earlier, in this article we think of random variables as laws and we do not assume that there is an underlying probability space $(\Omega, \mathcal{F}, \P)$ on which all the random variables are already defined. For instance, we say that we are given two random variables $(X,Y)$ and $(Y,Z)$ to mean that we are given two measures $\P_{(X,Y)}$ and $\P_{(Y,Z)}$ such that the marginals of $\P_{(X,Y)}$ are $\P_X$ and $\P_Y$ and the marginals of $\P_{(Y,Z)}$ are $\P_Y$ and $\P_Z$ without assuming that there exists an implicit probability space on which $X,Y$ and $Z$ are defined, indeed in that case the statement of the proposition would be trivial.

This convention allows to simplify the notation in the proofs and has the following consequence: when we are given two random variables $X$ and $Y$, we need to be careful to always construct a coupling between $X$ and $Y$ before introducing the random variables $X + Y$, $XY$ or any other display involving both $X$ and $Y$.
\end{remark}

The proof of Proposition~\ref{couplinglemma} relies on the existence of the conditional law which is recalled below.

\begin{proposition}[Theorem 33.3 and Theorem 34.5 of~\cite{bil}] \label{condlaw}
Let $(E, \mathcal{E})$ and $(F , \mathcal{F})$ be two Polish spaces equipped with their Borel $\sigma$-algebras. Assume that we are given two probability measures $\P_1$ and $\P_2$ on $E$ and $F$ respectively. Let $\P_{12}$ be a probability measure on $(E \times F, \mathcal{E} \otimes \mathcal{F})$ whose first and second marginals are $\P_1$ and $\P_2$ respectively, then there exists a mapping $\nu : E_1 \times \mathcal{F} \rightarrow \R_+$ such that
\begin{enumerate}
\item for each $x \in E$, $\nu(x , \cdot)$ is a probability measure on $(F , \mathcal{F})$;
\item for each $A \in \mathcal{F}$, the mapping
\begin{equation*}
\nu \left( \cdot, A \right) \left\{ \begin{aligned}
(E, \mathcal{E}) & \rightarrow \left( \R, \mathcal{B}(\R) \right) \\
x & \mapsto \nu(x, A)
\end{aligned} \right.
\end{equation*}
 is measurable;
\item For each $A_1 \in \mathcal{E}$ and each $A_2 \in \mathcal{F}$,
\begin{equation*}
\P_{12} \left( A_1 \times A_2 \right) = \int_{A_1} \nu \left( x, A_2 \right) \, \P_1 (dx).
\end{equation*}
\end{enumerate}
\end{proposition}

We can now prove Proposition~\ref{couplinglemma}.

\begin{proof}[Proof of Proposition~\ref{couplinglemma}]
The idea is to apply Proposition~\ref{condlaw} to the two laws $\P_{(X,Y)}$ and $\P_{(Y,Z)}$. This gives the existence of two conditional laws denoted by $\nu_X$ and $\nu_Z$ such that for each $B_1 , B_2 , B_3 \in \left( \mathcal{B}_1, \mathcal{B}_2, \mathcal{B}_3 \right) $,
\begin{equation*}
\P_{(X,Y)} (B_1 \times B_2) = \int_{B_2} \nu_X\left( y , B_1 \right) \P_Y( dy ) ~\mbox{and}~ \P_{(Y,Z)} (B_2 \times B_3) = \int_{B_2} \nu_Z\left( y , B_3 \right) \P_Y( dy ).
\end{equation*}
We can the define for each $B_1 , B_2 , B_3 \in \left( \mathcal{B}_1, \mathcal{B}_2, \mathcal{B}_3 \right)$,
\begin{equation*}
\P_{(X,Y,Z)} \left( B_1 \times B_2 \times B_3 \right) = \int_{B_2} \nu_X\left( y , B_1 \right) \nu_Z\left( y , B_3 \right) \P_Y( dy ).
\end{equation*}
Using standard tools from measure theory, one can then extend $\P_{(X,Y,Z)}$ into a measure on the $\sigma$-algebra $\mathcal{B}_1\otimes \mathcal{B}_2 \otimes \mathcal{B}_3$ and verify that this measure satisfies the property~\eqref{propPXYZ}.
\end{proof}

\subsection{Functional inequalities on the lattice}

In this section, we want to prove some functional inequalities for functions on the lattice $\Zd$, namely the Poincar\'e inequality, and the multiscale Poincar\'e inequality. These inequalities are known on $\Rd$ so the strategy of the proof is to extend functions defined on $\Zd$ to $\Rd$, to apply the inequalities to the extended functions and then show that the inequality obtained for the extended function is enough to prove the inequality for the discrete function.

The second inequality presented, called the multiscale Poincar\'e inequality, is a convenient tool to control the $L^2$-norm of a function by the spatial average of its gradient. It is proved in~\cite[Proposition 1.7 and Lemma 1.8]{armstrong2017quantitative}. The philosophy behind it comes from the theory of stochastic homogenization and roughly states that the usual Poincar\'e inequality can be refined by estimating the $L^2$-norm of a function by the spatial average of the gradient. This inequality is useful when one is dealing with rapidly oscillating functions, which frequently appear in homogenization. Indeed for these functions, the oscillations cancel out in the spatial average of the gradient, as a result these spatial averages are much smaller than the $L^2$-norm of the gradient. The resulting estimate is thus much more precise than the standard Poincar\'e inequality. Before stating the proposition, we recall the definition of the set $\mathcal{Z}_{m,n}$ given in~\eqref{def.mathcalzmn}.

\begin{proposition}[Poincar\'e and multiscale Poincar\'e inequalities]  \label{p.poincmultpoinc}
Let $\cu$ be a cube of $\Zd$ of size $R$ and a function $u : \cu \rightarrow \R$, then one has the inequality, for some constant $C := C(d) < \infty$,
\begin{equation} \label{e.poincWineq}
\sum_{x \in \cu} \left| u(x) - \left( u \right)_{\cu} \right|^2 \leq C R^2 \sum_{e \subseteq \cu} \left| \nabla u (e) \right|^2.
\end{equation}
If one assumes that $u = 0 $ on $\partial \cu$, then one has
\begin{equation} \label{e.poincineq}
\sum_{x \in \cu} \left| u(x)\right|^2 \leq C R^2 \sum_{e \subseteq \cu} \left| \nabla u (e) \right|^2.
\end{equation}
For each $n \in \N$, there exists a constant $C := C(d) < \infty$ such that for each function $u : \cu_n \rightarrow \R$,
\begin{equation} \label{e.multscalepoinc}
\frac{1}{\left| \cu_n\right|} \sum_{x \in \cu_n} \left| u(x) - \left( u \right)_{\cu_n} \right|^2  \leq C \sum_{e \subseteq \cu_n} \left| \nabla u (e) \right|^2 + C 3^n \sum_{k = 1}^n 3^k \left( \frac{1}{\left| \mathcal{Z}_{k,n} \right|} \sum_{y \in \mathcal{Z}_{k,n}} \left|\left\langle \nabla u \right\rangle_{z + \cu_k} \right|^2 \right) .
\end{equation}
If one assume that $u \in h^1_0 (\cu_n)$, then one has
\begin{equation} \label{e.multscalepoinch10}
\frac{1}{\left| \cu_n\right|} \sum_{x \in \cu_n} \left| u(x) \right|^2  \leq C \sum_{e \subseteq \cu_n} \left| \nabla u (e) \right|^2 + C 3^n \sum_{k = 1}^n 3^k \left( \frac{1}{\left|\mathcal{Z}_{k,n} \right|} \sum_{y \in \mathcal{Z}_{k,n}} \left|\left\langle \nabla u \right\rangle_{z + \cu_k} \right|^2 \right) .
\end{equation}
\end{proposition}

\begin{proof}
The idea is to construct a smooth function $\tilde u$ which is close to $u$ by first extending it to be piecewise constant on the cubes $z + \left( - \frac 12, \frac 12 \right)^d$, where $z \in \cu_m$. We then make this function smooth by taking the convolution with a smooth approximation of the identity supported in the ball $B_{1/2}$. It follows that $\tilde u (z) = u(z)$ for each $z \in \cu_m$ and that the following estimate holds: for each $z \in \cu_m$,
\begin{equation*}
\sup_{z +  \left( - \frac 12, \frac 12 \right)^d} \left| \nabla \tilde u (x) \right| \leq C \sum_{y\sim x} \left| u(y) - u (x) \right|.
\end{equation*}
One can then apply the Poincar\'e inequalities to the function $\tilde u$ and verify that this is enough to obtain~\eqref{e.poincWineq} and ~\eqref{e.poincineq}. The proof of~\eqref{e.multscalepoinc} follows the same lines, a proof of this inequality can be found in~\cite[Proposition A2]{AD2}. We note that the version stated here is a slight modification of the one which can be found there but can be deduced from it by applying the Cauchy-Schwarz inequality.

The version of the multiscale Poincar\'e inequality with $0$ boundary condition given in~\eqref{e.multscalepoinch10} cannot be found in~\cite[Proposition A2]{AD2}. Nevertheless the continous version of this inequality is a consequence of~\cite[Proposition 1.7 and Lemma 1.8]{armstrong2017quantitative}. The transposition to the discrete setting is identical to the proof given in~\cite[Proposition A2]{AD2}.
\end{proof}

\section{Subadditive quantities and their basic properties} \label{section3}

The goal of this section is to study the quantities $\nu$ and $\nu^*$ introduced in the previous sections. We prove a series of results about these quantities, which are reminiscent of the basic properties of $\nu$ and $\nu^*$ in stochastic homogenization, see~\cite[Lemma 1.1 and Lemma 2.2]{armstrong2017quantitative}. We first state these properties in the same proposition. Most of them are already known in the literature and in Remark~\ref{4.remar3.3}, we provide references for these results. We then prove the remaining results.
\begin{proposition}[Properties of $\nu$ and $\bar \nu^*$] \label{propofnunustar}
There exists a constant $C := C(d ,\lambda) < \infty$ such that the following properties hold:
\begin{itemize}
\item \textnormal{Subadditivity.} For each $n \in \N$ and each $p \in \Rd$,
\begin{equation}  \label{4.subaddnu}
\nu \left( \cu_{n+1} , p \right) \leq \nu \left( \cu_n  ,p \right) +  C \left( 1 + |p|^2 \right) 3^{-n}.
\end{equation}
For each $q \in \Rd$, 
\begin{equation}  \label{4.subaddnustar}
\nu^* \left( \cu_{n+1} , q \right) \leq \nu^* \left( \cu_n  ,q \right) +  C \left( 1 + |q|^2 \right) 3^{-n}.
\end{equation} \medskip
\item \textnormal{One-sided convex duality.} For each $p , q \in \Rd$ and each $n \in \N$,
\begin{equation}   \label{4.onesided}
\nu (\cu_n,p) + \nu^*(\cu_n,q) \geq p \cdot q - C 3^{-n}.
\end{equation} \medskip
\item \textnormal{Quadratic bounds.} there exists a small constant $c:= c(d , \lambda) > 0$ such that, for each $n \in \N$ and each $p\in \Rd$,
\begin{equation} \label{4.quadbounds}
- C + c |p|^2 \leq \nu \left( \cu_n , p \right) \leq C \left( 1 + |p|^2 \right),
\end{equation}
and for each $q \in \Rd$,
\begin{equation} \label{4.quadboundsnustar}
- C + c |q|^2 \leq \nu^* \left( \cu_n , q \right) \leq C \left( 1 + |q|^2 \right).
\end{equation}
\item \textnormal{Uniform convexity of $\nu$.}  For each $p_0 , p_1 \in \Rd$,
\begin{equation} \label{4.unifconv} 
\frac 1C |p_0 - p_1 |^2 \leq \frac 12 \nu(\cu_n, p_0) + \frac 12 \nu(\cu_n, p_1) -   \nu \left(\cu_n, \frac{p_0 + p_1}{2}\right) \leq C  |p_0 - p_1 |^2.
\end{equation}
\medskip
\item \textnormal{Convexity of $\nu^*$.} The mapping $q \rightarrow \nu^* (\cu_n , q)$ is convex. \medskip
\item \textnormal{$L^2$ bounds for the minimizers.} For each $p \in \Rd$, one has
\begin{equation} \label{4.L2boundnu}
\E \left[ \frac{1}{\left| \cu_n \right|} \sum_{e \subseteq \cu_n} |\nabla \phi_{n,p}(e)|^2 \right] \leq C (1 + |p|^2).
\end{equation}
For each $q \in \Rd$, one has
\begin{equation} \label{4.L2boundnustar}
\E \left[ \frac{1}{\left| \cu_n \right|} \sum_{e \subseteq \cu_n} |\nabla \psi_{n,q}(e)|^2 \right] \leq C (1 + |q|^2).
\end{equation}
\end{itemize}
\end{proposition}

\begin{remark} \label{4.remar3.3}
The reader can find in the literature the proofs of the following results.
\begin{enumerate}
\item The subadditivity of $\nu$ stated in~\eqref{4.subaddnu} is essentially proved by Funaki and Spohn in~\cite[Lemma II.1]{FS97}.
\item In Proposition~\ref{2sccomparison}, we prove a quantitative version of the subadditivity inequality~\eqref{4.subaddnustar} which is strictly stronger than the estimate~\eqref{4.subaddnustar}.
\item The quadratic bounds for $\nu$ stated in~\eqref{4.quadbounds} are elementary and we refer to the monograph of Funaki~\cite[Section 5.2]{Fu05}.
\item The uniform convexity of the finite volume surface tension $\nu$ stated in~\eqref{4.unifconv} was established by Deuschel Giacomin and Ioffe in~\cite[Lemma 3.6]{DGI00}.
\item The convexity of the mapping $q \rightarrow \nu^* \left( \cu_n , q \right)$ is a straightforward application of the Cauchy-Schwarz inequality.
\item The $L^2$ bounds for the minimizers~\eqref{4.L2boundnu} and~\eqref{4.L2boundnu} can be obtained the following way. Using the explicit formula for $\nu$ and $\nu^*$ together with a computation similar to~\cite[Lemma 2.11]{DGI00}, one derives the bounds
\begin{multline*}
\E \left[ \exp \left( \ep \sum_{e \subseteq \cu_n} \left| \nabla \phi_{n,q} \right|^2 \right) \right] \leq \exp \left( C \left| \cu_n \right| (1 + |p|^2) \right) \\ \mbox{and}~\E \left[ \exp \left( \ep \sum_{e \subseteq \cu_n} \left| \nabla \psi_{n,q} \right|^2 \right) \right] \leq \exp \left( C \left| \cu_n \right| (1 + |p|^2) \right),
\end{multline*}
for some $\ep := \ep(d , \lambda) > 0$ and $C := C(d , \lambda) < \infty$. By the Jensen inequality, this implies, and is in fact much stronger than, the desired estimates.
\end{enumerate}
\end{remark}

The two statements of Proposition~\ref{propofnunustar} which remain to be proved are the one-sided convex duality~\eqref{4.onesided} and the quadratic bound~\eqref{4.L2boundnustar} for $\nu^*$. They are established in Propositions~\ref{convexdual} and~\ref{p.quadbound} respectively.

\begin{remark}
From the subadditivity properties and the quadratic bounds, we obtain that for each $p,q \in \Rd$, the quantities $\nu \left( \cu_n,p \right)$ and $\nu^* \left( \cu_n , q \right)$ converge as $n \rightarrow \infty$. Moreover the limit satisfies the convexity and one-sided duality properties. This is summarized in the following proposition. 
\end{remark}

\begin{proposition} \label{def+propofarnu}
For each $p \in \Rd$ and $q \in \Rd$, the quantities $\nu \left( \cu_n, p \right)$ and $\nu^* \left( \cu_n , q \right)$ converge as $n \rightarrow \infty$. We denote by $\bar \nu (p)$ and $\bar \nu^* (q)$ their respective limits. Moreover, there exists a constant $C := C(d, \lambda) > \infty$ such that the following properties hold
\begin{itemize}
\item \textnormal{One-sided convex duality.} For each $p,q \in \Rd$,
\begin{equation*}
\bar \nu (p) + \bar \nu^*(q) \geq p \cdot q.
\end{equation*}
\item \textnormal{Quadratic bounds.} There exists a small constant $c := c(d , \lambda) > 0$ such that, for each $p \in \Rd$,
\begin{equation*}
- C + c |p|^2 \leq \bar \nu \left( p \right) \leq C \left( 1 + |p|^2 \right),
\end{equation*}
and, for each $q \in \Rd$,
\begin{equation*}
- C + c |q|^2 \leq \bar \nu^* \left( q \right) \leq C \left( 1 + |q|^2 \right).
\end{equation*}  \medskip
\item \textnormal{Convexity and uniform convexity.} The mapping $q \rightarrow \bar \nu^* \left( q \right)$ is convex and, for each $p_1 , p_2 \in \Rd$,
\begin{equation*} 
\frac 1C |p_0 - p_1 |^2 \leq \frac 12 \bar \nu(p_0) + \frac 12 \bar \nu( p_1) -   \bar \nu \left(\frac{p_0 + p_1}{2}\right) \leq C  |p_0 - p_1 |^2.
\end{equation*} \medskip
\end{itemize}
\end{proposition}

\begin{proof}
The properties mentioned in the proposition are valid for the quantities $\nu \left( \cu_n, p \right)$ and $\nu^* \left( \cu_n , q \right)$. Sending $n$ to infinity implies the results. For the surface tension $\nu$, this result was originally obtained by Funaki and Spohn~\cite[Proposition 1.1 (i)]{FS97}.
\end{proof}

\begin{remark}
The previous proposition proves the estimate~\eqref{unifconvbarnu} of Theorem~\ref{QuantitativeconvergencetothGibbsstate}. By the previous proposition, to prove~\eqref{dualconvnunustar}, there remains to show the upper bound
\begin{equation*}
\bar \nu^* \left( q \right) \leq \sup_{p \in \Rd} - \bar \nu \left( p \right) +p \cdot q,
\end{equation*}
since the lower bound follows from the one-sided convex duality. This upper bound will be proved later in the article. The uniform convexity of $\bar \nu^*$ stated in~\eqref{unifconvbarnustar} can then be deduced from~\eqref{dualconvnunustar} and~\eqref{unifconvbarnu}. 
\end{remark}

\subsection{Convex duality: lower bound}
We now turn to the proof of the one-sided convex duality for $\nu$ and $\nu^*$.

\begin{proposition}[One-sided convex duality] \label{convexdual}
There exists a constant $C := C(d , \lambda) < \infty$ such that for each $p , q \in \Rd$ and each $n \in \N$,
\begin{equation*}
\nu (\cu_n,p) + \nu^*(\cu_n,q) \geq p \cdot q - C 3^{-n}.
\end{equation*}
\end{proposition}

\begin{proof}
We recall the notation $\partial \cu_n$ and $\inte \cu_n$ to denote respectively the boundary and the interior of the cube $\cu_n$. We decompose the space $\mathring h^1(\cu_n)$ into three orthogonal subspaces
\begin{equation} \label{ortdecomp101}
\mathring h^1(\cu_n) = \mathring h^1\left( \partial \cu_n \right) \overset{\perp}{\oplus} \mathring h^1\left(\inte \cu_n \right) \overset{\perp}{\oplus} \R v,
\end{equation}
where we use a slight abuse of notation and denote by
\begin{equation*}
\mathring h^1\left( \partial \cu_n \right) := \left\{ \psi \in \mathring h^1(\cu_n)  ~:~\psi_{|\inte \cu_n} = 0   \right\} ~\mbox{and}~ \mathring h^1\left( \inte \cu_n \right) := \left\{ \psi \in \mathring h^1(\cu_n)  ~:~\psi_{|\partial \cu_n} = 0   \right\},
\end{equation*}
and where $v$ is the function defined by
\begin{equation*}
v =\frac{1}{|\cu_n|} \left(\sqrt{\frac{\left| \partial \cu_n \right|}{\left| \inte \cu_n \right|}}\indc_{ \inte \cu_n } -   \sqrt{\frac{\left| \inte \cu_n \right|}{\left| \partial \cu_n \right|}}   \indc_{ \partial \cu_n } \right),
\end{equation*}
so that it satisfies
\begin{equation*}
\sum_{x \in \cu_n} v(x) = 0 \mbox{ and } \sum_{x \in \cu_n} v(x)^2 = 1.
\end{equation*}
Since it is important in the proof, we note that, for each $n \in \N$,
\begin{equation} \label{dimh01bound}
\dim \,  \mathring h^1\left( \partial \cu_n \right) = \left| \partial \cu_n \right| - 1 \leq C 3^{(d-1)n}.
\end{equation}
We split the proof into 4 steps:
\begin{itemize}
\item in Step 1, we show that, for some constant $C := C(d , \lambda) < \infty$,
\begin{equation} \label{dualstep1}
\nu^*(\cu_n,q)  \geq \frac{1}{\left| \cu_n \right|} \log \left(  \int_{ \mathring h^1\left( \inte \cu_n \right) \oplus \R v} \exp \left(- \sum_{e \subset \cu_n }  V_e(  \nabla \psi (e))\right)  \, d\psi  \right) - C 3^{-n};
\end{equation}
\item in Step 2, we show that, for some constant $C := C(d) < \infty$,
\begin{equation} \label{dualstep22}
\nu \left( \cu_n , 0 \right) \geq - \frac{1}{\left| \cu_n \right|} \log \left(  \int_{ \mathring h^1\left( \inte \cu_n \right) \oplus \R v} \exp \left(- \sum_{e \subset \cu_n }  V_e(  \nabla \psi (e))  \right) \, d\psi \right) - C n3^{-dn};
\end{equation}
\item in Step 3, we combine the results of Steps 1 and 2 and prove that there exists a constant $C := C(d , \lambda) < \infty$ such that
\begin{equation*}
\nu^*(\cu_n,q) + \nu \left( \cu_n , 0 \right)   \geq -C 3^{-n};
\end{equation*}
\item in Step 4, we remove the assumption $p = 0$ and prove that, for each $p,q \in \Rd$,
\begin{equation*}
\nu^*(\cu_n,q) + \nu \left( \cu_n , p \right)   \geq p \cdot q -C 3^{-n}.
\end{equation*}
\end{itemize}

\medskip

\textit{Step 1.} First, by the identity~\eqref{ortdecomp101}, any function $\psi \in \mathring h^1(\cu_n)$ can be uniquely decomposed according to the formula
\begin{equation*}
\psi = \psi_1 + \psi_2 + t v,
\end{equation*}
with $\psi_1 \in \mathring h^1\left( \partial \cu_n \right)  $, $\psi_2 \in  \mathring h^1\left( \inte \cu_n \right)  $ and $t \in \R$. We note that, since each function $\psi_2$ in $\mathring h^1\left( \inte \cu_n \right)$ is equal to $0$ on the boundary of the cube $\cu_n$, we have, for each $q \in \Rd$,
\begin{equation*}
\sum_{e \subset \cu_n} q \cdot \nabla \psi_2 (e) = 0.
\end{equation*}
Since the function $v$ is constant on the boundary $\partial \cu_n$, we also have
\begin{equation*}
\sum_{e \subset \cu_n} q \cdot \nabla v(e) = 0.
\end{equation*}
To prove the inequality~\eqref{dualstep1}, it is sufficient to prove, for each $\psi_2 \in  \mathring h^1\left( \inte \cu_n \right) $ and each $t \in \R$,
\begin{multline} \label{dualstep2} 
\int_{ \mathring h^1\left( \partial \cu_n \right) } \exp \left( -\sum_{e \subset \cu_n } \left(V_e \left(\nabla \left( \psi_1 +  \psi_2 + tv \right) (e)  \right) - q \cdot \nabla \psi_1 (e) \right) \right) \, d\psi_1 \\
							\geq c^{3^{(d-1)n}} \exp \left(- \sum_{e \subset \cu_n } V_e \left( \nabla \left(   \psi_2 + tv \right) (e) \right)  \right),
\end{multline}
for some constant $c := c(d , \lambda) > 0$.
Indeed, the estimate~\eqref{dualstep1} is then obtained by integrating the previous inequality over $\mathring h^1\left( \inte \cu_n \right) \oplus \R v$. To prove~\eqref{dualstep2}, we use the following Taylor expansion
\begin{equation*}
V_e( \nabla \left( \psi_1 +  \psi_2 + tv \right) (e)) \leq  V_e \left( \nabla \left(  \psi_2 + tv \right) (e)\right) + V_e' \left( \nabla \left(  \psi_2 + tv \right) (e)  \right) \nabla \psi_1 (e) +   \frac{1}{2\lambda} \left| \nabla \psi_1 (e) \right|^2.
\end{equation*} 
This implies
\begin{align*}
\lefteqn{ \exp \left(  - \sum_{e \subset \cu_n} \left( V_e \left( \nabla \left( \psi_1 +  \psi_2 + tv \right) (e)\right)- q \cdot \nabla \psi_1 (e)\right) \right)}  \qquad \notag & \\ & 
			\geq \exp \left( - \sum_{e \subset \cu_n} \left(   V_e \left( \nabla \left(  \psi_2 + tv \right) (e)\right) + \left( V_e' \left( \nabla \left(  \psi_2 + tv \right) (e)  \right) - q(e) \right) \nabla \psi_1 (e) \right) \right) \\ & \qquad \times \exp \left(  - \sum_{e \subset \cu_n}  \frac{1}{2\lambda} \left| \nabla \psi_1 (e) \right|^2  \right).
\end{align*}
Using the inequality, for a bond $e = (x,y) \subseteq \cu_n$,
\begin{equation*}
\left| \nabla \psi_1 (e) \right|^2 = \left|  \psi_1 (x) - \psi_1 (y) \right|^2 \leq 2 \left|  \psi_1 (x)\right|^2 + 2 \left|\psi_1 (y) \right|^2,
\end{equation*}
and summing over all the bonds of the cube $\cu_n$ yields
\begin{equation*}
\sum_{e \subset \cu_n}  \left| \nabla \psi_1 (e) \right|^2 \leq 2 d \sum_{x \in \partial \cu_n} \psi_1(x)^2.
\end{equation*}
We then note that, for each $a \in \R$,
\begin{equation*}
\int_{\R} \exp \left( ax - \frac{d}{\lambda} x^2 \right) \, dx \geq \sqrt{ \frac{\lambda \pi}{d}}.
\end{equation*}
With the previous estimate and the equality~\eqref{dimh01bound}, one deduces
\begin{multline*}
\int_{\mathring h^1\left( \partial \cu_n \right)}  \exp \left( - \sum_{e \subset \cu_n}    \left( V_e' \left( \nabla \left(  \psi_2 + tv \right) (e)  \right) \nabla \psi_1 (e)- q \cdot \nabla \psi_1 (e) \right)  \right) \\ \times \exp \left(  \frac{d}{\lambda}  \sum_{x \in \partial \cu_n} \left| \psi_1 (x) \right|^2 \right) d \psi_1 \geq c^{3^{(d-1)n}},
\end{multline*}
for some constant $c := c(d , \lambda) > 0$.
Combining the few previous displays gives
\begin{multline*}
\int_{\mathring h^1\left( \partial \cu_n \right)}  \exp \left(  - \sum_{e \subset \cu_n}V_e \left( \nabla \left( \psi_1 +  \psi_2 + tv \right) (e)  \right)- q \cdot \nabla \psi_1 (e)\right) d \psi_1\\
			\geq c^{3^{(d-1)n}}  \exp \left( - \sum_{e \subset \cu_n}   V_e \left( \nabla \left(  \psi_2 + tv \right) (e) \right) \right).
\end{multline*}
This is precisely~\eqref{dualstep2}.

\medskip

\textit{Step 2.} We denote by 
\begin{equation*}
\tilde{v} :=\frac{1}{\sqrt{\left| \inte \cu_n \right|}}\indc_{\inte \cu_n}
\end{equation*}
so that $\sum_{x \in \cu_n} \tilde{v}(x)^2 = 1$.
Note that the two functions $v$ and $\tilde{v}$ are related by the formula
\begin{align*}
v + \frac{1}{|\cu_n|} \sqrt{\frac{\left| \inte \cu_n \right|}{\left| \partial \cu_n \right|}}  \indc_{\cu_n} & = \frac{1}{|\cu_n|} \left(  \sqrt{\frac{\left| \partial \cu_n \right|}{\left| \inte \cu_n \right|}} +  \sqrt{\frac{\left| \inte \cu_n \right|}{\left| \partial \cu_n \right|}} \right) \indc_{\inte \cu_n} \\
																					& =  \frac{\sqrt{\left| \inte \cu_n \right|}}{|\cu_n|} \left(  \sqrt{\frac{\left| \partial \cu_n \right|}{\left| \inte \cu_n \right|}} +  \sqrt{\frac{\left| \inte \cu_n \right|}{\left| \partial \cu_n \right|}} \right) \tilde{v}.
\end{align*}
To shorten the notation, we denote by
\begin{equation*}
\alpha_n := \frac{\sqrt{\left| \inte \cu_n \right|}}{|\cu_n|} \left(  \sqrt{\frac{\left| \partial \cu_n \right|}{\left| \inte \cu_n \right|}} +  \sqrt{\frac{\left| \inte \cu_n \right|}{\left| \partial \cu_n \right|}} \right).
\end{equation*}
Combining the two previous displays we obtain, for each edge $e \subseteq \cu_n$,
\begin{equation} \label{alphan.eq}
\nabla v (e) = \alpha_n \nabla \tilde{v} (e).
\end{equation} 
Note that, there exist two constants $c := c(d) > 0$ and $C := C(d) < \infty$ such that
\begin{equation} \label{boundalphandual}
c 3^{- \frac{(d-1)n}{2}} \leq \alpha_n \leq C 3^{- \frac{(d-1)n}{2}}.
\end{equation}
We then use the orthogonal decomposition
$
 h^1_0\left(\cu_{n}\right) = \mathring h^1\left( \inte \cu_n \right)  \overset{\perp}{\oplus} \R \tilde{v}
$
and the decomposition of the Lebesgue measure stated in~\eqref{decomplebmeasure} to obtain
\begin{multline*}
\int_{ h^1_0\left(\cu_{n}\right)} \exp \left(  - \sum_{e \subset \cu_{n}} V_e(\nabla \phi(e)) \right) \,   d\phi \\ = \int_\R \int_{  \mathring h^1\left( \inte \cu_n \right)  } \exp \left(  - \sum_{e \subset \cu_{n}} V_e(\nabla \phi(e) + t \nabla \tilde{v}(e)) \right) \,   d\phi dt.
\end{multline*}
Using the identity~\eqref{alphan.eq}, we obtain
\begin{align*}
\lefteqn{ \int_{ h^1_0\left(\cu_{n}\right)} \exp \left(  - \sum_{e \subset \cu_{n}} V_e(\nabla \phi(e)) \right) \,   d\phi } \qquad & \\ & 
					= \int_\R \int_{ \mathring h^1\left( \inte \cu_n \right)  } \exp \left(  - \sum_{e \subset \cu_{n}} V_e\left( \nabla \phi(e) + \frac{t}{\alpha_n} \nabla v(e) \right) \right) \,   d\phi dt \\
					&= \alpha_n \int_\R \int_{ \mathring h^1\left( \inte \cu_n \right)  } \exp \left(  - \sum_{e \subset \cu_{n}} V_e\left(\nabla \phi(e) + t \nabla v(e)\right) \right) \,   d\phi dt \\
					& = \alpha_n   \int_{ \mathring h^1\left( \inte \cu_n \right) \oplus \R v} \exp \left(- \sum_{e \subset \cu_n }  V_e(  \nabla \phi (e)) \right)  \, d\phi.
\end{align*}
Taking the logarithm, dividing by the volume $|\cu_n|$ and using the estimate~\eqref{boundalphandual}, we obtain the inequality~\eqref{dualstep22}.

\medskip

\textit{Step 3.} Combining the main results of Steps 1 and 2 gives
\begin{equation*}
\nu^*(\cu_n,q) + \nu \left( \cu_n , 0 \right)  \geq - C 3^{-n} - C n3^{-dn} \geq  - C 3^{-n}.
\end{equation*}

\medskip

\textit{Step 4.} Let $p \in \Rd $, define $\tilde{V}_e := V_e \left( p ( e) + \cdot \right)$ and denote by
\begin{equation*}
\tilde{\nu}(\cu_n,0) := - \frac 1 {|\cu_n|} \log \int_{ h^1_0(\cu_n)} \exp \left(  - \sum_{e \subset U} \tilde{V}_e(\nabla \phi(e)) \right) \, d\phi,
\end{equation*}
and, for every $q \in \Rd$,
\begin{equation*} 
\tilde{\nu}^*(\cu_n,q) := \frac{1}{|\cu_n|} \log \int_{\mathring h^1(\cu_n)} \exp \left( -\sum_{e \subset \cu_n} \left(\tilde{V}_e(\nabla \phi(e))  - q \cdot \nabla \phi(e)\right) \right) \, d\phi.
\end{equation*}
The functions $\tilde{V}_e$ satisfy the same property of uniform convexity property~\eqref{unifconvVi} as the functions $V_e$, thus one can apply the result of Steps 1, 2 and 3 with these functions. This gives, for every $q \in \Rd$,
\begin{equation*}
\tilde{\nu}(\cu_n,0) + \tilde{\nu}^*(\cu_n,q) \geq - C 3^{-n}.
\end{equation*}
We then note that
\begin{align*}
\tilde{\nu}(\cu_n,0) & = - \frac 1 {|\cu_n|} \log \int_{ h^1_0(\cu_n)} \exp \left(  - \sum_{e \subset U} \tilde{V}_e(\nabla \phi(e)) \right) \, d\phi \\
				& = - \frac 1 {|\cu_n|} \log \int_{ h^1_0(\cu_n)} \exp \left(  - \sum_{e \subset U}V_e(p \cdot e + \nabla \phi(e)) \right) \, d\phi \\
				& = \nu(\cu_n,p).
\end{align*}
Note also that, by translation invariance of the Lebesgue measure on $\mathring h^1(\cu_n)$, one can perform the change of variables $\phi := \phi - l_p$, where $l_p \in \mathring h^1 (\cu_n)$ is the affine function defined by $l_p(x) = p \cdot x$. This gives, for each $q \in \Rd$,
\begin{align*}
\tilde{\nu}^*(\cu_n,q) & =  \frac{1}{|\cu_n|} \log \int_{\mathring h^1(\cu_n)} \exp \left( -\sum_{e \subset \cu_n} \left(\tilde{V}_e(\nabla \phi(e))  - q \cdot \nabla \phi(e)\right) \right) \, d\phi \\
					& = \frac{1}{|\cu_n|} \log \int_{\mathring h^1(\cu_n)} \exp \left( -\sum_{e \subset \cu_n} \left( V_e( p \cdot e + \nabla \phi(e))  - q \cdot \nabla \phi(e)\right) \right) \, d\phi \\
				   	& = \frac{1}{|\cu_n|} \log \int_{\mathring h^1(\cu_n)} \exp \left( -\sum_{e \subset \cu_n} \left(V_e(\nabla \psi (e)) + (q \cdot e) (p \cdot e) -  q \cdot \nabla \psi (e)\right) \right) \, d\psi \\ 
					& = \nu^*(\cu_n,q) - p \cdot q.
\end{align*}
Combining the few previous displays yields, for each $p,q \in \Rd$,
\begin{equation*}
\nu(\cu_n,p) + \nu^*(\cu_n,q) \geq p \cdot q -  C 3^{-n}.
\end{equation*}
The proof of Proposition~\ref{convexdual} is complete.
\end{proof}

\subsection{Quadratic bounds for $\nu^*$}In this section, we prove the quadratic bounds for $\nu^\star$.

\begin{proposition}[Quadratic bounds for $\nu^*$] \label{p.quadbound}
There exist two constants $c := c(d,  \lambda) > 0$ and $C := C (d , \lambda) < \infty$ such that, for each $q \in \Rd$ and each $n \in \N$,
\begin{equation} \label{quadboundnustar} 
-C + c |q|^2 \leq \nu^* (\cu_n , q ) \leq C ( 1 + |q|^2).
\end{equation}
\end{proposition}

\begin{proof}
We now prove the estimate~\eqref{quadboundnustar}. We start with the upper bound. By the inequality~\eqref{4.subaddnustar}, we have, for each integer $n \in \N$ and each $q \in \R^d$, 
\begin{equation*}
\nu^* (\cu_n , q ) \leq \nu^* (\cu_1 , q ) + C(1 + |q|^2).
\end{equation*}
A straightforward computation gives the bound
\begin{equation*}
\nu^* (\cu_1 , q ) \leq C (1 + |q|^2).
\end{equation*}
Combining the two previous displays gives the upper bound of~\eqref{quadboundnustar}.

We then prove the lower bound, the idea is to use the one-sided convex duality proved in Proposition~\ref{convexdual} combined with the upper bound estimate~\eqref{4.quadbounds}. By Proposition~\ref{convexdual}, for each $p , q \in \Rd$,
\begin{equation*}
\nu (\cu_n,p) + \nu^*(\cu_n,q) \geq p \cdot q - C 3^{-n}.
\end{equation*}
Using~\eqref{4.quadbounds} and the bound $3^{-n} \leq 1$, the previous estimate becomes
\begin{equation*}
\nu^*(\cu_n,q) \geq p \cdot q - C (1 + |p|^2).
\end{equation*}
Choosing $p = q / 2C$ gives
\begin{equation*}
\nu^*(\cu_n,q) \geq \frac{|q|^2}{4C} - C.
\end{equation*}
This is the desired lower bound.
\end{proof}

\subsection{Two-scale comparison}  \label{section2scalesquant}The goal of this section is obtain a quantitative version of the subadditivity for the quantities $\nu$ and $\nu^*$ stated in Proposition~\ref{propofnunustar}. More precisely one wishes to derive a second variation type of statement, following the techniques of the calculus of variations: the objective is to construct, for $m < n$, a coupling between the random variables $\phi_{m,p}$ and $\phi_{n,p}$ (resp. $\psi_{m,p}$ and $\psi_{n,p}$) such that the $L^2$-norm of the gradient of their difference is controlled by the difference $\nu(\cu_m, p) - \nu(\cu_{n}, p)$. 

An interesting consequence of this estimate is that, since the sequence $\nu(\cu_n, p)$ converges, the difference $\nu(\cu_m, p) - \nu(\cu_{n}, p)$ is small when $m$ and $n$ are large: this implies that the gradient of the fields on two different scales are close in the $L^2$-norm.

\medskip

For any pair of integers $m,n \in \N$ with $m < n$, the triadic cube $\cu_{n}$ can be split into $3^{(n-m)d}$ cubes of the form $(z + \cu_m)$, with $z \in \mathcal{Z}_{m,n}$. We denote by $(\phi_z)_{z \in \mathcal{Z}_{m,n}}$ a family of random variables such that
\begin{itemize}
\item for each $z \in \mathcal{Z}_{m,n}$, $\phi_z$ takes values in $ h^1_0 \left( z + \cu_m\right)$ and the law of $\phi_z( \cdot - z)$ is $\P_{m,p}$;
\item the random variables $\phi_z$ are independent.
\end{itemize}

We can then, for each $z \in  \mathcal{Z}_{m,n}$, see $\phi_z$ as a random variable taking value in $ h^1_0 \left(\cu_{n}\right)$, by extending it to be $0$ on $\cu_{n} \setminus \left( z + \cu_m\right)$. We also denote by $\phi := \sum_{z \in  \mathcal{Z}_{m,n}}\phi_z$.

\begin{proposition}[Two-scale comparison for $\nu$] \label{2sccomparisonnu}
Given $m,n \in \N$ with $m<n$ and $p \in \Rd$, we consider the random variable $\phi$ defined in the previous paragraph and taking values in $ h^1_0 (\cu_{n})$. Then there exists a coupling between $\phi$ and $\phi_{n,p}$ and a constant $C := C(d , \lambda) < \infty$ such that,
\begin{equation} \label{e.2sccomparisonnu}
\frac{1}{|\cu_{n}|}\E \left[  \sum_{e \subseteq \cu_{n}} \left| \nabla \phi(e) - \nabla \phi_{n,p} (e)\right|^2  \right] \leq C \left( \nu (\cu_m , q) - \nu (\cu_{n} , q) \right) + C 3^{-\frac m2} (1 + |p|^2).
\end{equation}
\end{proposition}

\begin{proof}
We first introduce the set of vertices $\partial_{int} \cu_{m,n} \subseteq \cu_m$ which are on the boundary of one of the cubes $(z + \cu_m)$ but not on the boundary of the cube $\cu_{n}$, i.e.,
\begin{equation} \label{def.bn2}
\partial_{int} \cu_{m,n} := \left( \bigcup_{z \in  \mathcal{Z}_{m,n}} \partial \left( z + \cu_m \right) \right) \setminus \partial \cu_{n}.
\end{equation}
Note that the cardinality of this set satisfies the upper bound estimate
\begin{equation*}
\left| \partial_{int} \cu_{m,n} \right| \leq C 3^{dn-m}.
\end{equation*}

An idea to obtain~\eqref{e.2sccomparisonnu} relies on the second variation formula from calculus of variations applied to the functional $\mathcal{F}_{n,p}$: by considering the optimal coupling between the laws of $\phi$ and $\phi_{n,p}$, one can use the displacement convexity of the entropy and the uniform convexity of $V_e$ to obtain uniform convexity for the functional $\mathcal{F}_{n,p}$ and then apply the standard proof of the second variation formula for uniformly convex functional.

Unfortunately a technical problem has to be treated along the way: with the definition of the function $\phi$, one has
\begin{equation*}
\forall x \in \partial_{int} \cu_{m,n}, \, \phi(x) = 0.
\end{equation*} 
A consequence of the previous identity is that the law of $\phi$ is not absolutely continuous with respect to the Lebesgue measure on $h^1_0 (\cu_{n})$ and thus its entropy is infinite. Nevertheless, this is the only obstruction and to remedy this, the idea is to add a few additional random variables which are small and whose only purpose is to make the entropy of $\phi$ finite. We consequently introduce a random variable $X $ taking values in $h^1_0 (\cu_{n})$ and satisfying
\begin{itemize}
\item for each $x \in \partial_{int} \cu_{m,n}$, the law of $X(x)$ is uniform on $[0,1]$ and for each $x \in \cu_{n+1} \setminus \partial_{int} \cu_{m,n}$, $X(x) = 0$;
\item the $\R$-valued random variables $\left(X(x)\right)_{x \in \partial_{int} \cu_{m,n}}$ are independent;
\item the random variables $X$ and $\phi$ are independent.
\end{itemize}
We then consider the random variable $\phi' := \phi +  X$. It is a random variable taking values in the space~$h^1_0 (\cu_{n})$. Moreover, by construction, we see that this random variable is absolutely continuous with respect to the Lebesgue measure on $h^1_0 (\cu_{n})$. We denote by $\P_{\phi'}$ its law. The idea to keep in mind is that the random variable $\phi'$ is a small perturbation of the random variable $\phi$ and thanks to that it is possible to obtain estimates on $\phi$ from estimates on $\phi'$. This is carried out in Steps 3 and 4 of the proof.

\medskip

We then split the proof into 6 steps.
\begin{itemize}
\item In Step 1, we compute the entropy of $\P_{\phi'}$ and prove that
\begin{equation} \label{step12sccompnu}
H \left( \P_{\phi'} \right) = 3^{(n-m)d} H \left( \P_{m,p}  \right),
\end{equation}
where the entropy on the left-hand side is computed with respect to the Lebesgue measure on $h^1_0 (\cu_{n})$ and the entropy on the right-hand side is computed with respect to the Lebesgue measure on $h^1_0 (\cu_{m})$.
\item In Step 2, we consider the optimal coupling between the random variables $\phi'$ and $\phi_{n,p}$ and prove that, under this coupling,
\begin{multline} \label{2scnu.step2}
\E \left[ \frac1{|\cu_{n}|} \sum_{e \subseteq \cu_{n}} \left| \nabla \phi' (e) - \nabla \phi_{n,p} (e)\right|^2  \right] \\ \leq C \left( \frac{1}{|\cu_{n}|}\E \left[ \sum_{e \subseteq \cu_{n}} V_e \left( p ( e) +  \nabla \phi' (e) \right) \right] + \frac{1}{|\cu_{n}|} H \left( \P_{\phi'}  \right) - \nu (\cu_{n} , p) \right).
\end{multline}
\item In Step 3, we estimate the term on the right-hand side of~\eqref{2scnu.step2} and prove
\begin{equation} \label{2scnu.step3}
  \frac{1}{|\cu_{n}|}\E \left[ \sum_{e \subseteq \cu_{n}} V_e \left( p ( e) +  \nabla \phi' (e) \right) \right]  \leq \frac{1}{|\cu_{n}|}\E \left[ \sum_{e \subseteq \cu_{n}} V_e \left( p ( e) +  \nabla \phi (e) \right) \right]  + C 3^{-\frac m2} (1+|p|).
\end{equation}
\item In Step 4, we combine the main results of Steps 1 and 3 to obtain
\begin{equation*}
 \frac{1}{|\cu_{n}|}\E \left[ \sum_{e \subseteq \cu_{n}} V_e \left(p ( e) + \nabla \phi' (e) \right) \right] + \frac{1}{|\cu_{n}|} H \left( \P_{\phi'} \right) \leq \nu\left(\cu_n , p \right) + C 3^{-\frac m2} (1+|p|^2).
\end{equation*}
\item In Step 5, we estimate the term on the left-hand side of~\eqref{2scnu.step2}. We replace the random variable $\phi'$ by the random variable $\phi$ and show that this operation can be performed up to a small error term: we have the estimate
\begin{multline} \label{2scnu.step4}
\E \left[  \frac{1}{|\cu_{n}|} \sum_{e \subseteq \cu_{n}} \left| \nabla \phi (e) - \nabla \phi_{n,p} (e)\right|^2  \right] \\  \leq C \E \left[  \frac{1}{|\cu_{n}|} \sum_{e \subseteq \cu_{n}} \left| \nabla \phi' (e) - \nabla \phi_{n,p} (e)\right|^2  \right] + C3^{- \frac{m}{2}}(1 + |p|).
\end{multline}
We note that to compute the expectation on the left-hand side, we used the coupling between the random variables $\phi$ and $\phi_{n,p}$ which is induced by the coupling between the random variables $\phi'$ and $\phi_{n,p}$.
\item In Step 6, we combine the main results of Steps 2, 3, 4 and 5 to obtain~\eqref{e.2sccomparisonnu}.
\end{itemize}

\medskip

\textit{Step 1.} The idea to obtain~\eqref{step12sccompnu} is to use Proposition~\ref{entropycouplesum} pertaining to the entropy of a pair of random variables. To this end, we note that one has the orthogonal decomposition
\begin{equation*}
h^1_0 (\cu_{n}) = \underset{z \in \mathcal{Z}_{m,n}}{\oplus} h^1_0 (z + \cu_{m}) \overset{\perp}{\oplus} \R^{\partial_{int} \cu_{m,n}}, 
\end{equation*}
where $\R^{\partial_{int} \cu_{m,n}}$ stands for the set of functions from $\partial_{int} \cu_{m,n}$ to $\R$.

\medskip

Using the previous remark, one can apply Proposition~\ref{entropycouplesum2} with $Y:= \phi$, $Z:= X$ and consequently $Y + Z = \phi'$. This leads to
\begin{equation*}
H\left(\P_{\phi'} \right) = H\left(\P_\phi \right) + H\left( \P_X  \right),
\end{equation*}
where the entropy of the random variable $\phi'$ (resp. $\phi$ and $X$) is computed with respect to the Lebesgue measure on $h^1_0 (\cu_{n})$ (resp. $\underset{z \partial_{int} \cu_{m,n}}{\oplus} h^1_0 (z + \cu_{n}) $ and $\R^{\partial_{int} \cu_{m,n}}$).

On the one hand, since the random variables $(X(x))_{x \in \partial_{int} \cu_{m,n}}$ are independent and of law uniform on $[0,1]$, one has
\begin{equation*}
H\left(\P_X\right) = 0.
\end{equation*}

On the other hand, one also has the equality $\phi' = \sum_{z \in \partial_{int} \cu_{m,n}} \phi_z$. Using the independence of the family of random variables $\left( \phi_z \right)_{z \in \mathcal{Z}_{m,n} }$, that for each $z \in \mathcal{Z}_{m,n}$, $\phi_z(\cdot - z)$ has law $\P_{m,p}$ and Proposition~\ref{entropycouplesum2}, with $3^{(n-m)d}$ random variables instead of two, one deduces
\begin{equation*}
H\left(\P_\phi \right) = \sum_{z \in \mathcal{Z}_{m,n}} H\left(\P_{\phi_z} \right) = 3^{(n-m)d} H \left( \P_{m,p}  \right).
\end{equation*}
Combining the few previous displays yields~\eqref{step12sccompnu} and completes the proof of Step 1.

\medskip

\textit{Step 2.} Consider the optimal coupling with respect to the $L^2$ scalar product on $h^1_0 \left(\cu_{n}\right)$ between $\phi'$ and $\phi_{n,p}$ and denote by $\P_{\frac{\phi'+ \phi_{n,p}}{2}}$ the law of the random variable $\frac{\phi'+ \phi_{n,p}}{2}$ under this coupling. Using Proposition~\ref{p.disconv} about the displacement convexity of the entropy, one obtains
\begin{equation*}
H\left(\P_{\frac{\phi'+ \phi_{n,p}}{2}} \right) \leq \frac12 H\left(\P_{\phi'} \right)  + \frac12 H\left(\P_{n,p} \right) .
\end{equation*}
By the uniform convexity of $V_e$, one has
\begin{align*}
\lefteqn{\lambda \E \left[ \sum_{e \subseteq \cu_{n}} \left| \nabla \phi' (e) - \nabla \phi_{n,p} (e)\right|^2  \right]} \qquad & \\ & \leq \E \left[ \sum_{e \subseteq \cu_{n}} V \left( p (e )+ \nabla \phi'(e) \right) \right] +  \E \left[ \sum_{e \subseteq \cu_{n}} V \left( p ( e) + \nabla \phi_{n,p}(e) \right) \right] \\ & \quad - 2 \E \left[ \sum_{e \subseteq \cu_{n}} V \left( p( e) + \frac{\nabla \phi_{n,p}(e) + \nabla \phi'(e)}{2} \right) \right].
\end{align*}
By definition of $\nu$, the following equality holds
\begin{equation*}
\E \left[ \sum_{e \subseteq \cu_{n}} V \left( p \cdot e + \nabla  \phi_{n,p}(e) \right) \right] + H\left(\P_{n,p} \right) = \nu(\cu_{n},p).
\end{equation*}
Using the variational formulation for $\nu$ given in Proposition~\ref{varfornuprop}, one derives the inequality
\begin{equation*}
 \nu(\cu_{n},p)  \leq \E \left[ \sum_{e \subseteq \cu_{n}} V \left( p ( e) + \frac{\nabla \phi_{n,p}(e) + \nabla \phi'(e)}{2} \right) \right] + H\left(\P_{\frac{\phi'+ \phi_{n,p}}{2}}  \right).
\end{equation*}
Combining the few previous displays provides the inequality
\begin{multline*}
\E \left[ \sum_{e \subseteq \cu_{n}} \left| \nabla \phi' (e) - \nabla \phi_{n,p} (e)\right|^2  \right] \\ \leq C \left( \frac{1}{|\cu_{n}|}\E \left[ \sum_{e \subseteq \cu_{n}} V_e \left( p ( e) +  \nabla \phi' (e) \right) \right] + \frac{1}{|\cu_{n}|} H \left( \P_{\phi'}  \right) - \nu (\cu_{n} , p) \right).
\end{multline*}
This is precisely~\eqref{2scnu.step2} and the proof of Step 2 is complete.

\medskip

\textit{Step 3.} The main goal of this step is to prove the following estimate
\begin{equation*}
 \frac{1}{|\cu_{n}|}\E \left[ \sum_{e \subseteq \cu_{n}} V_e \left( p (e) +  \nabla \phi' (e) \right) \right]  \leq \frac{1}{|\cu_{n}|}\E \left[ \sum_{e \subseteq \cu_{n}} V_e \left( p (e) +  \nabla \phi (e) \right) \right]  + C 3^{-\frac n2} (1+|p|).
\end{equation*}
To this end, we recall that the field $\phi'$ is defined from the field $\phi$ according to the formula
\begin{equation*}
\phi' := \phi + X,
\end{equation*}
and that the random variable $X$ is supported on the vertices of $\partial_{int} \cu_{m,n}$. We denote by $B_{m,n}^{X}$ the set of bonds of $\cu_{n}$ where the gradient of $X$ is supported, i.e.,
\begin{equation*}
B_{m,n}^{X} := \left\{ (x,y) \in \B_d (\cu_{n}) \,: \, x \in \partial_{int} \cu_{m,n}  ~ \mbox{or} ~ y \in \partial_{int} \cu_{m,n} \right\}.
\end{equation*}
One can estimate the cardinality of $B_{m,n}^{X}$ according to the formula
\begin{equation} \label{e.cardbn1}
\left| B_{m,n}^{X}  \right| \leq C 3^{dn - m}.
\end{equation}
We then split the sum and use that the gradient of $X$ is supported on the set $B_{m,n}^{X}$ to obtain
\begin{align} \label{splitstep22Sccomp}
 \sum_{e \subseteq \cu_{n}} V_e \left( p (e) +  \nabla \phi' (e) \right) &= \sum_{e \subseteq \cu_{n} \setminus B_{m,n}^{X}} V_e \left( p ( e) +  \nabla \phi' (e) \right) + \sum_{e \in B_{m,n}^{X}}  V_e \left( p( e) +  \nabla \phi' (e) \right) \\
														& =  \sum_{e \subseteq \cu_{n} \setminus B_{m,n}^{X}} V_e \left( p (e )+  \nabla \phi (e) \right) + \sum_{e \in B_{m,n}^{X}}  V_e \left( p ( e) +  \nabla \phi' (e) \right). \notag
\end{align}
The second term on the right-hand side can be estimated by using the uniform convexity of $V_e$ and a Taylor expansion,
\begin{align*}
\lefteqn{ \sum_{e \in B_{m,n}^{X}}  V_e \left( p ( e) +  \nabla \phi' (e) \right) } \qquad & \\ & = \sum_{e \in B_{m,n}^{X}}  V_e \left( p( e) +  \nabla \phi (e) + \nabla X(e) \right) \\
													& \leq   \sum_{e \in B_{m,n}^{X}}  V_e \left( p ( e) +  \nabla \phi (e)  \right) + V_e' \left( p ( e) +  \nabla \phi (e)  \right) \nabla X(e) + \frac1\lambda |\nabla X(e)|^2.
\end{align*}
By definition of $X$, its gradient is bounded by $1$, and by the assumption made on the elastic potential $V_e$, one has
\begin{equation*}
\forall x \in \R, \hspace{3mm} |V_e'(x)| \leq \frac1\lambda |x|.
\end{equation*}
Using these ideas, the previous estimate can be rewritten
\begin{equation*}
\sum_{e \in B_{m,n}^{X}}  V_e \left( p ( e) +  \nabla \phi' (e) \right)  \leq \sum_{e \in B_{m,n}^{X}}  \left( V_e \left( p ( e) +  \nabla \phi (e)  \right)  +  |\nabla \phi(e)| \right) + C\left| B_{m,n}^{X} \right| (1 + |p|).
\end{equation*}
Using the estimate on the cardinality of $B_{m,n}^{X}$ given in~\eqref{e.cardbn1} and the Cauchy-Schwarz inequality, one has
\begin{multline*}
\sum_{e \in B_{m,n}^{X}}  V_e \left( p ( e) +  \nabla \phi' (e) \right) \leq \sum_{e \in B_n^{X}} V_e \left( p (e) +  \nabla \phi (e) \right) \\ + 3^{\frac{dn-m}{2}} \left( \sum_{e \in B_{m,n}^{X}}  |\nabla \phi (e)|^2\right)^\frac12 + C  3^{dn - m}  (1 + |p|). 
\end{multline*}
Taking the expectation and dividing by the volume $|\cu_{n}|$ gives
\begin{multline*}
\E \left[ \frac1{|\cu_{n}|}\sum_{e \in B_{m,n}^{X}}  V_e \left( p (e) +  \nabla \phi' (e) \right) \right] \leq \E \left[  \frac1{|\cu_{n}|} \sum_{e \in B_{m,n}^{X}} V_e \left( p (e) +  \nabla \phi (e) \right) \right] \\+ 3^{-\frac{dn +m}{2}}  \E \left[ \left( \sum_{e \in B_{m,n}^{X}}  |\nabla \phi (e)|^2\right)^\frac12\right] + C  3^{-m}  (1 + |p|).
\end{multline*}
By the Cauchy-Schwarz inequality, we further obtain
\begin{multline*}
\E \left[\frac1{|\cu_{n}|} \sum_{e \in B_{m,n}^{X}}  V_e \left( p (e) +  \nabla \phi' (e) \right) \right] \leq \E \left[ \frac1{|\cu_{n}|} \sum_{e \in B_{m,n}^{X}} V_e \left( p( e) +  \nabla \phi (e) \right) \right] \\+  3^{-\frac{dn+m}{2}}   \E \left[\sum_{e \in B_{m,n}^{X}}  |\nabla \phi (e)|^2\right]^\frac 12 + C  3^{-m}  (1 + |p|).
\end{multline*}
Combining the previous display with~\eqref{splitstep22Sccomp} gives
\begin{multline*}
\E \left[ \frac1{|\cu_{n}|} \sum_{e \subseteq \cu_{n}} V_e \left( p ( e) +  \nabla \phi' (e) \right) \right] \leq \E \left[ \frac1{|\cu_{n}|} \sum_{e \subseteq \cu_{n}} V_e \left( p ( e) +  \nabla \phi (e) \right) \right] \\+  3^{-\frac{dn+m}{2}}   \E \left[\sum_{e \in B_{m,n}^{X}}  |\nabla \phi (e)|^2\right]^\frac 12 + C  3^{-m}  (1 + |p|).
\end{multline*}
The proof of the estimate~\eqref{2scnu.step3} is almost complete: we note that by the bound~\eqref{4.L2boundnu} stated in Proposition~\ref{propofnunustar} and by the definition of the random variable $\phi$, one has the energy estimate 
\begin{equation} \label{L2boundpatch}
\E \left[ \frac{1}{|\cu_{n}|}\sum_{e \subseteq \cu_{n}}  |\nabla \phi (e)|^2\right] \leq C (1 + |p|^2).
\end{equation}
This immediately implies the desired result since the following estimate holds
\begin{equation*}
 \E \left[\sum_{e \in B_{m,n}^{X}}  |\nabla \phi (e)|^2\right]^\frac 12 \leq  \E \left[\sum_{e \subseteq \cu_{n}}  |\nabla \phi (e)|^2\right]^\frac 12.
\end{equation*}

\medskip

\textit{Step 4.} With the same proof as in Step 2, we can decompose the following sum
\begin{multline*}
\frac{1}{|\cu_{n}|}\E \left[ \sum_{e \subseteq \cu_{n}} V_e \left( p (e) +  \nabla \phi (e) \right) \right] \\ = \frac{1}{|\cu_{n}|}\E \left[ \sum_{z \in \mathcal{Z}_{m,n}} \sum_{e \subseteq \left( z +\cu_m \right)} V_e \left( p( e) +  \nabla \phi_z (e) \right) + \sum_{e \in B_{m,n}} V_e \left( p ( e) \right) \right].
\end{multline*}
Using the estimate $V(x) \leq \frac1 \lambda |x|^2$ and the bound on the cardinality of the set $B_{m,n}$
\begin{equation*}
|B_{m,n}| \leq C 3^{dn-m},
\end{equation*}
 one obtains
\begin{multline*}
\frac{1}{|\cu_{n}|}\E \left[ \sum_{e \subseteq \cu_{n}} V_e \left( p( e) +  \nabla \phi (e) \right) \right] \\ \leq \sum_{z \in \mathcal{Z}_{m,n}} \E \left[ \frac{1}{|\cu_{n}|}  \sum_{e \subseteq \left( z +\cu_m \right)} V_e \left( p (e) +  \nabla \phi_z (e) \right) \right] + C 3^{-m} (1 + |p|^2).
\end{multline*}
Combining this inequality with the main result~\eqref{2scnu.step3} of Step 3, we obtain
\begin{multline*}
 \frac{1}{|\cu_{n}|}\E \left[ \sum_{e \subseteq \cu_{n}} V_e \left( p( e) +  \nabla \phi' (e) \right) \right]  \leq \frac{1}{|\cu_{n}|} \sum_{z \in \mathcal{Z}_{m,n}} \E \left[   \sum_{e \subseteq \left( z + \cu_m \right)} V_e \left( p (e) +  \nabla \phi_z (e) \right) \right] \\
																																		+  C 3^{-m} (1 + |p|^2) + C 3^{-\frac m2} (1+|p|).
\end{multline*}
The error term can be simplified according to the formula
\begin{multline*}
\frac{1}{|\cu_{n}|}\E \left[ \sum_{e \subseteq \cu_{n}} V_e \left( p (e) +  \nabla \phi' (e) \right) \right] \\ \leq \frac{1}{|\cu_{n}|} \sum_{z \in \mathcal{Z}_{m,n}} \E \left[   \sum_{e \subseteq \left( z + \cu_m \right)} V_e \left( p \cdot e +  \nabla \phi_z (e) \right) \right] +  C 3^{-\frac m2} (1 + |p|^2).
\end{multline*}
Adding the entropy of the law $\P_{\phi'}$ and using the equality proved in Step 1, we obtain
\begin{align*}
\lefteqn{\frac{1}{|\cu_{n}|}\E \left[ \sum_{e \subseteq \cu_{n}} V_e \left( p(e) +  \nabla \phi' (e) \right) \right] + \frac{1}{|\cu_{n}|} H \left( \P_{\phi'}\right)} \qquad & \\ &
																							\leq 3^{-d(n-m)}\sum_{z \in \mathcal{Z}_{m,n}} \left( \E \left[ \frac{1}{|\cu_{m}|}  \sum_{e \subseteq \left( z+ \cu_m \right)} V_e \left( p ( e) +  \nabla \phi_z (e) \right) \right] + \frac{1}{|\cu_{m}|} H \left( \P_{m,p} \right) \right) \\ & \qquad  +  C 3^{-\frac m2} (1 + |p|^2).
\end{align*}
Since the law of $\phi_z$ in the cube $(z + \cu_m)$ is the probability measure $\P_{m,p}$, we have, for each $z \in \mathcal{Z}_{m,n}$,
\begin{equation*}
\E \left[ \frac{1}{|\cu_{m}|}  \sum_{e \subseteq z + \cu_m} V_e \left( p ( e) +  \nabla \phi_z (e) \right) \right] + \frac{1}{|\cu_{m}|} H \left( \P_{m,p}\right) = \nu(\cu_m , p).
\end{equation*}
Combining the two previous displays shows
\begin{equation*}
\frac{1}{|\cu_{n}|}\E \left[ \sum_{e \subseteq \cu_{n}} V_e \left( p ( e) +  \nabla \phi' (e) \right) \right] + \frac{1}{|\cu_{n}|} H \left( \P_{\phi'} \right) \leq  \nu(\cu_m , p) +   C 3^{-\frac m2} (1 + |p|^2)
\end{equation*}
and the proof of Step 4 is complete.

\medskip

\textit{Step 5.} We proceed as in Step 3 and use that the gradient of $X$ is supported on the set $B_{m,n}^{X}$ to decompose the sum
\begin{multline} \label{firstpartstep5.est}
 \sum_{e \subseteq \cu_{n}} \left| \nabla \phi' (e) - \nabla \phi_{n,p} (e)\right|^2 = \sum_{e \in B_{m,n}^{X}} \left| \nabla \phi (e) + \nabla X (e) - \nabla \phi_{n,p} (e)\right|^2 \\ +  \sum_{e \subseteq \cu_{n} \setminus B_{m,n}^{X}} \left| \nabla \phi (e) - \nabla \phi_{n,p} (e)\right|^2.
\end{multline}
We then expand the first term on the right-hand side
\begin{multline*}
\sum_{e \in B_{m,n}^{X}} \left| \nabla \phi (e) + \nabla X (e) - \nabla  \phi_{n,p} (e)\right|^2  \\= \sum_{e \in B_{m,n}^{X}} \left| \nabla  \phi (e) - \nabla  \phi_{n,p} (e)\right|^2 + 2 \sum_{e \in B_{m,n}^{X}} \left( \nabla  \phi (e) - \nabla  \phi_{n,p} (e)\right) \nabla X (e)  + \sum_{e \in B_{m,n}^{X}} \left|  \nabla X (e) \right|^2.
\end{multline*}
Using that the gradient of $X$ is bounded by $1$, the Cauchy-Schwarz inequality and the upper bound on the cardinality of $B_{m,n}^{X}$, one further obtains
\begin{multline*}
\sum_{e \in B_{m,n}^{X}} \left| \nabla  \phi (e) + \nabla X (e) - \nabla  \phi_{n,p} (e)\right|^2 \\ \geq  \sum_{e \in B_{m,n}^{X}} \left| \nabla  \phi (e) - \nabla  \phi_{n,p} (e)\right|^2  - C 3^{\frac{dn-m}{2}}\left( \sum_{e \in B_{m,n}^{X}} \left| \nabla  \phi (e) - \nabla  \phi_{n,p} (e)\right|^2\right)^\frac 12 - C 3^{dn-m}.
\end{multline*}
Dividing by $|\cu_{n}|$ on both sides of the previous inequality and taking the expectation gives
\begin{multline*}
\E \left[ \frac{1}{|\cu_{n}|}\sum_{e \in B_{m,n}^{X}} \left| \nabla  \phi (e) + \nabla X (e) - \nabla  \phi_{n,p} (e)\right|^2 \right]  \geq  \E \left[ \frac{1}{|\cu_{n}|} \sum_{e \in B_{m,n}^{X}} \left| \nabla  \phi (e) - \nabla  \phi_{n,p} (e)\right|^2 \right] \\  - C 3^{-\frac{dn+m}{2}} \E \left[ \left( \sum_{e \in B_{m,n}^{X}} \left| \nabla  \phi (e) - \nabla  \phi_{n,p} (e)\right|^2 \right)^\frac12 \right] - C 3^{-m}.
\end{multline*}
By the identity~\eqref{firstpartstep5.est} and the Cauchy-Schwarz inequality, one obtains
\begin{multline*}
\E \left[ \frac{1}{|\cu_{n}|}\sum_{e \subseteq \cu_{n}} \left| \nabla \phi' (e) - \nabla  \phi_{n,p} (e)\right|^2 \right] \geq  \E \left[ \frac{1}{|\cu_{n}|} \sum_{e \subseteq \cu_{n}} \left| \nabla  \phi (e) - \nabla  \phi_{n,p} (e)\right|^2 \right]  \\ - C 3^{-\frac{dn+m}{2}} \E \left[ \sum_{e \in B_{m,n}^{X}} \left| \nabla  \phi (e) - \nabla \phi_{n,p} (e)\right|^2 \right]^\frac 12- C 3^{-m}.
\end{multline*}
There remains to estimate the second term on the right-hand side of the previous equation. To do so, we recall the energy estimate already introduced in Step 4,
\begin{equation*} 
\E \left[ \frac{1}{|\cu_{n}|} \sum_{e \subseteq \cu_{n}}  |\nabla \phi (e)|^2 \right] \leq C (1 + |p|^2).
\end{equation*}
using this estimate and the inequality~\eqref{4.L2boundnu} of Proposition~\ref{propofnunustar}, one obtains
\begin{equation*}
\E \left[ \frac{1}{|\cu_{n}|} \sum_{e \subseteq \cu_{n}}  |\nabla \phi_{n,p} (e)|^2 \right] \leq C (1 + |p|^2).
\end{equation*}
Combining the three previous displays shows
\begin{equation*}
\E \left[ \frac{1}{|\cu_{n}|}\sum_{e \subseteq \cu_{n}} \left| \nabla \phi' (e) - \nabla \phi_{n,p} (e)\right|^2 \right] \geq  \E \left[ \frac{1}{|\cu_{n}|} \sum_{e \subseteq \cu_n} \left| \nabla \phi (e) - \nabla \phi_{n,p} (e)\right|^2 \right] - C 3^{-\frac{m}{2}} (1 + |p|),
\end{equation*}
and the proof of Step 5 is complete.

\medskip

\textit{Step 6.} The conclusion. First, combining the main results of Steps 2 and 4 gives,
\begin{equation*} 
\E \left[ \sum_{e \subseteq \cu_{n}} \left| \nabla \phi' (e) - \nabla \phi_{n,p} (e)\right|^2  \right] \leq C \left( \nu(\cu_m,p) - \nu (\cu_{n} , p) \right) + C 3^{- \frac m2} (1 + |p|^2).
\end{equation*}
Then by the main result of Step 5, we obtain
\begin{equation*}
\E \left[ \sum_{e \subseteq \cu_{n}} \left| \nabla \phi (e) - \nabla \phi_{n,p} (e)\right|^2  \right] \leq C \left( \nu(\cu_m,p) - \nu (\cu_{n} , p) \right) + C 3^{- \frac m2} (1 + |p|^2).
\end{equation*}
This is the estimate~\eqref{e.2sccomparisonnu} and the proof of Proposition~\ref{2sccomparisonnu} is complete.
\end{proof}

We now want to prove a version of the two-scale comparison for $\nu^*$. Similarly to what was done in Proposition~\ref{2sccomparisonnu}, we fix two integers $m,n \in \N$ such that $m<n$ and define a family of random variables $\psi_z$, for $z \in \mathcal{Z}_{m,n}$, such that
\begin{itemize}
\item for each $z \in \mathcal{Z}_{m,n}$, $\psi_z$ takes values in $\mathring h^1 (z + \cu_{n})$, is equal to $0$ in $\cu_{n} \setminus \left( z + \cu_m\right)$ and the law of $\psi_z \left( \cdot - z \right)$ is the probability measure $\P_{m,q}^*$; \smallskip
\item the random variables $\left( \psi_z\right)_{z \in \mathcal{Z}_{m,n}}$ are independent.
\end{itemize}

We also denote by $\psi' := \sum_{z \in \mathcal{Z}_{m,n}} \psi_z$ and by $\P_{\psi'}$ its law, it is a probability measure on $\mathring h^1 (\cu_{n})$.

\begin{proposition}[Two-scale comparison for $\nu^*$] \label{2sccomparison}
Given a pair of integers $n,m \in \N$ satisfying $m < n$ and $q \in \Rd$, we consider the random variables $\psi_z$ and $\psi'$ defined in the previous paragraph. Then there exist a coupling between $\psi'$ and $\psi_{n,q}$ and a constant $C := C(d , \lambda) < \infty$ such that,
\begin{multline} \label{e.2sccomparison}
\E \left[ \frac1{|\cu_{n}|}
 \sum_{z \in \mathcal{Z}_{m,n}} \sum_{e \subseteq \left( z + \cu_{m} \right)} \left| \nabla \psi_z (e) - \nabla \psi_{n,q} (e)\right|^2  \right] \\ \leq C \left( \nu^* (\cu_m , q) - \nu^* (\cu_{n} , q) \right) + C(1 + |q|^2) 3^{-m}.
\end{multline}
\end{proposition} 

\begin{proof}
The first idea of the proof is to consider the following orthogonal decomposition of the space $\mathring h^1 (\cu_n)$:
\begin{equation} \label{orthdecomp2sc}
\mathring h^1 (\cu_{n}) = \underset{z \in \mathcal{Z}_{m,n}}{\oplus} \mathring h^1 \left(z + \cu_{n} \right)  \overset{\perp}{\oplus} H,
\end{equation}
where we denote by
\begin{equation*}
\mathring h^1 \left(z + \cu_{m} \right) := \left\{ \psi \in \mathring h^1 (\cu_{n}) ~:~ \psi_{| \cu_{n} \setminus \left( z + \cu_m \right)} = 0 \right\},
\end{equation*}
with a slight abuse of notation: we extend the functions of $\mathring h^1 \left(z + \cu_{m} \right)$ by $0$ outside the cube~$(z + \cu_m).$

The remaining space $H$ is the space of functions of $\mathring h^1 (\cu_{n})$ which are constant on the subcubes~$\left( z + \cu_{n} \right)_{z \in \mathcal{Z}_{m,n}}$. It is a space of dimension $3^{d(n-m)} -1$ and each function $h \in H$ can be written in the following form
\begin{equation*}
h = \sum_{z \in \mathcal{Z}_{m,n}} \lambda_z \indc_{z  + \cu_{m}},
\end{equation*}
for some real constants $\left( \lambda_z \right)_{z \in \mathcal{Z}_{m,n}}$ satisfying $\sum_{z \in \mathcal{Z}_{m,n}} \lambda_z = 0$. 

\smallskip

For $z \in \mathcal{Z}_{m,n}$, we denote by $\psi_{n,q}^z$ the orthogonal projection of the random variable $\psi_{n,q}$ on the space $ \mathring h^1 \left(z + \cu_{m}\right)$. This defines a random variable taking values in $\mathring h^1 \left(z + \cu_{m}\right)$.

\smallskip

We also let $h$ the orthogonal projection of the random variable $\psi_{n,q}$ on the space $H$. Finally, we introduce the notation 
\begin{equation*}
\psi'_{n,q} : = \sum_{z \in \mathcal{Z}_{m,n}}\psi^z_{n,q}.
\end{equation*} 
This is a random variable taking values in $\underset{z \in \mathcal{Z}_{m,n}}{\oplus} \mathring h^1 \left(z + \cu_{m}\right) \subseteq \mathring h^1 \left(\cu_{n}\right)$. Its law is a probability measure defined on the vector space $\underset{z \in \mathcal{Z}_{m,n}}{\oplus} \mathring h^1 \left(z + \cu_{m}\right)$ and is denoted by $\P_{\psi'_{n,q}}$.

\smallskip

As in the proof of the previous proposition, we introduce $B_{m,n}$ the set of edges connecting two subcubes of the form $(z + \cu_m)$, i.e.,
\begin{equation} \label{def.bn}
B_{m,n} := \left\{ (x,y) \, : \, \exists z,z' \in \mathcal{Z}_{m,n}, \, z \neq z' \, \mbox{such that} \, x \in z + \cu_m \, \mbox{and} \, y \in z ' + \cu_m \right\},
\end{equation}
so as to have the decomposition of the sum
\begin{equation} \label{sum.decomp}
\sum_{e \subseteq \cu_{n}} = \sum_{z \in \mathcal{Z}_{m,n}} \sum_{e \subseteq z + \cu_m} + \sum_{e \in B_{m,n}}.
\end{equation}
We note that for every $h \in H$, the gradient $\nabla h$ is supported on the edges of $B_{m,n}$ and
\begin{equation} \label{propmenintro}
\mbox{for each} \,z , z' \in  \mathcal{Z}_{m,n} \, \mbox{with} \, z \neq z' \, \mbox{and for each} \, e \subseteq \left( z' + \cu_m \right), \nabla \psi_{n,q}^z(e) = 0. 
\end{equation} 
This implies, for each $z \in \mathcal{Z}_{m,n}$ and each $e \subseteq (z + \cu_m)$,
\begin{equation*}
\nabla \psi'_{n,q} (e) = \nabla \psi^z_{n,q}  (e).
\end{equation*}
The same result is valid for the field $\psi'$:  for each $z \in \mathcal{Z}_{m,n}$ and each $e \subseteq (z + \cu_m)$,
\begin{equation*}
\nabla \psi' (e) = \nabla \psi_z  (e).
\end{equation*}

\medskip

We now split the proof into 5 Steps. In Steps 1 to 4, we assume $q = 0$. We then remove this additional assumption in Step 5.
\begin{itemize}
\item In Step 1, we show that the law of the random variable $\psi'$ is the minimizer of the variational problem
\begin{equation*}
\inf_{\P} ~ \E \left[ \sum_{z \in \mathcal{Z}_{m,n}} \sum_{e \subseteq \left( z + \cu_m \right)} V_e(\nabla \psi(e)) \right] + H \left( \P \right) ,
\end{equation*}
where the infimum is chosen over all the probability measures on $\underset{z \in \mathcal{Z}_{m,n}}{\oplus} \mathring h^1 \left( z + \cu_{m}\right)$ and the entropy is computed with respect to the Lebesgue measure on this space. \smallskip

\item In Step 2, we consider the optimal coupling between the random variables $\psi'$ and $\psi'_{n,0}$ and prove the following inequality,
\begin{multline*}
\E \left[ \frac1{|\cu_{n}|} \sum_{z \in \mathcal{Z}_{m,n}} \sum_{e \subseteq \left( z + \cu_m \right)} \left| \nabla \psi_z(e) - \nabla \psi_{n,0}^z(e) \right|^2  \right] \\ \leq C \left( \nu^*(\cu_n , 0)  + \E \left[   \frac1{|\cu_{n}|} \sum_{z \in \mathcal{Z}_{m,n}} \sum_{e \subseteq \left( z + \cu_n \right)} V_e \left(\nabla  \psi_{n,0}^z(e)\right) \right]
					+ \frac1{|\cu_{n}|} H \left( \P_{\psi'_{n,0}} \right)  \right).
\end{multline*}
\item In Step 3, we prove the following estimate pertaining to the random variables $\psi_{n,0}^z$,
\begin{equation} \label{2sccomp.step3}
 \sum_{z \in \mathcal{Z}_{m,n}}  \E \left[   \frac1{|\cu_{n}|} \sum_{e \subseteq \left( z + \cu_m \right)} V\left(\nabla  \psi_{n,0}^z(e)\right) \right] + \frac1{|\cu_{n}|} H \left( \P_{\psi_{n,0}^z} \right) \leq - \nu^*(\cu_{m}, 0) +Cm3^{-dm}.
\end{equation}
\item In Step 4, we complete the proof and show that there exists a coupling between the random variables $\psi_{n,0}$ and $\psi'$ such that
\begin{multline*}
\E \left[ \frac1{|\cu_{n}|} \sum_{z \in \mathcal{Z}_{m,n}} \sum_{e \subseteq \left( z + \cu_m \right)} \left| \nabla \psi' (e)- \nabla \psi_{n,0}(e) \right|^2  \right] \\ \leq C (\nu^*(\cu_m , 0) - \nu^*(\cu_{n}, 0) ) +Cm3^{-dm}.
\end{multline*}
\item In Step 5, we remove the assumption $q = 0$ and prove the more general result: for each $q \in \Rd$, there exists a coupling between the random variables $\psi'$ and $\psi_{n,q}$ such that
\begin{multline*}
\E \left[ \frac1{|\cu_{n}|} \sum_{z \in \mathcal{Z}_{m,n}} \sum_{e \subseteq \left( z +\cu_m \right)} \left| \nabla \psi_z (e)- \nabla \psi_{n,q} (e) \right|^2  \right] \\ \leq C (\nu^*(\cu_m , q) - \nu^*(\cu_{n}, q) ) + C (1 + |q|^2) 3^{-m}.
\end{multline*}
\end{itemize}

\medskip

\textit{Step 1.} By Proposition~\ref{varfornuprop}, one has
\begin{multline} \label{e.step12sccomp1}
\inf_{\P} \left( \E \left[ \sum_{z \in \mathcal{Z}_{m,n}} \sum_{e \subseteq \left( z + \cu_m \right) } V_e(\nabla \psi(e)) \right] + H \left( \P \right) \right) \\ = - \log \int_{\underset{z \in \mathcal{Z}_{m,n}}{\oplus} \mathring h^1 \left( z + \cu_{m}\right)}  \exp \left( - \sum_{z \in \mathcal{Z}_{m,n}} \sum_{e \subseteq \left( z + \cu_n\right) } V_e(\nabla \psi(e))  \right) \, d\psi,
\end{multline}
where the infimum is considered over all the probability measures on $\underset{z \in \mathcal{Z}_{m,n}}{\oplus} \mathring h^1 \left( z + \cu_{m}\right)$.

But on the one hand, one has the equality
\begin{multline} \label{e.step12sccomp2}
\int_{\underset{z \in \mathcal{Z}_{m,n}}{\oplus} \mathring h^1 \left( z + \cu_{m}\right)}  \exp \left( - \sum_{z \in \mathcal{Z}_{m,n}} \sum_{e \subseteq \left( z + \cu_m\right)} V_e(\nabla \psi(e))  \right) \, d\psi \\ = \left( \int_{ \mathring h^1 \left(\cu_m\right)}  \exp \left( -  \sum_{e \subseteq \cu_m} V_e(\nabla \psi(e))  \right) \, d\psi  \right)^{3^{d(n-m)}}.
\end{multline}
On the other hand, since by assumption the random variables $\left( \psi_z \right)_{z \in \mathcal{Z}_{m,n}}$ are independent, one has
\begin{equation*}
H \left( \P_{\psi'}\right) = \sum_{z \in \mathcal{Z}_{m,n}} H \left( \P_{\psi_z} \right) = 3^{d(n-m)} H \left( \P_{m , 0}^* \right).
\end{equation*}
As a remark, we note that the entropy of $\P_{\psi'} $ is computed with respect to the Lebesgue measure on $\oplus_{z \in \mathcal{Z}_{m,n}} \mathring h^1 \left(z + \cu_{m}\right)$, while the entropies of the laws $ \P_{\psi_z}$ and $\P_{m , 0}^*$ are computed with respect to the Lebesgue measures on $\mathring h^1 \left( z+ \cu_{m}\right)$ and on $\mathring h^1 \left(\cu_m\right)$ respectively.
The energy part of the random variable $\psi'$ can be computed explicitly and one derives
\begin{align*}
\E \left[ \sum_{z \in \mathcal{Z}_{m,n}} \sum_{e \subseteq \left( z + \cu_m\right)} V_e(\nabla \psi'(e)) \right] & = \sum_{z \in \mathcal{Z}_{m,n}}  \E \left[ \sum_{e \subseteq \left( z + \cu_m\right)} V_e(\nabla \psi_z(e)) \right] \\ & = 3^{d(n-m)} \E \left[ \sum_{e \subseteq \cu_n} V_e(\nabla \psi_{n,0}(e)) \right].
\end{align*}
Consequently
\begin{multline*}
\E \left[ \sum_{z \in \mathcal{Z}_{m,n}} \sum_{e \subseteq \left( z + \cu_m\right)} V_e(\nabla \psi'(e)) \right] + H \left( \P_{\psi'}  \right) \\ =  3^{d(n-m)} \left(  \E \left[ \sum_{e \subseteq \cu_m} V_e(\nabla \psi_{m,0}(e)) \right] + H \left( \P_{m , 0}^* \right) \right).
\end{multline*}
By definition of law $\P_{m,0}^*$, the term on the right-hand side can be explicitly computed, and one obtains
 \begin{multline*}
\E \left[ \sum_{z \in \mathcal{Z}_{m,n}} \sum_{e \subseteq \left( z + \cu_m\right)} V_e(\nabla \psi'(e)) \right] + H \left( \P_{\psi'}  \right)  \\ = -3^{d(n-m)} \log \left( \int_{ \mathring h^1 \left(\cu_m\right)}  \exp \left( -  \sum_{e \subseteq \cu_m} V_e(\nabla \psi(e))  \right) \, d\psi \right)  .
\end{multline*}
Combining the previous display with~\eqref{e.step12sccomp1} and ~\eqref{e.step12sccomp2}, we obtain
\begin{multline*}
\E \left[ \sum_{z \in \mathcal{Z}_{m,n}} \sum_{e \subseteq \left( z + \cu_m\right)} V_e(\nabla \psi'(e)) \right] + H \left( \P_{\psi'}  \right) \\= \inf_{\P} \left( \E \left[ \sum_{z \in \mathcal{Z}_{m,n}} \sum_{e \subseteq \left( z + \cu_m \right) } V_e(\nabla \psi(e)) \right] + H \left( \P \right) \right).
\end{multline*}
The proof of Step 1 is complete.

\medskip

\textit{Step 2.} The idea of this step is similar to the strategy adopted in Step 2 of Proposition~\ref{2sccomparisonnu}: one uses the uniform convexity of the elastic potential $V_e$ and the displacement convexity of the entropy to prove a uniform convexity result for the functional $\mathcal{F}_{n,0}^*$ defined in Definition~\ref{chapitre444.energyentrpyfunctional}.

First, we consider the optimal coupling between the random variables $\psi'$ and $\psi'_{n,0}$. In particular, by the displacement convexity of the entropy, the law of the random variable~$\frac{\psi' + \psi'_{n,0}}{2}$ satisfies the inequality
\begin{equation} \label{e.step22sccomp1}
H \left( \P_{\frac{\psi' + \psi'_{n,0}}{2}} \right) \leq \frac{H \left( \P_{\psi'}  \right) + H \left( \P_{\psi'_{n,0}} \right) }{2}.
\end{equation}
Moreover, by the uniform convexity of $V_e$, one has
\begin{align}\label{e.step22sccomp2}
2\lefteqn{\E \left[ \sum_{z \in \mathcal{Z}_{m,n}} \sum_{e \subseteq \left( z +\cu_m \right)} V_e\left( \frac{\nabla \psi'(e)+ \nabla \psi'_{n,0}(e)}{2} \right) \right] } \qquad & \\ & \leq \E \left[ \sum_{z \in \mathcal{Z}_{m,n}} \sum_{e \subseteq \left( z +\cu_m \right)} V_e\left( \nabla \psi'_{n,0}(e)\right) \right]
														+ \E \left[ \sum_{z \in \mathcal{Z}_{m,n}} \sum_{e \subseteq \left( z +\cu_m \right)} V_e\left( \nabla  \psi' (e)\right) \right]\notag \\
														&  \qquad- \lambda \E \left[ \sum_{z \in \mathcal{Z}_{m,n}} \sum_{e \subseteq \left( z +\cu_m \right)} \left| \nabla \psi'(e)- \nabla \psi'_{n,0} (e)\right|^2 \right]. \notag
\end{align}
We can then use Step 1 to compute
\begin{align*}
\lefteqn{\E \left[ \sum_{z \in \mathcal{Z}_{m,n}} \sum_{e \subseteq \left( z +\cu_m \right)}   V_e(\nabla \psi'(e)) \right]  + H \left( \P_{\psi'} \right)} \qquad & \\  & =  \inf_{\P} \E \left[\sum_{z \in \mathcal{Z}_{m,n}} \sum_{e \subseteq \left( z +\cu_m \right)}  V_e(\nabla \psi(e)) \right] + H \left( \P \right)  \\
																							& \leq \E \left[ \sum_{z \in \mathcal{Z}_{m,n}} \sum_{e \subseteq \left( z +\cu_m \right)}  V_e\left( \frac{\nabla \psi'(e)+ \nabla \psi'_{n,0}(e)}{2} \right) \right]+ H \left( \P_{\frac{\nabla \psi'(e)+ \nabla \psi'_{n,0}(e)}{2}} \right).
\end{align*}
One can then apply~\eqref{e.step22sccomp1} and~\eqref{e.step22sccomp2} to deduce
\begin{align*}
\lefteqn{\E \left[\sum_{z \in \mathcal{Z}_{m,n}} \sum_{e \subseteq \left( z +\cu_m \right)} V_e(\nabla \psi'(e)) \right]  + H \left( \P_{\psi'} \right) } \qquad & \\ &  \leq \frac 12 \left( \E \left[ \sum_{z \in \mathcal{Z}_{m,n}} \sum_{e \subseteq \left( z +\cu_m \right)} V_e(\nabla \psi'(e)) \right]  + H \left( \P_{\psi'}  \right) \right) \\  & + \frac 12 \left(  \E \left[ \sum_{z \in \mathcal{Z}_{m,n}} \sum_{e \subseteq \left( z +\cu_m \right)}  V_e\left( \nabla  \psi'_{n,0} (e)\right) \right]  + H \left( \P_{\psi'_{n,0}}  \right)  \right) \\ &- \frac{\lambda}{2} \E \left[ \sum_{z \in \mathcal{Z}_{m,n}} \sum_{e \subseteq \left( z +\cu_m \right)} \left| \nabla \psi'(e)- \nabla \psi'_{n,0}(e) \right|^2 \right].
\end{align*}
From Step 1, one also has
\begin{align*}
\frac1{|\cu_{n}|} \left( \E \left[ \sum_{z \in \mathcal{Z}_{m,n}} \sum_{e \subseteq \left( z +\cu_m \right)}  V_e(\nabla \psi'(e)) \right]  + H \left( \P_{\psi'}\right) \right) & = - \frac{3^{d(n-m)} |\cu_m|}{|\cu_{n}|}   \nu^* \left( \cu_{m},0 \right)  \\ & = - \nu^* \left( \cu_{m},0 \right).
\end{align*}
Combining the two previous displays and dividing by $\left| \cu_{n}\right|$ gives
\begin{multline*}
\frac{\lambda}{2} \E \left[ \frac1{|\cu_{n}|} \sum_{z \in \mathcal{Z}_{m,n}} \sum_{e \subseteq \left( z +\cu_m \right)} \left| \nabla \psi'(e)- \nabla \psi'_{n,0} (e)\right|^2 \right] \\ 
		\leq \frac12  \nu^* \left( \cu_{m},0 \right) + \frac 12 \left(  \E \left[\frac1{|\cu_{n}|}\sum_{z \in \mathcal{Z}_{m,n}} \sum_{e \subseteq \left( z +\cu_m \right)} V_e\left( \nabla  \psi'_{n,0}(e) \right) \right]  + \frac1{|\cu_{n}|} H \left( \P_{\psi'_{n,0}} \right)  \right).
\end{multline*}
This completes the proof of Step 2.

\medskip

\textit{Step 3.}  Before starting this step, we recall the notation $B_{m,n}$ for the bonds connecting two subcubes introduced in~\eqref{def.bn} and the decomposition of the sum~\eqref{sum.decomp}. This step is the most technical one, and the main difficulty is to compare the entropies of the random variables $\psi_{n,0}$ and $\psi_{n,0}'$. The only property which is known about these random variables comes from their definitions: one has
\begin{equation*}
\psi_{n,0} -\psi_{n,0}' \in H,
\end{equation*}
where $H$ is the space of functions of mean zero which are constant on the cubes $z + \cu_m$, it was first introduced in~\eqref{orthdecomp2sc}. In this situation, it is difficult to compare the two entropies but one can use the elastic energy of the random variable $\nabla \psi_{n,0}$ on the edges of $B_{m,n}$, where the gradient of $\psi_{n,0} -\psi_{n,0}'$ is supported, to obtain some information: more specifically, we prove the inequality
\begin{equation} \label{2sctechstep3}
\E \left[ \frac{1}{\left|\cu_{n} \right|}\sum_{e \in B_{m,n}} V_e(\nabla \psi_{n,0}(e)) \right] + \frac{1}{\left|\cu_{n} \right|}  H \left( \P_{n,0}^* \right) \geq   \frac{1}{\left|\cu_{n} \right|} H \left( \P_{\psi'_{n,0}} \right) - Cm3^{-dm}.
\end{equation}
The previous inequality is the main argument of this step. In the next paragraph, we explain how to obtain the main result~\eqref{2sccomp.step3} of this step, assuming that~\eqref{2sctechstep3} holds. First one deduces from the estimate~\eqref{2sctechstep3}
\begin{multline*}
\E \left[  \frac1{|\cu_{n}|} \sum_{z \in \mathcal{Z}_{m,n}} \sum_{e \subseteq \left( z +\cu_m \right)} V_e\left(\nabla  \psi_{n,0} (e)\right) \right]  + \E \left[  \frac1{|\cu_{n}|} \sum_{e \in B_{m,n}} V_e\left(\nabla \psi_{n,0}\right) \right] +  \frac1{|\cu_{n}|} H \left( \P_{n,0} \right) \\ \geq \E \left[   \frac1{|\cu_{n}|} \sum_{z \in \mathcal{Z}_{m,n}} \sum_{e \subseteq \left( z +\cu_m \right)} V_e\left(\nabla  \psi_{n,0}(e)\right) \right] + \frac1{|\cu_{n}|}  H \left( \P_{\psi'_{n,0}}  \right)  -  Cm3^{-dm}.
\end{multline*}
By the decomposition of the sum stated in~\eqref{sum.decomp}, the term on the left-hand side can be rewritten
\begin{equation*}
\E \left[\sum_{z \in \mathcal{Z}_{m,n}} \sum_{e \subseteq \left( z +\cu_m \right)} V_e\left(\nabla  \psi_{n,0}(e)\right) \right]  + \E \left[ \sum_{e \in B_{m,n}} V_e(\nabla \psi_{n,0} (e)) \right] = \E \left[ \sum_{e \subseteq \cu_{n}}  V_e\left(\nabla  \psi_{n,0} (e)\right) \right].
\end{equation*}
We then use the equality
\begin{equation*}
\E \left[ \sum_{e \subseteq \cu_{n}}  V_e\left(\nabla  \psi_{n,0} (e)\right) \right] + H \left( \P_{n,0}^*  \right) = - |\cu_{n}| \nu^*\left( \cu_{n},0\right).
\end{equation*}
Combining the few previous results and dividing by the volume term $|\cu_{n}|$ yields
\begin{equation} \label{e.almoststep3}
- \nu^*\left( \cu_{n},0\right) \geq \E \left[   \frac1{|\cu_{n}|} \sum_{z \in \mathcal{Z}_{m,n}} \sum_{e \subseteq \left( z +\cu_m \right)}  V_e\left(\nabla  \psi_{n,0}(e)\right) \right] + \frac1{|\cu_{n}|}  H \left( \P_{\psi'_{n,0}} \right)  -  Cm3^{-dm}.
\end{equation}
This is the estimate~\eqref{2sccomp.step3} up to two points:
\begin{enumerate}
\item First there should be a term $\psi_{n,0}^z$ instead of a term $\psi_{n,0}$ on the right-hand side. By definition, the random variable $\psi_{n,0}^z$ is the orthogonal projection of $\psi_{n,0}$ on the space $ \mathring h^1 \left(z + \cu_m \right)$, and using the property~\eqref{propmenintro}, one has
\begin{equation*}
 \mbox{for each} \, z \in \mathcal{Z}_{m,n} \, \mbox{and each} \, e \subseteq z + \cu_{m}, \, \nabla \psi_{n,0}(e) = \nabla \psi_{n,0}^z(e).
\end{equation*}
With this identity, the inequality~\eqref{e.almoststep3} can be rewritten
\begin{equation*}
- \nu^*\left( \cu_{n},0\right) \geq \E \left[   \frac1{|\cu_{n}|} \sum_{z \in \mathcal{Z}_{m,n}} \sum_{e \subseteq \left( z +\cu_m \right)} V_e\left(\nabla  \psi_{n,0}^z \right) \right] + \frac1{|\cu_{n}|}  H \left( \P_{\psi'_{n,0}} \right)  -  C m 3^{-dm}.
\end{equation*}
\item The second point which needs to be addressed is the entropy which is not exactly the same as in~\eqref{2sccomp.step3}. By Proposition~\ref{entropycouplesum2}, one has
\begin{equation*}
H \left( \P_{\psi'_{n,0}} \right) \geq  \sum_{z \in \mathcal{Z}_{m,n}} H \left( \P_{\psi_{n,0}^z} \right),
\end{equation*}
which provides a solution.
\end{enumerate}

\medskip

We now turn to the proof of~\eqref{2sctechstep3}. We recall the notation $Z_{0}^*(\cu_{n})$ introduced in~\eqref{partfctnustar}. We also let $\rho$ be the density associated to the law $\P_{n,0}^*$; it is defined on $\mathring h^1 \left(\cu_{n}\right)$ by
\begin{equation*}
\rho : \left\{ \begin{aligned}
\mathring h^1 \left(\cu_{n}\right) & \rightarrow \hspace{40mm} \R \\
\psi \hspace{7mm}& \mapsto \frac{1}{Z_{0}^*(\cu_{n})} \exp \left( - \sum_{e \subseteq \cu_{n}} V_e\left(\nabla \psi (e)\right) \right).
\end{aligned} \right.
\end{equation*}
Using the orthogonal decomposition~\eqref{orthdecomp2sc}, and the definition of $\psi'_{n,0}$, we can compute the density $\rho'$ of the random variable $\psi'_{n,0}$. It is defined on the space $\underset{z \in \mathcal{Z}_{m,n}}{\oplus} \mathring h^1 \left(z + \cu_{m} \right) $ according to the formula
\begin{equation*}
\rho' : \left\{ \begin{aligned}
\underset{z \in  \mathcal{Z}_{m,n}}{\oplus} \mathring h^1 \left(z + \cu_{m} \right)  & \rightarrow \hspace{40mm} \R \\
\psi \hspace{7mm}& \mapsto \frac{1}{Z_0^* \left( \cu_{n}\right) } \int_H \exp \left( - \sum_{e \subseteq \cu_{n}} V_e\left(\nabla \psi (e) + \nabla h(e)\right) \right) \, d h,
\end{aligned} \right.
\end{equation*}
where the integral is considered with respect to the Lebesgue measure on the space $H$ defined in~\eqref{orthdecomp2sc}. Using that for every $h \in H$, $\nabla h$ is supported in $B_{m,n}$, we can split the sum according to the formula~\eqref{sum.decomp} to obtain, for each $\psi \in \underset{z \in  \mathcal{Z}_{m,n}}{\oplus} \mathring h^1 \left(z + \cu_{m} \right) $,
\begin{multline} \label{def.rho''}
\rho' (\psi) =  \frac{\exp \left( - \sum_{z \in  \mathcal{Z}_{m,n}} \sum_{e \subseteq z + \cu_m} V\left(\nabla \psi(e)  \right) \right)}{Z_0^* \left( \cu_{n} \right) } \\ \times  \int_H \exp \left( - \sum_{e\in B_{m,n}} V_e\left(\nabla \psi (e) + \nabla h (e)\right) \right) d h.
\end{multline}
We now prove an upper bound for the term inside the integral and show that its logarithm is relatively small: there exists a constant $C := C(d , \lambda) < \infty$ such that, for each function $\psi \in \underset{z \in  \mathcal{Z}_{m,n}}{\oplus} \mathring h^1 \left(z + \cu_{m} \right) $,
\begin{equation} \label{suptecest}
\log  \int_H \exp \left( - \sum_{e\in B_{m,n}} V_e\left(\nabla \psi (e) + \nabla h(e)\right) \right) \, d h \leq C m 3^{d(n-m)}.
\end{equation}
The proof of this estimate is essentially technical and relies on the fact that the dimension of $H$ is equal to $3^{d(n-m)} -1 $, which justifies the form of the right-hand side. The proof is postponed to Appendix A, Proposition~\ref{5dest.}. We now show how to deduce~\eqref{2sctechstep3} from~\eqref{suptecest}. We first compute the entropy of the law $\P_{n,0}^*$. This gives
\begin{align*}
H \left( \P_{n,0}^* \right) & = \int_{\mathring{h}^1 (\cu_{n})} \rho(\psi ) \log \rho (\psi) \, d \psi \\
								& = \int_{\mathring{h}^1 (\cu_{n})} \rho(\psi ) \left( - \sum_{e \subseteq \cu_{n}} V_e \left( \nabla \psi\right) \right) \, d \psi - \log Z_0^* \left(\cu_{n} \right).
\end{align*}
Adding the term $\E \left[ \sum_{e \in B_{m,n}} V_e(\nabla \psi(e)) \right] $ and splitting the sum according to the formula~\eqref{sum.decomp} gives
\begin{multline*}
\E \left[ \frac{1}{\left|\cu_{n} \right|}\sum_{e \in B_{m,n}} V_e(\nabla \psi_{n,0}(e)) \right]  + H \left( \P_{n,0}^* \right) 
				\\ = \int_{\mathring{h}^1 (\cu_{n})} \rho(\psi ) \left( - \sum_{z \in  \mathcal{Z}_{m,n}} \sum_{e \subseteq z + \cu_m} V\left(\nabla \psi(e)  \right) \right) \, d \psi - \log Z_0^* \left(\cu_{n}\right).
\end{multline*}
We focus on the integral on the right-hand side: using the decomposition~\eqref{orthdecomp2sc}, one can apply Fubini's Theorem and first integrate over $H$ then over $\underset{z \in  \mathcal{Z}_{m,n}}{\oplus} \mathring h^1 \left(z + \cu_{m}\right)$. This gives
\begin{multline*}
\int_{\mathring{h}^1 (\cu_{n})} \rho(\psi ) \left( - \sum_{z \in  \mathcal{Z}_{m,n}}\sum_{e \subseteq \left( z + \cu_m \right)} V_e \left( \nabla \psi \right) \right) \, d \psi \\ =  \int_{\underset{z \in  \mathcal{Z}_{m,n}}{\oplus} \mathring h^1 \left(z + \cu_{m}\right)} \int_{H} \rho(\psi + h ) \left( - \sum_{z \in  \mathcal{Z}_{m,n}}\sum_{e \subseteq \left( z + \cu_m \right)} V_e \left( \nabla \psi \right) \right) \, d h \, d \psi.
\end{multline*}
Note that here we have used that the gradient of an element of $H$ is supported on $B_{m,n}$. Using the definition of $\rho'$ stated in~\eqref{def.rho''}, the previous equality can be rewritten
\begin{multline*}
\int_{\mathring{h}^1 (\cu_{n})} \rho(\psi ) \left( - \sum_{z \in  \mathcal{Z}_{m,n}}\sum_{e \subseteq \left( z + \cu_m \right)} V_e \left( \nabla \psi \right) \right) \, d \psi \\ =  \int_{\underset{z \in  \mathcal{Z}_{m,n}}{\oplus} \mathring h^1 \left(z + \cu_{m}\right)}  \rho'(\psi ) \left( - \sum_{z \in  \mathcal{Z}_{m,n}} \sum_{e \subseteq \left( z + \cu_m \right)} V_e \left( \nabla \psi \right) \right) \, d \psi.
\end{multline*}
Combining the few previous displays then gives
\begin{multline} \label{2scstep3finalest}
\E \left[ \sum_{e \in B_{m,n}} V_e(\nabla  \psi_{n,0}(e)) \right] + H \left( \P_{n,0}^* \right) \\
					= \int_{\underset{z \in  \mathcal{Z}_{m,n}}{\oplus} \mathring h^1 \left(z + \cu_{m}\right)}  \rho'(\psi ) \left( - \sum_{z \in \mathcal{Z}_{m,n}} \sum_{e \subseteq \left( z + \cu_m \right)} V_e \left( \nabla \psi \right) \right) \, d \psi - \log Z_{\cu_{n}}^*.
\end{multline}
By the definition of $\rho'$ in~\eqref{def.rho''} and the estimate~\eqref{suptecest}, one has, for each function $\psi \in \underset{z \in  \mathcal{Z}_{m,n}}{\oplus} \mathring h^1 \left(z + \cu_{m}\right)$,
\begin{equation*}
\log \rho' (\psi) \leq  - \sum_{z \in  \mathcal{Z}_{m,n}} \sum_{e \subseteq \left( z + \cu_m\right)} V\left(\nabla \psi(e) \right) \, d  \psi - \log Z_0^* \left(\cu_{n}\right) + C m 3^{d(n-m)}.
\end{equation*}
This allows to perform the following computation
\begin{align} \label{2scstep3finalest2}
H \left( \P_{\psi'_{n,0}} \right) & = \int_{\underset{z \in  \mathcal{Z}_{m,n}}{\oplus} \mathring h^1 \left(z + \cu_{m}\right) }  \rho'(\psi)  \log \rho' (\psi) \, d \psi \\
							& \leq \int_{\underset{z \in \mathcal{Z}_{m,n}}{\oplus} \mathring h^1 \left(z + \cu_{m}\right)}  \rho'(\psi) \left( - \sum_{z \in  \mathcal{Z}_{m,n}} \sum_{e \subseteq \left( z +\cu_m \right)} V\left(\nabla \psi (e)\right) \right) \, d  \psi \notag \\ & \qquad  - \log Z_0^* \left(\cu_{n}\right) + C m 3^{d(n-m)}. \notag
\end{align}
Combining~\eqref{2scstep3finalest} and~\eqref{2scstep3finalest2} gives
\begin{equation*}
H \left( \P_{\psi'_{n,0}}  \right)  \leq \E \left[ \sum_{e \in B_{m,n}} V_e(\nabla \psi_{n,0}(e)) \right] + H \left( \P_{n,0}^* \right) + C m 3^{d(n-m)}.
\end{equation*}
This is~\eqref{2sctechstep3} and the proof of Step 3 is complete.

\medskip

\textit{Step 4.} Combining the main results of Steps 2 and 3, one obtains the existence of a coupling between the random variables $\psi'$ and $\psi_{n,0}'$, such that
\begin{equation}\label{2sccompalmostfinal.est}
\E \left[ \frac1{|\cu_{n}|} \sum_{z \in  \mathcal{Z}_{m,n}} \sum_{e \subseteq \left( z + \cu_m \right)} \left| \nabla \psi_z(e) - \nabla \psi_{n,0}'(e) \right|^2  \right] \\ \leq C \left( \nu^*(\cu_m , 0) - \nu^*(\cu_{n} , 0) \right) + C m 3^{-dm}.
\end{equation} 
The objective of this step is to find a coupling between the random variables $\psi_{n,0}$ and $\psi'$, instead of $\psi'_{n,0}$, and $\psi'$. Recall that we denoted by $h$ the orthogonal projection of $\psi_{n,0}$ on the space $H$. Denote by $\P_h$ its law, it is a probability measure on $H$.

\smallskip

 By Lemma~\ref{couplinglemma}, there exists a coupling between the three probability measures $\P_{\psi'}, \P_{\psi'_{n,0}}$ and $\P_h$ such that, under this coupling, the law of $(\psi', \psi'_{n,0})$ is the optimal coupling between $\P_{\psi'}$ and  $\P_{\psi'_{n,0}}$ and the law of $(\psi'_{n,0} , h )$ is $\P_{n , 0}^*$. This provides the desired coupling between the random variables $\psi'$ and $\psi_{n,0}$: under this coupling the inequality~\eqref{2sccompalmostfinal.est} is satisfied. Step 4 is complete.

\medskip

\textit{Step 5.} We remove the assumption $q=0$. This can be achieved by applying the result obtained for $q=0$ with the tilted elastic potential
\begin{equation*}
V_{e,q}(x) := V_e(x) - q(e)  x.
\end{equation*}
This new elastic potential satisfies the same uniform ellipticity assumption as $V_e$ and one can do the same proof with $V_{e,q}$ instead of $V_e$, the difference is mostly notational: the only place where it as an impact is in the estimate~\eqref{suptecest} and it provides an additional error term which can be proved to be bounded by $C(1+|q|^2) 3^{dn-m}$. Dividing by the volume of the triadic cube $\left|\cu_n\right|$ eventually shows
\begin{multline*}
\E \left[ \frac1{|\cu_{n}|} \sum_{z \in  \mathcal{Z}_{m,n}} \sum_{e \subseteq \left( z + \cu_m \right)} \left| \nabla \psi'(e) - \nabla \psi_{n,q}(e) \right|^2  \right] \\ \leq C \left( \nu^*(\cu_m , q) - \nu^*(\cu_{n} , q) \right) +  C m 3^{-dm} + C(1+|q^2|) 3^{-m}.
\end{multline*}
This can be simplified into
\begin{multline*}
\E \left[ \frac1{|\cu_{n}|} \sum_{z \in  \mathcal{Z}_{m,n}} \sum_{e \subseteq \left( z + \cu_m \right)} \left| \nabla \psi'(e) - \nabla \psi_{n,q}(e) \right|^2  \right] \\ \leq C \left( \nu^*(\cu_m , q) - \nu^*(\cu_{n} , q) \right) + C(1+|q^2|) 3^{-m}.
\end{multline*}
The proof of Proposition~\ref{2sccomparison} is complete.
\end{proof}

\section{Quantitative convergence of the subadditive quantities} \label{section4}

The main goal of this section is to use the tools developed in Section~\ref{section3} to prove Theorem~\ref{QuantitativeconvergencetothGibbsstate}. To this end, we first introduce a notation for the subadditivity defect of the surface tensions $\nu$ and $\nu^*$.

\begin{definition}
For each $p\in \Rd$ and each $n \in \N$, we define
\begin{equation*}
\tau_n (p) := \nu \left( \cu_n, p \right) - \nu \left( \cu_{n+1}, p \right),
\end{equation*}
and, for each $q \in \Rd$,
\begin{equation*}
\tau_n^* (q) := \nu^* \left( \cu_n, q \right) - \nu^* \left( \cu_{n+1}, q \right).
\end{equation*}
\end{definition}

We also recall the following notation from the introduction: for a bounded subset $U \subseteq \Zd$ and a vector field $F : E_d (U) \rightarrow \R$, we let $\left\langle F\right\rangle_U$ be the unique vector in $\Rd$ such that, for each $p \in \Rd$,
\begin{equation*}
p \cdot \left\langle F\right\rangle_U = \frac1{|U|} \sum_{e \subseteq U} p \cdot F (e).
\end{equation*}
In the rest of this section, this is applied in the case where $U$ is a triadic cube and where $F$ is the gradient of a function. We may also refer to the quantity $\left\langle \nabla \phi \right\rangle_U$ as the slope of the function $\phi$ over the set $U$.

This section is organized as follows. We first prove, using Proposition~\ref{2sccomparison}, that the variance of the slope of the random variable $\psi_{n,q}$ over the cube $\cu_n$ converges to~$0$ as $n$ tends to infinity. More precisely, we bound the variance of the slope by the subadditivity defect $\tau_n^* (q)$, which is expected to be small as $n$ tends to infinity. This is proved in Proposition~\ref{contslope} and essentially relies of the two-scale comparison for $\nu^*$ stated in Proposition~\ref{2sccomparison}.

Once the slope of the random variable $\psi_{n,q}$ is controlled, we apply the multiscale Poincar\'e inequality, stated in Proposition~\ref{p.poincmultpoinc}, to prove that $\psi_{n,q}$ is close, in the expectation of the $L^2$-norm over the cube $\cu_n$, to an affine function. With all these tools at hand, we prove the technical estimate of this article, Proposition~\ref{lemma3.3}, thanks to a patching construction. This proposition, combined with the one-sided convex duality property proved in Proposition~\ref{convexdual}, shows that on a large scale, the functions $p \rightarrow \nu(\cu_n , p)$ and $q \rightarrow \nu^*(\cu_n,q)$ are approximately convex dual, i.e., they satisfy up to a small error
\begin{equation*}
\nu^* \left( \cu_n , q \right) \simeq \sup_{p \in \Rd} -\nu \left( \cu_n , p \right) + p \cdot q.
\end{equation*}
Once such a result is established, we turn to the proof of the quantitative convergence of the surface tension, Theorem~\ref{QuantitativeconvergencetothGibbsstate}.

\subsection{Contraction of the variance of the slope of the field $\psi_{n,q}$} We first prove the contraction of the variance of the slope of the random interface $\psi_{n,q}$. This is stated in the following proposition.

\begin{proposition}[Contraction of the slope of the field $\psi_{n,q}$] \label{contslope}
There exists a constant $C := C(d , \lambda) < \infty$ such that, for each $n \in \N$ and each $q \in \R^d$,
\begin{equation} \label{contslope.nustar}
\var \left[\left\langle \nabla \psi_{n+1,q} \right\rangle_{\cu_{n+1}} \right] \leq C (1 + |q|^2) 3^{- n} + C \sum_{m=0}^n 3^{\frac{(m-n)}{2}} \tau_m^*(q).
\end{equation}
\end{proposition}

\begin{proof}
The proof of this inequality relies on Proposition~\ref{2sccomparison} and it is necessary to reintroduce the objects used in the statement of this proposition. This is the subject of the following paragraph.

Consider the family of random variables $\left( \psi_z \right)_{z \in \mathcal{Z}_{n,n+1}}$ and the random variable $\psi' := \sum_{z \in \mathcal{Z}_{n,n+1}} \psi_z$ which were introduced before the statement of Proposition~\ref{2sccomparison}. We recall that they satisfy the following properties:
\begin{itemize}
\item for each $z \in\mathcal{Z}_{n,n+1}$, $\psi_z$ is a random variable valued in $\mathring h^1 (\cu_{n+1})$. It is equal to $0$ outside the cube $(z + \cu_n)$ and has law $\P_{n,q}^*$ in the cube $(z + \cu_n)$;
\item the random variables $\psi_z$ are independent.
\end{itemize}
We also consider the coupling between $\psi'$ and $\psi_{n+1,q}$ which was introduced in Proposition~\ref{2sccomparison}. In particular estimate~\eqref{e.2sccomparison} holds. 

\smallskip

The main idea of the proof is the following. By Proposition~\ref{2sccomparison}, one knows that, up to an error of size $\tau_n^*(q)$, the gradient of the random interface~$\psi_{n+1,q}$ is close to the gradients of $3^d$ independent random variables: the random variables $\psi_z$. As a consequence, the slope of the random interface $\psi_{n+1,q}$ can be written, up to an error of size $\tau_n^*(q)$, as a sum of independent random variables. One can then apply a concentration inequality for the variance of independent random variables to show the contraction of the slope.

\medskip

We split the proof into 2 steps:
\begin{itemize}
\item in Step 1, we use Proposition~\ref{2sccomparison} to prove
\begin{equation} \label{contvarstep1}
\var^{\frac 12} \left[ \left\langle \nabla \psi_{n+1,q} \right\rangle_{\cu_{n+1}} \right]  \leq 3^{-\frac d2} \var^{\frac 12} \left[ \left\langle \nabla \psi_{n,q}(e) \right\rangle_{\cu_n} \right] + C \tau^*_n(q)^\frac 12 + C3^{-\frac n2} (1 + |q|) ;
\end{equation}
\item in Step 2, we iterate the inequality obtained in Step 1 to derive~\eqref{contslope.nustar}.
\end{itemize}

\medskip 

\textit{Step 1.} First we recall the definition of the set of bonds connecting two cubes of the form~$(z + \cu_n)$,
\begin{equation*}
B_n := \left\{ (x,y) \, : \, \exists z,z' \in 3^n \Zd \cap \cu_{n+1} \, \mbox{such that} \, z \neq z' , \,  x \in z + \cu_n \, \mbox{and} \, y \in z' + \cu_n \right\}
\end{equation*}
and the decomposition of the sum~\eqref{sum.decomp},
\begin{equation*}
\sum_{e \subseteq \cu_{n+1}} = \sum_{z \in \mathcal{Z}_{n}} \sum_{e \subseteq z+\cu_n} + \sum_{e \in B_n}.
\end{equation*}
The set $B_n$ is equal to the set $B_{n,n+1}$ from the previous sections, but since it depends only on one parameter, we use the shortcut notation $B_n$. From the previous decomposition of the sum, one has the estimate
\begin{multline*}
\left| \left\langle \nabla \psi_{n+1,q} \right\rangle_{\cu_{n+1}} - 3^{-d} \sum_{z \in \mathcal{Z}_{n}} \left\langle \nabla \psi_{n+1,q} - \nabla \psi_{z} \right\rangle_{z + \cu_{n}} - 3^{-d}\sum_{z \in \mathcal{Z}_{n}} \left\langle \nabla \psi_{z} \right\rangle_{z + \cu_{n}} \right| \\ \leq  \frac1{|\cu_{n+1}|} \sum_{e \in B_n} \left| \nabla \psi_{n+1,q}(e) \right|.
\end{multline*}
Taking the square-root of the variance and using the triangle inequality, one obtains
\begin{multline*}
\var^{\frac 12} \left[\left\langle \nabla \psi_{n+1,q} \right\rangle_{\cu_{n+1}}  \right] \leq \var^{\frac 12} \left[ 3^{-d} \sum_{z \in \mathcal{Z}_{n}}\left\langle \nabla \psi_{n+1} - \nabla \psi_{z} \right\rangle_{z + \cu_n} \right] \\+ \var^{\frac 12} \left[3^{-d} \sum_{z \in \mathcal{Z}_{n}}\left\langle \nabla \psi_{z} \right\rangle_{z + \cu_n}  \right] + \E \left[ \left( \frac1{|\cu_{n+1}|} \sum_{e \in B_n} \left| \nabla \psi_{n+1,q}(e) \right| \right)^2 \right]^{\frac 12}.
\end{multline*}
We then estimate the three terms on the right-hand side separately. The first term is an error term which can be estimated by Proposition~\ref{2sccomparison},
\begin{align*}
\var\left[ 3^{-d} \sum_{z \in \mathcal{Z}_{n}}\left\langle \nabla \psi_{n+1} - \nabla \psi_{z} \right\rangle_{z + \cu_n} \right]  &  \leq \E \left[   \frac1{|\cu_{n+1}|} \sum_{z \in \mathcal{Z}_{n}} \sum_{e \subseteq z+ \cu_n} \left| \nabla \psi_{n+1}(e) - \nabla \psi_{z}(e) \right|^2 \right] \\
																						& \leq C\tau_n^*(q) + C(1+|q|^2 ) 3^{-n}.
\end{align*}
To estimate the second term, we use the independence of the random variables $\psi_z$ and the concentration inequality for a sum of independent random variables,
\begin{equation*}
 \var \left[3^{-d} \sum_{z \in \mathcal{Z}_{n}}\left\langle \nabla \psi_{z} \right\rangle_{z + \cu_n}  \right] = 3^{-2d} \sum_{z \in \mathcal{Z}_{n}} \var \left[ \left\langle \nabla \psi_{z} \right\rangle_{z + \cu_n} \right].
\end{equation*}
Using that the law of $\psi_z$ on the cube $(z + \cu_n)$ is $\P_{n,q}^*$, one obtains
\begin{equation*}
 \var \left[3^{-d} \sum_{z \in \mathcal{Z}_{n}}\left\langle \nabla \psi_{z} \right\rangle_{z + \cu_n}  \right] =3^{-d} \var \left[ \left\langle \nabla \psi_{n,q} \right\rangle_{ \cu_n} \right].
\end{equation*}
The third term is also an error term and can be estimated thanks to the Cauchy-Schwarz inequality together with the bound $|B_n| \leq C 3^{-n} |\cu_{n+1}|$,
\begin{align*}
 \E \left[ \left| \frac1{|\cu_{n+1}|} \sum_{e \in B_n} \nabla \psi_{n+1}(e) \right|^2 \right] & \leq \E \left[  \frac{|B_n|}{|\cu_{n+1}|^2} \sum_{e \in B_n} \left| \nabla \psi_{n+1}(e) \right|^2 \right] \\
															& \leq C3^{-n} \E \left[  \frac{1}{|\cu_{n+1}|} \sum_{e \in B_n} \left| \nabla \psi_{n+1}(e) \right|^2 \right].
\end{align*}
By the bound on the $L^2$-norm of the gradient $\nabla \psi_{n+1}$ obtained in Proposition~\ref{propofnunustar}, one obtains
\begin{equation*}
 \E \left[ \left| \frac1{|\cu_{n+1}|} \sum_{e \in B_n} \nabla \psi_{n+1}(e) \right|^2 \right]  \leq C3^{-n} (1 + |q|^2),
\end{equation*}
for some constant $C := C(d, \lambda) < \infty$. Combining the few previous displays gives the estimate
\begin{equation*}
\var^{\frac 12} \left[ \left\langle \nabla \psi_{n+1,q} \right\rangle_{\cu_{n+1}} \right]  \leq 3^{-\frac d2} \var^{\frac 12} \left[ \left\langle \nabla \psi_{n}(e) \right\rangle_{\cu_n} \right] + C \tau^*_n(q)^\frac 12 + C3^{-\frac n2} (1 + |q|)  .
\end{equation*}

\medskip

\textit{Step 2.} Iteration and conclusion. We denote by
\begin{equation*}
\sigma_n :=  \var^{\frac 12} \left[\left\langle \nabla \psi_{n} \right\rangle_{\cu_n} \right].
\end{equation*} 
The main estimate of Step 1 can be rewritten with this new notation
\begin{equation*}
\sigma_{n+1} \leq 3^{-\frac d2}  \sigma_n + C \tau_n^*(q)^\frac 12 + C3^{-\frac n2} (1 + |q|).
\end{equation*}
 An iteration of the previous display gives
\begin{align*}
\sigma_n & \leq 3^{-\frac{dn}2} \sigma_0 + C \sum_{m=0}^{n} 3^{-\frac{d(n-m)}2} \tau_m^*(q)^\frac 12 + C(1 + |q|) \sum_{m=0}^{n} 3^{-\frac{d(n-m)}2} 3^{-\frac m2} \\
							& \leq 3^{-\frac{dn}2} \sigma_0 + C \sum_{m=0}^{n} 3^{-\frac{d(n-m)}2}  \tau_m^*(q)^\frac 12 + C(1 + |q|) 3^{-\frac{n}2}.
\end{align*}
To simplify the previous estimate, we note that by the estimate~\eqref{4.quadboundsnustar} of Proposition~\ref{propofnunustar}, one has the bound $\sigma_0 \leq C(1 + |q|)$. This implies
\begin{equation*}
\sigma_n  \leq C(1 + |q|) 3^{-\frac n2} + C \sum_{m=0}^{n} 3^{\frac{(m-n)}2} \tau_m^*(q)^\frac 12.
\end{equation*}
Squaring the previous inequality gives
\begin{equation*}
\sigma_n^2 \leq C(1 + |q|^2) 3^{-n} + C \left( \sum_{m=0}^{n} 3^{\frac{(m-n)}2} \tau_m^*(q)^\frac 12 \right)^2 \leq C(1 + |q|^2) 3^{-n} + C  \sum_{m=0}^{n} 3^{\frac{(m-n)}2} \tau_m^*(q).
\end{equation*}
The proof of Step 2 is complete.
\end{proof}

To end this section, we record another useful property of the slope of the random interface $\psi_{n,q}$: thanks to an explicit computation, one can relate the expectation of the slope of the random interface $\psi_{n,q}$ to the gradient in the $q$ variable of the dual surface tension $\nu^*$ according to the formula
\begin{equation*}
 \nabla_q \nu^* (\cu_n, q) = \E \left[ \left\langle \nabla \psi_{n,q} \right\rangle_{\cu_n} \right].
\end{equation*}
This has some interesting consequences: first one derives the bound, for each $q \in \Rd,$
\begin{align} \label{boundnustarqn}
\left| \nabla_q \nu^* (\cu_n, q) \right| & \leq \left(  \E \left[ \frac{1}{|\cu_n|} \sum_{e \subseteq \cu_n} \left| \nabla \psi_{n,q} (e) \right|^2 \right] \right)^\frac 12 \\
						& \leq C(1 + |q|). \notag
\end{align} 

Second, and for future reference, we note that the difference of the gradient of $\nu^*$ on different scales can be controlled by the subadditivity defect. This is stated in the following lemma.
\begin{lemma}
For each $m \leq n$ and each $q \in \Rd$
\begin{equation} \label{splitstep2}
\left|  \nabla_q \nu^* (\cu_m, q) -  \nabla_q \nu^* (\cu_n, q)\right|^2 \leq C \sum_{k=m}^{n} \tau_k^*(q) + C (1 + |q|^2 ) 3^{-m}.
\end{equation}
\end{lemma}

\begin{proof}
The arguments for the proof of this lemma are essentially contained in the proof of Proposition~\ref{contslope}. The main idea is to apply Proposition~\ref{2sccomparison}. We thus consider the coupling between the random variables $\psi_{n,q}$ and $\psi_z$ introduced in this proposition. With this notation, one has 
\begin{equation} \label{Dnucomp1}
\left|  \nabla_q \nu^* (\cu_m, q) -  \nabla_q \nu^* (\cu_n, q)\right|^2 =  \left| E \left[ \left\langle \nabla \psi_{n,q} (e) \right\rangle_{\cu_n} - \frac1{| \mathcal{Z}_{m,n}|}\sum_{z \in \mathcal{Z}_{m,n}} \left\langle \nabla \psi_{z} \right\rangle_{z + \cu_m} \right] \right|^2.
\end{equation}
We reintroduce the set of bonds connecting two cubes of the form $(z + \cu_m)$,
\begin{equation*}
B_{m,n} := \left\{ e = (x,y) \subseteq \cu_n \, : \,  \exists z,z' \in 3^m \Zd \cap \cu_{n}, \, z \neq z', \, x \in z + \cu_m ~ \mbox{and} ~ y \in z' + \cu_m \right\}.
\end{equation*}
We also recall that its cardinality can be estimated according to the formula
\begin{equation*}
| B_{m,n}| \leq C 3^{-m} \left| \cu_n \right|,
\end{equation*}
and that one has the decomposition of the sum
\begin{equation*}
\sum_{e \subseteq \cu_n} = \sum_{e\in B_{m,n}} + \sum_{ z \in \mathcal{Z}_{m,n}}\sum_{e \subseteq z +\cu_m}.
\end{equation*}
Combining this with the equality~\eqref{Dnucomp1}, one obtains
\begin{multline*}
\left|  \nabla_q \nu^* (\cu_m, q) - \nabla_q \nu^* (\cu_n, q)\right|^2 \leq 2  \E \left[  \frac1{|\cu_n|}\sum_{z \in  \mathcal{Z}_{m,n}} \sum_{e \subseteq z + \cu_m} \left| \nabla \psi_{n,q} (e) -\nabla \psi_{z} (e)\right|^2 \right] 
					\\ + 2 \E \left[  \left| \frac{1}{|\cu_n|}   \sum_{e \in B_{m,n}} \nabla \psi_{n,q} (e)  \right|^2 \right].
\end{multline*}
The first term on the right-hand side can be estimated by Proposition~\ref{2sccomparison}, and the second one by the Cauchy-Schwarz inequality, similarly to what is written in Proposition~\ref{contslope}. This implies the estimate~\eqref{splitstep2}.
\end{proof}

\subsection{$L^2$ contraction of the field $\psi_{n,q}$ to an affine function} The main objective of this section is to combine the multiscale Poincar\'e inequality with the contraction of the variance of the slope of the field $\psi_{n,q}$ proved in the previous section, to obtain that the field $\psi_{n,q}$ is close in the $L^2$-norm to an affine function. The right-hand side of the estimate still depends on the subadditivity defects $\tau_n^*(q)$ which are small as $n$ tends to infinity. This is stated in the following proposition.

\begin{proposition}\label{p.ustarminimizerareflat}
There exists a constant $C := C(d, \lambda) < \infty$ such that, for every $n \in \N$ and every $q \in \Rd$,
\begin{equation} \label{nustarminimizerareflat}
\E \left[\frac1{|\cu_{n}|} \sum_{x \in \cu_{n}} \left| \psi_{n,q}(x) -  \nabla_q \nu^*(\cu_{n}, q) \cdot x  \right|^2 \right] \leq C 3^{2n} \left((1 + |q|^2) 3^{-\frac n2} + \sum_{m=0}^n 3^{\frac{(m-n)}{2}}\tau_m^*(q)\right).
\end{equation} 
\end{proposition}

\begin{proof}
By the discrete version of the multiscale Poincar\'e inequality stated in Proposition~\ref{p.poincmultpoinc}, one has
\begin{multline} \label{rmultscalepoincare}
\frac1{|\cu_{n}|} \sum_{x \in \cu_{n}} \left| \psi_{n,q}(x) -  \nabla_q \nu^*(\cu_{n}, q) \cdot x \right|^2  \leq C \frac1{|\cu_{n}|} \sum_{e \subseteq \cu_{n+1}} \left| \nabla \psi_{n,q}(e) -  \nabla_q \nu^*(\cu_{n}, q) \cdot e \right|^2 
						\\ + C 3^n \sum_{m = 0}^n 3^m \left( \frac{1}{|\mathcal{Z}_{m,n}|} \sum_{z \in \mathcal{Z}_{m,n}} \left| \left\langle \nabla \psi_{n,q} \right\rangle_{z + \cu_m} -  \nabla_q \nu^*(\cu_{n}, q) \right|^2 \right).
\end{multline}
By Proposition~\ref{propofnunustar} and the inequality~\eqref{boundnustarqn}, we can bound the expectation of the first term on the right-hand side,
\begin{align*}
 \lefteqn{\E \left[ \frac1{|\cu_{n}|} \sum_{e \subseteq \cu_{n}} \left| \nabla \psi_{n,q}(e) -  \nabla_q \nu^*(\cu_{n}, q) \cdot e \right|^2  \right]} \qquad & \\ &
				\leq \E \left[ \frac1{|\cu_{n}|} \sum_{e \subseteq \cu_{n}} \left| \nabla \psi_{n,q}(e) \right|^2 + \left|  \nabla_q \nu^*(\cu_{n}, q) \right|^2  \right]  \\
				& \leq C(1 + |q|^2).
\end{align*}

We then split the proof into 2 steps:
\begin{itemize}
\item in Step 1, we estimate the expectation of the second term on the right-hand side and prove the estimate, for any integer $m \leq n$,
\begin{multline} \label{splitmultscale2}
 \frac{1}{|\mathcal{Z}_{m,n}|} \sum_{z \in \mathcal{Z}_{m,n}}  \E  \left[ \left| \left\langle \nabla \psi_{n,q} \right\rangle_{z + \cu_m} -  \nabla_q \nu^*(\cu_{n}, q) \right|^2 \right] \\ \leq C(1 + |q|^2) 3^{-m} +  C \sum_{k=0}^m 3^{\frac{(k-m)}{2}} \tau_k^*(q) + C \sum_{k=m}^n \tau_k^*(q) ;
\end{multline}
\item in Step 2, we deduce the estimate~\eqref{nustarminimizerareflat} from the previous display.
\end{itemize}

\medskip

\textit{Step 1.} To prove the inequality~\eqref{splitmultscale2}, the two main ingredients are Proposition~\ref{contslope} and Proposition~\ref{2sccomparison}. To apply Proposition~\ref{2sccomparison}, we first recall the definition of the family random variables $\psi_z$, for $z \in \mathcal{Z}_{m,n}$ introduced in its statement as well as the coupling between the random interfaces $\psi$ and $\psi_z$ which satisfies 
\begin{equation} \label{split1}
\E \left[ \frac{1}{|\cu_n|}\sum_{z \in\mathcal{Z}_{m,n}} \sum_{e \subseteq z+ \cu_{m}} \left| \nabla \psi_z (e) - \nabla \psi_{n,q} (e)\right|^2  \right] \leq C \sum_{k=m}^n \tau_k^*(q) + C(1 + |q|^2) 3^{-m}.
\end{equation}
We decompose the term in the left-hand side of~\eqref{splitmultscale2}
\begin{align*}
\lefteqn{ \frac{1}{|\mathcal{Z}_{m,n}|} \E \left[ \sum_{z \in \mathcal{Z}_{m,n}} \left| \left\langle \nabla \psi_{n,q} \right\rangle_{z + \cu_m} -  \nabla_q \nu^*(\cu_{n}, q) \right|^2 \right]}  \qquad\qquad\qquad & \\ &
					\leq \frac{3}{|\mathcal{Z}_{m,n}|} \sum_{z \in \mathcal{Z}_{m,n}}  \frac1{|\cu_m|}  \E \left[ \sum_{e \subseteq z+ \cu_m} \left| \nabla \psi_{n,q}(e) - \nabla \psi_{z}(e) \right|^2 \right] \\
					& \qquad	+\frac{3}{|\mathcal{Z}_{m,n}|} \sum_{z \in \mathcal{Z}_{m,n}}  \E \left[ \left| \left\langle \nabla \psi_{z} \right\rangle_{z + \cu_m}  - \nabla_q \nu^*(\cu_{m}, q) \right|^2 \right]	\\
					& \qquad	+3\left| \nabla_q \nu^*(\cu_{n}, q) - \nabla_q \nu^*(\cu_{m}, q) \right|^2,
\end{align*}
and estimate the three terms separately. The first and third terms term on the right-hand side are an error terms which can be estimated thanks to ~\eqref{split1} and~\eqref{splitstep2} respectively. This gives
\begin{align*}
\lefteqn{ \frac{1}{|\mathcal{Z}_{m,n}|} \E \left[ \sum_{z \in \mathcal{Z}_{m,n}} \left|\left\langle \nabla \psi_{n,q} \right\rangle_{z + \cu_m} - \nabla_q \nu^*(\cu_{n}, q)\right|^2 \right]}  \qquad\qquad\qquad & \\ &
				\leq \frac{3}{|\mathcal{Z}_{m,n}|} \sum_{z \in \mathcal{Z}_{m,n}}  \frac1{|\cu_m|} \E \left[ \sum_{e \subseteq  z+\cu_m} \left| \nabla \psi_{z}(e)  - \nabla_q \nu^*(\cu_{m}, q) \cdot e \right|^2 \right]	 \\
				& \qquad + C \sum_{k=m}^n \tau_k^*(q) + C(1 + |q|^2) 3^{-m}.
\end{align*}
Using the identity $\nabla_q \nu^*(\cu_{m}, q) :=\E \left[ \left\langle \nabla \psi_{m,q} \right\rangle_{\cu_m} \right] $ and Proposition~\ref{contslope}, one can estimate the remaining term,
\begin{align*}
 \frac{1}{|\mathcal{Z}_{m,n}|} \sum_{z \in \mathcal{Z}_{m,n}}  \E \left[ \left| \left\langle \nabla \psi_{z} \right\rangle_{z + \cu_m}  - \nabla_q \nu^*(\cu_{m}, q) \right|^2 \right]  
		& = \var \left[\left\langle \nabla \psi_{n,q} \right\rangle_{ \cu_m} \right] \\
		& \leq C (1 + |q|^2) 3^{- m} + C \sum_{k=0}^m 3^{\frac{(k-m)}{2}} \tau_k^*(q).
\end{align*}
Combining the previous displays yields
\begin{multline*}
\frac{1}{|\mathcal{Z}_{m,n}|} \sum_{z \in \mathcal{Z}_{m,n}} \E \left[  \left|\left\langle \nabla \psi_{n,q} \right\rangle_{z + \cu_m} - \nabla_q \nu^*(\cu_{n}, q) \right|^2 \right] \\ \leq   C(1 + |q|^2) 3^{-m} +  C \sum_{k=0}^m 3^{\frac{(k-m)}{2}} \tau_k^*(q) + C \sum_{k=m}^n \tau_k^*(q) .
\end{multline*}
The proof of Step 1 is complete.
\medskip

\textit{Step 2.} To ease the notation, we denote by, for each $m \in \{ 1, \ldots, n \}$,
\begin{equation*}
 X_m := \frac{1}{|\mathcal{Z}_{m,n}|} \sum_{z \in \mathcal{Z}_{m,n}} \left|\left\langle \nabla \psi_{n,q} \right\rangle_{z + \cu_m} - \nabla_q \nu^*(\cu_{n}, q) \right|^2.
\end{equation*}
The main result of Step 1 can be reformulated with this new notation
\begin{equation*}
\E \left[ X_m \right] \leq  C(1 + |q|^2) 3^{-m} +  C \sum_{k=0}^m 3^{\frac{(k-m)}{2}} \tau_k^*(q) + C \sum_{k=m}^n \tau_k^*(q) .
\end{equation*}
By the multiscale Poincar\'e inequality stated in~\eqref{rmultscalepoincare}, one has
\begin{equation*}
\frac1{|\cu_{n}|} \sum_{x \in \cu_{n}} \left| \psi_{n,q}(x) - \nabla_q \nu^*(\cu_{n}, q) \cdot x \right|^2  \leq C (1 + |q|^2)  + C 3^n \sum_{m = 0}^n 3^m  X_m.
\end{equation*}
Taking the expectation gives
\begin{align*}
\E \left[ 3^n \sum_{m = 0}^n 3^m  X_m \right] & \leq C 3^n \sum_{m = 0}^n 3^m  \left( C(1 + |q|^2) 3^{-m} +  C \sum_{k=0}^m 3^{\frac{(k-m)}{2}} \tau_k^*(q) + C \sum_{k=m}^n \tau_k^*(q)  \right) \\
													& \leq C 3^{2n} \left( (1 + |q|^2)  n 3^{-n} + \sum_{k=0}^n  3^{\frac{(k-n)}{2}} \tau_k^*(q) + \sum_{k=0}^n  3^{(k-n)} \tau_k^*(q) \right).
\end{align*}
The previous inequality can be further simplified by appealing to the estimates $n \leq C 3^{\frac n2}$ and $3^{(k-n)} \leq 3^{\frac{(k-n)}{2}}$ and we deduce
\begin{equation*}
\E \left[\frac1{|\cu_{n}|} \sum_{x \in \cu_{n}} \left| \psi_{n,q}(x) - \nabla_q \nu^*(\cu_{n}, q) \cdot x  \right|^2 \right] \leq C 3^{2n} \left((1 + |q|^2) 3^{-\frac n2} + \sum_{m=0}^n 3^{\frac{(m-n)}{2}}\tau_m^*(q)\right).
\end{equation*}
The proof of Proposition~\ref{p.ustarminimizerareflat} is complete.
\end{proof}

\subsection{Convex duality: upper bound} \label{section4.3} The objective of this section is to use the results of the previous sections, and in particular Proposition~\ref{p.ustarminimizerareflat}, to show that the two surface tensions $\nu$ and $\nu^*$ are approximately convex dual to one another. We first introduce the following notation: for each $n \in \N$, we denote by $\cu_n^+$ the triadic cube $\cu_n$ to which one has added a boundary layer of size $1$, i.e.,
\begin{equation*}
\cu_{n}^+ := \left( - \frac{3^n + 1}{2}, \frac{3^n + 1}{2} \right)^d.
\end{equation*}
It is a cube of size $3^n + 2$ and satisfies the following property
\begin{equation*}
\big( \cu_n^+ \big)^{\mathrm{o}} = \cu_n.
\end{equation*}
The statement of the next proposition can be formulated as follows: if the coefficients $\tau_n^*(q)$ are small, then for each $q \in \Rd$, there exists $p \in \Rd$ such that
\begin{equation*}
\nu \left(\cu_{2n}^+ ,p \right) + \nu^* \left(\cu_{n} , q \right) - p  \cdot q \hspace{5mm} \mbox{is small,}
\end{equation*}
and one can choose the value of $p = \nabla_q \nu^*(\cu_n, q) $ in the previous display. In a later statement, we remove the condition $\cu_{2n}^+ $ and prove that, for each $q \in \Rd$, there exists $p \in \Rd$ such that
\begin{equation*}
\nu \left(\cu_n ,p \right) + \nu^* \left(\cu_{n} , q \right) - p  \cdot q \hspace{5mm} \mbox{is small.}
\end{equation*}
Combining this result with the lower bound on the convex duality proved in Proposition~\ref{convexdual}, one obtains
\begin{equation*}
\nu^* \left(\cu_{n} , q \right) = \inf_{p \in \Rd} \left(  - \nu \left(\cu_n ,p \right)   + p  \cdot q \right)  \hspace{5mm} \mbox{up to a small error.}
\end{equation*}
The main argument in the proof of Proposition~\ref{lemma3.3} is a patching construction: we need to patch functions defined on the cubes $(z + \cu_n)$ of laws $\P_{n,q}^*$ in the (much larger) cube $\cu_{2n}^+$, to construct a random variable on the space $h^1_0 \left( \cu_{2n}^+ \right)$ whose law is then used as a test measure in the variational formulation for $\nu$ stated in~\eqref{varfornu}.

\begin{proposition} \label{lemma3.3}
There exist a constant $C := C(d , \lambda) < \infty$ and an exponent $\beta := \beta (d , \lambda) > 0$ such that for each $q \in \Rd$ and each $n \in \N$,
\begin{multline*}
\nu \left(\cu_{2n}^+ , \nabla_q \nu^*(\cu_n, q) \right) + \nu^* \left(\cu_{n} , q \right) -\nabla_q \nu^*(\cu_n, q)  \cdot q \\ \leq  C \left(( 1 + |q|^2) 3^{-\beta n} + \sum_{m=0}^n 3^{\frac{(m-n)}{2}}\tau_m^*(q)\right).
\end{multline*}
\end{proposition}

\begin{proof}
Fix $q \in \Rd$ and for each $p \in \Rd$, we denote by $l_p$ the linear function of slope $p$.
 To simplify the notation, we also write, for each $q \in \Rd$ and each $n \in \N$,
\begin{equation*}
\nabla \nu^*_n(q) :=  \nabla_q \nu^*(\cu_n, q)  \in \Rd.
\end{equation*}
We also recall the notation~\eqref{constantvectfield} introduced in Section~\ref{section1} which is used in the proof with the value $p = \nabla \nu^*_n(q)$.

The strategy of the proof is the following. We construct a random variable taking values in $ h^1_0 \left( \cu_{2n}^+ \right)$, denoted by $\kappa^+_{2n}$ in the proof. This random variable is essentially constructed by patching together independent random variables, which are defined on the triadic cubes $(z + \cu_n)$, for $z \in \mathcal{Z}_{n,2n}$ and whose laws are the law of $\psi_{n, q} - l_{\nabla \nu^*_n(q)}$. We then prove that the random variable $\kappa^+_{2n}$ satisfies
\begin{multline*}
\E \left[ \frac{1}{\left| \cu_{2n}^+\right|} \sum_{e \subseteq \cu_{2n}^+}  V_e \left(\nabla \nu^*_n (q) (e) +  \nabla \kappa^+_{2n} (e) \right) \right] +  \frac{1}{\left| \cu_{2n}^+\right|}  H \left( \P_{\kappa^+_{2n}}  \right) \leq - \nu^* \left(\cu_{n} , q \right) + \nabla \nu^*_{n} (q) \cdot q \\ + C \left(( 1 + |q|^2) 3^{-\beta n} + \sum_{m=0}^n 3^{\frac{(m-n)}{2}}\tau_m^*(q)\right).
\end{multline*}
We finally use $\kappa^+_{2n}$ as a test function in the variational formulation of $\nu \left(\cu_{2n}^+ , \nabla \nu^*_{n} (q) \right)$, to obtain the inequality
\begin{equation*}
\nu \left(\cu_{2n}^+ , \nabla \nu^*_{n} (q) \right) \leq \E \left[ \frac{1}{\left| \cu_{2n}^+\right|} \sum_{e \subseteq \cu_{2n}^+}  V_e \left( \nabla \nu^*_n (q)  +  \nabla \kappa^+_{2n} (e) \right) \right] +  \frac{1}{\left| \cu_{2n}^+\right|}  H \left( \P_{\kappa^+_{2n}}\right) .
\end{equation*}
Combining the two previous displays completes the proof.

\medskip

We split the proof into 4 steps:
\begin{itemize}
\item in Step 1, we construct the random variable $\kappa^+_{2n}$ taking values in $ h^1_0 \left( \cu_{2n}^+ \right)$;
\item in Step 2, we show that the entropy of $\kappa^+_{2n}$ is controlled by the entropy of $\P^*_{n,q}$. Precisely, we prove
\begin{equation} \label{mainstep2entropy}
\frac{1}{\left| \cu_{2n}^+\right|} H \left( \P_{\kappa^+_{2n}}  \right) \leq \frac{1}{\left| \cu_{n}\right|} H \left( \P^*_{n,q} \right) + C n 3^{-n},
\end{equation}
where the entropy on the left-hand side is computed with respect to the Lebesgue measure on $h^1_0 \left( \cu_{2n}^+ \right)$ and the entropy on the right-hand side is computed with respect to the Lebesgue measure on $\mathring h^1 (\cu_n)$;
\item in Steps 3 and 4, we show that the energy of the random variable $\kappa^+_{2n}$ is controlled by the energy of the field $\psi_{n,q}$. Precisely, we prove
\begin{multline} \label{mainstep2energy}
\E \left[ \frac{1}{\left| \cu_{2n}^+\right|} \sum_{e \subseteq \cu_{2n}^+}  V_e \left( \nabla \nu^*_n (q) + \nabla \kappa^+_{2n} (e) \right) \right]  \leq \E \left[ \frac{1}{\left| \cu_{n}\right|} \sum_{e \subseteq \cu_{n}}  V_e \left( \nabla \psi_{n,q} (e) \right) \right] \\+ C \left(( 1 + |q|^2) 3^{-\beta n} + \sum_{m=0}^n 3^{\frac{(m-n)}{2}}\tau_m^*(q)\right);
\end{multline}
\item in Step 5, we combine the results of Steps 3 and 4 to prove
\begin{equation*}
\nu \left(\cu_{2n}^+ , \nabla \nu^*_{n} (q) \right)  - \nu^*(\cu_n,q) + q \cdot \nabla \nu^*_n (q) \leq C \left(( 1 + |q|^2) 3^{-\beta n} + \sum_{m=0}^n 3^{\frac{(m-n)}{2}}\tau_m^*(q)\right).
\end{equation*}
\end{itemize}

\medskip

\textit{Step 1.} Denote by $h^1 (\cu_{2n})$ the set of functions from the cube $\cu_{2n}$ to $\R$. There exists a canonical bijection between the spaces $h^1 (\cu_{2n})$ and $h^1_0 (\cu_{2n}^+)$ obtained by extending the functions of $h^1 (\cu_{2n})$ to be $0$ on the boundary of the cube $\cu_{2n}^+$. We first explain the strategy to construct the random interface $\kappa_{2n}^+$. We consider a family $\left(\psi_z \right)_{z \in \mathcal{Z}_{n,2n}}$ of random variables satisfying:
\begin{itemize}
\item for each $z \in \mathcal{Z}_{n,2n}$,  $\psi_z$ is valued in $\mathring h^1 \left( z + \cu_{n} \right)$ and its law is $\P_{n,q}^*$;
\item the random variables $\psi_z$ are independent.
\end{itemize}
We recall the definition for the set of edges connecting two triadic cubes of the form~$(z + \cu_n)$ in the cube $\cu_{2n}$,
\begin{equation*}
B_{n,2n} := \left\{ e = (x,y) \subseteq \cu_{2n} \, : \,  \exists z,z' \in \mathcal{Z}_{n,2n}, \, z \neq z', \, x \in z + \cu_n \, \mbox{and} \, y \in z' + \cu_n \right\},
\end{equation*}
as well as the corresponding partition of edges of the cube $\cu_{2n}$,
\begin{equation*}
 e \subseteq \cu_{2n} ~ \implies ~ \exists z \in 3^n \Zd \cap \cu_{2n}, \, e \subseteq z +\cu_{n} \, \mbox{or} \, e \in B_{n,2n}.
\end{equation*}
We then construct a random vector field $\mathbf{f}$ defined on the edges of the cube $\cu_{2n}$ by patching together the vector fields $\nabla \psi_z - \nabla \nu^*_n(q)$ defined on the edges of the cube $(z + \cu_n)$. Precisely, the vector field $\mathbf{f}$ is defined as follows, for each edge $e \subseteq \cu_{2n}$,
\begin{equation} \label{def.mathbff}
\mathbf{f}(e) =  \left\{ \begin{aligned}
\nabla \psi_z (e) - \nabla \nu^*_n(q)(e) ~& \mbox{if}~ e \subseteq z + \cu_n, \,  \mbox{for some} \, z \in \mathcal{Z}_{n,2n}, \\
0 \hspace{6mm} & \mbox{if}~ e \in B_{n,2n}.
\end{aligned}  \right.
\end{equation}
The objective is to construct a random variable $\kappa_{2n}^+$ which belongs to the space $h^1_0 \left( \cu_{2n}^+ \right)$ and whose gradient is close to the vector field $\mathbf{f}$.

A first obstruction is that the vector field $\mathbf{f}(e)$ is not in general the gradient of a function in $h^1_0 \left(\cu_{2n}^+\right)$. To remedy this, we consider the orthogonal projection the vector field $\mathbf{f}$ on the space of gradients of functions in $h^1_0 \left(\cu_{2n}^+\right)$. This corresponds to solving the Dirichlet boundary value problem
\begin{equation} \label{projgradh10}
\left\{ \begin{aligned}
\Delta \kappa & = \div \mathbf{f} ~ \mbox{in} ~ \cu_{2n}, \\
\kappa & \in h^1_0 \left( \cu_{2n}^+ \right). 
\end{aligned}  \right.
\end{equation}

A second obstruction is that the random variable $\kappa$ defined in~\eqref{projgradh10} almost surely belongs to a strict linear subspace of $h^1_0 \left(\cu_{2n}^+\right)$, consequently its law is not absolutely continuous with respect to the Lebesgue measure on $h^1_0 \left(\cu_{2n}^+\right)$ and its entropy is infinite. To remedy this we add some independent random variables whose law uniform on $[0,1]$, as was done in Proposition~\ref{2sccomparisonnu}.

\medskip

We now give the details of the construction. We consider the orthogonal decomposition with respect to the standard $L^2$ scalar product 
\begin{equation} \label{orthdecompL}
h^1 (\cu_{2n}) = \underset{z \in \mathcal{Z}_{n,2n}}{\oplus} \mathring h^1(z + \cu_n) \,  \overset{\perp}{\oplus} H,
\end{equation}
where $H := \left(\underset{z \in \mathcal{Z}_{n,2n}}{\oplus} \mathring h^1(z + \cu_n) \right)^\perp$ is the vector space of functions which are constant on the subcubes~$\left(z + \cu_n\right)$. Its dimension is $3^{dn}$.

\smallskip

We then consider a linear operator denoted by $L$ defined on $ \underset{z \in \mathcal{Z}_{n,2n}}{\oplus} \mathring h^1(z + \cu_n) $ valued in $h^1_0 \left(\cu_{2n}^+\right)$ defined according to the following procedure. 

\medskip

For each $\psi \in \underset{z \in \mathcal{Z}_{n,2n}}{\oplus} \mathring h^1(z + \cu_n) $, we let $L(\psi)$ be the unique solution to
\begin{equation} \label{rigorousprojh10grad}
\left\{ \begin{aligned}
\Delta L(\psi) & = \div f ~ \mbox{in} ~ \cu_{2n} \\
L(\psi) & \in h^1_0 \left( \cu_{2n}^+ \right),
\end{aligned}  \right.
\end{equation}
where $f$ is the vector field defined by, for each $e \subseteq \cu_{2n}$,
\begin{equation*}
 f(e) =  \left\{ \begin{aligned}
\nabla \psi (e) ~& \mbox{if}~ e \subseteq z + \cu_n, \,  \mbox{for some} \, z \in \mathcal{Z}_{n,2n}, \\
0 \hspace{6mm} & \mbox{if}~ e \in B_{2n,n}.
\end{aligned}  \right.
\end{equation*}
We first verify that the operator $L$ is injective. To this end, we check that the kernel of $L$ is reduced to $\{ 0 \}$. Let $\psi \in \underset{z \in \mathcal{Z}_{n,2n}}{\oplus} \mathring h^1(z + \cu_n)$ such that $L(\psi) = 0$, we want to prove that $\psi = 0$. 

\smallskip

First by definition of $L$, the condition $L(\psi) = 0$ implies $\div f = 0$. The function $\div f$ can be computed explicitly and we have, for each $z \in \mathcal{Z}_{n,2n}$
\begin{equation*}
\Delta_{z + \cu_n} \psi = 0~ \mbox{in}~ (z + \cu_n),
\end{equation*}
where $\Delta_{z + \cu_n}$ is the Laplacian on the graph $(z + \cu_n)$ and is defined by, for each $x \in \cu_n$,
\begin{equation*}
\Delta_{z + \cu_n}  \psi  (x)= \sum_{y \sim x, y \in z + \cu_n} \left( \psi (y) - \psi(x)\right).
\end{equation*} 
Note that this Laplacian is different from the standard Laplacian on the cube $\cu_{2n}^+$ which is used in~\eqref{rigorousprojh10grad} and defined by, for each $x \in \cu_{2n},$
\begin{equation*}
\Delta  \psi  (x)= \sum_{y \sim x} \left( \psi (y) - \psi(x)\right).
\end{equation*}
This difference comes from the fact that $f$ was set to be $0$ on the set $B_{2n,n}$. For the Laplacian on the cube $(z + \cu_n)$, one can apply the maximum principle to obtain, for each $z \in \mathcal{Z}_{n,2n}$,
\begin{equation*}
\Delta_{z + \cu_n} \psi = 0~ \mbox{in}~ (z + \cu_n) ~\implies ~ \psi ~ \mbox{ is constant in}~ z + \cu_n.
\end{equation*}
Combining the previous remark with the assumption that the function $\psi$ belongs to the space $\underset{z \in \mathcal{Z}_{n,2n}}{\oplus} \mathring h^1(z + \cu_n) $ gives $\psi = 0$ and thus $\ker L = \{0\}$. In particular, if we denote by $\im L$ the image of $L$, one has
\begin{equation*}
\dim(\im L) = \dim \left( \underset{z \in\mathcal{Z}_{n,2n}}{\oplus} \mathring h^1(z + \cu_n) \right) = 3^{2dn} - 3^{dn}. 
\end{equation*}
We then extend $L$ into an isomorphism of $h^1 (\cu_{2n})$ into $h^1_0 \left(\cu_{2n}^+\right)$. We consider the orthogonal decomposition
\begin{equation*}
h^1 (\cu_{2n}) = \underset{z \in \mathcal{Z}_{n,2n}}{\oplus} \mathring h^1(z + \cu_n) \,  \overset{\perp}{\oplus} H,
\end{equation*}
and we let $h_1, \ldots, h_{3^{nd}}$ be an orthonormal basis of $H$. Consider now the $L^2$ orthogonal decomposition
\begin{equation} \label{orthdecompL-1}
h^1 (\cu_{2n}) = \im L \, \oplus \, (\im L)^\perp.
\end{equation}
By the injectivity of $L$, we have $\dim \, (\im L)^\perp = 3^{dn}$. We let $\tilde h_1, \ldots, \tilde h_{3^{nd}}$ be an orthonormal basis of $\left(\im L\right)^\perp$ and extend the linear operator $L$ to the space $h^1 (\cu_{2n})$ by setting
\begin{equation} \label{4.155}
L(h_i) = \tilde h_i, ~\forall i \in \{ 1 , \ldots, 3^{dn} \}.
\end{equation}
By construction, the linear mapping $L$ is an isomorphism between $h^1 (\cu_{2n})$ and $h^1_0 \left(\cu_{2n}^+\right)$.

\smallskip

We now construct the random variable $\kappa_{2n}^+$ using the operator $L$. To this end, we consider two families $\left(\psi_z \right)_{z \in \mathcal{Z}_{n,2n}}$ and $\left(X_i \right)_{i = 1 , \ldots, 3^{nd}}$ of random variables satisfying:
\begin{itemize}
\item for each $z \in \mathcal{Z}_{n,2n}$,  $\psi_z$ is valued in $\mathring h^1 \left( z + \cu_{n} \right)$ and its law is $\P_{n,q}^*$. We extend it by $0$ outside the cube $ (z + \cu_{n}) $ so that it can be seen as a random variable taking values in $h^1 \left( \cu_{2n} \right)$;
\item for each $i \in \{ 1 , \ldots, 3^{nd} \}$, $X_i$ is valued in $[0,1]$ and its law is $\mathrm{Unif}[0,1]$;
\item the random variables $\psi_z$ and $X_i$ are independent.
\end{itemize}
We also define, for each $z \in \mathcal{Z}_{n,2n}$, the random variable $\sigma_z$ taking values in $\mathring h^1 \left( \cu_{2n} \right)$ defined by subtracting the affine function of slope $\nabla \nu^*(q)$ to $\psi_z$, i.e., for each $x \in \cu_{2n}$,
\begin{equation} \label{def.sigmaz}
 \sigma_z(x) :=  \left\{ \begin{aligned}
\psi_z(x) - \nabla \nu^*_n(q) \cdot (x - z) ~& \mbox{if}~x \in  z + \cu_n, \\
0 \hspace{6mm} & \mbox{otherwise}.
\end{aligned}  \right.
\end{equation}
Let $\kappa$ and $\kappa_{2n}^+$ be the random variables valued in $h^1 \left( \cu_{2n} \right)$ defined by the formulas
\begin{equation} \label{defofkappa}
\kappa := L \left( \sum_{z \in \mathcal{Z}_{n,2n}} \sigma_z \right) \qquad \mbox{and} \qquad \kappa_{2n}^+  := L \left( \sum_{z \in \mathcal{Z}_{n,2n}} \sigma_z + \sum_{i=1}^{3^{nd}} X_i h_i \right).
\end{equation}
With this definition, it is clear that $\kappa_{2n}^+$ is a random variable valued in $h^1_0 \left(\cu_{2n}^+\right)$ and that its law is absolutely continuous with respect to the Lebesgue measure on this space. Moreover by construction of the operator $L$, the random variable $\kappa_{2n}^+$ is one of the functions of $h^1_0 \left(\cu_{2n}^+\right)$ which is the closest in the $L^2$-norm to the vector field $\mathbf{f}$ defined in~\eqref{def.mathbff}. 

\medskip

\textit{Step 2.} In this step, we compute the entropy of the random variable $\kappa_{2n}^+$.

Using the canonical bijection between $h^1 \left( \cu_{2n} \right)$ and $h^1_0 \left( \cu_{2n}^+ \right)$, one can see the operator $L$ as an automorphism of $h^1_0 \left( \cu_{2n}^+ \right)$. Combining this remark with the change of variables formula for the differential entropy, one computes the entropy of the random variable $\kappa_{2n}^+$,
\begin{equation} \label{entropiekappa.eq}
H \left( \P_{\kappa_{2n}^+}  \right) = H \left( \P_{\sum_{z \in \mathcal{Z}_{n,2n}} \sigma_z + \sum_{i=1}^{3^{nd}} X_i h_i} \right) - \ln \left| \det L \right|.
\end{equation}
We first focus on the first term on the right-hand side. By construction of $\sigma_z, X_i, h_i$ and using the formula to compute the entropy of two independent random variables given in Proposition~\ref{entropycouplesum}, one has
\begin{equation*}
H \left( \P_{\sum_{z \in \mathcal{Z}_{n,2n}} \sigma_z + \sum_{i=1}^{3^{nd}} X_i h_i}  \right) = \sum_{z \in \mathcal{Z}_{n,2n}} H \left( \P_{\sigma_z} \right) + \sum_{i = 1}^{3^{nd}} H \left( \P_{X_i}  \right).
\end{equation*} 
Using that the law of $X_i$ is uniform in $[0,1]$ and since the entropy of a random variable is translation invariant  the previous display can be further simplified
\begin{equation*}
H \left( \P_{\sum_{z \in \mathcal{Z}_{n,2n}} \sigma_z + \sum_{i=1}^{3^{nd}} X_i h_i} \right)  = \sum_{z \in \mathcal{Z}_{n,2n}}  H \left( \P_{\psi_z} \right). 
\end{equation*}
Since the law of the random interface $\psi_z$ is $\P_{n,q}^*$, one obtains
\begin{equation} \label{4.16bis}
H \left( \P_{\kappa_{2n}^+}  \right) = 3^{dn} H \left(\P_{n,q}^*\right) - \ln \left| \det L \right|.
\end{equation}
We now focus on the second term on the right-hand side of~\eqref{entropiekappa.eq}. More precisely, we prove that the logarithm of the determinant of $L$ is small compared to the volume of the cube $\cu_n$: more precisely, we show the bound
\begin{equation} \label{est0detL}
\left|  \ln \left| \det L \right| \right| \leq C 3^{(2d-1)n}n .
\end{equation} 
Combining~\eqref{4.16bis} with~\eqref{est0detL} gives the main result~\eqref{mainstep2entropy} of Step 2.

To prove~\eqref{est0detL}, note that the dimension of the vector space $h^1_0 \left( \cu_{2n}^+ \right)$ is $3^{2dn}$. Denote by $(l_1 , \ldots , l_{3^{2dn}})$ the (potentially complex) eigenvalues of $L$. Note that since $L$ is bijective, none of these eigenvalues is equal to $0$. With this notation the determinant of $L$ can be computed,
\begin{equation} \label{e.lndetL1}
 \ln \left| \det L \right| = \sum_{i = 1}^{3^{2dn}}  \ln \left| l_i \right| .
\end{equation}
To prove that the logarithm of the determinant of $L$ is small, the strategy relies on the two following ingredients:
\begin{enumerate}
\item we prove that most of the eigenvalues of $L$ are equal to $1$;
\item we prove that the remaining eigenvalues are bounded from above and below by the values $C3^{4n}$ and $c3^{-4n}$, respectively.
\end{enumerate}

We focus on the first item. To prove this result, we consider the interior of the cube~$\cu_n$:
\begin{equation*}
\cu_n^{\mathrm{o}} := \left( - \frac{3^n - 2}{2} , \frac{3^n - 2}{2} \right)^d \cap \Zd = \cu_n \setminus \partial \cu_n.
\end{equation*}
In particular, one has the inclusion $\cu_n^{\mathrm{o}} \subseteq \cu_n$ and thus the vector space $\mathring h^1 \left( \cu_n^{\mathrm{o}}\right)$ is a linear subspace of $\mathring h^1 \left(\cu_n\right)$. This implies that the space $\underset{z \in \mathcal{Z}_{n,2n}}{\oplus} \mathring h^1 \left(z + \cu_n^{\mathrm{o}}\right)$ is a linear subspace of $\underset{z \in \mathcal{Z}_{n,2n}}{\oplus} \mathring h^1 \left(z + \cu_n\right)$. The important observation is that for each $\psi \in \underset{z \in \mathcal{Z}_{n,2n}}{\oplus} \mathring h^1 \left(z + \cu_n^{\mathrm{o}}\right)$ and each edge $e \in B_{2n,n}$,
\begin{equation*}
\nabla \psi (e) = 0.
\end{equation*}
This is due to the fact that, by definition of the space $\underset{z \in \mathcal{Z}_{n,2n}}{\oplus} \mathring h^1 \left(z + \cu_n^{\mathrm{o}}\right)$, any function belonging to this space is equal to $0$ on the boundaries of the cubes $\left( z + \cu_n \right)$, for any $z \in \mathcal{Z}_{n,2n}$. Consequently the vector field $f$ defined from $\psi$ according to the formula
\begin{equation} \label{def.mathbfff}
f(e) =  \left\{ \begin{aligned}
\nabla \psi (e) ~& \mbox{if}~ e \subseteq z + \cu_n, \,  \mbox{for some} \, z \in \mathcal{Z}_{n,2n}, \\
0 \hspace{6mm} & \mbox{if}~ e \in B_{2n,n},
\end{aligned}  \right.
\end{equation}
satisfies
\begin{equation*}
f  = \nabla \psi.
\end{equation*}
Thus $L(\psi)$ is the solution of
\begin{equation*}
\left\{ \begin{aligned}
\Delta L(\psi) & = \Delta \psi ~ \mbox{in} ~ \cu_{2n}, \\
L(\psi) & \in h^1_0 \left( \cu_{2n}^+ \right). 
\end{aligned}  \right.
\end{equation*}
This implies $L(\psi) = \psi$. This proves $$\psi \in \underset{z \in \mathcal{Z}_{n,2n}}{\oplus} \mathring h^1 \left(z + \cu_n^{\mathrm{o}}\right), ~ L(\psi) = \psi.$$ Consequently, the vector space $\underset{z \in \mathcal{Z}_{n,2n}}{\oplus} \mathring h^1 \left(z + \cu_n^{\mathrm{o}}\right)$ is an eigenspace for $L$ associated to the eigenvalue~$1$, its dimension can be estimated by the following computation
\begin{align*}
\dim \left(  \underset{z \in\mathcal{Z}_{n,2n}}{\oplus} \mathring h^1 \left(z + \cu_n^{\mathrm{o}}\right) \right) & = \sum_{z \in \mathcal{Z}_{n,2n}} \dim \left( \mathring h^1 \left(z + \cu_n^{\mathrm{o}}\right)   \right)\\
																				& = 3^{dn}  \dim \left( \mathring h^1 \left(\cu_n^{\mathrm{o}}\right)   \right)\\
																				& =  3^{dn}  \left( \left| \cu_n^{\mathrm{o}} \right| -1   \right).
\end{align*}
The volume of $\cu_n^{\mathrm{o}}$ can then be estimated according to
\begin{equation*}
 \left| \cu_n^{\mathrm{o}} \right| \geq 3^{dn} - C 3^{(d-1)n}.
\end{equation*}
Combining the two previous displays gives
\begin{equation} \label{estimatedimensioneigenspace1}
\dim \left(  \underset{z \in \mathcal{Z}_{n,2n}}{\oplus} \mathring h^1 \left(z + \cu_n^{\mathrm{o}}\right) \right) \geq 3^{2dn} - C 3^{(2d -1)n}.
\end{equation}
Thus we can, without loss of generality, assume that for each $i \geq C 3^{(2d -1)n}$, $l_i = 1$. The equality~\eqref{e.lndetL1} can be rewritten
\begin{equation*}
 \ln \left| \det L \right| = \sum_{i = 1}^{ C 3^{(2d -1)n}}  \ln \left| l_i \right| .
\end{equation*}
We then use the inequalities
\begin{equation*}
\left| \left| \left| L^{-1} \right| \right| \right|^{-1} \leq \inf_{i \in \{1 , \ldots, 3^{2dn}\} } |l_i| \leq   \sup_{i \in \{1 , \ldots, 3^{2dn}\} } |l_i| \leq ||| L |||,
\end{equation*}
where $ ||| L |||$ (resp. $||| L^{-1} |||$) denotes the operator norm of $L$ (resp. $L^{-1}$) with respect to the $L^2$-norm on $h^1_0 \left( \cu_{2n}^+ \right)$. A combination of the two previous displays gives
\begin{equation} \label{est1detL}
\left| \ln \left| \det L \right|  \right| \leq C 3^{(2d -1)n} \max \left(  \ln  ||| L |||  , \ln  \left| \left| \left| L^{-1} \right| \right| \right| \right).
\end{equation}
To complete the proof, there remains to prove an estimate on the operator norms of $L$ and $L^{-1}$. Specifically, we prove,
\begin{equation} \label{est2detL}
 ||| L ||| \leq C 3^{2n} \qquad \mbox{and} \qquad  \left| \left| \left| L^{-1} \right| \right| \right| \leq C 3^{2n}.
\end{equation}
We first focus on the estimate of the operator norm of $L$. Let $\phi \in h^1_0 \left( \cu_{2n}^+ \right)$ such that $\sum_{x \in \cu_{2n}^+} \phi(x)^2 \leq 1$, one aims to prove
\begin{equation} \label{eigenvaluesnottoolarge}
\sum_{x \in \cu_{2n}^+} \left| L(\phi)(x) \right|^2 \leq C 3^{4n}.
\end{equation}
To this end, we decompose the field $\phi$ according to the orthogonal decomposition~\eqref{orthdecompL}. This gives
\begin{equation*}
\phi  = \psi + h, ~ \mbox{with} ~ \psi \in \underset{z \in \mathcal{Z}_{n,2n}}{\oplus} \mathring h^1(z + \cu_n)  ~ \mbox{and} ~ h \in H.
\end{equation*}
In particular, 
\begin{equation*}
\sum_{x \in \cu_{2n}^+} \left| \psi(x) \right|^2 + \sum_{x \in \cu_{2n}^+} \left| h(x) \right|^2  = \sum_{x \in \cu_{2n}^+} \left|\phi(x) \right|^2 \leq 1.
\end{equation*}
By definition of the operator $L$, the functions $L(\psi)$ and $L(h)$ are orthogonal in $h^1_0 \left( \cu_{2n}^+ \right)$ and 
\begin{equation*}
 \sum_{x \in \cu_{2n}^+} \left| L(h)(x) \right|^2 = \sum_{x \in \cu_{2n}^+}\left| h(x) \right|^2
\end{equation*}
From this we deduce
\begin{align*}
 \sum_{x \in \cu_{2n}^+} \left| L(\phi)(x) \right|^2  =  \sum_{x \in \cu_{2n}^+} \left| L(\psi)(x) \right|^2  + \sum_{x \in \cu_{2n}^+} \left| L(h)(x) \right|^2  & = \sum_{x \in \cu_{2n}^+} \left| L(\psi)(x) \right|^2  + \sum_{x \in \cu_{2n}^+} \left| h(x) \right|^2  \\
																& \leq   \sum_{x \in \cu_{2n}^+} \left| L(\psi)(x) \right|^2   + 1.
\end{align*}
Thus to prove~\eqref{eigenvaluesnottoolarge}, it is sufficient to prove: $$\forall \psi \in \underset{z \in \mathcal{Z}_{n,2n}}{\oplus} \mathring h^1(z + \cu_n) ~\mbox{such that}~ \sum_{x \in \cu_{2n}^+} \left| \psi(x) \right|^2  \leq 1,~
 \sum_{x \in \cu_{2n}^+} \left| L(\psi)(x) \right|^2  \leq C 3^{4n}.$$
For each $\psi \in  \underset{z \in \mathcal{Z}_{n,2n}}{\oplus} \mathring h^1(z + \cu_n)$, we know that the map $L(\psi)$ is a solution to
\begin{equation} \label{testLagitself}
\left\{ \begin{aligned}
\Delta L(\psi) & =\div f~ \mbox{in} ~ \cu_{2n}, \\
L(\psi) & \in h^1_0 \left( \cu_{2n}^+ \right),
\end{aligned}  \right.
\end{equation}
where the vector field $f $ is defined by the formula
\begin{equation*}
f(e) =  \left\{ \begin{aligned}
\nabla \psi (e) ~& \mbox{if}~ e \subseteq z + \cu_n, \,  \mbox{for some} \, z \in \mathcal{Z}_{n,2n}, \\
0 \hspace{6mm} & \mbox{if}~ e \in B_{2n,n}.
\end{aligned}  \right.
\end{equation*}
Using the function $L(\psi)$ as a test function in~\eqref{testLagitself} shows
\begin{equation*}
\sum_{e \subseteq \cu_{2n}^+} \left| \nabla L(\psi)(e) \right|^2 = \sum_{e \subseteq \cu_{2n}^+} \nabla L(\psi) (e) f(e).
\end{equation*}
By the Cauchy-Schwarz inequality, this implies
\begin{equation*}
\sum_{e \subseteq \cu_{2n}^+} \left| \nabla L(\psi)(e) \right|^2 \leq \sum_{e \subseteq \cu_{2n}^+} |f(e)|^2.
\end{equation*}
By definition of $f$, one obtains
\begin{equation} \label{Lgradientcontract}
\sum_{e \subseteq \cu_{2n}^+} \left| \nabla L(\psi)(e) \right|^2 \leq \sum_{z \in \mathcal{Z}_{n,2n}} \sum_{e \subseteq z + \cu_n} |\nabla \psi (e)|^2.
\end{equation}
Using the inequality, for each bond $e = (x,y) \subseteq \cu_{2n}$,
\begin{equation*}
\left| \nabla \psi (e) \right|^2 = \left|  \psi (x) - \psi(y)\right|^2  \leq 2 \left|  \psi (x)\right|^2 + 2 \left|  \psi(y)\right|^2,
\end{equation*}
we derive the estimate, 
\begin{equation*}
\sum_{z \in 3^n \Zd \cap \cu_{2n}} \sum_{e \subseteq z + \cu_n} |\nabla \psi (e)|^2  \leq C \sum_{x \in \cu_{2n}} \left| \psi(x) \right|^2 .
\end{equation*}
Combining the previous displays and using the assumption $\sum_{x \in \cu_{2n}} \left| \psi(x) \right|^2 \leq 1$ shows
\begin{equation*}
\sum_{e \subseteq \cu_{2n}^+} \left| \nabla L(\psi)(e) \right|^2 \leq C \sum_{x \in \cu_{2n}} \left| \psi(x) \right|^2 \leq C.
\end{equation*}
Since $L(\psi) \in h^1_0 \left( \cu_{2n}^+ \right)$, the Poincar\'e inequality implies
\begin{equation*}
\sum_{x \in \cu_{2n}^+} \left| L (\psi)(x) \right|^2 \leq 3^{4n} \sum_{e \subseteq \cu_{2n}^+} \left| \nabla L(\psi)(e) \right|^2.
\end{equation*}
Combining the two previous displays gives
\begin{equation*}
\sum_{x \in \cu_{2n}^+} \left| L (\psi)(x) \right|^2  \leq C 3^{4n}. 
\end{equation*}
This is the desired result.
\smallskip
We now turn to the bound on the operator norm of $L^{-1}$, we aim to prove
\begin{equation*}
 \left| \left| \left| L^{-1} \right| \right| \right| \leq C 3^{2n}.
\end{equation*}
To this end, let $\psi \in h^1_0 \left( \cu_{2n}^+ \right)$. We first prove
\begin{equation} \label{boundL-1.eqq}
\sum_{x \in \cu_{2n}^+} \left| \psi(x) \right|^2 \leq C 3^{4n} \sum_{x \in \cu_{2n}^+} \left| L (\psi)(x) \right|^2.
\end{equation}
Using the same idea as in the proof of the bound for the operator norm of $L$, we see that it is enough to prove~\eqref{boundL-1.eqq} under the additional assumption $\psi \in \underset{z \in \mathcal{Z}_{n,2n}}{\oplus} \mathring h^1(z + \cu_n)$. In this case, one has
\begin{equation*} 
\left\{ \begin{aligned}
\Delta L(\psi) & = \div f~ \mbox{in} ~ \cu_{2n}, \\
L(\psi) & \in h^1_0 \left( \cu_{2n}^+ \right). 
\end{aligned}  \right.
\end{equation*}
Testing this equation with the function $\psi$ gives
\begin{equation*}
\sum_{x \in \cu_{2n}} \Delta L(\psi)(x) \psi(x) =  \sum_{x \in \cu_{2n}} \div f (x) \psi(x) = \sum_{e \subseteq \cu_{2n}}  f (e) \nabla \psi(e).
\end{equation*}
We then use the definition of the function $\psi$ to get
\begin{equation*}
\sum_{e \subseteq \cu_{2n}} f (e) \nabla \psi(e) = \sum_{z \in \mathcal{Z}_{n,2n}} \sum_{e \subseteq z + \cu_{n}} \left| \nabla \psi(e) \right|^2.
\end{equation*}
Since the function $\psi$ belongs to the space $\underset{z \in \mathcal{Z}_{n,2n}}{\oplus} \mathring h^1(z + \cu_n)$, it has mean $0$ on each of the subcubes~$(z + \cu_n)$. We can thus apply the Poincar\'e inequality on each of the cubes $(z + \cu_n)$ to get, for some constant $C := C(d) < \infty$,
\begin{equation*}
 \sum_{x \in z + \cu_{n}} |\psi(x)|^2 \leq  C 3^{2n} \sum_{e \subseteq z + \cu_n} \left| \nabla \psi(e) \right|^2, \qquad \forall z \in  \mathcal{Z}_{n,2n}.
\end{equation*}
Summing the previous inequality over $z \in  \mathcal{Z}_{n,2n}$ gives
\begin{equation*}
 \sum_{x \in \cu_{2n}} |\psi(x)|^2 \leq  C 3^{2n}  \sum_{z \in  \mathcal{Z}_{n,2n}} \sum_{e \subseteq z + \cu_{n}} \left| \nabla \psi(e) \right|^2.
\end{equation*}
Combining the few previous displays gives
\begin{equation*}
 \sum_{x \in  \cu_{2n}} |\psi(x)|^2 \leq  C 3^{2n} \sum_{x \in \cu_{2n}} \Delta L(\psi)(x) \psi(x).
\end{equation*}
By the Cauchy-Schwarz inequality, one further obtains
\begin{equation*}
\sum_{x \in  \cu_{2n}} |\psi(x)|^2 \leq  C 3^{4n} \sum_{x \in \cu_{2n}} \left| \Delta L(\psi)(x) \right|^2.
\end{equation*}
But by definition of the Laplacian, one has, for each $x \in \cu_{2n}$,
\begin{equation*}
 \left| \Delta L(\psi)(x) \right|^2 =  \left| \sum_{y \sim x} \left( \psi(y) - \psi (x) \right) \right|^2 \leq C \sum_{y \sim x} |\psi(y)|^2 + C | \psi (x)|^2.  
\end{equation*}
From this we obtain
\begin{equation*}
\sum_{x \in \cu_{2n}} \left| \Delta L(\psi)(x) \right|^2 \leq C \sum_{x \in \cu_{2n}} \left|  L(\psi)(x) \right|^2
\end{equation*}
and consequently
\begin{equation*}
\sum_{x \in  \cu_{2n}} |\psi(x)|^2 \leq  C 3^{4n} \sum_{x \in \cu_{2n}} \left| L(\psi)(x) \right|^2.
\end{equation*}
This is~\eqref{boundL-1.eqq}. 
\medskip

We now complete the proof of the bound $|\ln |\det L| |$ stated in~\eqref{est0detL}. Combining~\eqref{est1detL} and~\eqref{est2detL} gives
\begin{align*}
\left| \ln \left| \det L \right|  \right| & \leq C 3^{(2d -1)n} \max \left(  \ln  ||| L |||  , \ln  \left| \left| \left| L^{-1} \right| \right| \right| \right) \\
							&\leq C 3^{(2d -1)n} \ln  \left( C 3^{2n} \right) \\ 
							& \leq C 3^{(2d -1)n} n.
\end{align*}
The proof of Step 2 is complete.

\medskip

\textit{Step 3.} The goal of this step is to show the following estimate
\begin{equation} \label{mainresultstep3sptech0}
\E \left[  \frac1{\left|\cu_{2n}^+\right|} \sum_{e \subseteq \cu_{2n}^+} \left| \mathbf{f} (e) - \nabla \kappa_{2n}^+ (e) \right|^2  \right]  \leq C \left(( 1 + |q|^2) 3^{-\beta n} + \sum_{m=0}^n 3^{\frac{(m-n)}{2}}\tau_m^*(q)\right),
\end{equation}
where we recall that $\mathbf{f}$ is the random vector field defined in~\eqref{def.mathbff}.
To prove this estimate, we proceed as follows:
\begin{itemize}
\item we first remove the additional random variable $L\left(  \sum_{i=1}^{3^{nd}} X_i h_i \right)$. Precisely we prove
\begin{equation*}
\E \left[ \frac1{\left|\cu_{2n}^+\right|} \sum_{x \subseteq \cu_{2n}^+} \left| \kappa_{2n}^+ (x) -  \kappa (x) \right|^2  \right] \leq 3^{-dn},
\end{equation*}
where the random variable $\kappa$ is defined in~\eqref{defofkappa};
\item then we construct a random function $\Psi$ taking values in $h^1_0 \left(\cu_{2n}^+ \right)$ such that
\begin{equation} \label{projcloseid.e.0}
\E \left[  \frac1{\left|\cu_{2n}^+\right|} \sum_{e \subseteq \cu_{2n}^+} \left| \mathbf{f} (e) - \nabla \Psi (e) \right|^2  \right] \leq C \left(( 1 + |q|^2) 3^{-\beta n} + \sum_{m=0}^n 3^{\frac{(m-n)}{2}}\tau_m^*(q)\right),
\end{equation}
for some small exponent $\beta := \beta(d) > 0$;
\item we deduce from~\eqref{projcloseid.e.0} that 
\begin{equation*}
\E \left[  \frac1{\left|\cu_{2n}^+\right|} \sum_{e \subseteq \cu_{2n}^+} \left| \mathbf{f} (e) - \nabla \kappa (e) \right|^2  \right] \leq C \left(( 1 + |q|^2) 3^{-\beta n} + \sum_{m=0}^n 3^{\frac{(m-n)}{2}}\tau_m^*(q)\right).
\end{equation*}
\end{itemize}

\textit{Step 1.} We first prove that, up to a small error, we can remove the additional random variable $L\left(  \sum_{i=1}^{3^{nd}} X_i h_i \right)$ which was added to $\kappa$ to obtain $\kappa_{2n}^+$. These random variables were added so that the law of $\kappa_{2n}^+$ is absolutely continuous with respect to the Lebesgue measure on $h^1_0 \left( \cu_{2n}^+ \right)$, in order not to obtain an infinite entropy. They were also chosen in a way that their role in the energy is negligible. More precisely, we prove the following statement
\begin{equation*}
\E \left[ \frac1{\left|\cu_{2n}^+\right|} \sum_{x \subseteq \cu_{2n}^+} \left| \kappa_{2n}^+ (x) -  \kappa (x) \right|^2  \right] \leq 3^{-dn}.
\end{equation*}
We first recall the definition of the operator $L$ on the linear subspace $H$ given in~\eqref{4.155}. The previous estimate is then a consequence of the following computation
\begin{align*}
\E \left[ \frac1{\left|\cu_{2n}^+\right|} \sum_{x \in \cu_{2n}^+} \left| \kappa_{2n}^+ (x) -  \kappa (x) \right|^2  \right] & = \E \left[ \frac1{\left|\cu_{2n}^+\right|} \sum_{x \in \cu_{2n}^+} \left|  \sum_{i=1}^{3^{nd}} X_i \tilde h_i (x) \right|^2  \right] \\
																										& =  \E \left[ \frac1{\left|\cu_{2n}^+\right|}\sum_{i=1}^{3^{nd}}  |X_i |^2  \right].
\end{align*}
Since the family $\tilde h_i$, for $i \in \{ 1, \ldots, 3^{nd}\}$ is orthonormal with respect to the $L^2$ scalar product in $h^1_0 \left( \cu_{2n}^+ \right)$. Since the random variables $\left( X_i \right)_{i \in \{ 1 , \ldots, 3^{nd} \}}$ are i.i.d of law uniform in $[0,1]$, one can complete the computation
\begin{equation} \label{compkappakappa+}
\E \left[ \frac1{\left|\cu_{2n}^+\right|} \sum_{x \in \cu_{2n}^+} \left| \kappa_{2n}^+ (x) -  \kappa (x) \right|^2  \right]  = \frac{3^{nd}}{3 \left|\cu_{2n}^+\right|} \leq C 3^{-dn}.
\end{equation}
Using this and the inequality, for each $e = (x,y) \subseteq \cu_{2n}$,
\begin{align*} 
\left| \nabla \left( \kappa_{2n}^+ -  \kappa  \right) (e) \right|^2 & = \left|  \left( \kappa_{2n}^+ -  \kappa  \right) (x) - \left( \kappa_{2n}^+ -  \kappa  \right) (y)\right|^2 \\ & \leq 2 \left| \left( \kappa_{2n}^+ -  \kappa  \right) (x)\right|^2 + 2 \left|  \left( \kappa_{2n}^+ -  \kappa  \right) (y)\right|^2,
\end{align*}
one derives
\begin{equation} \label{compgradkappakappa+}
\E \left[ \frac1{\left|\cu_{2n}^+\right|} \sum_{e \subseteq \cu_{2n}^+} \left| \nabla \kappa_{2n}^+ (e) - \nabla \kappa (e) \right|^2  \right] \leq C 3^{-dn}.
\end{equation}
The proof of Step 1 is complete.

\medskip

\textit{Step 2.} We now prove~\eqref{projcloseid.e.0} and construct the random variable $\Psi$. We recall the definition of the family~$\left(\psi_z \right)$, for $z \in \mathcal{Z}_{n,2n}$, and extend it to each point $z \in 3^n \Zd$ according to
\begin{enumerate}
\item[(i)] for each point $z \in 3^n \Zd$,  $\psi_z$ is a random function from $\Zd$ to $\R$ equal to $0$ outside $(z + \cu_n)$ and the law of $\psi_z \left( \cdot - z \right)$ restricted to $\cu_n$ is $\P_{n,q}^*$;
\item[(ii)] the random variables $\psi_z$ are independent.
\end{enumerate}
It is the same family as in Step 1, except that it was extended to each $z \in 3^n \Zd$ and not only for $z \in \mathcal{Z}_{n,2n}$. The reason behind this extension will become clear later in the proof.

\smallskip
We also define $$ \psi := \sum_{z \in 3^n \Zd} \psi_z.$$
\smallskip

Moreover for each $z \in \mathcal{Z}_{n,2n}$, we let $\psi_{z,n+1}$be a random variable such that
\begin{equation*}
\psi_{z,n+1} ~ \mbox{is valued in}~\mathring h^1 \left( z + \cu_{n+1} \right)~\mbox{and the law of}~\psi_{z,n+1} \left(\cdot - z \right)~\mbox{is}~\P_{n+1,q}^*.
\end{equation*}
We extend this function by $0$ outside the cube $ (z + \cu_{n+1})$ so that one can think of the random variable $\psi_{z,n+1}$ as a random function from $\Zd$ to $\R$.

The goal of the following argument is to construct a suitable coupling between the random variables $\psi_{z,n+1}$, for $z \in \mathcal{Z}_{n,2n}$.
\smallskip

For some fixed $z \in \mathcal{Z}_{n,2n}$, we apply Proposition~\ref{2sccomparison} and Proposition~\ref{couplinglemma}, with the random variables $X = \psi$, $Y = \sum_{z' \in 3^n \Zd \setminus  \left( z + \cu_{n+1} \right) } \psi_{z'}$ and $Z = \psi_{z,n+1}$; we obtain that there exists a coupling between the random variables $\psi$ and $\psi_{z,n+1}$ such that
\begin{equation} \label{2sccouploin.0.ineq}
\E \left[ \frac1{|\cu_n|} \sum_{z' \in 3^n \Zd \cap \left( z + \cu_{n+1} \right)} \sum_{e \subseteq z' + \cu_{n}} \left| \nabla \psi_{z,n+1} (e) - \nabla \psi_{z'} (e)\right|^2  \right] \leq C\tau_n^*(q) + C(1 + |q|^2) 3^{-n}.
\end{equation}
This is where we used that $\psi_{z'}$ is defined for some $z'$ outside the cube $\cu_{2n}$. Indeed for some $z \in 3^n \Zd \cap \cu_{2n} $, close to the boundary of the cube $\cu_{2n}$, the set $3^n \Zd \cap \left( z + \cu_{n+1} \right)$ is not included in the cube $\cu_{2n}$.

\medskip

Thanks to the previous argument, we have constructed, for each $z \in \mathcal{Z}_{n,2n}$, a coupling between $\psi$ and $\psi_{z,n+1}$. Let $(z_1 , \ldots, z_{3^{nd}})$ be an enumeration of the elements of $\mathcal{Z}_{n,2n}$.

\smallskip

Applying Proposition~\ref{couplinglemma} with $X = \psi$, $Y = \psi_{z_1,n+1}$ and $Z = \psi_{z_2,n+1}$ constructs a coupling between $\psi$, $ \psi_{z_1,n+1}$ and $ \psi_{z_2,n+1}$ such that~\eqref{2sccouploin.0.ineq} is satisfied.

\smallskip

We then apply Proposition~\ref{couplinglemma} a second time, with this time the random variables $X =  \left( \psi_{z_1,n+1} ,\psi_{z_2,n+1} \right) $, $Y = \psi$ and $Z  =  \psi_{z_3,n+1}$ to construct a coupling between $\psi, \psi_{z_1,n+1} ,\psi_{z_2,n+1}$ and $\psi_{z_3,n+1}$ such that~\eqref{2sccouploin.0.ineq} is satisfied.

\smallskip

Iterating this construction $3^{nd}$ times constructs a coupling between the random variables $\psi$ and $\psi_{z,n+1}$, for $z \in \mathcal{Z}_{n,2n}$, such that~\eqref{2sccouploin.0.ineq} is satisfied.

\medskip

From the previous construction, we derive a coupling between the random variables $\psi_{z,n+1}$, for $z \in \mathcal{Z}_{n,2n}$, which satisfies the estimate~\eqref{2sccouploin.0.ineq}.

\medskip

We now build the function $\Psi$ by patching together the random variables $\psi_{z,n+1}$. The argument relies on a partition of unity: we let $\chi_0 \in h^1_0 (\cu_{n})$ be a cutoff function satisfying
\begin{equation*}
0 \leq \chi_0 \leq C 3^{-dn}, \qquad \sum_{x \in \Zd} \chi_0(x) = 1, \qquad |\nabla \chi_0| \leq C 3^{-(d+1)n}, \qquad \supp \chi_0 \subseteq \frac 12 \cu_n. 
\end{equation*}
We then define, for each $z \in \Zd$
\begin{equation*}
\chi(y) := \sum_{x \in \cu_ n} \chi_0 (y - x).
\end{equation*}
Note that the function $\chi$ is supported in the cube $\frac 34 \cu_{n+1}$, satisfies $0 \leq \chi \leq 1$ and the translations of $\chi$ form a partition of unity:
\begin{equation*}
\sum_{z \in 3^n \Zd} \chi(\cdot - z ) = 1.
\end{equation*} 
Moreover, one has the bound on the gradient of $\chi$
\begin{equation} \label{boundgradientpartofunitychi}
| \nabla \chi | \leq C 3^{-n}.
\end{equation}
We next consider the cutoff function $\zeta \in h^1_0 \left( \cu_{2n}^+ \right)$ satisfying
\begin{equation} \label{defcutoozeta}
 0 \leq \zeta \leq 1, \qquad \zeta = 1 ~\mbox{on}~ \left\{ x \in \cu_{2n} ~:~ \dist(x , \partial \cu_{2n}) \geq 3^n \right\}, \qquad \left| \nabla \zeta \right| \leq C 3^{-n},
\end{equation}
which is used to remove a boundary layer in the patching construction. We also define the following discrete set
\begin{equation*}
\mathcal{Z}_{n,2n} := \left\{ z \in 3^n \Zd ~:~ z \in \mathcal{Z}_{n,2n} ~\mbox{or}~ \dist (z , \partial \cu_{2n} ) \leq 3^n \right\}.
\end{equation*}
It represents the set $\mathcal{Z}_{n,2n}$ with an additional boundary layer of size $1$ of points in $3^n \Zd$ around it. This set is useful in the proof because it satisfies the following property
\begin{equation*}
\forall y \in \cu_{2n}^+,~ \sum_{z \in \mathcal{Z}_{n,2n}^+} \chi(y - z ) = 1.
\end{equation*} 
We then define the function $\Psi$ by the formula
\begin{equation*}
\Psi(x) = \zeta(x) \sum_{z \in \mathcal{Z}_{n,2n}^+} \chi(x - z) \left( \psi_{z , n+1} (x) - \nabla \nu^*_{n+1}(q) \cdot (x - z) \right).
\end{equation*}
Now that $\Psi$ has been constructed, we prove ~\eqref{projcloseid.e.0}. The main ingredients to prove this estimate are the inequality~\eqref{2sccouploin.0.ineq}, Proposition~\ref{p.ustarminimizerareflat} and the interior Meyers estimate, Proposition~\ref{meyersgradphi} stated in Appendix~\ref{appendixB}.

We first compute the derivative of $\Psi$. An explicit computation gives, for each edge $e = (x, y) \subseteq \cu_{2n}$,
\begin{align} \label{gradientPdi}
\lefteqn{ \nabla \Psi(e) = \zeta( y) \sum_{z \in \mathcal{Z}_{n,2n}^+}  \chi(y - z) \left( \nabla \psi_{z , n+1} (e) - \nabla \nu^*_{n+1}(q)(e) \right)} \qquad & \\ & \qquad \qquad  +  \zeta( y) \sum_{z \in \mathcal{Z}_{n,2n}^+}  \nabla \chi(\cdot - z)(e)  \left(  \psi_{z , n+1} (x) - \nabla \nu^*_{n+1}(q) \cdot x \right)  \notag \\ & \qquad \qquad +  \nabla \zeta( e) \sum_{z \in \mathcal{Z}_{n,2n}^+}  \chi(x - z) \left( \psi_{z , n+1} (x) - \nabla \nu^*_{n+1}(q) \cdot x \right). \notag
\end{align}
The second and third terms in the previous display are error terms, which will be proved to be small: the interesting term is the first one. The $L^2$-norm of second term can be estimated thanks to the bound~\eqref{boundgradientpartofunitychi} on the gradient of $\chi$,
\begin{align*}
\lefteqn{ \E \left[  \frac1{\left| \cu_{2n}^+ \right|}\sum_{e = (x,y) \subseteq \cu_{2n}^+} \left|  \zeta( y) \sum_{z \in \mathcal{Z}_{n,2n}^+}  \nabla \chi(\cdot - z)(e)  \left( \psi_{z , n+1} (x) -\nabla \nu^*_{n+1}(q) \cdot x \right) \right|^2 \right]}  \qquad\qquad & \\
					& \leq C 3^{-2n} \E \left[\frac1{\left| \cu_{2n}^+ \right|} \sum_{z \in \mathcal{Z}_{n,2n}^+}  \sum_{x \in z + \cu_{n+1}}\left|  \psi_{z , n+1} (x) - \nabla \nu^*_{n+1}(q)  \cdot x \right|^2 \right] \\
					& \leq C 3^{-2n} \E \left[\frac1{\left| \cu_{n+1} \right|}  \sum_{x \in  \cu_{n+1}}\left|  \psi_{n+1,q} (x) - \nabla \nu^*_{n+1}(q)  \cdot x \right|^2 \right].
\end{align*}
We then apply Proposition~\ref{p.ustarminimizerareflat} to obtain
\begin{multline*}
\E \left[  \frac1{\left| \cu_{2n}^+ \right|}\sum_{e = (x,y) \subseteq \cu_{2n}^+} \left|  \zeta( y) \sum_{z \in \mathcal{Z}_{n,2n}^+}  \nabla \chi(\cdot - z)(e)  \left(  \psi_{z , n+1} (x) -\nabla \nu^*_{n+1}(q) \cdot x \right) \right|^2 \right] \\ \leq C \left((1 + |q|^2) 3^{-\frac n2} + \sum_{m=0}^n 3^{\frac{(m-n)}{2}}\tau_m^*(q)\right).
\end{multline*}
The third term of~\eqref{gradientPdi} can be estimated in a similar manner, using~\eqref{defcutoozeta} this time,
\begin{align*}
\lefteqn{ \E \left[  \frac1{\left| \cu_{2n}^+ \right|} \sum_{e = (x,y) \subseteq \cu_{2n}^+} \left|  \nabla \zeta( e) \sum_{z \in \mathcal{Z}_{n,2n}^+}  \chi(x - z) \left( \psi_{z , n+1} (x) -\nabla \nu^*_{n+1}(q) \cdot x \right)  \right|^2 \right]}  \qquad\qquad & \\
					& \leq C 3^{-2n} \E \left[\frac1{\left| \cu_{2n}^+ \right|} \sum_{z \in \mathcal{Z}_{n,2n}^+}  \sum_{x \in z + \cu_{n+1}}\left|  \psi_{z , n+1} (x) - \nabla \nu^*_{n+1}(q)  \cdot x \right|^2 \right] \\
					& \leq C 3^{-2n} \E \left[\frac1{\left| \cu_{n+1} \right|}  \sum_{x \in  \cu_{n+1}}\left|\psi_{n+1,q} (x) - \nabla \nu^*_{n+1}(q)  \cdot x \right|^2 \right].
\end{align*}
We then apply Proposition~\ref{p.ustarminimizerareflat} again to obtain
\begin{multline*}
 \E \left[  \frac1{\left| \cu_{2n}^+ \right|} \sum_{e = (x,y) \subseteq \cu_{2n}^+} \left|  \nabla \zeta( e) \sum_{z \in \mathcal{Z}_{n,2n}^+}  \chi(x - z) \left(  \psi_{z , n+1} (x) - \nabla \nu^*_{n+1}(q)  \cdot x \right)  \right|^2 \right] \\ \leq C \left((1 + |q|^2) 3^{-\frac n2} + \sum_{m=0}^n 3^{\frac{(m-n)}{2}}\tau_m^*(q)\right).
\end{multline*}
Combining the few previous displays then yields
\begin{multline*}
\E \left[  \frac1{\left| \cu_{2n}^+ \right|} \sum_{e = (x,y) \subseteq \cu_{2n}^+}  \left| \nabla \Psi(e) - \zeta( y) \sum_{z \in \mathcal{Z}_{n,2n}^+}  \chi(y - z) \left( \nabla \psi_{z , n+1} (e) -\nabla \nu^*_{n+1}(q) (e) \right) \right|^2 \right] \\ \leq C \left((1 + |q|^2) 3^{-\frac n2} + \sum_{m=0}^n 3^{\frac{(m-n)}{2}}\tau_m^*(q)\right).
\end{multline*}
Thus to prove~\eqref{projcloseid.e.0}, it is sufficient to show
\begin{multline} \label{proving3.23.1}
\E \left[  \frac1{\left| \cu_{2n}^+ \right|} \sum_{e = (x,y) \subseteq \cu_{2n}^+}  \left| \mathbf{f}(e) - \zeta( y) \sum_{z \in \mathcal{Z}_{n,2n}^+}  \chi(y - z) \left( \nabla \psi_{z , n+1} (e) -\nabla \nu^*_{n+1}(q)(e) \right) \right|^2 \right] \\ \leq C\tau_n^*(q) + C(1 + |q|^2) 3^{-\beta n},
\end{multline}
for some small exponent $\beta := \beta (d, \lambda) > 0$.
We simplify the previous display by removing the function~$\zeta$. Note that if we denote by
\begin{equation*}
\partial  \mathcal{Z}_{n,2n}:= \left\{ z \in \mathcal{Z}_{n,2n}^+ ~:~ \dist(z , \partial \cu_{2n}) \leq 2\cdot 3^n \right\},
\end{equation*}
then we have the following computation, using the properties of the functions $\zeta$ and $\chi$,
\begin{align*}
\lefteqn{ \E \left[  \frac1{\left| \cu_{2n}^+ \right|} \sum_{e = (x,y) \subseteq \cu_{2n}^+}  \left| \left(1- \zeta( y)\right) \sum_{z \in \mathcal{Z}_{n,2n}^+}  \chi(y - z) \left( \nabla \psi_{z , n+1} (e) -\nabla \nu^*_{n+1}(q)(e) \right) \right|^2 \right] } \qquad & \\ &
				\leq  \E \left[  \frac1{\left| \cu_{2n}^+ \right|} \sum_{e = (x,y) \subseteq \cu_{2n}^+}  \left| \sum_{z \in \partial  \mathcal{Z}_{n,2n}}  \chi(y - z) \left( \nabla \psi_{z , n+1} (e) - \nabla \nu^*_{n+1}(q)(e) \right) \right|^2 \right] \\
				& \leq  \E \left[  \frac1{\left| \cu_{2n}^+ \right|} \sum_{e = (x,y) \subseteq \cu_{2n}^+}  \sum_{z \in \partial  \mathcal{Z}_{n,2n}}  \chi(y - z) \left| \nabla \psi_{z , n+1} (e) - \nabla \nu^*_{n+1}(q)(e) \right|^2 \right].
\end{align*}
The previous display can be simplified by using that the function $\zeta$ is equal to $1$ on the set $\left\{ x \in \cu_{2n} ~:~ \dist(x , \partial \cu_{2n}) \geq 3^n \right\}$. We obtain
\begin{multline*}
 \E \left[  \frac1{\left| \cu_{2n}^+ \right|} \sum_{e = (x,y) \subseteq \cu_{2n}^+}  \left| \left(1- \zeta( y)\right) \sum_{z \in \mathcal{Z}_{n,2n}^+}  \chi(y - z) \left( \nabla \psi_{z , n+1} (e) - \nabla \nu^*_{n+1}(q)(e) \right) \right|^2 \right] \\
			 \leq  \E \left[  \frac1{\left| \cu_{2n}^+ \right|}  \sum_{z \in \partial \mathcal{Z}_{n,2n}}  \sum_{e \subseteq z + \cu_{n+1}} \left| \nabla \psi_{z , n+1} (e) - \nabla \nu^*_{n+1}(q)(e) \right|^2 \right].
\end{multline*}
Using that all the random variables $ \psi_{z , n+1}$ have the same law, which is $\P_{n+1,q}^*$, one has
\begin{multline*}
 \E \left[ \frac1{\left| \cu_{2n}^+ \right|} \sum_{e = (x,y) \subseteq \cu_{2n}^+}  \left| \left(1- \zeta( y)\right) \sum_{z \in \mathcal{Z}_{n,2n}^+}  \chi(y - z) \left( \nabla \psi_{z , n+1} (e) - \nabla \nu^*_{n+1}(q)(e) \right) \right|^2 \right] \\
			 \leq \frac{\left|  \partial  \mathcal{Z}_{n,2n} \right|}{\left| \cu_{2n}^+ \right|}  \E \left[ \sum_{e \subseteq \cu_{n+1}} \left| \nabla \psi_{n+1 , q} (e) -\nabla \nu^*_{n+1}(q)(e) \right|^2 \right].
\end{multline*}
But by the estimate~\eqref{4.L2boundnustar} of Proposition~\ref{propofnunustar}, the term on the right-hand side is bounded and one has
\begin{multline*}
 \E \left[ \frac1{\left| \cu_{2n}^+ \right|} \sum_{e = (x,y) \subseteq \cu_{2n}^+}  \left| \left(1- \zeta( y)\right) \sum_{z \in \mathcal{Z}_{n,2n}^+}  \chi(y - z) \left( \nabla \psi_{z , n+1} (e) -  \nabla \nu^*_{n+1}(q)(e) \right) \right|^2 \right] \\
			 \leq \frac{\left|  \partial  \mathcal{Z}_{n,2n} \right|\cdot |\cu_{n+1}|}{\left| \cu_{2n}^+ \right|}  C (1 + |q|^2).
\end{multline*}
One then appeals to the estimate
$
\left|  \partial \mathcal{Z}_{n,2n} \right| \leq C 3^{(d-1)n} $ and the equality $\left| \cu_{n+1} \right| = 3^{d(n+1)}
$
to obtain
\begin{multline*}
 \E \left[ \frac1{\left| \cu_{2n}^+ \right|} \sum_{e = (x,y) \subseteq \cu_{2n}^+}  \left| \left(1- \zeta( y)\right) \sum_{z \in \mathcal{Z}_{n,2n}^+}  \chi(y - z) \left( \nabla \psi_{z , n+1} (e) - \nabla \nu^*_{n+1}(q)(e)  \right) \right|^2 \right] 
			\\  \leq C 3^{-n} (1 + |q|^2).
\end{multline*}
By the previous display and~\eqref{proving3.23.1}, it is enough to prove~\eqref{projcloseid.e.0} to show
\begin{multline} \label{toshow3.245}
\E \left[  \frac1{\left| \cu_{2n}^+ \right|} \sum_{e = (x,y) \subseteq \cu_{2n}^+}  \left| \mathbf{f}(e) - \sum_{z \in \mathcal{Z}_{n,2n}^+}  \chi(y - z) \left( \nabla \psi_{z , n+1} (e) -  \nabla \nu^*_{n+1}(q)(e)  \right) \right|^2 \right] \\ \leq C\tau_n^*(q) + C(1 + |q|^2) 3^{-\beta n},
\end{multline}
for some small exponent $\beta := \beta(d, \lambda) > 0$.
We now prove this estimate. Using that the map $\chi$ is a partition of unity, we rewrite
\begin{multline*}
\E \left[  \frac1{\left| \cu_{2n}^+ \right|} \sum_{e = (x,y) \subseteq \cu_{2n}^+}  \left| \mathbf{f}(e) - \sum_{z \in \mathcal{Z}_{n,2n}^+}  \chi(y - z) \left( \nabla \psi_{z , n+1} (e) -  \nabla \nu^*_{n+1}(q)(e) \right) \right|^2 \right]  \\ = \E \left[  \frac1{\left| \cu_{2n}^+ \right|} \sum_{e = (x,y) \subseteq \cu_{2n}^+}  \left|  \sum_{z \in \mathcal{Z}_{n,2n}^+}  \chi(y - z) \left(  \mathbf{f}(e) - \nabla \psi_{z , n+1} (e) +  \nabla \nu^*_{n+1}(q)(e)\right) \right|^2 \right] .
\end{multline*}
Using that the function $\chi$ is supported in the cube $\frac 34 \cu_{n+1}$, one obtains
\begin{multline*}
\E \left[  \frac1{\left| \cu_{2n}^+ \right|} \sum_{e = (x,y) \subseteq \cu_{2n}^+}  \left|  \sum_{z \in \mathcal{Z}_{n,2n}^+}  \chi(y - z) \left(  \mathbf{f}(e) - \nabla \psi_{z , n+1} (e) +  \nabla \nu^*_{n+1}(q)(e) \right) \right|^2 \right]  \\ \leq \E \left[  \frac1{\left| \cu_{2n}^+ \right|}  \sum_{z \in \mathcal{Z}_{n,2n}^+} \sum_{e \subseteq \left( z + \frac 34 \cu_{n+1} \right) \cap \cu_{2n}}  \left|   \mathbf{f}(e) - \nabla \psi_{z , n+1} (e) +  \nabla \nu^*_{n+1}(q)(e) \right|^2 \right] .
\end{multline*}
Using the definition of the vector field $\mathbf{f}$ given in~\eqref{def.mathbfff}, and splitting the sum according to the partition of bonds,
\begin{equation*}
e \subseteq \cu_{2n} \implies e \in B_{2n,n} \quad \mbox{or} \quad \exists z \in \mathcal{Z}_{n,2n}, ~ e \subseteq z + \cu_n,
\end{equation*}
one derives
\begin{align} \label{comaroingftopatchup}
\lefteqn{ \E \left[  \frac1{\left| \cu_{2n}^+ \right|} \sum_{e = (x,y) \subseteq \cu_{2n}^+}  \left| \mathbf{f}(e) - \sum_{z \in \mathcal{Z}_{n,2n}}  \chi(y - z) \left( \nabla \psi_{z , n+1} (e) - \nabla \nu^*_{n+1} (q) ( e )\right) \right|^2 \right] \notag  } \qquad & \\ &
				\leq \E \left[  \frac1{\left| \cu_{2n}^+ \right|}  \sum_{z, z' \in \mathcal{Z}_{n,2n}, \, z \in z' + \cu_{n+1}} \sum_{e \subseteq z' +\cu_{n}}  \left|  \nabla \psi_{z'} (e) - \nabla \psi_{z,n+1} (e)  \right|^2 \right]   \\
						& \quad + \E \left[  \frac1{\left| \cu_{2n}^+ \right|} \sum_{z \in  \mathcal{Z}_{n,2n}} \sum_{e \in B_{2n,n} \cap (z + \frac 34 \cu_{n+1} )}\left| \nabla \psi_{z , n+1} (e) \right|^2 \right]  \notag \\
						& \qquad + C \left|\nabla \nu^*_{n+1} (q) - \nabla \nu^*_{n} (q) \right|^2 \notag.
\end{align}
The first term on the right-hand side is estimated thanks to~\eqref{2sccouploin.0.ineq}. This gives
\begin{align} \label{comaroingftopatchup2}
\lefteqn{ \E \left[  \frac1{\left| \cu_{2n}^+ \right|}  \sum_{z, z' \in \mathcal{Z}_{n,2n}, \, z \in z' + \cu_{n+1}} \sum_{e \subseteq z' +\cu_{n}}  \left|  \nabla \psi_{z'} (e) - \nabla \psi_{z,n+1} (e)  \right|^2 \right]} \qquad & \\ &
						= \frac1{\left| \cu_{2n}^+ \right|}  \sum_{z \in \mathcal{Z}_{n,2n}} \E \left[\sum_{z' \in 3^n \Zd \cap \left( z + \cu_{n+1} \right)} \sum_{e \subseteq z' + \cu_{n}} \left| \nabla \psi_{z,n+1} (e) - \nabla \psi_{z'} (e)\right|^2  \right] \notag \\
						& \leq C\tau_n^*(q) + C(1 + |q|^2) 3^{-n}. \notag
\end{align}
The third term can be estimated thanks to~\eqref{splitstep2},
\begin{equation*}
\left|\nabla \nu^*_{n+1} (q) - \nabla \nu^*_{n} (q) \right|^2 \leq C \tau_n^*(q) + C (1 + |q|^2) 3^{-n}.
\end{equation*}
To estimate the second term on the right-hand side of~\eqref{comaroingftopatchup}, we first use that all the random variables $\psi_{z,n+1}$ have the same law, which is $\P_{n+1,q}^*$. This gives
\begin{align*}
\lefteqn{\E \left[  \frac1{\left| \cu_{2n}^+ \right|} \sum_{z \in  \mathcal{Z}_{n,2n}} \sum_{e \in B_{2n,n} \cap (z + \frac 34 \cu_{n+1} )}\left| \nabla \psi_{z , n+1} (e) \right|^2 \right] } \qquad & \\ &
				= \frac1{\left| \cu_{2n}^+ \right|} \sum_{z \in  \mathcal{Z}_{n,2n}}  \E \left[ \sum_{e \in B_{2n,n} \cap \frac 34 \cu_{n+1} }\left| \nabla \psi_{n+1,q} (e) \right|^2 \right] \\ &
				\leq C \E \left[ \frac 1{\left| \cu_{n+1} \right|}  \sum_{e \in B_{2n,n} \cap \frac 34 \cu_{n+1} }\left| \nabla \psi_{n+1 , q} (e) \right|^2 \right].
\end{align*}
We then estimate this term by the Meyers estimate, Proposition~\ref{meyersgradphi} with $\gamma = \frac 34$. We denote by $\delta$ the exponent of Proposition~\ref{meyersgradphi} and compute
\begin{align*}
\lefteqn{ \E \left[ \frac 1{\left| \cu_{n+1} \right|}  \sum_{e \in B_{2n,n} \cap \frac 34 \cu_{n+1} }\left| \nabla \psi_{n+1 ,q} (e) \right|^2 \right] }\qquad & \\ 
					& \leq  \frac C{\left| \frac 34 \cu_{n+1} \right|}   \sum_{e \in B_{2n,n} \cap \frac 34 \cu_{n+1} }  \E \left[  \left| \nabla \psi_{n+1 ,q} (e) \right|^2 \right] \\
					& \leq \frac{C \left|B_{2n,n} \cap \frac 34 \cu_{n+1}\right|^{\frac{\delta}{1 + \delta}} }{\left| \frac 34 \cu_{n+1} \right|}  \left(\sum_{e \subseteq \frac 34 \cu_{n+1} }  \E \left[  \left| \nabla \psi_{n+1 ,q} (e) \right|^2 \right]^{1 + \delta} \right)^{\frac1{1 +\delta}} \\
					& \leq C \left( \frac{\left|B_{2n,n} \cap \frac 34 \cu_{n+1}\right|}{\left| \frac 34 \cu_{n+1} \right|}\right)^{\frac{\delta}{1 + \delta}} ( 1 + |q|^2 ).
\end{align*}
We then use that
\begin{equation*}
\frac{\left|B_{2n,n} \cap \frac 34 \cu_{n+1}\right|}{\left| \frac 34 \cu_{n+1} \right|} \leq C3^{-n},
\end{equation*}
to derive
\begin{equation*}
 \E \left[ \frac 1{\left| \cu_{n+1} \right|}  \sum_{e \in B_{2n,n} \cap \frac 34 \cu_{n+1} }\left| \nabla \psi_{n+1,q} (e) \right|^2 \right] \leq C 3^{-\frac{\delta }{1 + \delta}n} ( 1 + |q|^2 ).
\end{equation*}
Combining the few previous displays gives the following estimate for the second term on the right-hand side of~\eqref{comaroingftopatchup}
\begin{equation*}
\E \left[  \frac1{\left| \cu_{2n}^+ \right|} \sum_{z \in \mathcal{Z}_{n,2n}} \sum_{e \in B_{2n,n} \cap (z + \frac 34 \cu_{n+1} )}\left| \nabla \psi_{z , n+1} (e) \right|^2 \right]  \leq C 3^{-\frac{\delta }{1 + \delta}n} ( 1 + |q|^2 ).
\end{equation*}
Combining~\eqref{comaroingftopatchup} with~\eqref{comaroingftopatchup2}, the previous displays and setting $\beta := \frac{\delta }{1 + \delta}$ yields
\begin{multline*}
 \E \left[  \frac1{\left| \cu_{2n}^+ \right|} \sum_{e = (x,y) \subseteq \cu_{2n}^+}  \left| \mathbf{f}(e) - \sum_{z \in 3^n \Zd}  \chi(y - z) \left( \nabla \psi_{z , n+1} (e) - \nabla \nu^*_{n+1}(q) \cdot e \right) \right|^2 \right] \\ \leq C\tau_n^*(q) + C(1 + |q|^2) 3^{-\beta n}.
\end{multline*}
This is precisely~\eqref{toshow3.245}. The proof of~\eqref{projcloseid.e.0} is complete.

\medskip

\textit{Step 3.} In this step, we deduce from~\eqref{projcloseid.e.0} the estimate 
\begin{equation*}
\E \left[  \frac1{\left|\cu_{2n}^+\right|} \sum_{e \subseteq \cu_{2n}^+} \left| \mathbf{f} (e) - \nabla \kappa (e) \right|^2  \right]  \leq C \left(( 1 + |q|^2) 3^{-\beta n} + \sum_{m=0}^n 3^{\frac{(m-n)}{2}}\tau_m^*(q)\right).
\end{equation*}
We recall that $\kappa$ is defined as the solution of the problem
\begin{equation*}
\left\{ \begin{aligned}
\Delta \kappa & = \div \mathbf{f}~ \mbox{in} ~ \cu_{2n}, \\
\kappa & \in h^1_0 \left( \cu_{2n}^+ \right). 
\end{aligned}  \right.
\end{equation*}
This implies the almost sure inequality
\begin{align} \label{kappainff}
\sum_{e \subseteq \cu_{2n}^+} \left|  \mathbf{f}(e) - \nabla \kappa (e) \right|^2 & = \inf_{\kappa' \in h^1_0 \left(\cu_{2n}^+ \right)} \sum_{e \subseteq \cu_{2n}^+} \left| \mathbf{f}(e) - \nabla \kappa' (e) \right|^2 \\
							& \leq \sum_{e \subseteq \cu_{2n}^+} \left|  \mathbf{f}(e) - \nabla \Psi (e) \right|^2. \notag
\end{align}
Taking the expectation and using the inequality~\eqref{projcloseid.e.0} gives
\begin{equation*}
\E \left[  \frac1{\left|\cu_{2n}^+\right|} \sum_{e \subseteq \cu_{2n}^+} \left| \mathbf{f} (e) - \nabla \kappa (e) \right|^2  \right]  \leq C \left(( 1 + |q|^2) 3^{-\beta n} + \sum_{m=0}^n 3^{\frac{(m-n)}{2}}\tau_m^*(q)\right).
\end{equation*}
Combining the previous display with~\eqref{compkappakappa+} proves the estimate
\begin{equation} \label{mainresultstep3sptech}
\E \left[  \frac1{\left|\cu_{2n}^+\right|} \sum_{e \subseteq \cu_{2n}^+} \left| \mathbf{f} (e) - \nabla \kappa_{2n}^+ (e) \right|^2  \right]  \leq C \left(( 1 + |q|^2) 3^{-\beta n} + \sum_{m=0}^n 3^{\frac{(m-n)}{2}}\tau_m^*(q)\right).
\end{equation}

\medskip 

\textit{Step 4.} The goal of this step is to use the main result~\eqref{mainresultstep3sptech} of Step 3 to prove
\begin{multline} \label{Step4quanthomog.maintech}
\E \left[ \frac{1}{\left| \cu_{2n}^+\right|} \sum_{e \subseteq \cu_{2n}^+}  V_e \left( \nabla \nu^*_n (q) (e) + \nabla \kappa^+_{2n} (e) \right) \right]  \leq \E \left[ \frac{1}{\left| \cu_{n}\right|} \sum_{e \subseteq \cu_{n}}  V_e \left( \nabla \psi_{n,q} (e) \right) \right] \\ + C \left(( 1 + |q|^2) 3^{-\beta n} + \sum_{m=0}^n 3^{\frac{(m-n)}{2}}\tau_m^*(q)\right).
\end{multline}

We recall the definition of the random variable $\sigma_z(x) := \psi_z(x) - \nabla \nu^*_n (q) \cdot (x - z)$ introduced in~\eqref{def.sigmaz}. The proof relies on the following technical estimate, the proof of which is postponed to Appendix~\ref{appendixA}, Proposition~\ref{p.lemmaA2}.
\begin{align} \label{pp.lemmaA2}
\lefteqn{\E \left[ \frac{1}{\left| \cu_{2n}\right|} \sum_{z \in \mathcal{Z}_{n,2n}} \sum_{e \subseteq z + \cu_n}  V_e \left(  \nabla \nu^*_n (q) (e) + \nabla \kappa (e) \right) \right]} \qquad & \\ &  \leq  \E \left[  \frac{1}{\left| \cu_{2n}\right|} \sum_{z \in \mathcal{Z}_{n,2n}} \sum_{e \subseteq z + \cu_n}  V_e \left( \nabla \psi_z (e) \right)  \right] + C (1 + |q|^2) 3^{-\frac n2}\notag  \\ & \qquad + C \E \left[ \frac{1}{\left| \cu_{2n}\right|} \sum_{z \in \mathcal{Z}_{n,2n}} \sum_{e \subseteq z + \cu_n} \left| \nabla \kappa (e) - \nabla \sigma_z (e) \right|^2 \right]. \notag
\end{align}
We now show how to deduce~\eqref{Step4quanthomog.maintech} from the previous inequality. First, since all the random variables $\psi_z$ have the same law, one can simplify the first term on the right-hand side,
\begin{align*}
\E \left[  \frac{1}{\left| \cu_{2n}\right|} \sum_{z \in \mathcal{Z}_{n,2n}} \sum_{e \subseteq z + \cu_n}  V_e \left( \nabla \psi_z (e) \right)  \right]
			& =  \frac{|\mathcal{Z}_{n,2n}|}{\left| \cu_{2n}\right|} \E \left[ \sum_{e \subseteq \cu_n}  V_e \left( \nabla \psi_{n,q} (e) \right)  \right] \\
			& =   \E \left[ \frac1{\left| \cu_n\right|} \sum_{e \subseteq \cu_n}  V_e \left( \nabla \psi_{n,q} (e) \right)  \right].
\end{align*}
We now estimate the last term on the right-hand side of~\eqref{pp.lemmaA2}. One has
\begin{align} \label{exacsamecompkappapsiz}
\lefteqn{\E \left[ \frac{1}{\left| \cu_{2n}\right|} \sum_{z \in \mathcal{Z}_{n,2n}} \sum_{e \subseteq z + \cu_n} \left| \nabla \sigma_z (e) - \nabla \kappa (e)  \right|^2 \right] } \qquad & \\ & \leq  \E \left[  \frac1{\left|\cu_{2n}\right|} \sum_{e \subseteq \cu_{2n}^+} \left| \mathbf{f} (e) - \nabla \kappa (e) \right|^2  \right] \notag \\ &
		  \leq \E \left[  \frac1{\left|\cu_{2n}^+\right|} \sum_{e \subseteq \cu_{2n}^+} \left| \mathbf{f} (e) - \nabla \kappa (e) \right|^2  \right] \notag \\
		 & \quad+ \left(  \frac{\left|\cu_{2n}\right| - \left|\cu_{2n}^+\right| }{\left|\cu_{2n}^+\right|\cdot\left|\cu_{2n}\right|} \right)  \E \left[  \sum_{e \subseteq \cu_{2n}^+} \left| \mathbf{f} (e) - \nabla \kappa (e) \right|^2  \right] . \notag
\end{align}
We first estimate the second term on the right-hand side of the previous display. Note that
\begin{equation*}
\frac{\left|\cu_{2n}\right| - \left|\cu_{2n}^+\right| }{\left|\cu_{2n}^+\right|\cdot\left|\cu_{2n}\right|} \leq 3^{-2n} \frac{1}{\left|\cu_{2n}\right|},
\end{equation*}
and using the estimate~\eqref{kappainff}, one has
\begin{align*}
 \E \left[   \frac1{\left|\cu_{2n}\right|} \sum_{e \subseteq \cu_{2n}^+} \left| \mathbf{f} (e) - \nabla \kappa (e) \right|^2  \right] &  \leq  \E \left[   \frac1{\left|\cu_{2n}\right|} \sum_{e \subseteq \cu_{2n}^+} \left| \mathbf{f} (e) \right|^2  \right]  \\
				& \leq \E \left[  \frac{1}{\left|\cu_{2n}\right|} \sum_{z \in \mathcal{Z}_{n,2n}}\sum_{e \subseteq z + \cu_n} \left| \nabla \sigma_z (e) \right|^2  \right] \\		
				& \leq C \E \left[  \frac{1}{\left|\cu_n\right|} \sum_{e \subseteq  \cu_n} \left| \nabla \psi_{n,q} (e) \right|^2  \right].
\end{align*}
We can then bound the last term on the right-hand side thanks to Proposition~\ref{propofnunustar}. This gives
\begin{equation*}
 \E \left[   \frac1{\left|\cu_{2n}\right|} \sum_{e \subseteq \cu_{2n}^+} \left| \mathbf{f} (e) - \nabla \kappa (e) \right|^2  \right]  \leq C \E \left[  \frac{1}{\left|\cu_n\right|} \sum_{e \subseteq  \cu_n} \left| \nabla \psi_{n,q} (e) \right|^2  \right]  \leq C(1+|q|^2).
\end{equation*}
Combining the few previous displays shows
\begin{multline*}
\E \left[ \frac{1}{\left| \cu_{2n}\right|} \sum_{z \in \mathcal{Z}_{n,2n}} \sum_{e \subseteq z + \cu_n} \left|  \nabla \sigma_z (e) - \nabla \kappa (e) \right|^2 \right] \leq \E \left[  \frac1{\left|\cu_{2n}^+\right|} \sum_{e \subseteq \cu_{2n}^+} \left| \mathbf{f} (e) - \nabla \kappa (e) \right|^2  \right]  \\ + C 3^{-2n} (1 + |q|^2).
\end{multline*}
We then use~\eqref{mainresultstep3sptech0} to deduce the estimate
\begin{equation*}
\E \left[ \frac{1}{\left| \cu_{2n}\right|} \sum_{z \in \mathcal{Z}_{n,2n}} \sum_{e \subseteq z + \cu_n} \left| \nabla \kappa (e) - \nabla \sigma_z (e) \right|^2 \right]  \leq C \left(( 1 + |q|^2) 3^{-\beta n} + \sum_{m=0}^n 3^{\frac{(m-n)}{2}}\tau_m^*(q)\right).
\end{equation*}
Combining this estimate with~\eqref{pp.lemmaA2} shows
\begin{align*}
\lefteqn{ \E \left[ \frac{1}{\left| \cu_{2n}\right|} \sum_{z \in \mathcal{Z}_{n,2n}} \sum_{e \subseteq z + \cu_n}  V_e \left( \nabla \nu^*_n (q) (e)  + \nabla \kappa (e) \right) \right] } \qquad & \\ & \leq  \E \left[  \frac{1}{\left| \cu_{2n}\right|} \sum_{z \in \mathcal{Z}_{n,2n}} \sum_{e \subseteq z + \cu_n}  V_e \left( \nabla \psi_z (e) \right)  \right] \\  &\qquad + C \left(( 1 + |q|^2) 3^{-\beta n} + \sum_{m=0}^n 3^{\frac{(m-n)}{2}}\tau_m^*(q)\right).
\end{align*}
To complete the proof of~\eqref{Step4quanthomog.maintech}, it is thus sufficient to prove
\begin{align} \label{of(0)proof}
  \lefteqn{\E \left[ \frac{1}{\left| \cu_{2n}^+\right|} \sum_{e \subseteq \cu_{2n}^+}  V_e \left( \nabla \nu^*_n (q) (e) +  \nabla \kappa^+_{2n} (e) \right) \right]  } \qquad & \\ & \leq \E \left[ \frac{1}{\left| \cu_{2n}\right|} \sum_{z \in \mathcal{Z}_{n,2n}} \sum_{e \subseteq z + \cu_n}  V_e \left( \nabla \nu^*_n (q) (e)  + \nabla \kappa (e) \right) \right] \notag
					\\ & \qquad +  C \left(( 1 + |q|^2) 3^{-\beta n} + \sum_{m=0}^n 3^{\frac{(m-n)}{2}}\tau_m^*(q)\right), \notag
\end{align}
for some constant $C := C(d,\lambda) < \infty$ and some exponent $\beta := \beta(d , \lambda) > 0$. To this end, we prove the two following inequalities:
\begin{enumerate}
\item we first prove the inequality
\begin{align} \label{of(1)proof}
\lefteqn{  \E \left[ \frac{1}{\left| \cu_{2n}\right|} \sum_{e \subseteq \cu_{2n}^+}  V_e \left( \nabla \nu^*_n (q) (e) +  \nabla \kappa (e) \right) \right] } \qquad & \\ & \leq \E \left[ \frac{1}{\left| \cu_{2n}\right|} \sum_{z \in\mathcal{Z}_{n,2n}} \sum_{e \subseteq z + \cu_n}  V_e \left( \nabla \nu^*_n (q) (e)  + \nabla \kappa (e) \right) \right] \notag \\ & +  C \left(( 1 + |q|^2) 3^{-\beta n} + \sum_{m=0}^n 3^{\frac{(m-n)}{2}}\tau_m^*(q)\right); \notag
\end{align}
\item we then prove
\begin{multline} \label{{of(2)proof}}
\E \left[ \frac{1}{\left| \cu_{2n}^+\right|} \sum_{e \subseteq \cu_{2n}^+}  V_e \left( \nabla \nu^*_n (q) (e)  + \nabla \kappa^+_{2n} (e) \right) \right] \\ \leq  \E \left[ \frac{1}{\left| \cu_{2n}\right|} \sum_{e \subseteq \cu_{2n}^+}  V_e \left( \nabla \nu^*_n (q) (e)  + \nabla \kappa (e) \right) \right]  +  C 3^{-\frac d2 n} (1 + |q|) .
\end{multline}
\end{enumerate}

\medskip

Proof of~\eqref{of(1)proof}.  We define $B_{n,2n}^+$ to be the set of edges of $\cu_{2n}^+$ which do not belong to a cube of the form $(z + \cu_n)$, for $z \in \mathcal{Z}_{n,2n}$, i.e.,
\begin{equation*}
B_{n,2n}^+ := \left\{ e \subseteq \cu_{2n}^+ ~:~ \forall z \in \mathcal{Z}_{n,2n}, \, e \not\subseteq z + \cu_n \right\}.
\end{equation*}
This set has been defined to have the following decomposition of the sum
\begin{equation*}
\sum_{e \subseteq \cu_{2n}^+} = \sum_{z \in \mathcal{Z}_{n,2n}} \sum_{e \subseteq z + \cu_n} + \sum_{e \in B_{n,2n}^+}.
\end{equation*}
Note also that the set $B_{n,2n}^+$ is almost equal to the set $B_{n,2n}$, the only difference is that we added the bonds which belong to the cube $\cu_{2n}^+$ but not to the cube $\cu_{2n}$, which is a small boundary layer of bonds. Additionally, one has the estimate on the cardinality of~$B_{n,2n}^+$,
\begin{equation*}
\left| B_{n,2n}^+ \right| \leq C 3^{-n} \left| \cu_{2n} \right|.
\end{equation*}
We first prove the estimate
\begin{equation*}
 \E \left[ \frac{1}{\left| \cu_{2n}\right|} \sum_{e \in B_{2n,n}^+}  V_e \left(  \nabla \nu^*_n (q) (e) + \nabla \kappa (e) \right) \right] \leq   C \left(( 1 + |q|^2) 3^{-\beta n} + \sum_{m=0}^n 3^{\frac{(m-n)}{2}}\tau_m^*(q)\right).
\end{equation*}
We first use that, for each $x \in \R$, one has the inequality $V_e(x) \leq \frac 1 \lambda x^2$. This shows
\begin{align*}
\lefteqn{ \E \left[ \frac{1}{\left| \cu_{2n}\right|} \sum_{e \in B_{2n,n}^+}  V_e \left(  \nabla \nu^*_n (q) (e) + \nabla \kappa (e) \right) \right]} \qquad & \\ & \leq C \E \left[ \frac{1}{\left| \cu_{2n}\right|} \sum_{e \in B_{2n,n}^+}  |\nabla \kappa (e)|^2 \right] + C \frac{\left|B_{2n,n}^+\right|}{\left| \cu_{2n} \right|} \left|\nabla \nu^*_n (q)\right|^2  \\ &
\leq C \E \left[ \frac{1}{\left| \cu_{2n}\right|} \sum_{e \in B_{2n,n}^+}  |\nabla \kappa (e)|^2 \right] + C (1 + |q|^2 ) 3^{-2n}.
\end{align*}
For each bond $e$ in the set $B_{2n,n}^+$, one has the equality $\mathbf{f}(e) =0$. This implies
\begin{equation*}
 \E \left[ \frac{1}{\left| \cu_{2n}\right|} \sum_{e \in B_{2n,n}^+}  |\nabla \kappa (e)|^2 \right]  \leq  \E \left[  \frac1{\left|\cu_{2n}\right|} \sum_{e \subseteq \cu_{2n}^+} \left| \mathbf{f} (e) - \nabla \kappa (e) \right|^2  \right].
\end{equation*}
Using the exact same computation as in~\eqref{exacsamecompkappapsiz}, one obtains
\begin{equation*}
\E \left[  \frac1{\left|\cu_{2n}\right|} \sum_{e \subseteq \cu_{2n}^+} \left| \mathbf{f} (e) - \nabla \kappa (e) \right|^2  \right] \leq  C \left(( 1 + |q|^2) 3^{-\beta n} + \sum_{m=0}^n 3^{\frac{(m-n)}{2}}\tau_m^*(q)\right).
\end{equation*}
Combining the few previous displays shows
\begin{equation*}
 \E \left[ \frac{1}{\left| \cu_{2n}\right|} \sum_{e \in B_{2n,n}^+}  V_e \left(  \nabla \nu^*_n (q) (e)  + \nabla \kappa (e) \right) \right] \leq  C \left(( 1 + |q|^2) 3^{-\beta n} + \sum_{m=0}^n 3^{\frac{(m-n)}{2}}\tau_m^*(q)\right),
\end{equation*}
and consequently
\begin{align*}
\lefteqn{\E \left[ \frac{1}{\left| \cu_{2n}\right|} \sum_{e \subseteq \cu_{2n}^+}  V_e \left(  \nabla \nu^*_n (q) (e)  + \nabla \kappa (e) \right) \right] } \qquad & \\ & = \E \left[ \frac{1}{\left| \cu_{2n}\right|} \sum_{z \in \mathcal{Z}_{n,2n}} \sum_{e \subseteq z + \cu_n}   V_e \left(  \nabla \nu^*_n (q) (e)  + \nabla \kappa (e) \right) \right] \\ &
				\quad + \E \left[ \frac{1}{\left| \cu_{2n}\right|} \sum_{e \in B_{2n,n}^+} V_e \left(  \nabla \nu^*_n (q) (e)  + \nabla \kappa (e) \right) \right] \\
				& \leq  \E \left[ \frac{1}{\left| \cu_{2n}\right|} \sum_{z \in \mathcal{Z}_{n,2n}} \sum_{e \subseteq z + \cu_n}   V_e \left(  \nabla \nu^*_n (q) (e)  + \nabla \kappa (e) \right) \right] \\ & \quad + C \left(( 1 + |q|^2) 3^{-\beta n} + \sum_{m=0}^n 3^{\frac{(m-n)}{2}}\tau_m^*(q)\right).
\end{align*}
This is~\eqref{of(1)proof}.

\medskip

Proof of~\eqref{{of(2)proof}}. The main tool is the estimate~\eqref{compgradkappakappa+}, which we recall
\begin{equation*}
\E \left[ \frac1{\left|\cu_{2n}^+\right|} \sum_{e \subseteq \cu_{2n}^+} \left| \nabla \kappa_{2n}^+ (e) - \nabla \kappa (e) \right|^2  \right] \leq C 3^{-dn}.
\end{equation*}
Using this inequality and a Taylor expansion, together with the assumption $V_e'' \leq \frac 1\lambda$, one obtains
\begin{align*}
\lefteqn{\E \left[ \frac{1}{\left| \cu_{2n}\right|} \sum_{e \subseteq \cu_{2n}^+}  V_e \left( \nabla \nu^*_n (q) (e) + \nabla \kappa_{2n}^+ (e) \right) \right] } \qquad & \\ & \leq \E \left[ \frac{1}{\left| \cu_{2n}\right|} \sum_{e \subseteq \cu_{2n}^+}  V_e \left( \nabla \nu^*_n (q) (e)  + \nabla \kappa_{2n} \right) \right] \\ & \qquad
+ \E \left[ \frac{1}{\left| \cu_{2n}\right|} \sum_{e \subseteq \cu_{2n}^+}  V_e' \left( \nabla \nu^*_n (q) (e) + \nabla \kappa (e) \right)  \left( \nabla \kappa_{2n}^+ (e) - \nabla \kappa (e)  \right) \right] \\ & \qquad
+ \frac1{2\lambda} \E \left[ \frac{1}{\left| \cu_{2n}\right|} \sum_{e \subseteq \cu_{2n}^+}  \left| \nabla \kappa_{2n}^+ (e) - \nabla \kappa (e) \right|^2 \right].
\end{align*} 
First combining the two previous displays, one has
\begin{multline} \label{displaywithC3-dn}
\E \left[ \frac{1}{\left| \cu_{2n}\right|} \sum_{e \subseteq \cu_{2n}^+}  V_e \left( \nabla \nu^*_n (q) (e) + \nabla \kappa_{2n}^+ (e) \right) \right]  \leq \E \left[ \frac{1}{\left| \cu_{2n}\right|} \sum_{e \subseteq \cu_{2n}^+}  V_e \left( \nabla \nu^*_n (q) (e) +  \nabla \kappa_{2n} \right) \right] \\ + C 3^{-dn} + \E \left[ \frac{1}{\left| \cu_{2n}\right|} \sum_{e \subseteq \cu_{2n}^+}  V_e' \left( \nabla \nu^*_n (q) (e) + \nabla \kappa (e) \right)  \left( \nabla \kappa_{2n}^+ (e) - \nabla \kappa (e)  \right) \right],
\end{multline} 
so that there only remains to study the last term on the right-hand side of the previous display. This is achieved thanks to the Cauchy-Schwarz inequality
\begin{align*}
\lefteqn{ \left| \E \left[ \frac{1}{\left| \cu_{2n}\right|} \sum_{e \subseteq \cu_{2n}^+}  V_e' \left( \nabla \nu^*_n (q) (e) + \nabla \kappa (e) \right)  \left( \nabla \kappa_{2n}^+ (e) - \nabla \kappa (e)  \right) \right] \right|} \qquad &  \\
			& \leq  \E \left[ \frac{1}{\left| \cu_{2n}\right|} \sum_{e \subseteq \cu_{2n}^+} \left| V_e' \left( \nabla \nu^*_n (q) (e) + \nabla \kappa (e) \right) \right|^2 \right]^{\frac 12}  \E \left[ \frac{1}{\left| \cu_{2n}\right|} \sum_{e \subseteq \cu_{2n}^+}  \left| \nabla \kappa_{2n}^+ (e) - \nabla \kappa (e)  \right|^2 \right]^{\frac 12}  \\
			& \leq C 3^{-\frac d2 n}  \E \left[ \frac{1}{\left| \cu_{2n}\right|} \sum_{e \subseteq \cu_{2n}^+} \left| V_e' \left( \nabla \nu^*_n (q) (e) + \nabla \kappa (e) \right) \right|^2 \right]^{\frac 12} .
\end{align*}
We then use that, for each $x \in \R$, $| V_e' \left( x \right)| \leq \frac 1\lambda |x|$ to obtain
\begin{equation*}
\E \left[ \frac{1}{\left| \cu_{2n}\right|} \sum_{e \subseteq \cu_{2n}^+} \left| V_e' \left( \nabla \nu^*_n (q) (e) +  \nabla \kappa (e) \right) \right|^2 \right] \leq  C \left|\nabla \nu^*_n (q) \right|^2 + C \E \left[ \frac{1}{\left| \cu_{2n}\right|} \sum_{e \subseteq \cu_{2n}^+} \left|  \nabla \kappa (e) \right|^2 \right] 
\end{equation*}
and by the definition of $\kappa$ given in~\eqref{defofkappa}, we have
\begin{align*}
\sum_{e \subseteq \cu_{2n}^+} \left|  \nabla \kappa (e) \right|^2  &\leq \sum_{e \subseteq \cu_{2n}^+} \left| \mathbf{f} (e) \right|^2 \\
												& \leq \sum_{z \in \mathcal{Z}_{n,2n}} \sum_{e \subseteq z + \cu_{n}} \left| \nabla \psi_z (e) \right|^2.
\end{align*}
Taking the expectation and using the estimate~\eqref{4.L2boundnustar} of Proposition~\ref{propofnunustar}, one derives
\begin{align*}
\E \left[ \frac 1{|\cu_{2n}|}\sum_{e \subseteq \cu_{2n}^+} \left|  \nabla \kappa (e) \right|^2 \right]&  \leq \E \left[  \frac 1{|\cu_{2n}|} \sum_{z \in \mathcal{Z}_{n,2n}} \sum_{e \subseteq z + \cu_{n}} \left| \nabla \psi_z (e) \right|^2 \right] \\
						& \leq \E \left[  \frac 1{|\cu_n|}  \sum_{e \subseteq  \cu_{n}} \left| \nabla \psi_{n,q} (e) \right|^2 \right] \\
						& \leq C (1 + |q|^2).
\end{align*}
Combining the few previous displays and the bound $\left|\nabla \nu^*_n (q)\right|^2 \leq C (1 + |q|^2) $ proved in~\eqref{boundnustarqn} gives
\begin{equation*}
\left| \E \left[ \frac{1}{\left| \cu_{2n}\right|} \sum_{e \subseteq \cu_{2n}^+}  V_e' \left( \nabla \nu^*_n (q) +  \nabla \kappa (e) \right)  \left( \nabla \kappa_{2n}^+ (e) - \nabla \kappa (e)  \right) \right] \right| \leq C 3^{-\frac d2 n} (1 + |q|).
\end{equation*}
Combining this with the estimate~\eqref{displaywithC3-dn} gives
\begin{multline*}
\E \left[ \frac{1}{\left| \cu_{2n}\right|} \sum_{e \subseteq \cu_{2n}^+}  V_e \left( \nabla \nu^*_n (q) +  \nabla \kappa_{2n}^+ (e) \right) \right]  \\ \leq \E \left[ \frac{1}{\left| \cu_{2n}\right|} \sum_{e \subseteq \cu_{2n}^+}  V_e \left( \nabla \nu^*_n (q) +\nabla \kappa_{2n} \right) \right]  + C (1 + |q|) 3^{-\frac d2 n}.
\end{multline*}
We complete the proof of~\eqref{{of(2)proof}} by noting that $V_e$ is positive and that $|\cu_{2n}| \leq |\cu_{2n}^+|$. This implies
\begin{align*}
\E \left[ \frac{1}{\left| \cu_{2n}^+\right|} \sum_{e \subseteq \cu_{2n}^+}  V_e \left( \nabla \nu^*_n (q) +  \nabla \kappa^+_{2n} (e) \right) \right] & \leq \E \left[ \frac{1}{\left| \cu_{2n}\right|} \sum_{e \subseteq \cu_{2n}^+}  V_e \left( \nabla \nu^*_n (q) + \nabla \kappa^+_{2n} (e) \right) \right] \\
					&  \leq   \E \left[ \frac{1}{\left| \cu_{2n}\right|} \sum_{e \subseteq \cu_{2n}^+}  V_e \left( \nabla \nu^*_n (q) + \nabla \kappa (e) \right) \right] \\ & \qquad+  C 3^{-\frac d2 n} (1 + |q|).
\end{align*}
This completes the proof of~\eqref{{of(2)proof}}.

\smallskip

We can now conclude this step. Combining~\eqref{of(1)proof} and~\eqref{{of(2)proof}} implies~\eqref{of(0)proof} and thus completes the proof of~\eqref{Step4quanthomog.maintech}.

\medskip

\textit{Step 5.} The conclusion. Combining the main results~\eqref{mainstep2entropy} of Step 2 and~\eqref{mainstep2energy} of Steps 3 and 4, one obtains
\begin{align*}
\lefteqn{\E \left[ \frac{1}{\left| \cu_{2n}^+\right|} \sum_{e \subseteq \cu_{2n}^+}  V_e \left( \nabla \nu^*_n (q) + \nabla \kappa^+_{2n} (e) \right) \right]  +  \frac{1}{\left| \cu_{2n}^+\right|} H \left( \P_{\kappa^+_{2n}}  \right)} \qquad &
						\\ & \leq \E \left[ \frac{1}{\left| \cu_{n}\right|} \sum_{e \subseteq \cu_{n}}  V_e \left( \nabla \psi_{n,q} (e) \right) \right]   + \frac{1}{\left| \cu_{n}\right|} H \left( \P^*_{n,q}  \right) \\ & \qquad +  C \left(( 1 + |q|^2) 3^{-\beta n} + \sum_{m=0}^n 3^{\frac{(m-n)}{2}}\tau_m^*(q)\right).
\end{align*}
But one knows that
\begin{align*}
\nu \left(\cu_{2n}^+ , \nabla \nu^*_{n} (q) \right) & = \inf_{\P \in \mathcal{P}\left( h^1_0 \left( \cu_{2n}^+ \right) \right) } \E \left[ \frac1{|\cu_{2n}^+|} \sum_{e \subseteq \cu_{2n}^+} V_e \left(\nabla \nu^*_{n} (q)(e) + \nabla \phi (e)  \right) \right] + \frac1{|\cu_{2n}^+|}  H\left(\P \right) \\
									& \leq \E \left[ \frac{1}{\left| \cu_{2n}^+\right|} \sum_{e \subseteq \cu_{2n}^+}  V_e \left( \nabla \nu^*_{n} (q)(e) + \nabla \kappa^+_{2n} (e) \right) \right]  +  \frac{1}{\left| \cu_{2n}^+\right|} H \left( \P_{\kappa^+_{2n}} \right).
\end{align*}
Moreover, by the definition of $\P_{n,q}^*$ and the equality $\nabla \nu^*_n (q) = \E \left[ \frac{1}{|\cu_n|} \sum_{e \subseteq \cu_n} \nabla \psi_{n,q} (e) \right]$, one has
\begin{equation*}
\nu^*(\cu_n,q) =  - \E \left[ \frac{1}{\left| \cu_{n}\right|} \sum_{e \subseteq \cu_{n}}  V_e \left( \nabla \psi_{n,q} (e) \right) \right] +  q \cdot  \nabla \nu^*_n (q) -  \frac{1}{\left| \cu_{n}\right|} H \left( \P^*_{n,q}  \right) .
\end{equation*}
Combining the three previous displays shows
\begin{equation*}
\nu \left(\cu_{2n}^+ , \nabla \nu^*_{n} (q) \right)  + \nu^*(\cu_n,q) - q \cdot \nabla \nu^*_n (q) \leq C \left(( 1 + |q|^2) 3^{-\beta n} + \sum_{m=0}^n 3^{\frac{(m-n)}{2}}\tau_m^*(q)\right).
\end{equation*}
The proof of Proposition~\ref{lemma3.3} is complete.
\end{proof}

\subsection{Quantitative convergence of the partitions functions} \label{section4.4gradphi} Now that Proposition~\ref{lemma3.3} is proved, we can deduce the main result of the article. The theorem is recalled below.

\begin{reptheorem}{QuantitativeconvergencetothGibbsstate}[Quantitative convergence to the Gibbs state]
There exist a constant $C := C(d , \lambda) < \infty$ and an exponent $\alpha := \alpha(d , \lambda) > 0$ such that, for each $p,q \in \Rd$,
\begin{equation*}
\left| \nu (\cu_n, p) - \bar \nu (p) \right| \leq C 3^{-\alpha n} (1 + |p|^2)
\end{equation*}
and 
\begin{equation*}
\left| \nu^* (\cu_n, q) - \bar \nu^* (q) \right| \leq C 3^{-\alpha n} (1 + |q|^2).
\end{equation*}
\end{reptheorem}

Before starting the proof, we mention that the argument only relies on the properties of the surface tensions $\nu$ and $\nu^*$ stated in Proposition~\ref{propofnunustar} together with upper bound for the convex duality given in Proposition~\ref{lemma3.3}. In particular, we do not use any specific properties of the gradient field model in the rest of this section.

\begin{proof}
From Proposition~\ref{def+propofarnu}, we know that, for each $p,q \in \Rd$, the two sequences $\left( \nu(\cu_n,p)\right)_{n \in \N}$ and $\left( \nu^*(\cu_n,p) \right)_{n \in \N}$ converge and that, for each $p,q \in \Rd$,
\begin{equation} \label{e.step1quanthomog}
\bar \nu (p) +  \bar \nu^* (q) \geq p \cdot q.
\end{equation}
We split the proof into 5 steps.
\begin{itemize}
\item In Step 1, the objective is to remove the $\cu_{2n}^+$ condition which appears in Proposition~\ref{lemma3.3}: the idea is to appeal to the subadditivity of the surface tension $\nu$ to prove the estimate, for each $p \in \Rd$,
\begin{equation} \label{e.step2quanthomog}
\nu (\cu_{3n} , p ) \leq \nu (\cu_{2n}^+ , p ) + C 3^{-n} (1 + |p|^2).
\end{equation}
\item In Step 2, we show that, for each $n \in \N$ and each $q \in \Rd$,
\begin{equation*}
|\nu^* (\cu_{n},q) - \bar \nu^* (q) | \leq  C \left(( 1 + |q|^2) 3^{-\beta n} + \sum_{m=0}^n 3^{\frac{(m-n)}{2}}\tau_m^*(q)\right).
\end{equation*}
\item In Step 3, we deduce that there exists an exponent $\alpha := \alpha (d , \lambda) > 0$ such that
\begin{equation} \label{e.step4quanthomog}
| \nu^* (\cu_{n},q) - \bar \nu^* (q) | \leq  C 3^{-\alpha n}.
\end{equation}
\item In Step 4, we show that the limiting surface tensions $\bar \nu$ and $\bar \nu^*$ are dual convex to one another: one has the equality, for each $q \in \Rd$,
\begin{equation} \label{e.step5quanthomog}
\bar \nu^*(q) = \sup_{p \in \Rd} - \bar \nu (p)  + p \cdot q.
\end{equation}
\item In Step 5, we show that  there exists an exponent $\alpha := \alpha (d , \lambda) > 0$ such that,
\begin{equation} \label{e.step6quanthomog}
| \nu (\cu_{n},q) - \bar \nu (q) | \leq  C 3^{-\alpha n}.
\end{equation}
\end{itemize}

\medskip

\textit{Step 1.} The main idea of this step is to consider a cube $\cu \subseteq \Zd$ satisfying the two following properties:
\begin{enumerate}
\item the cube $\cu$ is included in $\cu_{3n}$ and is almost as large as $\cu_{3n}$ in the sense that it satisfies the volume estimate
\begin{equation} \label{compcu3ncu}
|\cu_{3n}| \leq |\cu| + C 3^{3dn} \times 3^{-n}.
\end{equation}
\item the cube $\cu$ can be decomposed as a disjoint union of cubes of the same size as the cube $\cu_{2n}^+$, i.e., a union of disjoint translated of the cube $\cu_{2n}^+$. 
\end{enumerate}
More precisely, the cube $\cu$ can be constructed as follows: we denote by $\mathcal{Z}$ the set
\begin{equation*}
\mathcal{Z} := \left\{ z \in (3^{2n} + 2) \Zd ~:~ z + \cu_{2n}^+ \subseteq \cu_{3n} \right\}
\end{equation*}
and then define
\begin{equation*}
\cu = \bigcup_{z \in \mathcal{Z}} (z + \cu_{2n}^+).
\end{equation*}
With this definition, it is clear that the cube $\cu$ satisfies the two properties (1) and (2). Following the proof of the subadditivity of the finite volume surface tension $\nu$ given by Funaki and Spohn in~\cite{FS97}, on obtains the estimate
\begin{equation} \label{2012}
\nu \left( \cu_{3n} , p \right) \leq  \frac{|\cu_{3n} \setminus \cu |}{|\cu_{3n}|}\nu \left( \cu_{3n} \setminus \cu , p \right)  + \sum_{z \in \mathcal{Z}} \frac{\left|z + \cu_{2n}^+\right|}{|\cu_{3n}|} \nu (z + \cu_{2n}^+, p) + C 3^{-2n} (1 + |p|^2).
\end{equation}
Using Proposition~\ref{upperboundnugeneralsubset}, proved in Appendix~\ref{appendixA}, one knows that $\nu \left( \cu_{3n} \setminus \cu , p \right)$ is bounded by $C(1 + |p|^2)$, thus by the inequality~\eqref{compcu3ncu}, one can estimate
\begin{equation*}
\frac{|\cu_{3n} \setminus \cu |}{|\cu_{3n}|}\nu \left( \cu_{3n} \setminus \cu , p \right)  \leq C 3^{-n} (1 + |p|^2).
\end{equation*}
From the estimate~\eqref{2012} and the previous display, one obtains
\begin{equation*}
\nu \left( \cu_{3n} , p \right) \leq \sum_{z \in \mathcal{Z}} \frac{\left|z + \cu_{2n}^+\right|}{|\cu_{3n}|} \nu (z + \cu_{2n}^+, p) + C 3^{-2n} (1 + |p|^2).
\end{equation*}
But, one has the equality, for each $z \in \mathcal{Z}$, $\nu (z + \cu_{2n}^+, p) = \nu (\cu_{2n}^+, p) $. This implies
\begin{equation*}
\sum_{z \in \mathcal{Z}} \frac{\left|z + \cu_{2n}^+\right|}{|\cu_{3n}|} \nu (z + \cu_{2n}^+, p) = \frac{\left| \cu_{3n} \setminus \cu\right|}{|\cu_{3n}|} \nu (\cu_{2n}^+, p).
\end{equation*}
Using Proposition~\ref{p.quadbound}, we have $\nu (\cu_{2n}^+, p) \geq - C + \lambda |p|^2$. Combining this bound with~\eqref{compcu3ncu}, the previous display can be refined
\begin{equation*}
\frac{\left| \cu_{3n} \setminus \cu\right|}{|\cu_{3n}|} \nu (\cu_{2n}^+, p) \leq \nu (\cu_{2n}^+, p) + C (1 + |p|^2) 3^{-n}.
\end{equation*}
Combining the few previous displays shows
\begin{equation*}
\nu \left( \cu_{3n} , p \right) \leq \nu ( \cu_{2n}^+, p) + C 3^{-n} (1 + |p|^2),
\end{equation*}
which is the desired result. The proof of Step 1 is complete.

\medskip

\textit{Step 2.} First, by the formula, for each $q \in \Rd,$
\begin{equation*}
\nabla_q \nu^* \left( \cu_{n} , q \right) = \E \left[ \frac1{|\cu_{n}|} \sum_{e \subseteq \cu_n} \nabla \psi_{n,q}(e) \right]
\end{equation*}
and by the estimate~\eqref{4.L2boundnustar} of Proposition~\ref{propofnunustar}, one has
\begin{equation*}
\left| \nabla_q \nu^* \left( \cu_{n} , q \right) \right| \leq C (1 + |q|).
\end{equation*}
As a consequence, by the main result~\eqref{e.step2quanthomog} of Step 1, one has
\begin{equation*}
\nu \left( \cu_{3n} , \nabla_q \nu^* \left( \cu_{n} , q \right) \right) \leq \nu ( \cu_{2n}^+, \nabla_q  \nu^* \left(\cu_{n} , q \right)) + C 3^{-n} (1 + |q|^2).
\end{equation*}
Combining the previous display with Proposition~\ref{lemma3.3}, one obtains
\begin{equation} \label{reallemma3.3}
\nu \left( \cu_{3n} , \nabla_q \nu^* \left(\cu_{n} , q \right) \right) + \nu^*(\cu_n ,q) + q \cdot \nabla_q \nu^* (\cu_n, q) \leq C \left(( 1 + |q|^2) 3^{-\beta n} + \sum_{m=0}^n 3^{\frac{(m-n)}{2}}\tau_m^*(q)\right).
\end{equation}
Moreover using the inequality~\eqref{e.step1quanthomog} applied with $p = \nabla_q \nu^* \left( \cu_{n} , q \right)$ and $q$ gives
\begin{equation*}
0 \leq \bar \nu (\nabla_q  \nu^* \left( \cu_{n} , q \right)) +  \bar \nu^* (q) - \nabla_q  \nu^* \left( \cu_{n}, q \right) \cdot q.
\end{equation*} 
A combination of the two previous displays gives
\begin{multline} \label{quantdiffnubarnuetnustar}
\left( \nu \left( \cu_{3n} , \nabla_q \nu^* \left( \cu_{n} , q \right) \right) - \bar \nu (\nabla_q \nu^* \left(\cu_{n}, q \right))  \right) + \left( \nu^*(\cu_n,q) - \bar \nu^* (q) \right) \\ \leq C \left(( 1 + |q|^2) 3^{-\beta n} + \sum_{m=0}^n 3^{\frac{(m-n)}{2}}\tau_m^*(q)\right).
\end{multline}
By the subadditivity for the surface tension $\nu$ stated in Proposition~\ref{propofnunustar}, there exists a constant $C := C(d, \lambda) < \infty$ (in particular larger than the one appearing in the proposition) such the sequence $n \rightarrow \nu \left( \cu_n,p \right) + C (1 + |p|^2) 3^{-n}$ is decreasing. As a consequence, for each $n \in \N$ and each $p\in \Rd$,
\begin{equation*}
 \nu \left( \cu_{n} , p\right)  \geq \bar \nu (p)  - C (1 + |p|^2) 3^{-n}.
\end{equation*}
This implies, for each $q \in \Rd$,
\begin{equation*}
\nu \left( \cu_{3n} , \nabla_q \nu^* \left( \cu_{n} , q \right) \right)  - \bar \nu (\nabla_q \nu^* \left(\cu_{n}, q \right)) \geq - C (1 + |q|^2) 3^{-3n}.
\end{equation*}
Combining the previous inequality with~\eqref{quantdiffnubarnuetnustar} shows
\begin{equation} \label{mainresultstep3}
\nu^*(\cu_n,q) - \bar \nu^* (q) \leq C \left(( 1 + |q|^2) 3^{-\beta n} + \sum_{m=0}^n 3^{\frac{(m-n)}{2}}\tau_m^*(q)\right).
\end{equation}
The proof of Step 2 is complete.

\medskip

\textit{Step 3.} Let $C := C(d , \lambda) < \infty$ be a constant large enough so that the sequence $\nu^*(\cu_n, q) + C (1+|q|^2)3^{-n}$ is decreasing. To shorten the notation, we denote by, for $q \in \Rd$,
\begin{equation} \label{definitionofF_n}
F_n(q) := \nu^*(\cu_n, q) + C (1+|q|^2)3^{-n} -  \bar \nu^* (q),
\end{equation}
so that the sequence $F_n(q)$ is decreasing and tends to $0$ as $n$ tends to infinity. Moreover, one has the following inequality
\begin{equation*}
\tau_n^*(q) \leq  F_n(q) - F_{n+1}(q) .
\end{equation*}
We can rewrite the main result~\eqref{mainresultstep3} of Step 3 with this notation
\begin{equation} \label{Fncontracts}
F_n(q) \leq C \left(( 1 + |q|^2) 3^{-\beta n} + \sum_{m=0}^n 3^{\frac{(m-n)}{2}}\left( F_m(q) - F_{m+1}(q) \right)\right).
\end{equation}
We then define
\begin{equation*}
\tilde F_n (q):= 3^{-\frac{n}{4}} \sum_{m=0}^{n} 3^{\frac m4}  F_m(q).
\end{equation*}
We next show that there exist two constants $\theta := \theta(d, \lambda) \in (0,1)$, $ C:=C(d , \lambda) < \infty$ and an exponent $\beta := \beta(d, \lambda) > 0$ such that, for every $n \in \N$,
\begin{equation} \label{2.60}
\tilde F_{n+1}(q) \leq \theta \tilde F_n(q) + C 3^{-\beta n}.
\end{equation}
Using the inequality $F_0(q) \leq C (1 + |q|^2)$, one has
\begin{equation} \label{Fn(q)-tildeFn+1(q)}
\tilde F_n(q) - \tilde F_{n+1}(q) \geq 3^{-\frac{n}{4}} \sum_{m=0}^{n} 3^{\frac m4}  ( F_m(q) - F_{m+1}(q) ) - C (1 + |q|^2)3^{-\frac n4} .
\end{equation}
Since $\left( F_n(q) \right)_{n \in \N}$ is a decreasing sequence, we deduce from the previous display that, for each $n \in \N$,
\begin{equation*}
\tilde F_{n+1}(q) \leq \tilde F_n(q)  + C (1 + |q|^2)3^{-\frac n4}.
\end{equation*}
Using the inequality~\eqref{Fncontracts} and reducing the size of the exponent exponent $\beta$ if necessary, one deduces
\begin{align*}
\tilde{F}_{n+1}(q) & \leq \tilde F_n(q)  + C (1 + |q|^2)3^{-\frac n4} \\ & \leq 3^{-\frac{n}{4}} \sum_{m=0}^{n} 3^{\frac m4}  F_m(q)  + C (1 + |q|^2)3^{-\frac n4} \\
& \leq C 3^{-\frac{n}{4}} \sum_{m=0}^{n} 3^{\frac m4}   \left( \left(( 1 + |q|^2) 3^{-\beta m} + \sum_{k=0}^m 3^{\frac{(k-m)}{2}}\left( F_k(q) - F_{k+1}(q) \right)\right) \right) \\ & \qquad + C (1 + |q|^2)3^{- \frac{n}{4}} \\
& \leq C 3^{-\frac{n}{4}} \sum_{m = 0}^n  \sum_{k=0}^m 3^{- \frac{m}{4}}  3^{\frac k2} \left( F_k(q) - F_{k+1}(q) \right) + C(1 + |q|^2) 3^{- \beta n} \\
& \leq C 3^{-\frac{n}{4}} \sum_{k = 0}^n  \sum_{m=k}^n 3^{- \frac{m}{4}}  3^{\frac k2} \left( F_k(q) - F_{k+1}(q) \right) + C(1 + |q|^2) 3^{- \beta  n} \\
& \leq C 3^{-\frac{n}{4} } \sum_{k = 0}^m  3^{\frac{k}{4}} \left( F_k(q) - F_{k+1}(q) \right) + C(1 + |q|^2) 3^{- \beta  n} .
\end{align*} 
Comparing the previous display with~\eqref{Fn(q)-tildeFn+1(q)} gives
\begin{equation*}
\tilde{F}_{n+1}(q) \leq C \left(  \tilde F_n(q) - \tilde F_{n+1}(q) \right) +  C(1 + |q|^2) 3^{- \beta  n}.
\end{equation*}
A rearrangement of this inequality gives~\eqref{2.60}. An iteration of~\eqref{2.60} yields
\begin{equation*}
\tilde F_n(q) \leq \theta^n \tilde F_0 + C (1 + |q|^2) \sum_{k =0}^n \theta^k 3^{- \beta  (n -k)}.
\end{equation*} 
By making $\theta$ closer to $1$ if necessary, one has
\begin{equation*}
\sum_{k =0}^n \theta^k 3^{- \beta  (n -k)} \leq C \theta^n.
\end{equation*}
Combining the few previous displays shows
\begin{equation*}
\tilde F_n(q) \leq C(1 + |q|^2) \theta^n.
\end{equation*}
Setting $\alpha := -\frac{\ln \theta}{\ln 3}$ so that $\theta = 3^{-\alpha}$ gives the bound $\tilde F_n(q) \leq C 3^{-\alpha n}$. By the definition of $\tilde F_n(q)$, one has the inequality
\begin{equation*}
F_n(q) \leq \tilde F_n(q),
\end{equation*}
and thus
\begin{equation*}
F_n(q) \leq C (1 + |q|^2) 3^{-\alpha n}.
\end{equation*}
We conclude the proof by noting that $F_n(q)$ was defined so that it is decreasing and tends to $0$. In particular, it is positive. By the explicit formula~\eqref{definitionofF_n} for $F_n$ and the previous display, one obtains
\begin{equation*}
- C (1+|q|^2)3^{-n} \leq \nu^*(\cu_n, q)  -  \bar \nu^* (q) \leq C (1 + |q|^2) 3^{-\alpha n},
\end{equation*}
for some constant $C := C(d , \lambda) < \infty$ and exponent $\alpha := \alpha (d, \lambda) > 0$. By reducing the size of the exponent $\alpha$, one eventually obtains
\begin{equation*}
\left| \nu^*(\cu_n, q)  -  \bar \nu^* (q) \right| \leq  C (1 + |q|^2) 3^{-\alpha n}.
\end{equation*}
The proof of Step 3 is complete.

\medskip

\textit{Step 4.} First note that, by~\eqref{e.step1quanthomog}, for each $p,q \in \Rd$
\begin{equation} \label{3.51}
0 \leq \bar \nu (p) +  \bar \nu^* (q) - p \cdot q.
\end{equation}
This implies
\begin{equation*}
\bar \nu^* (q) \geq \sup_{p \in \Rd} - \bar \nu (p) + p \cdot q.
\end{equation*}
The main idea of this step is to use Proposition~\ref{lemma3.3} to show the two following results:
\begin{enumerate}
\item For each $q \in \Rd$, the sequence $\nabla_q \nu^*(\cu_n, q)$ converges as $n$ tends to infinity. We denote its limit by $\bar P(q)$. Moreover one has that the following quantitative estimate
\begin{equation} \label{mainresultstep4inside}
\left| \nabla_q \nu^*(\cu_n, q) - \bar P(q) \right| \leq C (1 + |q|) 3^{-\alpha n}.
\end{equation}
\begin{remark}
We would like to say that the limit is equal to $\nabla_q \bar \nu^* (q) $ but at this point of the argument, one only knows that the function $q \rightarrow \bar \nu^*(q)$ is convex and in particular we do not know that it is differentiable everywhere. We will prove later that $\bar \nu^* (q)$ is in fact $C^1(\R)$ and this will imply $\bar P (q) = \nabla_q \bar \nu^* (q) $.
\end{remark}
\item We deduce from (1) that, for each $q \in \Rd$, one has the following quantitative convergence estimate
\begin{equation} \label{mainresultstep4outside}
\left| \nu(\cu_{3n}, \nabla_q \nu^*(\cu_n, q) ) - \bar \nu \left( \bar P(q)\right) \right| \leq C (1  + |q|^2 ) 3^{- \alpha n}.
\end{equation}
\end{enumerate}

\medskip

We first prove (1). From the main result of the previous step~\eqref{e.step5quanthomog}, we deduce that, for each $q \in \Rd$,
\begin{equation*}
\tau_n^*(q) = \nu^* \left( \cu_n , q \right) - \nu^* \left( \cu_{n+1} , q \right) \leq C (1 + |q|^2 ) 3^{-\alpha n}.
\end{equation*}
Combining this result with~\eqref{splitstep2} gives, for each $q \in \Rd$,
\begin{equation*}
\left|  \nabla_q \nu^* (\cu_{n+1}, q) -  \nabla_q \nu^* (\cu_n, q)\right| \leq C (1 + |q|^2 ) 3^{- \alpha n}.
\end{equation*}
The previous display implies that the sequence $\nabla_q \nu^* (\cu_{n+1}, q)$ converges for each $q \in \Rd$ together with the quantitative rate of convergence stated in~\eqref{mainresultstep4inside}. We denote by $\bar P (q)$ its limit and the proof of (1) is complete.

\medskip

We now prove (2). We first use the triangle inequality to estimate the left-hand side of~\eqref{mainresultstep4outside}
\begin{multline*}
\left| \nu(\cu_{3n}, \nabla_q \nu^*(\cu_n, q) ) - \bar \nu \left(\bar P(q) \right)\right|  \leq \left| \nu(\cu_{3n}, \nabla_q \nu^*(\cu_n, q) ) - \nu\left(\cu_{3n}, \bar P(q) \right)\right| \\ + \left|\nu\left(\cu_{3n}, \bar P(q) \right)  - \bar \nu \left(\bar P(q) \right)\right| .
\end{multline*}
Since one has the bound, for each $n \in \N$, $\left| \nabla_q \nu^*(\cu_n, q) \right| \leq C ( 1 + |q|)$, one obtains by taking the limit $n$ tends to infinity,
\begin{equation} \label{boundbarp}
\left| \bar P (q) \right| \leq C ( 1 + |q|).
\end{equation}
Combining the previous bound with~\eqref{e.step4quanthomog} gives
\begin{equation*}
 \left|\nu(\cu_{3n}, \bar P(q) )  - \bar \nu \left(\bar P(q) \right)\right| \leq C (1 + |q|^2 ) 3^{- \alpha n}.
\end{equation*}
To prove~\eqref{mainresultstep4outside}, it is sufficient to prove
\begin{equation*}
\left| \nu(\cu_{3n}, \nabla_q \nu^*(\cu_n, q) ) - \nu\left(\cu_{3n}, \bar P(q) \right) \right| \leq C (1 + |q|^2) 3^{- \alpha n}.
\end{equation*}
To this end, we first prove the following property: for each $p,p' \in \Rd$, and each $n \in \N$,
\begin{equation*}
| \nu(\cu_{n},p ) - \nu(\cu_{n}, p' )| \leq C (|p| + |p'| ) |p-p'|.
\end{equation*}
The idea to prove the previous inequality is to compute the gradient of $p \mapsto  \nu(\cu_{n},p )$. A straightforward computation gives
\begin{equation*}
\left| \nabla_p \nu(\cu_{n},p )  \right| \leq \E \left[ \frac1{|\cu_n|} \sum_{e \in \cu_n} \left| V_e'( p \cdot e + \nabla \phi_{n,p}(e)) \right| \right].
\end{equation*}
Using the bound, for each $x \in \R$ $ V_e'(x) \leq \frac 1\lambda |x|$ and the Jensen inequality, we obtain
\begin{align*}
|\nabla_p \nu(\cu_{n},p ) | & \leq |p| + \E \left[ \frac1{|\cu_n|} \sum_{e \in \cu_n} \left|\nabla \phi_{n,p}(e)\right| \right] \\
				& \leq |p| + \E \left[ \frac1{|\cu_n|} \sum_{e \in \cu_n} \left|\nabla \phi_{n,p}(e)\right|^2 \right]^\frac 12.
\end{align*}
We then apply the estimate~\eqref{4.L2boundnu} Proposition~\ref{propofnunustar} to derive the bound
\begin{equation} \label{Dnubounded}
|\nabla_p \nu(\cu_{n},p ) | \leq C (1 + |p|).
\end{equation}
This implies that, for each $n \in \N$ and each $p,p'\in \Rd$, $\nu(\cu_n, \cdot)$ is $C (1 + |p| + |p'|)$-Lipschitz in the ball $B(0 , |p|+ |p'|)$. Since both $p$ and $p'$ belongs to $B(0 , |p|+ |p'|)$, one has
\begin{equation} \label{3.66}
| \nu(\cu_{n},p ) - \nu(\cu_{n}, p' )| \leq C (1 +|p| + |p'| ) |p-p'|.
\end{equation}
This is the desired result. Applying the previous estimate with $p = \nabla_p \nu^*(\cu_n, q)$ and $p' =  \bar P(q)$ gives
\begin{equation*}
\left| \nu(\cu_{3n}, \nabla_p \nu^*(\cu_n, q) ) - \nu(\cu_{3n}, \bar P(q) ) \right| \leq C ( 1+ |\nabla_q \nu^*(\cu_n, q)| + |\bar P(q)| ) |\nabla_q \nu^*(\cu_n, q) - \bar P(q)|.
\end{equation*}
By~\eqref{boundbarp}, for each $q \in \Rd$,
\begin{equation*}
\left|\bar P (q)\right|  \leq C (1 + |q|).
\end{equation*}
Combining the three previous displays with~\eqref{mainresultstep4inside} gives
\begin{equation} \label{3.67}
\left| \nu(\cu_{3n}, \nabla_p \nu^*(\cu_n, q) ) - \nu(\cu_{3n}, \bar P(q) ) \right| \leq C (1 + |q|^2) 3^{- \alpha n}
\end{equation}
and completes the proof of (2).

\medskip

We now prove the main result~\eqref{e.step5quanthomog} of Step 4. By~\eqref{reallemma3.3} and the main result~\eqref{e.step4quanthomog} of Step 3, one has, for each $q \in \Rd$,
\begin{equation*}
\nu \left(\cu_{3n}, \nabla_q \nu^*\left( \cu_{n}, q\right) \right) + \nu^* \left(\cu_{n} , q \right) -\nabla_q \nu^*\left( \cu_{n}, q\right) \cdot q \leq C (1 + |q|^2) 3^{-\alpha n}.
\end{equation*} 
By~\eqref{mainresultstep4inside} and~\eqref{mainresultstep4outside}, one also has the convergence
\begin{equation*}
\nu \left(\cu_{3n}, \nabla_q \nu^* \left( \cu_{n}, q\right) \right) + \nu^* \left(\cu_{n} , q \right) -\nabla_q \nu^* \left( \cu_{n} ,q \right) \cdot q \underset{n \rightarrow \infty}{\longrightarrow} \bar \nu \left(\bar P (q) \right) + \bar \nu^* \left(  q \right) - \bar P (q) \cdot q.
\end{equation*}
A combination of the two previous displays gives
\begin{equation*}
\bar \nu \left(\bar P (q) \right) + \bar \nu^* \left( q \right) - \bar P (q) \cdot q \leq 0.
\end{equation*}
Together with~\eqref{3.51}, the previous estimate gives 
\begin{equation*}
\bar \nu \left(\bar P (q) \right) + \bar \nu^* \left(q \right) - \bar P (q) \cdot q = 0,
\end{equation*}
and thus
\begin{equation*}
\bar \nu^*(q) = \sup_{p \in \Rd} - \bar \nu (p)  + p \cdot q.
\end{equation*}
This is precisely~\eqref{e.step5quanthomog} and the proof of Step 4 is complete.

\medskip

\textit{Step 5.} The main result~\eqref{e.step5quanthomog} of Step 4 asserts that $\bar \nu^*$ is the Legendre-Fenchel transform of $\bar \nu$. But by Proposition~\ref{def+propofarnu}, one knows that for each $p_1,p_2 \in \Rd$,
\begin{equation*}
\frac 1C |p_0 - p_1 |^2 \leq \frac 12 \bar \nu(\cu_n, p_0) + \frac 12 \bar \nu(\cu_n, p_1) -   \bar \nu \left(\cu_n, \frac{p_0 + p_1}{2}\right) \leq C  |p_0 - p_1 |^2.
\end{equation*}
With the two previous properties, one deduces that $\bar \nu^*$ is also uniformly convex. As a consequence, it is in the space $C^{1,1} (\Rd)$ and one has the following equalities, for each $p,q \in \Rd$,
\begin{equation} \label{unifconvofbarnustar}
\nabla_p \bar \nu \left( \nabla_q \bar \nu^* (q) \right) = q , ~\nabla_q \bar \nu^*\left( \nabla_p \bar \nu  (q) \right) = q, ~\mbox{and}~ \bar P (q) = \nabla_q \bar \nu^* (q). 
\end{equation}
We are now ready to prove~\eqref{e.step6quanthomog}. We start from~\eqref{quantdiffnubarnuetnustar}, which reads, for each $q \in \Rd$,
\begin{multline} \label{quantdiffnubarnuetnustarbis}
\left( \nu \left( \cu_{3n} , \nabla_q \nu^* \left( \cu_{n} , q \right) \right) - \bar \nu (\nabla_q \nu^* \left(\cu_{n}, q \right))  \right) + \left( \nu^*(\cu_n,q) - \bar \nu^* (q) \right) \\ \leq C \left(( 1 + |q|^2) 3^{-\beta n} + \sum_{m=0}^n 3^{\frac{(m-n)}{2}}\tau_m^*(q)\right).
\end{multline}
We then apply the estimate~\eqref{e.step4quanthomog} which allows to estimate most of the terms in the previous display. Precisely, one has the inequalities
\begin{equation*}
 \nu^*(\cu_n,q) - \bar \nu^* (q) \leq C (1+ |q|^2) 3^{-\alpha n}
\end{equation*}
and
\begin{equation*}
\sum_{m=0}^n 3^{\frac{(m-n)}{2}}\tau_m^*(q)  \leq C (1+ |q|^2) 3^{-\alpha n}.
\end{equation*}
With these estimates, the inequality~\eqref{quantdiffnubarnuetnustarbis} becomes
\begin{equation*}
\nu \left( \cu_{3n} , \nabla_q \nu^* \left( \cu_{n} , q \right)\right) - \bar \nu (\nabla_q \nu^* \left(\cu_{n}, q \right))   \leq C (1+ |q|^2) 3^{-\alpha n}.
\end{equation*}
Then by~\eqref{3.67}, one has
\begin{equation*}
\left| \nu(\cu_{3n}, \nabla_q \nu^*(\cu_n, q) ) - \nu(\cu_{3n}, \bar P(q) ) \right| \leq C (1 + |q|^2) 3^{- \alpha n}.
\end{equation*}
Then by sending $n$ to infinity in~\eqref{3.66}, one obtains, for each $p,p' \in \Rd$,
\begin{equation}\label{Linftboundgradbarnu}
| \bar \nu(p ) - \bar \nu(p' )| \leq C (|p| + |p'| ) |p-p'|.
\end{equation}
With the same proof as the one which gives the bound~\eqref{3.67}, one obtains
\begin{equation*}
\left| \bar \nu( \nabla_q \nu^*(\cu_n, q) ) - \bar \nu( \bar P(q) ) \right| \leq C (1 + |q|^2) 3^{- \alpha n}.
\end{equation*}
Combining the few previous displays shows, for each $q \in \Rd$,
\begin{equation*}
\nu \left( \cu_{3n} , \bar P(q) \right) - \bar \nu \left( \bar P(q) \right) \leq C (1 + |q|^2) 3^{- \alpha n}.
\end{equation*}
Applying the previous inequality with $q = \nabla_p \bar \nu (p)$ gives, thanks to~\eqref{unifconvofbarnustar},
\begin{equation*}
\nu \left( \cu_{3n} ,p \right) - \bar \nu \left( p \right) \leq C (1 + |\nabla_p \bar \nu (p)|^2) 3^{- \alpha n}.
\end{equation*}
We then simplify the term on the right-hand side. Thanks to~\eqref{Linftboundgradbarnu}, one obtains the bound on the gradient of $\bar \nu$, for each $p \in \Rd$,
\begin{equation*}
\left| \nabla_p  \bar \nu \left( p \right) \right| \leq C ( 1 +|p|).
\end{equation*}
A combination of the two previous displays gives, for each $n \in \N$ and each $p \in \Rd$,
\begin{equation*}
\nu \left( \cu_{3n} ,p \right) - \bar \nu \left( p \right) \leq C (1 + |p|^2) 3^{- \alpha n}.
\end{equation*}
We now want to remove the $3n$ term on the left-hand side. To this end, we use the subadditivity of $\nu$ stated in Proposition~\ref{propofnunustar}, to obtain, for each $p\in \Rd$ and each $n \in \N$,
\begin{align*}
\nu \left( \cu_{3n+ 2} ,p \right) - \bar \nu \left( p \right) & \leq \nu \left( \cu_{3n+ 1} ,p \right) - \bar \nu \left( p \right) + C (1+|p|^2) 3^{-n} \\ & \leq \nu \left( \cu_{3n} ,p \right) - \bar \nu \left( p \right) + C (1+|p|^2) 3^{-n} \\
																										& \leq  C (1 + |p|^2) 3^{- \alpha n}.
\end{align*}
From the previous display and by reducing the size of the exponent $\alpha$, one obtains for each $n \in \N$ and each $p \in \Rd$,
\begin{equation*}
\nu \left( \cu_n ,p \right) - \bar \nu \left( p \right) \leq C (1 + |p|^2) 3^{- \alpha n}.
\end{equation*}
The proof of~\eqref{e.step6quanthomog} is almost complete, there only remains to prove a lower bound for $\nu \left( \cu_n ,p \right) - \bar \nu \left( p \right)$. But one knows that there exists a constant $C:= C(d, \lambda) < \infty$ such that the sequence $\nu \left( \cu_n ,p \right)  + C (1+|p|^2) 3^{-n}$ is decreasing and converges to $\bar \nu \left( p \right)$. This implies in particular that, for each $n \in \N$ and each $p \in \Rd$,
\begin{equation*}
\nu \left( \cu_n ,p \right) - \bar \nu \left( p \right) \geq  - C (1 + |p|^2) 3^{- n}
\end{equation*}
and provides the lower bound. A combination of the two previous displays shows
\begin{equation*}
| \nu \left( \cu_n ,p \right) - \bar \nu \left( p \right)  | \leq C (1 + |p|^2) 3^{- \alpha n}
\end{equation*}
and completes the proof of Step 5 and of Theorem~\ref{QuantitativeconvergencetothGibbsstate}.
\end{proof}

\subsection{Quantitative contraction of the fields $\phi_{n,p}$ and $\psi_{n,q}$ to affine functions} Now that Theorem~\ref{QuantitativeconvergencetothGibbsstate} is proved, we deduce the $L^2$ estimate on the random variables $\phi_{n,p}$ and $\psi_{n,q}$ stated in Theorem~\ref{conhvergencetoaffine}. The theorem is recalled below.
\begin{reptheorem}{conhvergencetoaffine}[$L^2$ contraction of the Gibbs measure]
There exist a constant $C := C(d , \lambda) < \infty$ and an exponent $\alpha := \alpha (d, \lambda) > 0$ such that for each $n \in \N$, $p, q \in \Rd$,
\begin{multline*}
\frac{1}{\left(\diam \cu_n\right)^2} \E \left[  \frac1{\left| \cu_n \right|} \sum_{x \in \cu_n} \left(  \left| \phi_{n,p}(x) \right|^2 + \left| \psi_{n,q}(x) -  \nabla_q \bar \nu^* (q ) \cdot x \right|^2 \right) \right] \\ \leq C 3^{ - \alpha n} \left( 1 + |p|^2 + |q|^2 \right).
\end{multline*}
\end{reptheorem}

\begin{proof}
We first prove the estimate for the random variable $\psi_{n,q}$, i.e.,
\begin{equation} \label{psiisclosetoaffinefinal1342}
\E \left[  \frac1{\left| \cu_n \right|} \sum_{x \in \cu_n} \left| \psi_{n,q}(x) -  \nabla_q \bar \nu^* (q ) \cdot x \right|^2 \right] \leq C 3^{n (2 - \alpha)} \left( 1 +|q|^2 \right).
\end{equation}
Indeed in that case all the tools have already been developed and this allows for a short proof. First by Theorem~\ref{QuantitativeconvergencetothGibbsstate}, one knows that, for each $q \in \Rd$,
\begin{equation*}
\tau_n^*(q) \leq 3^{-\alpha n } (1 + |q|^2 ).
\end{equation*}
Using the previous display together with Proposition~\ref{p.ustarminimizerareflat}, we obtain
\begin{equation*}
\E \left[\frac1{|\cu_{n}|} \sum_{x \in \cu_{n}} \left| \psi_{n,q}(x) -  \nabla_q \nu^*(\cu_{n}, q) \cdot x  \right|^2 \right] \leq C 3^{(2-\alpha )n} (1 + |q|^2).
\end{equation*} 
But by~\eqref{mainresultstep4inside} and~\eqref{unifconvofbarnustar}, one also has
\begin{equation*}
\left| \nabla_q \nu^*(\cu_{n}, q) -  \nabla_q \bar \nu^*( q) \right| \leq C 3^{-\alpha n} (1 + |q|^2 ).
\end{equation*}
A combination of the two previous displays gives~\eqref{psiisclosetoaffinefinal1342} and completes the proof.
\medskip

We now want to prove the estimate with the random variable $\phi_{n,p}$, i.e.,
\begin{equation*}
\E \left[  \frac1{\left| \cu_n \right|} \sum_{x \in \cu_n} \left| \phi_{n,p}(x) - p \cdot x \right|^2 \right] \leq C 3^{(2 - \alpha)n} \left( 1 + |p|^2 \right).
\end{equation*}
The proof follows the same lines as the proof of~\eqref{psiisclosetoaffinefinal1342} except that we have not proved an equivalent version of Proposition~\ref{p.ustarminimizerareflat}. The proof of this statement is split into 2 steps.
\begin{itemize}
\item In Step 1, we show that, for each $m\in \N$ with $m\leq n$,
\begin{equation*}
\frac{1}{|\mathcal{Z}_{m,n}|} \sum_{z \in \mathcal{Z}_{m,n}} \E \left[ \left| \left\langle \nabla \phi_{n,p} \right\rangle_{z + \cu_m} \right|^2 \right] \leq  C ( 1 + |p|^2) 3^{-\alpha m}.
\end{equation*}
\item In Step 2, we deduce from the previous step and the multiscale Poincar\'e inequality
\begin{equation*}
\E \left[ \frac{1}{\left| \cu_n\right|} \sum_{x \in \cu_n} \left| \phi_{n,p} \right|^2 \right] \leq  C ( 1 + |p|^2) 3^{(2-\alpha) n}.
\end{equation*}
\end{itemize}
\medskip

\textit{Step 1.} Consider the random variable $\phi = \sum_{z \in \mathcal{Z}_{m,n}} \phi_z$ introduced in Proposition~\ref{2sccomparisonnu} as well as the coupling between $\phi$ and $\phi_{n,p}$ introduced in the same proposition. The following estimate holds
\begin{equation*}
\frac{1}{|\cu_{n}|}\E \left[  \sum_{e \subseteq \cu_{n}} \left| \nabla \phi(e) - \nabla \phi_{n,p} (e)\right|^2  \right] \leq C \left( \nu (\cu_m , q) - \nu (\cu_{n} , q) \right) + C 3^{-\frac m2} (1 + |p|^2).
\end{equation*} 
Using Theorem~\ref{QuantitativeconvergencetothGibbsstate}, the previous estimate can be refined and one obtains
\begin{equation*}
\frac{1}{|\cu_{n}|}\E \left[  \sum_{e \subseteq \cu_{n}} \left| \nabla \phi(e) - \nabla \phi_{n,p} (e)\right|^2  \right] \leq C ( 1 + |p|^2 ) 3^{- \alpha m}.
\end{equation*}
Using this inequality, one has
\begin{align*}
\lefteqn{\frac{1}{|\mathcal{Z}_{m,n}|} \sum_{z \in \mathcal{Z}_{m,n}} \E \left[ \left| \left\langle \nabla \phi_{n,p} \right\rangle_{z + \cu_m} \right|^2 \right] } \qquad & \\ &\leq \frac{1}{|\mathcal{Z}_{m,n}|}  \sum_{z \in \mathcal{Z}_{m,n}} \E \left[ \left| \left\langle \nabla \phi \right\rangle_{z + \cu_m} \right|^2 \right] + \frac1{|\cu_n|}\E \left[ \sum_{e \subseteq \cu_n} \left|\nabla \phi(e)  - \nabla \phi_{n,p}(e) \right|^2 \right] \\
				& \leq \frac{1}{|\mathcal{Z}_{m,n}|}  \sum_{z \in \mathcal{Z}_{m,n}} \E \left[ \left| \left\langle \nabla \phi \right\rangle_{z + \cu_m} \right|^2 \right] +  C ( 1 + |p|^2) 3^{-\alpha m}.
\end{align*}
We then note that, for each $z \in \mathcal{Z}_{m,n}$, one has the equality $\left\langle \nabla \phi \right\rangle_{z + \cu_m} = \left\langle \nabla \phi_z \right\rangle_{z + \cu_m} $ and that, since $\phi_z \in h^1_0 (z + \cu_m)$, one has $\left\langle \nabla \phi_z \right\rangle_{z + \cu_m} = 0$. Consequently, the previous display can be simplified and one obtains
\begin{equation*}
\frac{1}{|\mathcal{Z}_{m,n}|} \sum_{z \in \mathcal{Z}_{m,n}} \E \left[ \left| \left\langle \nabla \phi_{n,p} \right\rangle_{z + \cu_m} \right|^2 \right] \leq  C ( 1 + |p|^2) 3^{-\alpha m}.
\end{equation*} 

\medskip

\textit{Step 2.} We now apply the multiscale Poincar\'e inequality stated in Proposition~\ref{p.poincmultpoinc} for functions in $h^1_0 (\cu_n)$. This gives
\begin{equation*}
\frac{1}{\left| \cu_n\right|} \sum_{x \in \cu_n} \left| \phi_{n,p}(x) \right|^2  \leq C \sum_{e \subseteq \cu_n} \left| \nabla \phi_{n,p} (e) \right|^2 + C 3^n\sum_{m = 1}^n 3^m \left( \frac{1}{\left| \mathcal{Z}_{m,n} \right|} \sum_{y \in \mathcal{Z}_{m,n}} \left|\left\langle \nabla \phi_{n,p} \right\rangle_{z + \cu_m} \right|^2 \right).
\end{equation*}
Taking the expectation and using Proposition~\ref{p.quadbound} to estimate the first term on the right-hand side and the main result of Step 1 to estimate the second term gives
\begin{align*}
\E \left[ \frac{1}{\left| \cu_n\right|} \sum_{x \in \cu_n} \left| \phi_{n,p}(x) \right|^2 \right] &\leq C ( 1 + |p|^2) + C (1 + |p|^2) 3^n \sum_{m = 1}^n 3^m 3^{- \alpha m} \\
					& \leq  C ( 1 + |p|^2) 3^{(2-\alpha) n}.
\end{align*}
This is the desired result. The proof of Theorem~\ref{conhvergencetoaffine} is complete.
\end{proof}

\appendix

\section{Technical estimates} \label{appendixA}

Before stating the first proposition of this appendix, we recall that the space $H$ mentioned in the following proposition is the space of functions of $\mathring h^1 (\cu_{n})$ which are constant on the cubes $(z +\cu_{m})$, for $z \in \mathcal{Z}_{m,n}$. It is a space of dimension $3^{d(n-m)} -1$ and each function $h \in H$ can be written in the following form
\begin{equation*}
h = \sum_{z \in \mathcal{Z}_{m,n}} \lambda_z \indc_{z  + \cu_{m}},
\end{equation*}
for some constants $\left( \lambda_z \right)_{z \in \mathcal{Z}_{m,n}}$ satisfying $\sum_{z \in \mathcal{Z}_{m,n}} \lambda_z = 0$.

\begin{proposition}\label{5dest.}
There exists a constant $C := C(d , \lambda) < \infty$ such that the following estimate holds, for each $\psi \in \mathring h^1 (\cu_n),$
\begin{equation*}
\log \int_H \exp \left( - \sum_{e\in B_{m,n}} V_e\left(\nabla \psi (e) + \nabla h(e)\right) \right) \, d h \leq C m 3^{d(n-m)}.
\end{equation*}
\end{proposition}

\begin{proof}
By the assumptions made on the elastic potential $V_e$, one has, for each $x \in \R$,
\begin{equation*}
V_e(x) \geq \lambda x^2.
\end{equation*}
This gives the following estimate, for each $\psi \in \underset{z \in \mathcal{Z}_{m,n}}{\oplus} \mathring h^1 \left(z + \cu_{m}\right)$,
\begin{equation} \label{est.e.H.small}
 - \sum_{e\in B_{m,n}} V_e\left(\nabla \psi (e) + \nabla h(e)\right) \leq - \sum_{e\in B_{m,n}}  \lambda \left(\nabla \psi (e) + \nabla h(e)\right)^2 .
\end{equation}
For the rest of the proof, we introduce the following notation: for $z,z' \in \mathcal{Z}_{m,n}$, we write
\begin{equation*}
z \sim z' \, \mbox{if and only if} \, |z - z'|_1 = 3^m.
\end{equation*}
That is to say, we write $z \sim z'$ if and only if $z$ and $z'$ are neighbors in the rescaled lattice $3^m \Zd$. With this notation, the set $B_{m,n}$ can be partitioned according to
\begin{equation*}
B_{m,n} = \bigcup_{z,z' \in \mathcal{Z}_{m,n}, \, z \sim z'} F_{z,z'},
\end{equation*}
where we introduce the notation
\begin{equation*}
F_{z,z'} := \left\{ e = (x,y) \subseteq \cu_{n+1} \, : \, x \in z + \cu_n ~ \mbox{and}~ y \in z' + \cu_n \right\}.
\end{equation*}
With this notation, the right-hand side of~\eqref{est.e.H.small} can be rewritten
\begin{equation*}
 \sum_{e\in B_{m,n}}  \left(\nabla \psi (e) + \nabla h (e)\right)^2  =  \sum_{z,z' \in \mathcal{Z}_{m,n}, \, z \sim z'} \sum_{e\in F_{z,z'}}  \left(\nabla \psi (e) + \lambda_{z'} - \lambda_z \right)^2.
\end{equation*}
Expanding the square gives, for each $z , z' \in \mathcal{Z}_{m,n}$ satisfying $z \sim z'$,
\begin{align*}
\sum_{e\in F_{z,z'}}  \left(\nabla \psi (e) + \lambda_z' - \lambda_z \right)^2 &= \sum_{e\in F_{z,z'}}  |\nabla \psi (e)|^2 + 2 \nabla \psi (e) \left(\lambda_{z'} - \lambda_z \right) +  |\lambda_{z'} - \lambda_z|^2 \\
										& = |F_{z,z'}| \left( \lambda_{z'} - \lambda_z + \frac1{|F_{z,z'}|}  \sum_{e\in F_{z,z'}} \nabla \psi (e)  \right)^2  \\ & \quad -  \frac1{|F_{z,z'}|}  \left| \sum_{e\in F_{z,z'}} \nabla \psi (e) \right|^2 +  \sum_{e\in F_{z,z'}} |\nabla \psi (e)|^2 .
\end{align*}
But one has
\begin{equation*}
 -  \frac1{|F_{z,z'}|}  \left| \sum_{e\in F_{z,z'}} \nabla \psi (e) \right|^2 + \sum_{e\in F_{z,z'}} |\nabla \psi (e)|^2 \geq 0,
\end{equation*}
thus one obtains
\begin{equation*}
\sum_{e\in F_{z,z'}}  \left(\nabla \psi (e) + \lambda_z' - \lambda_z \right)^2 \geq  |F_{z,z'}| \left( \lambda_{z'} - \lambda_z + \frac1{|F_{z,z'}|}  \sum_{e\in F_{z,z'}} \nabla \psi (e)  \right)^2.
\end{equation*}
Note that the cardinality $|F_{z,z'}|$ is the same for every pair of points $z,z' \in \mathcal{Z}_{m,n}$ such that $z \sim z'$. It is indeed the cardinality of a face of the cube $\cu_m$ and is equal to $3^{(d-1)m}$. This cardinality is denoted by $|F_m|$ in the rest of the proof.

\smallskip

The next step of the proof is to construct an isometry between the spaces $H$ and $\mathring h^1 (\cu_{n-m})$. To do so, we note that the following equality holds
\begin{equation*}
\mathcal{Z}_{m,n} = 3^m \cu_{n-m}.
\end{equation*}
In particular if $z \in \mathcal{Z}_{m,n}$, then $z / 3^m \in \cu_{n-m}$. From this one obtains that the existence of an isometry between the spaces $H$ and $\mathring h^1 \left(\cu_{n-m}\right)$ given by
\begin{equation} \label{def.iso.phi}
\Phi : \left\{ \begin{array}{ccccc}
 H & \to & \mathring h^1 \left(\cu_{n-m}\right) \\
 h : =  \sum_{z \in \mathcal{Z}_{m,n}} \lambda_z \indc_{\{ z  + \cu_{m}\}} & \mapsto & \Phi(h) = 3^{\frac{dm}{2}} \sum_{z \in \mathcal{Z}_{m,n}} \lambda_z  \delta_{z/ 3^m},\\
\end{array} \right.
\end{equation}
where $\delta_z$ is the function defined by $\delta_z\left(z'\right) = 1$ if $z = z'$ and $\delta_z\left(z'\right) = 0$ otherwise. The scalar $ 3^{\frac{dm}{2}} $ is here to ensure that 
$$\sum_{x \in \cu_{n}} h(x)^2 = \sum_{z \in \mathcal{Z}_{m,n}} 3^{dm} |\lambda_z|^2$$
is equal to 
$$ \sum_{x \in \cu_{x \in \cu_{n-m}}}  \Phi(h)(x)^2  = \sum_{z \in \mathcal{Z}_{m,n}}  \left| 3^{\frac{dm}{2}} \lambda_z\right|^2 .$$
We also denote by $X_{ \psi}$ the vector field defined on the edges of $\cu_{n-m}$ by
\begin{equation*}
X_{\psi} \left(\frac z{3^m} , \frac{z'}{3^m} \right) =  \frac1{|F_m|}  \sum_{e\in F_{z,z'}} \nabla \psi (e) .
\end{equation*}
Performing the change of variables by the isometry $\Phi$ shows
\begin{multline*}
\int_H \exp \left(  - \sum_{e\in B_{m,n}}  \lambda \left(\nabla \psi (e) + \nabla h(e)\right)^2 \right) \, d h
\\ \leq
\int_{\mathring h^1 (\cu_{n-m})} \exp \left(  - \sum_{e\subseteq \cu_{n-m}}  \lambda |F_m| \left(3^{-\frac{dm}2} \nabla h(e) + X_{\psi}(e) \right)^2 \right) \, d h.
\end{multline*}
Using the equality $|F_m| = 3^{(d-1)m}$, one obtains
\begin{multline} \label{est.techn25}
\int_H \exp \left(  - \sum_{e\in B_{m,n}}  \lambda \left(\nabla \psi (e) + \nabla h(e)\right)^2 \right) \, d h
\\ \leq
\int_{\mathring h^1 (\cu_{n-m})} \exp \left(  - \sum_{e\subseteq \cu_{m,n}}  \lambda  3^{(d-1)m} \left( 3^{-\frac{dm}2} \nabla h(e) -   X_{\psi}(e) \right)^2 \right) \, d h.
\end{multline}
We denote by $V(\cu_{n-m})$ the space of vector fields of the cube $\cu_{n-m}$ and equip it with the standard $L^2$ scalar product. The idea is then to consider the following orthogonal decomposition
\begin{equation*}
V(\cu_{n-m}) = \nabla \mathring h^1 (\cu_{n-m}) \overset{\perp}{\oplus} \left( \nabla \mathring h^1 (\cu_{n-m}) \right)^\perp,
\end{equation*}
so that the vector field $X_{\psi}$ can be decomposed according to the formula
\begin{equation*}
X_{\psi} = \nabla h_{\psi} + X_{\psi}^\perp,
\end{equation*}
where $h_{\psi} \in \mathring h^1 (\cu_{n-m})$ and $ X_{\psi}^\perp \in  \left( \nabla \mathring h^1 (\cu_{n-m}) \right)^\perp$. Using this decomposition, one has
\begin{align*}
 \sum_{e\subseteq \cu_{n-m}} \left( 3^{-\frac{dm}2}\nabla h(e) -  X_{\psi}(e) \right)^2 & =  \sum_{e\subseteq \cu_{n-m}} \left(  3^{-\frac{dm}2} \nabla h(e) -  \nabla h_{\psi}  \right)^2 + \sum_{e\subseteq \cu_{n-m}} \left(  X_{\psi}^\perp  \right)^2 \\
									& \geq \sum_{e\subseteq \cu_{n-m}} \left(  3^{-\frac{dm}2} \nabla h(e) - \nabla h_{\psi}  \right)^2 .
\end{align*}
Using the previous inequality, the estimate~\eqref{est.techn25} becomes
\begin{multline*}
\int_H \exp \left(  - \sum_{e\in B_{m,n}}  \lambda \left(\nabla \psi (e) + \nabla h(e)\right)^2 \right) \, d h \\
 \leq 
\int_{\mathring h^1 (\cu_{n-m})} \exp \left(  - \sum_{e\subseteq \cu_{n-m}}  \lambda 3^{(d-1)m} \left( 3^{-\frac{dm}2}\nabla h(e) -   \nabla h_{\psi}  \right)^2 \right) \, d h.
\end{multline*}
We then use the translation invariance of the Lebesgue measure to prove
\begin{multline*}
\int_{\mathring h^1 (\cu_{n-m})} \exp \left(  - \sum_{e\subseteq \cu_{n-m}}  \lambda 3^{(d-1)m} \left(  3^{-\frac{dm}2}\nabla h(e) -  \nabla h_{\psi}  \right)^2 \right) \, d h \\ = \int_{\mathring h^1 (\cu_{n-m})} \exp \left(  - \sum_{e\subseteq \cu_{n-m}}  \lambda  3^{-m} \nabla h(e)^2 \right) \, d h  .
\end{multline*}
Combining the two previous displays yields
\begin{multline*}
\int_H \exp \left(  - \sum_{e\in B_{m,n}}  \lambda \left(\nabla \psi (e) + \nabla h(e)\right)^2 \right) \, dh
 \\ \leq 
\int_{\mathring h^1 (\cu_{n-m})} \exp \left(  - \sum_{e\subseteq \cu_{n-m}}  \lambda 3^{-m}\nabla h(e)^2 \right) \, d h.
\end{multline*}
We then perform a change of variables to obtain
\begin{multline*}
\int_H \exp \left(  - \sum_{e\in B_{m,n}}  \lambda \left(\nabla \psi (e) + \nabla h(e)\right)^2 \right) \, d h \\
 \leq 
 \left( \frac{3^m}{\lambda}\right)^{\frac{\left( 3^{d(n-m)} - 1 \right) }{2}}  \int_{\mathring h^1 (\cu_{n-m})} \exp \left(  - \sum_{e\subseteq \cu_{n-m}} \nabla h(e)^2 \right) \, d h,
\end{multline*}
since $\dim \mathring h^1(\cu_{n-m}) = 3^{d(n-m)} -1$. Taking the logarithm and applying Proposition~\ref{p.quadbound}, one deduces
\begin{align*}
\log \int_H \exp \left(  - \sum_{e\in B_{m,n}}  \lambda \left(\nabla \psi (e) + \nabla h(e)\right)^2 \right) \, d h & \leq C m \left( 3^{d(n-m)} - 1 \right) + C |\cu_{n-m}|  \\
																					& \leq C m 3^{d(n-m)}.
\end{align*}
This completes the proof of Proposition~\ref{5dest.}.
\end{proof}

We now turn to the proof of the second technical lemma of the appendix. This lemma and the notation are used in Step 4 of the proof of Proposition~\ref{lemma3.3}.

\begin{proposition} \label{p.lemmaA2}
There exists a constant $C := C(d , \lambda) < \infty$ such that, for each $n \in \N$,
\begin{align} \label{e.lemmaA2}
\lefteqn{\E \left[ \frac{1}{\left| \cu_{2n}\right|} \sum_{z \in \mathcal{Z}_{n,2n}} \sum_{e \subseteq z + \cu_n}  V_e \left(  \nabla \nu^*_n (q) (e) + \nabla \kappa (e) \right) \right]} \qquad & \\ & \leq  \E \left[  \frac{1}{\left| \cu_{2n}\right|} \sum_{z \in \mathcal{Z}_{n,2n}} \sum_{e \subseteq z + \cu_n}  V_e \left( \nabla \psi_z (e) \right)  \right] + C (1 + |q|^2) 3^{-\frac n2} \notag  \\ & \qquad + C \E \left[ \frac{1}{\left| \cu_{2n}\right|} \sum_{z \in \mathcal{Z}_{n,2n}} \sum_{e \subseteq z + \cu_n} \left| \nabla \kappa (e) - \nabla \sigma_z (e) \right|^2 \right]. \notag
\end{align}
\end{proposition}

\begin{proof}
Denote by $\sigma$ the random variable taking values in $\oplus_{z \in \mathcal{Z}_{n,2n}} \mathring h^1 (z + \cu_{n})$ given by
\begin{equation*}
\sigma = \sum_{z \in \mathcal{Z}_{n,2n}} \sigma_z.
\end{equation*}
We also recall that the random variable $\sigma_z$ is defined by the identity
\begin{equation*}
\forall z \in \mathcal{Z}_{n,2n}, ~\forall x \in (z + \cu_n), ~\sigma_z (x) = \psi_z (x) - \nabla \nu^*_n (q)  \cdot x.
\end{equation*}
Let $P$ be the orthogonal projection from $h^1 (\cu_{2n})$ to $ \oplus_{z \in \mathcal{Z}_{n,2n}} \mathring h^1 (z + \cu_{n})$. Note the operator $P$ satisfies the following property
\begin{equation} \label{usefulpropertyofP}
\forall g \in h^1 (\cu_{2n}), \, \forall z \in \mathcal{Z}_{n,2n}, \, \forall e \subseteq (z + \cu_n), \, \nabla Pg(e) = \nabla g(e).
\end{equation}
Denote by $\xi$ the random variable taking values in $\oplus_{z \in \mathcal{Z}_{n,2n}} \mathring h^1 (z + \cu_{n})$, defined according to the formula
\begin{equation*}
\xi := 2 \sigma - P \kappa,
\end{equation*}
so that $\sigma=\frac{\xi + P \kappa}2 $. Using the uniform convexity of the elastic potential $V_e$, one has
\begin{align*}
\lefteqn{2 \E \left[ \frac{1}{\left| \cu_{2n}\right|} \sum_{z \in \mathcal{Z}_{n,2n}} \sum_{e \subseteq z + \cu_n}  V_e \left( \nabla \psi_z (e) \right) \right]} \qquad & \\ & \geq \E \left[ \frac{1}{\left| \cu_{2n}\right|} \sum_{z \in \mathcal{Z}_{n,2n}} \sum_{e \subseteq z + \cu_n}  V_e \left(\nabla \nu^*_n (q) (e)  + \nabla \xi (e) \right) \right] & \\
		 & \quad + \E \left[ \frac{1}{\left| \cu_{2n}\right|} \sum_{z \in \mathcal{Z}_{n,2n}} \sum_{e \subseteq z + \cu_n}  V_e \left( \nabla \nu^*_n (q) (e)  + \nabla P\kappa (e) \right) \right] \\
		& \quad - C \E \left[  \sum_{z \in \mathcal{Z}_{n,2n}} \sum_{e \subseteq z + \cu_n} | \nabla \sigma_z (e)  - \nabla  P \kappa(e)|^2 \right].
\end{align*}
We then use~\eqref{usefulpropertyofP} to get
\begin{align}\label{compenergypsixi}
\lefteqn{2 \E \left[ \frac{1}{\left| \cu_{2n}\right|} \sum_{z \in \mathcal{Z}_{n,2n}} \sum_{e \subseteq z + \cu_n}  V_e \left( \nabla \psi_z (e) \right) \right] } \qquad & \\ & \geq \E \left[ \frac{1}{\left| \cu_{2n}\right|} \sum_{z \in \mathcal{Z}_{n,2n}} \sum_{e \subseteq z + \cu_n}  V_e \left( \nabla \nu^*_n (q) (e)  + \nabla \xi (e) \right) \right]  \notag \\ 
& \quad+ \E \left[ \frac{1}{\left| \cu_{2n}\right|} \sum_{z \in \mathcal{Z}_{n,2n}} \sum_{e \subseteq z + \cu_n}  V_e \left( \nabla \nu^*_n (q) (e)  + \nabla \kappa (e) \right) \right]  \notag \\ 
& \quad - C  \E \left[ \sum_{z \in \mathcal{Z}_{n,2n}} \sum_{e \subseteq z + \cu_n} | \nabla \psi (e) - \nabla  \kappa(e)|^2 \right]. \notag
\end{align} 
Note that the random variable $\sum_{z \in \mathcal{Z}_{n,2n}} \psi_z$ is the minimizer of the problem
\begin{equation*}
\inf_{\P } \, \E \left[ \sum_{z \in \mathcal{Z}_{n,2n}} \sum_{e \subseteq z + \cu_n} \left( V_e(\nabla \psi') - q \cdot \nabla \psi'(e) \right) \right]+ H\left( \P \right),
\end{equation*}
where the infimum is taken over all the probability measures on $ \oplus_{z \in \mathcal{Z}_{n,2n}} \mathring h^1 (z + \cu_{n})$ and $\psi'$ is a random variable of law $\P$. In particular, using the translation invariance of the entropy gives
\begin{multline} \label{entropyxipsior0}
\E \left[ \frac{1}{\left| \cu_{2n}\right|} \sum_{z \in \mathcal{Z}_{n,2n}} \sum_{e \subseteq z + \cu_n}   V_e \left( \nabla \nu^*_n (q) (e) + \nabla \xi (e) \right)  - q \cdot \left( \nabla \nu^*_n (q) + \nabla \xi \right) (e)   \right] + \frac{1}{\left| \cu_{2n}\right|} H \left( \P_{\xi}  \right) \\ 
		\geq  \E \left[ \frac{1}{\left| \cu_{2n}\right|} \sum_{z \in \mathcal{Z}_{n,2n}} \sum_{e \subseteq z + \cu_n}  V_e \left( \nabla \psi_z (e) \right) - q \cdot  \nabla \psi_z (e) \right]  +  \frac{1}{\left| \cu_{n}\right|} H \left(  \P_{n,q}^* \right).
\end{multline} 
We first simplify a slightly the previous display by removing the linear terms in the left and right-hand side. Using that $ \nabla \nu^*_n (q)  = \E \left[ \left\langle \nabla \psi_{n,q} \right\rangle_{\cu_{n}} \right]$, we obtain
\begin{equation*}
 \E \left[  \frac{1}{\left| \cu_{2n}\right|} \sum_{z \in \mathcal{Z}_{n,2n}} \sum_{e \subseteq z + \cu_n}  q \cdot  \nabla \psi_z (e) \right]   = q \cdot \nabla \nu^*_n (q),
\end{equation*}
and, using the definition of $\sigma$,
\begin{multline*}
\E \left[ \frac{1}{\left| \cu_{2n}\right|} \sum_{z \in \mathcal{Z}_{n,2n}} \sum_{e \subseteq z + \cu_n} q \cdot \left( \nabla \nu^*_n (q) + \nabla \xi \right) (e)   \right] = q \cdot \nabla \nu^*_n (q) \\ + \E \left[ \frac{1}{\left| \cu_{2n}\right|} \sum_{z \in \mathcal{Z}_{n,2n}} \sum_{e \subseteq z + \cu_n} q \cdot \nabla P \kappa (e)  \right].
\end{multline*}
We denote by $B_{n,2n}^+$ the set of bonds of the cube $\cu_{2n}^+$ which do not belong to a cube of the form $(z + \cu_n)$, for $z \in \mathcal{Z}_{n,2n}$, i.e.,
\begin{equation*}
B_{n,2n}^+ := \left\{ e \subseteq \cu_{2n}^+ ~:~ \forall z \in \mathcal{Z}_{n,2n}, \, e \not\subseteq z + \cu_n \right\}.
\end{equation*}
Using this set, one can decompose the sum
\begin{equation*}
\sum_{e \subseteq \cu_{2n}^+} = \sum_{z \in \mathcal{Z}_{n,2n}} \sum_{e \subseteq z + \cu_n} + \sum_{e \in B_{n,2n}^+}.
\end{equation*}
Note also that the set $B_{n,2n}^+$ is almost equal to the set $B_{n,2n}$, the only difference is that we added the bonds which belong to $\cu_{2n}^+$ but not $\cu_{2n}$, which is a small boundary layer of bonds. Additionally, one has the estimate on the cardinality of $B_{n,2n}^+$,
\begin{equation*}
\left| B_{n,2n}^+ \right| \leq C 3^{-n} \left| \cu_{2n} \right|.
\end{equation*}
Using the decomposition of the sum, the fact that $\kappa \in h^1_0 \left( \cu_{2n}^+ \right)$ and the property~\eqref{usefulpropertyofP}, one obtains,
\begin{align*}
\E \left[ \frac{1}{\left| \cu_{2n}\right|} \sum_{z \in \mathcal{Z}_{n,2n}} \sum_{e \subseteq z + \cu_n} q \cdot \left( \nabla P \kappa \right) (e)  \right] & = \E \left[ \frac{1}{\left| \cu_{2n}\right|} \sum_{z \in \mathcal{Z}_{n,2n}} \sum_{e \subseteq z + \cu_n}  q \cdot \nabla  \kappa  (e)  \right] \\
																											& = \E \left[ \frac{1}{\left| \cu_{2n}\right|} \sum_{e \in B_{n,2n}^+} q \cdot \nabla  \kappa  (e)  \right].
\end{align*}
We then apply the Cauchy-Schwarz inequality as well as the definition of $\kappa$ given in~\eqref{defofkappa} to obtain
\begin{align*}
\E \left[ \frac{1}{\left| \cu_{2n}\right|} \sum_{e \in B_{n,2n}^+} q \cdot \nabla  \kappa  (e)  \right] & \leq |q| \left( \frac{\left| B_{n,2n}^+ \right|}{\left| \cu_{2n}\right|} \right)^\frac 12 \left( \E \left[ \frac{1}{\left| \cu_{2n}\right|} \sum_{e \in B_{n,2n}^+}  \left|  \nabla  \kappa (e) \right|^2  \right]\right)^\frac 12 \\
									& \leq C |q|  3^{-\frac n2} \left( \E \left[ \frac{1}{\left| \cu_{2n}\right|} \sum_{e \subseteq \cu_{2n}^+}  \left|  \nabla  \kappa (e) \right|^2  \right]\right)^\frac 12 \\
									& \leq C |q|  3^{-\frac n2} \left( \E \left[ \frac{1}{\left| \cu_{2n}\right|} \sum_{e \subseteq \cu_{2n}^+}  \left|  \mathbf{f} (e) \right|^2  \right]\right)^\frac 12 \\
									& \leq C |q| 3^{-\frac n2} ( 1 + |q|).
\end{align*}
But using the definition of $\mathbf{f}$ given in~\eqref{def.mathbff} together with the estimate~\eqref{4.L2boundnustar} of Proposition~\ref{propofnunustar}, one has
\begin{equation*}
\left( \E \left[ \frac{1}{\left| \cu_{2n}\right|} \sum_{e \subseteq \cu_{2n}^+}  \left|  \mathbf{f} (e) \right|^2  \right]\right)^\frac 12 \leq C (1 + |q|).
\end{equation*}
A combination of the previous displays gives
\begin{equation*}
\left| \E \left[ \frac{1}{\left| \cu_{2n}\right|} \sum_{z \in \mathcal{Z}_{n,2n}} \sum_{e \subseteq z + \cu_n} q \cdot \nabla P \kappa (e)  \right] \right| \leq C 3^{-\frac n2} (1 + |q|^2).
\end{equation*}
Combining the previous estimate with the inequality~\eqref{entropyxipsior0}, one obtains the simplified display
\begin{multline} \label{realentropyxipsior0}
\E \left[ \frac{1}{\left| \cu_{2n}\right|} \sum_{z \in \mathcal{Z}_{n,2n}} \sum_{e \subseteq z + \cu_n}   V_e \left( \nabla \nu^*_n (q) (e) + \nabla \xi (e) \right)   \right] + \frac{1}{\left| \cu_{2n}\right|} H \left( \P_{\xi}  \right) \\ 
		\geq  \E \left[ \frac{1}{\left| \cu_{2n}\right|} \sum_{z \in \mathcal{Z}_{n,2n}} \sum_{e \subseteq z + \cu_n}  V_e \left( \nabla \psi_z (e) \right)  \right]  +  \frac{1}{\left| \cu_{2n}\right|} H \left( \P_{\psi}  \right) - C 3^{-\frac n2} (1 + |q|^2).
\end{multline}
We now show the following estimate comparing the entropies of $\P_{\xi}$ and $\P_{\psi}$,
\begin{equation} \label{entropyxipsior}
\frac{1}{\left| \cu_{2n}\right|} H \left( \P_{\xi}\right) \leq  \frac{1}{\left| \cu_{n}\right|} H \left( \P_{n,q}^*  \right) + C 3^{-n}.
\end{equation}
First we recall the definition of the linear operator $L$ given in~\eqref{rigorousprojh10grad} and that the random variable $\kappa$ is defined by 
\begin{equation*}
 \kappa := L \left( \sum_{z \in \mathcal{Z}_{n,2n}} \sigma_z \right).
\end{equation*}
We consequently have
\begin{equation*}
\xi = (2 \mathrm{Id} - P \circ L ) \psi,
\end{equation*}
where $2 \mathrm{Id} - P \circ L$ is seen as a linear operator from $ \oplus_{z \in \mathcal{Z}_{n,2n}} \mathring h^1 (z + \cu_{n})$ into itself. Using the change of variables formula for the entropy, one obtains
\begin{equation} \label{entropyxipsi}
\frac{1}{\left| \cu_{2n}\right|} H \left( \P_{\xi} \right) =  \frac{1}{\left| \cu_{n}\right|} H \left(  \P_{n,q}^*\right)  - \ln | \det (2 \mathrm{Id} - P \circ L)|.
\end{equation}
Since the dimension of the vector space $ \oplus_{z \in \mathcal{Z}_{n,2n}} \mathring h^1 (z + \cu_{n})$ is $3^{2dn} - 3^{dn}$, we denote by $l_1, \ldots , l_{3^{2dn} - 3^{dn}}$ the eigenvalues (potentially complex and with repetition) of the operator $P \circ L$. We thus have
\begin{equation*}
 \ln | \det (2 \mathrm{Id} - P \circ L)| = \sum_{i= 0}^{3^{2dn} - 3^{dn}} \ln | 2 - l_i|.
\end{equation*}
We now prove the two following statements on $l_i$:
\begin{enumerate}
\item for each $i \in \{ 1 , \ldots,  3^{2dn} - 3^{dn}\}$, $|l_i| \leq 1$;
\item there exists a constant $C := C(d) < \infty$ such that at least $ 3^{2dn} -  C 3^{dn}$ eigenvalues $l_i$ are equal to $1$.
\end{enumerate}
To prove the first fact, note that, by~\eqref{usefulpropertyofP}, for each $\psi \in \oplus_{z \in \mathcal{Z}_{n,2n}} \mathring h^1 (z + \cu_{n})$,
\begin{equation*}
\sum_{z \in \mathcal{Z}_{n,2n}} \sum_{e \subseteq z + \cu_n} \left| \nabla P \circ L (\psi)(e) \right|^2 = \sum_{z \in \mathcal{Z}_{n,2n}} \sum_{e \subseteq z + \cu_n} \left| \nabla L (\psi)(e) \right|^2.
\end{equation*}
Moreover by~\eqref{Lgradientcontract}, one has
\begin{equation*}
\sum_{e \subseteq \cu_{2n}^+} \left| \nabla L(\psi)(e) \right|^2 \leq \sum_{z \in \mathcal{Z}_{n,2n}} \sum_{e \subseteq z + \cu_n} |\nabla \psi (e)|^2.
\end{equation*}
Since one clearly has
\begin{equation*}
\sum_{z \in \mathcal{Z}_{n,2n}} \sum_{e \subseteq z + \cu_n} \left| \nabla L (\psi)(e) \right|^2 \leq \sum_{e \subseteq \cu_{2n}^+} \left| \nabla L(\psi)(e) \right|^2,
\end{equation*}
one obtains
\begin{equation*}
\sum_{z \in \mathcal{Z}_{n,2n}} \sum_{e \subseteq z + \cu_n} \left| \nabla P \circ L (\psi)(e) \right|^2 \leq  \sum_{z \in \mathcal{Z}_{n,2n}} \sum_{e \subseteq z + \cu_n} |\nabla \psi (e)|^2.
\end{equation*}
We consider an eigenvalue $l_i$ of $P \circ L$ and an eigenvector $\psi_i$ (which may be complex) associated to this eigenvalue. Then one has
\begin{equation*}
|l_i|^2 \sum_{z \in \mathcal{Z}_{n,2n}} \sum_{e \subseteq z + \cu_n} \left| \nabla \psi_i (e) \right|^2 \leq  \sum_{z \in \mathcal{Z}_{n,2n}} \sum_{e \subseteq z + \cu_n} |\nabla \psi_i (e)|^2,
\end{equation*}
which implies $|l_i| \leq 1$ as soon as $\sum_{z \in \mathcal{Z}_{n,2n}} \sum_{e \subseteq z + \cu_n} |\nabla \psi_i (e)|^2$ is not equal $0$, but since $\psi_i \in \oplus_{z \in \mathcal{Z}_{n,2n}} \mathring h^1 (z + \cu_{n})$, one also has
\begin{equation*}
\sum_{z \in \mathcal{Z}_{n,2n}} \sum_{e \subseteq z + \cu_n} |\nabla \psi_i (e)|^2 = 0 \Longleftrightarrow \psi_i =0.
\end{equation*}
This completes the proof of the first item.
\smallskip

We now prove the second item. To this end we proceed exactly as in Step 1 of Proposition~\ref{lemma3.3}, where it is proved proved that if one considers the interior $\cu_n^{\mathrm{o}}$,
\begin{equation*}
\cu_n^{\mathrm{o}} := \left( - \frac{3^n - 2}{2} , \frac{3^n - 2}{2} \right)^d \cap \Zd = \cu_n \setminus \partial \cu_n,
\end{equation*}
then for each $\psi \in \oplus_{z \in \mathcal{Z}_{n,2n}} \mathring h^1 \left(z + \cu_{n}^{\mathrm{o}} \right)$ one has
\begin{equation*}
L(\psi) = \psi.
\end{equation*}
Moreover, from~\eqref{estimatedimensioneigenspace1} one has the estimate on the dimension of $\oplus_{z \in \mathcal{Z}_{n,2n}} \mathring h^1 \left(z + \cu_{n}^- \right)$,
\begin{equation*}
\dim \left(  \underset{z \in \mathcal{Z}_{n,2n}}{\oplus} \mathring h^1 \left(z + \cu_n^-\right) \right) \geq 3^{2dn} - C 3^{(2d -1)n}.
\end{equation*}
This implies that among the $l_i$, at least $3^{2dn} - C 3^{(2d -1)n}$ of them are equal to $1$. Without loss of generality, we can thus assume that for each $i \in \{1 , \ldots , 3^{2dn} - C 3^{(2d -1)n} \}$, $l_i =1$.

\smallskip

Combining (1) and (2), we obtain
\begin{align*}
 \frac{1}{|\cu_{2n}|} \left| \ln | \det (2 \mathrm{Id} - P \circ L)| \right| & \leq \sum_{i= 0}^{3^{2dn} - 3^{dn}} \left| \ln | 2 - l_i| \right| \\
											& \leq  \sum_{i= 3^{2dn} - C 3^{(2d -1)n}}^{3^{2dn} - 3^{dn}}  \left| \ln | 2 - l_i| \right|\\
											& \leq C  3^{(2d -1)n}.
\end{align*}
Thus
\begin{equation*}
 \frac{1}{|\cu_{2n}|} \left| \ln | \det (2 \mathrm{Id} - P \circ L)| \right| \leq C 3^{-n}.
\end{equation*}
Combining this estimate with~\eqref{entropyxipsi} gives
\begin{equation*}
\frac{1}{\left| \cu_{2n}\right|} H \left( \P_{\xi}  \right) \leq   \frac{1}{\left| \cu_{n}\right|} H \left(  \P_{n,q}^* \right)  + C3^{-n}.
\end{equation*}
This is precisely~\eqref{entropyxipsior}.

\medskip

\noindent We then combine~\eqref{realentropyxipsior0} and~\eqref{entropyxipsior} to obtain
 \begin{multline*}
\E \left[ \frac{1}{\left| \cu_{2n}\right|} \sum_{z \in \mathcal{Z}_{n,2n}} \sum_{e \subseteq z + \cu_n}  V_e \left( \nabla \xi (e) \right) \right]
		\geq  \E \left[ \frac{1}{\left| \cu_{2n}\right|} \sum_{z \in \mathcal{Z}_{n,2n}} \sum_{e \subseteq z + \cu_n}  V_e \left( \nabla \psi_z (e) \right) \right] \\ -C (1 + |q|^2) 3^{-\frac n2}.
\end{multline*}
Using this inequality together with~\eqref{compenergypsixi} gives
\begin{multline*}
 \E \left[ \frac{1}{\left| \cu_{2n}\right|} \sum_{z \in \mathcal{Z}_{n,2n}} \sum_{e \subseteq z + \cu_n}  V_e \left( \nabla \psi (e) \right) \right]  \geq    \E \left[ \frac{1}{\left| \cu_{2n}\right|} \sum_{z \in \mathcal{Z}_{n,2n}} \sum_{e \subseteq z + \cu_n}  V_e \left( \nabla \kappa (e) \right) \right] \\  - C (1 + |q|^2) 3^{-\frac n2}   - C  \E \left[ \sum_{z \in \mathcal{Z}_{n,2n}} \sum_{e \subseteq z + \cu_n} | \nabla \psi (e) - \nabla  \kappa(e)|^2 \right].
\end{multline*}
This is~\eqref{e.lemmaA2} and thus the proof of Lemma~\ref{p.lemmaA2} is complete.
\end{proof}

We then prove the last lemma of this appendix. It gives a quadratic upper bound for the value of $\nu (U,p)$ for any bounded domain $U \subseteq \Zd$.

\begin{proposition} \label{upperboundnugeneralsubset}
There exists a constant $C := C(d,\lambda) < \infty$ such that for each bounded domain $U \subseteq \Zd$ and each $p \in \Rd$,
\begin{equation*}
\nu \left( U , p \right) \leq C (1 + |p|^2 ).
\end{equation*}
\end{proposition}

\begin{remark}
This statement is a more general version of the upper bound for $\nu$ than the one given in Proposition~\ref{p.quadbound}, since it is valid for any bounded domain  $U \subseteq \Zd$, nevertheless the argument presented here does not give a lower bound as the one we computed in the case of cubes. It also does not give bounds on $\nu^*$, this is why this statement is presented in the appendix.
\end{remark}

\begin{proof}
Consider a random variable $X$, taking values in $h^1_0 (U)$ whose law is defined by
\begin{itemize}
\item for each $x \in U$, the law of $X(x)$ is uniform in $[0,1]$
\item the random variables $X(x)$, for $x \in U$ are independent.
\end{itemize}
Using that the entropy of the law uniform in $[0,1]$ is equal to $0$ together with Proposition~\ref{entropycouplesum}, one obtains
\begin{equation*}
H \left( \P_X \right) = 0.
\end{equation*}
Then by Proposition~\ref{varfornuprop}, one has the following computation
\begin{align*}
\nu (U , p) & \leq \E \left[ \frac{1}{|U|} \sum_{e \subseteq U} V_e \left( p(e) + \nabla X(e) \right) \right] +  \frac{1}{|U|} H \left( \P_X \right) \\
		& \leq \E \left[ \frac{1}{|U|} \sum_{e \subseteq U} V_e \left( p(e) + \nabla X(e) \right) \right] .
\end{align*}
We then use the bound $V_e (x) \leq \frac 1\lambda |x|^2$ combined with the estimate $\left|\nabla X(e)\right| \leq 1$ for each bond $e \subseteq U$ to obtain
\begin{equation*}
\E \left[ \frac{1}{|U|} \sum_{e \subseteq U} V_e \left( p(e) + \nabla X(e) \right) \right]  \leq C \left( 1 + |p|^2 \right).
\end{equation*}
A combination of the two previous displays completes the proof of the proposition.
\end{proof}

\section{Functional inequalities} \label{appendixB}

The goal of this second appendix is to prove some classic inequalities from the theory of elliptic equations in the setting of the $\nabla \phi$ model. These inequalities are proved with the random variable $\psi_{n,q}$ associated to the law $\P^*_{n,q}$ because it is needed in the proof of Theorem~\ref{conhvergencetoaffine}, nevertheless similar statements, with similar proofs, are available for the random variable $\phi_{n,p}$ associated to the law $\P_{n,p}$.

\begin{proposition}[Interior Caccioppoli inequality] \label{e.caccioppoli}
There exists a constant $C := C(d , \lambda) < \infty$ such that for every integer $n \geq 1$, every $x \in \cu_n$ and every $r \geq 1$ such that $B(x , 2r ) \subseteq \cu_n$,
\begin{equation*}
\E \left[ \sum_{e \subseteq B(x,r)} \left| \nabla \psi_{n,q}(e) \right|^2 \right] \leq  \frac C{r^2}\E \left[ \sum_{x \in B(x , 2r)} \left|  \psi_{n,q}(y) - (\psi_{n,q})_{B(x , 2r)}\right|^2   \right] + C r^d.
\end{equation*}
\end{proposition}

\begin{proof}
Let $\eta$ be a cutoff function defined on the discrete lattice $\cu_n$, taking values in $\R$ and satisfying
\begin{equation*}
\indc_{B(x, r)} \leq \eta \leq \indc_{B(x , 2r)} ~\mbox{and} ~ \forall e = (x,y) \subseteq \cu_n, \, |\nabla \eta(e) |^2 \leq C r^{-2} \left( \eta(x) + \eta(y)\right).
\end{equation*}
For $t \geq 0$, we denote by $L_t$ the following linear operator
\begin{equation*}
L_t := \left\{ \begin{array}{lcl}
     \mathring h^1 (\cu_n) &\rightarrow &\mathring h^1 (\cu_n) \\
     \psi &\mapsto & \psi +t\eta  \left(\psi - ( \psi )_{B(x , 2r)} \right) - \left( \psi +t\eta  \left( \psi - \left( \psi \right)_{B(x, 2r)} \right) \right)_{\cu_n} .
  \end{array} \right.
\end{equation*}
As a remark, we note that the last term on the right-hand side can be rewritten
\begin{equation*}
 \left(\psi +t\eta  \left( \psi - \left( \psi \right)_{B(x, 2r)} \right) \right)_{\cu_n} = t  \left( \eta  \left( \psi - \left( \psi \right)_{B(x, 2r)} \right) \right)_{\cu_n},
\end{equation*}
since the function $\psi$ belongs to the space $\mathring h^1 (\cu_n)$. We note that $L_0$ is the identity of $h^1 (\cu_n)$. We now show the following inequality which estimates the distance between $L_t$ and the identity of $\mathring h^1 (\cu_n)$, in the $L^2$ operator norm:
\begin{equation*}
\forall \psi \in \mathring h^1 (\cu_n), \, \sum_{x \in \cu_n} \left| \psi(x) - L_t(\psi)(x) \right|^2 \leq  |t|^2\sum_{x \in \cu_n}  \left|\psi(x)  \right|^2.
\end{equation*}
This is a consequence of the computation
\begin{align*}
 \sum_{x \in \cu_n} \left| \psi(x) - L_t(\psi)(x) \right|^2 & \leq  \sum_{x \in \cu_n} \left| t\eta(x)  \left( \psi(x) - ( \psi )_{B(x , 2r)} \right) \right|^2 \\
					&  \leq |t|^2 \sum_{x \in B(x,2r)} \left|  \psi(x) - ( \psi )_{B(x , 2r)} \right|^2 \\
					& \leq |t|^2\sum_{x \in B(x,2r)} \left|\psi(x) \right|^2 .
\end{align*} 
This implies in particular that for each $t \in (-1,1)$, the operator $L_t$ is bijective.
We also note that by definition of the operator $L_t$,
\begin{equation} \label{useproj}
\forall  \psi \in \mathring h^1 (\cu_n), \,  \forall e \subseteq \cu_n, \, \nabla L_t( \psi) (e) = \nabla \psi(e) +t \nabla \left(\eta  \left( \psi - ( \psi )_{B(x , 2r)} \right) \right) (e) .
\end{equation}
We fix a vector $q \in \R$ and use the random variable $L_t\left( \psi_{n,q} \right)$ as a test random variable in the variational formulation for $\nu^*$ stated in Proposition~\ref{varfornuprop}. This yields
\begin{multline*}
\E \left[ - \sum_{e \subseteq \cu_n} \left( V_e \left( \nabla \psi_{n,q} (e) \right)  - q \cdot \nabla \psi_{n,q}(e) \right) \right] - H(\P_{n , q}^*) \\ \geq \E \left[ - \sum_{e \subseteq \cu_n} \left(V_e \left( \nabla L_t \left(\psi_{n,q}\right) (e) \right) -  q \cdot \nabla L_t \left(\psi_{n,q}\right) (e)  \right) \right] - H\left(\P_{L_t (\psi_{n,q})} \right).
\end{multline*}
First note that, since $\eta$ is supported in $B(x,2r) \subseteq \cu_n$,
\begin{equation*}
\left\langle \nabla L_t \left(\psi_{n,q}\right) \right\rangle_{\cu_n} = \left\langle \nabla \psi_{n,q} \right\rangle_{\cu_n},
\end{equation*}
consequently,
\begin{equation*}
\E \left[  \sum_{e \subseteq \cu_n} q \cdot \nabla \psi_{n,q}(e)  \right] =  \E \left[ \sum_{e \subseteq \cu_n}   q \cdot \nabla L_t \left(\psi_{n,q}\right) (e)  \right].
\end{equation*}
Thus one can simplify the previous display
\begin{equation*}
\E \left[ - \sum_{e \subseteq \cu_n} V_e \left( \nabla \psi_{n,q} (e) \right)  \right] - H(\P_{n , q}^*) \\ \geq \E \left[ - \sum_{e \subseteq \cu_n} V_e \left( \nabla L_t \left(\psi_{n,q}\right) (e) \right) \right] - H\left(\P_{L_t (\psi_{n,q})} \right).
\end{equation*}
By Proposition~\ref{transandchofvar}, we compute the entropy
\begin{equation*}
H\left(\P_{L_t (\psi_{n,q})} \right) = H(\P_{n , q}^*) - \ln \det L_t.
\end{equation*}
Using the previous display and the formula for $L_t$, one obtains, for each $t \in (-1 , 1)$,
\begin{equation*}
\E \left[ \sum_{e \subseteq \cu_n} V_e \left(  \nabla \psi_{n,q}(e) +t \nabla \left(\eta  \left( \psi_{n,q} - ( \psi_{n,q} )_{B(x , 2r)} \right) \right) (e) \right) -  V_e \left( \nabla \psi_{n,q} (e) \right)  \right] -  \ln \det L_t \geq 0.
\end{equation*}
It is clear that the function $t \rightarrow \ln \det L_t$ is smooth for $t \in (-1,1)$. In particular, dividing the previous display by $t$ and sending $t$ to $0$ gives
\begin{equation} \label{almost.caccioppolinustar}
\E \left[ \sum_{e \subseteq \cu_n} V_e' \left( \nabla \psi_{n,q} )(e) \right)  \nabla \left(\eta  \left( \psi_{n,q} - ( \psi_{n,q} )_{B(x , 2r)} \right) \right) (e) \right] -  \frac{d}{d t}_{ | t = 0} \ln \det L_t  = 0.
\end{equation}
We first deal with the term coming from the entropy. By the chain rule, one has the formula
\begin{equation*}
\frac{d}{d t}_{ | t = 0} \ln \det L_t = \tr L_0',
\end{equation*}
where $L_0'$ denote the derivative of the operator $L_t$ at $t=0$, it is given by the explicit formula
\begin{equation*}
L_0' := \left\{ \begin{array}{lcl}
     \mathring h^1 (\cu_n) &\rightarrow &\mathring h^1 (\cu_n) \\
     \psi &\mapsto & \eta  \left( \psi - ( \psi )_{B(x , 2r)} \right) - \left( \eta  \left( \psi - \left( \psi \right)_{B(x, 2r)} \right) \right)_{\cu_n} .
  \end{array} \right.
\end{equation*}
In particular, for each $\psi \in \mathring h^1 (\cu_n)$,
\begin{align*}
\sum_{x \in \cu_n} \left| L_0' (\psi)(x) \right|^2 & \leq \sum_{x \in \cu_n} \left| \eta(x) \left( \psi(x) - ( \psi )_{B(x , 2r)} \right) \right|^2 \\
										& \leq  \sum_{x \in B(x,2r)}  \left| \psi(x) - ( \psi )_{B(x , 2r)} \right|^2 \\
										& \leq \sum_{x \in B(x,2r)}  \left| \psi(x) \right|^2.
\end{align*}
This implies that every function $\psi$ supported in $ \cu_n \setminus B(x,2r)$ is in the kernel of $L_0'$ and thus one has
\begin{equation*}
\dim \ker L_0' \geq |\cu_n| - C r^d.
\end{equation*}
Combining the two previous displays shows
\begin{equation} \label{est.trace.Lo'}
\left| \tr L_0' \right| \leq \dim \mathring h^1 (\cu_n ) -  \dim \ker L_0'  \leq C r^d.
\end{equation}
We now turn to the first term on the right-hand side of~\eqref{almost.caccioppolinustar}. To simplify the notation in the following computation, we set
\begin{equation*}
\chi_{n,q} := \psi_{n,q} - ( \psi_{n,q} )_{B(x , 2r)}
\end{equation*}
and compute
\begin{align*}
\lefteqn{ \sum_{e \subseteq B(x,2r)} V_e' \left( \nabla \chi_{n,q} )(e) \right)  \nabla (\eta \chi_{n,q})(e) }\qquad & \\&
 = \sum_{x,y\in B(x,2r), x\sim y} \left( \eta(x)\chi_{n,q}(x) - \eta(y)\chi_{n,q}(y) \right) V_{(x,y)}' \left( \chi_{n,q}(x) - \chi_{n,q}(x) \right) \\
& = \sum_{x,y\in B(x,2r), x\sim y} \eta(x) \left(\chi_{n,q}(x) - \chi_{n,q}(y) \right)  V_{(x,y)}' \left( \chi_{n,q}(x) -\chi_{n,q}(x) \right) \\
& \qquad + \sum_{x,y\in B(x,2r), x\sim y}\chi_{n,q}(y) \left( \eta(x) - \eta(y) \right)  V_{(x,y)}' \left(\chi_{n,q}(x) -\chi_{n,q}(x) \right).
\end{align*}
Using the uniform convexity of $V_e$ and~\eqref{almost.caccioppolinustar}, one obtains, by taking the expectation,
\begin{align*}
\lefteqn{
\lambda \E \left[ \sum_{x,y\in B(x,2r), x\sim y} \eta(x) \left(\chi_{n,q}(x) -\chi_{n,q}(y) \right)^2  \right]
} \qquad & \\
& \leq  \E \left[ \sum_{x,y\in B(x,2r), x\sim y} \eta(x) \left( \chi_{n,q}(x) -\chi_{n,q}(y) \right)  V_{(x,y)}' \left( \chi_{n,q}(x) - \chi_{n,q}(x) \right)  \right] \\
& \leq   \E \left[ \sum_{x,y\in B(x,2r), x\sim y}  |\chi_{n,q}(y) | \left| \eta(x) - \eta(y) \right|  \left|V_{(x,y)}' \left( \chi_{n,q}(x) - \chi_{n,q}(x) \right)\right| \right] + |\tr L_0'|.
\end{align*}
We then use the bound $V_e'( x) \leq \lambda  |x|$. This yields
\begin{align*}
\lefteqn{
\lambda  \E \left[ \sum_{x,y\in B(x,2r), x\sim y} \eta(x) \left(\chi_{n,q}(x) - \chi_{n,q}(y) \right)^2 \right]
} \qquad & \\
& \leq C  \E \left[ \sum_{x,y\in B(x,2r), x\sim y} \frac{\left|\eta(x) - \eta(y)\right|^2}{\eta(x)+ \eta(y)} \left|  \chi_{n,q}(y)\right|^2 \right] \\
& \qquad  +  \E \left[ \frac\lambda4 \sum_{x,y\in B(x,2r), x\sim y} \left( \eta(x) + \eta(y) \right) \left( \chi_{n,q}(x) -  \chi_{n,q}(y) \right)^2  \right]+  |\tr L_0'| \\
& \leq 3^{-2n} \E \left[  \sum_{x,y\in B(x,2r), x\sim y}  \left|  \chi_{n,q}(y)\right|^2   \right]+ \frac\lambda2  \E \left[ \sum_{x,y\in U, x\sim y}  \eta(x) \left|  \chi_{n,q}(x) - \chi_{n,q}(y)\right|^2 \right] \\ & \qquad+ |\tr L_0'|.
\end{align*}
Absorbing the second term on the right back into the left-hand side and using the estimate~\eqref{est.trace.Lo'} gives,
\begin{equation*}
\E \left[ \sum_{x,y\in B(x,2r), x\sim y} \eta(x) \left(\chi_{n,q}(x) - \chi_{n,q}(y) \right)^2 \right] \leq C r^{-2} \E \left[  \sum_{x}  \left|  \chi_{n,q}(x)\right|^2   \right] + C r^d.
\end{equation*}
Now we replace the term $\chi_{n,q}$ by the expression $\psi_{n,q} - ( \psi_{n,q} )_{B(x , 2r)}$ to obtain
\begin{equation*}
\E \left[ \sum_{e \subseteq B(x,r)} \left| \nabla \psi_{n,q}(e) \right|^2 \right] \leq C r^{-2} \E \left[  \sum_{e \subseteq B(x,2r)}  \left|  \psi_{n,q}(e) - ( \psi_{n,q} )_{B(x , 2r)} \right|^2   \right] + C r^d.
\end{equation*}
This is the desired result.
\end{proof}

The next statement one wishes to obtain is a reverse H\"older inequality for the random variable $\psi_{n,q}$. It is obtained by combining the Caccioppoli inequality proved in the previous proposition with the Sobolev inequality recalled below.

\begin{proposition}[Sobolev inequality on $\Zd$] \label{discretesobolev}
There exists a constant $C := C(d ) < \infty$ such that for each $x \in \Zd$, each $r\geq 1$, each exponent $s \in \left(\frac d{d-1}, \infty \right)$ and each function $f : B(x, r) \rightarrow \R$ satisfying
\begin{equation*}
\sum_{x \in B(x,r)} f(x) =0,
\end{equation*}
one has the estimate 
\begin{equation*}
\left( \sum_{x \in B(x,r)} |f(x)|^s \right)^{\frac 1s}\leq C \left( \sum_{e \subseteq B(x,r)} |\nabla f(e)|^{s_\star} \right)^{\frac 1{s_\star}},
\end{equation*}
where $s_\star$ is the Sobolev conjugate defined from $s$ by the formula
\begin{equation*}
s_\star := \frac{sd}{s+d}.
\end{equation*}
\end{proposition}

This inequality can be deduced from the continuous Sobolev inequality (on $\Rd$) by an interpolation argument. From the Sobolev inequality and the Cacioppoli inequality, we deduce the following reverse H\"older inequality.
\begin{proposition}[Reverse H\"older inequality for $\P_{n,q}^*$] \label{revholdPnqstar}
There exists a constant $C := C(d , \lambda) < \infty$ such that for every integer $n \geq 1$, every $x \in \cu_n$, every $r \geq 1$ such that $B(x , 2r ) \subseteq \cu_n$, and every $q \in \Rd$,
\begin{equation*}
\E \left[ \frac{1}{|B(x,r)|} \sum_{e \subseteq B(x,r)} \left| \nabla \psi_{n,q}(e) \right|^2 \right] \leq  C\left( \frac{1}{|B(x,2r)|} \sum_{e \subseteq B(x,2r)} \E \left[ \left| \nabla \psi_{n,q}(e) \right|^2  \right]^\frac{d}{d+2} \right)^{\frac{d+2}{d}} + C.
\end{equation*}
\end{proposition}

\begin{proof}
Fix $q \in \Rd$,  an integer $n \geq 1$, and a ball $B(x,r)$ with $x \in \cu_n$ such that $B(x , 2r) \subseteq \cu_n$. By Proposition~\ref{e.caccioppoli}, one has the inequality
\begin{equation*}
\E \left[ \sum_{e \subseteq B(x,r)} \left| \nabla \psi_{n,q}(e) \right|^2 \right] \leq  \frac C{r^2}\E \left[ \sum_{x \in B(x , 2r)} \left|  \psi_{n,q}(y) - (\psi_{n,q})_{B(x , 2r)}\right|^2   \right] + C r^d.
\end{equation*}
We then apply Proposition~\ref{discretesobolev} with $s = 2$ and $s_\star = \frac{2d}{d+2}$ and obtain
\begin{equation} \label{almost.rev.hold}
\E \left[ \sum_{e \subseteq B(x,r)} \left| \nabla \psi_{n,q}(e) \right|^2 \right]  \leq \frac C{r^2} \left( \sum_{e \subseteq B(x,r)} \E \left[ \left| \nabla \psi_{n,q}(e) \right|^2  \right]^\frac{d}{d+2} \right)^{\frac{d+2}{d}}+ C r^d.
\end{equation}
Then for each $N \in \N$, we introduce the following notation for the half space
\begin{equation*}
\R^N_+ := \left\{ (x_1 , \ldots, x_N ) \in  \R^N \, : \, x_1 \geq 0, \ldots, x_N \geq 0 \right\}.
\end{equation*}
Note that the mapping
\begin{equation*}
F : = \left\{  \begin{array}{lcl}
    \R^N_+ &\rightarrow & \R \\
     (x_1 , \cdots, x_N) &\mapsto & \left( \sum_{e \subseteq B(x,r)} \left| x_i \right|^\frac{d}{d+2} \right)^{\frac{d+2}{d}}
  \end{array} \right.
\end{equation*}
is concave. We then let $N$ be the number of bonds of the cube $\cu_n$ and apply Jensen's inequality to the random variable $\left( \left| \nabla \psi_{n,q}(e) \right|^2 \right)_{e \subseteq \R^N}$, which is valued in $\R^N_+$, to obtain
\begin{equation*}
 \E \left[\left( \sum_{e \subseteq B(x,r)} \left| \nabla \psi_{n,q}(e) \right|^\frac{2d}{d+2} \right)^{\frac{d+2}{d}} \right] \leq \left( \sum_{e \subseteq B(x,r)} \E \left[ \left| \nabla \psi_{n,q}(e) \right|^2  \right]^\frac{d}{d+2} \right)^{\frac{d+2}{d}}.
\end{equation*}
Combining this estimate with~\eqref{almost.rev.hold} and dividing by $r^d$ completes the proof of~Proposition~\ref{revholdPnqstar}.
\end{proof}

We then combine the previous estimate with the discrete version of the Gehring's Lemma, which is stated in the following proposition. The continuous version of this result can be found in~\cite{Giu}.

\begin{proposition}[Discrete Gehring's Lemma] \label{discGehrLemma}
Fix $q < 1$, $K \geq 1$ and $R>0$. Suppose that we are given two (discrete) functions $f,g : B(0,R) \rightarrow \R$, and that $f$ satisfies the following reverse H\"older inequality, for each $z \in \Zd$ and each $r \geq 1$ such that $B(z,2r) \subseteq B(0,R)$,
\begin{equation*}
\frac{1}{|B(x,r)|} \sum_{x \in B(x ,r)} |f(x)| \leq K \left(\frac{1}{|B(x,2r)|} \sum_{x \in B(x ,2r)} |f(x)|^q \right)^\frac1q+\frac{K}{|B(x,2r)|} \sum_{x \in B(x ,2r)} |g(x)|,
\end{equation*}
then there exist an exponent $\delta := \delta(q , K ,d) >0$ and a constant $C := C(q , K ,d ) < \infty $ such that 
\begin{multline*}
\left( \frac{1}{\left|B\left(x,\frac R2\right)\right|} \sum_{x \in B\left(x,\frac R2\right)} |f(x)|^{1+\delta} \right)^\frac1{1 + \delta} \leq C \left(\frac{1}{|B(x,R)|} \sum_{x \in B(x ,R)} |f(x)|  \right) \\ +  C\left(\frac{1}{|B(x,R)|} \sum_{x \in B(x ,R)} |g(x)|^{1+\delta} \right)^\frac1{1 + \delta} .
\end{multline*}
\end{proposition}

From this estimate, one obtains the following version of the interior Meyers estimate for the $\nabla \phi$ model. The idea is to combine the reverse H\"older inequality and the Gehring's Lemma to improve the integrability of the expectation of the field $\psi_{n,q}$ (seen as a function from the triadic cube $\cu_n$ to $\R$) from $L^2$ to $L^{2+\delta}$.
\begin{proposition}[Interior Meyers estimate for $\P_{n,q}^*$] \label{meyersgradphi}
For each $\gamma \in (0,1 ]$ and each $n \in \N$, denote by~$\gamma \cu_n$ the cube
\begin{equation*}
\gamma \cu_n := \left( - \frac{\gamma3^n}{2},  \frac{\gamma 3^n}{2}  \right)^d \cap \Zd.
\end{equation*}
Fix $q \in \Rd$. For each $\gamma \in (0,1)$, each $n \in \N$, there exist an exponent $\delta := \delta (d , \lambda ) > 0$ and a constant $C := C(d , \lambda, \gamma) < \infty$ such that
\begin{equation*}
\left( \frac{1}{\left|\gamma \cu_n \right|} \sum_{e \subseteq  \gamma \cu_n} \E \left[|\nabla \psi_{n,q}(e)|^2\right]^{1+ \delta} \right)^\frac1{1+\delta} \leq  \frac{C}{|\cu_n|} \sum_{e \subseteq \cu_n} \E \left[|\nabla \psi_{n,q}(e)|^2\right]  + C.
\end{equation*}
\end{proposition}

\begin{proof}
The main idea of the proof is to apply the Gehring's Lemma, Proposition~\ref{discGehrLemma}, with the following choice of functions
\begin{equation*}
\forall x \in \cu_n, \, f(x) := \E \left[  \sum_{y \in \cu_n, x \sim y} \left| \psi_{n,q} (y) \right|^2 \right]~ \mbox{and}~ g(x) =1.
\end{equation*}
By Proposition~\ref{revholdPnqstar}, one has the following reverse H\"older inequality: there exists a constant $C := C(d, \lambda) < \infty$ such that for each $z \in \Zd$ and each $r \geq 1$ satisfying $B(z,2r) \subseteq \cu_n$,
\begin{multline*}
 \frac{1}{|B(x,r)|} \sum_{x \in B(x ,r)} |f(x)| \\ \leq C \left(\frac{1}{|B(x,2r)|} \sum_{x \in B(x ,2r)} |f(x)|^\frac{d+2}{d} \right)^\frac{d}{d+2} +\frac{C}{|B(x,2r)|} \sum_{x \in B(x ,2r)} |g(x)|.
\end{multline*}
Applying Proposition~\ref{discGehrLemma}, there exist an exponent $\delta := \delta(d , \lambda ) > 0$ and a constant $C := C(d , \lambda) < \infty$, such that for each point $z \in \cu_n$, and each radius $R \geq 1$ satisfying $B(x, 2R) \subseteq \cu_n$,
\begin{multline*}
\left( \frac{1}{\left|B\left(x, R\right)\right|} \sum_{x \in B\left(x , R \right)} |f(x)|^{1+ \delta} \right)^\frac1{1+\delta} \leq \frac{C}{|B(x,2R)|} \sum_{x \in B(x ,2R)} |f(x)| \\+ C \left(  \left(\frac{1}{|B(x,2R)|} \sum_{x \in B(x ,2R)} |g(x)|^{1+ \delta} \right)^\frac1{1+ \delta} \right).
\end{multline*}
The previous display can be rewritten
\begin{equation*}
\left( \frac{1}{\left|B\left(x, R\right)\right|} \sum_{e \subseteq  B\left(x , R \right)} \E [|\nabla \psi_{n,q}(x)|^2]^{1+ \delta} \right)^\frac1{1+\delta} \leq  \frac{C}{|B(x,2R)|} \sum_{x \in B(x ,2R)} \E [|\nabla \psi_{n,q}(x)|^2]  + C.
\end{equation*}
We then conclude that, for each $\gamma \in [0,1)$, the cube $\gamma \cu_{n}$ can be covered by finitely many balls of the form $B(x , R)$ such that $B(x , 2R) $ is included in $\cu_n$. The cardinality of this covering family can be bounded from above by a constant depending only on $d$ and $\gamma$. This implies that, for each $\gamma \in (0,1)$, there exist an exponent $\delta := \delta(d , \lambda) > 0$ and a constant $C := C(d , \lambda, \gamma) < \infty$ such that
\begin{equation*}
\left( \frac{1}{\left|\gamma \cu_n \right|} \sum_{e \subseteq  \gamma \cu_n} \E [|\nabla \psi_{n,q}(x)|^2]^{1+ \delta} \right)^\frac1{1+\delta} \leq  \frac{C}{|\cu_n|} \sum_{x \in \cu_n} \E [|\nabla \psi_{n,q}(x)|^2]  + C.
\end{equation*}
\end{proof}

\begin{remark}
Combining the Meyers estimate with~\eqref{4.L2boundnustar} of Proposition~\ref{propofnunustar}, one obtains
\begin{equation*}
\left( \frac{1}{\left|\gamma \cu_n \right|} \sum_{e \subseteq  \gamma \cu_n} \E \left[|\nabla \psi_{n,q}(e)|^2\right]^{1+ \delta} \right)^\frac1{1+\delta} \leq  C (1 + |q|^2).
\end{equation*}
\end{remark}

\small
\bibliographystyle{abbrv}
\bibliography{holes}

\end{document}